\documentclass[a4paper]{amsart}
\usepackage{amssymb, amsthm, diagrams, enumerate, mathtools, mathrsfs, stmaryrd}
\usepackage[headings]{fullpage}
\usepackage[colorlinks, pagebackref]{hyperref}

\date{Submitted February 7, 2013; revised \today}

\newcommand{\into}{\hookrightarrow}
\newcommand{\fp}{\mathfrak{p}}

\newcommand{\fz}{\mathfrak{z}}

\newcommand{\bff}{\boldsymbol{f}}
\newcommand{\bfg}{\boldsymbol{g}}
\newcommand{\bfz}{\boldsymbol{\mathrm{z}}}

\newcommand{\tbt}[4]{\begin{pmatrix}#1 & #2 \\ #3 & #4\end{pmatrix}}
\newcommand{\stbt}[4]{\left(\begin{smallmatrix}#1 & #2 \\ #3 & #4\end{smallmatrix}\right)}
\newcommand{\quot}[1]{\text{``}{#1}\text{''}}

\renewcommand{\AA}{\mathbb{A}}
\newcommand{\QQ}{\mathbb{Q}}
\newcommand{\FF}{\mathbb{F}}
\newcommand{\RR}{\mathbb{R}}
\newcommand{\CC}{\mathbb{C}}
\newcommand{\ZZ}{\mathbb{Z}}
\newcommand{\Zp}{\ZZ_p}
\newcommand{\Qp}{\QQ_p}
\newcommand{\QQbar}{\overline{\QQ}}
\newcommand{\dR}{\mathrm{dR}}
\newcommand{\cris}{\mathrm{cris}}
\newcommand{\syn}{\mathrm{syn}}
\newcommand{\et}{\text{\textrm{\'et}}}
\newcommand{\Betti}{\mathrm{Betti}}
\newcommand{\cO}{\mathcal{O}}
\newcommand{\cY}{\mathcal{Y}}
\newcommand{\cH}{\mathcal{H}}

\newcommand{\Dcris}{\mathbb{D}_{\cris}}
\newcommand{\vp}{\varphi}
\newcommand{\cX}{\mathcal{X}}
\newcommand{\cA}{\mathcal{A}}
\newcommand{\cC}{\mathcal{C}}
\newcommand{\cE}{\mathcal{E}}
\newcommand{\cZ}{\mathcal{Z}}
\newcommand{\cS}{\mathcal{S}}
\newcommand{\BB}{\mathbb{B}}

\renewcommand{\d}{\,\mathrm{d}}

\DeclareMathOperator{\ur}{ur}

\DeclareMathOperator{\res}{res}

\DeclareMathOperator{\cores}{cores}
\DeclareMathOperator{\eord}{e_{ord}}
\newcommand{\esord}{\operatorname{e}'_{\mathrm{ord}}}
\DeclareMathOperator{\SL}{SL}
\DeclareMathOperator{\GL}{GL}
\DeclareMathOperator{\PGL}{PGL}
\DeclareMathOperator{\Var}{Var}
\DeclareMathOperator{\trace}{trace}
\DeclareMathOperator{\Sm}{Sm}
\DeclareMathOperator{\Gerst}{Gerst}
\DeclareMathOperator{\Frob}{Frob}
\DeclareMathOperator{\Ind}{Ind}

\DeclareMathOperator{\tr}{tr}

\DeclareMathOperator{\cont}{cont}
\DeclareMathOperator{\Hom}{Hom}

\DeclareMathOperator{\Prime}{prime}

\DeclareMathOperator{\CH}{CH}
\DeclareMathOperator{\ord}{ord}

\DeclareMathOperator{\id}{id}
\DeclareMathOperator{\Div}{Div}
\DeclareMathOperator{\reg}{reg}
\DeclareMathOperator{\Iw}{Iw}
\DeclareMathOperator{\Fil}{Fil}
\DeclareMathOperator{\End}{End}
\DeclareMathOperator{\Hyp}{Hyp}
\DeclareMathOperator{\Gal}{Gal}
\DeclareMathOperator{\Spec}{Spec}
\DeclareMathOperator{\pr}{pr}

\DeclareMathOperator{\norm}{norm}
\newcommand{\ah}{\mathrm{ah}}
\newcommand{\nh}{\mathrm{nh}}

\newcommand{\htimes}{\mathop{\widehat\otimes}}

\newtheorem{theorem}{Theorem}[subsection]
\newtheorem{lemma}[theorem]{Lemma}
\newtheorem{proposition}[theorem]{Proposition}
\newtheorem{corollary}[theorem]{Corollary}
\newtheorem{assumption}[theorem]{Assumption}
\newtheorem{conjecture}[theorem]{Conjecture}
\newtheorem{definition}[theorem]{Definition}

\theoremstyle{remark}
\newtheorem{remark}[theorem]{Remark}
\newtheorem{note}[theorem]{Note}

\newtheorem*{hypothesis}{Hypothesis}
\newtheorem*{notation}{Notation}

\begin{document}

 \begin{abstract}
  We construct an Euler system in the cohomology of the tensor product of the Galois representations attached to two modular forms, using elements in the higher Chow groups of products of modular curves. We use these elements to prove a finiteness theorem for the strict Selmer group of the Galois representation when the associated $p$-adic Rankin--Selberg $L$-function is non-vanishing at $s = 1$.
 \end{abstract}

\title{Euler systems for Rankin--Selberg convolutions of modular forms}

\author[A.~Lei]{Antonio Lei}
\address[Lei]{Department of Mathematics and Statistics\\
Burnside Hall, McGill University\\
Montreal, QC, Canada H3A 2K6}
\email{antonio.lei@mcgill.ca}

\author[D.~Loeffler]{David Loeffler}
\address[Loeffler]{Mathematics Institute\\
Zeeman Building, University of Warwick\\
Coventry CV4 7AL, UK}
\email{d.a.loeffler@warwick.ac.uk}

\author[S.L.~Zerbes]{Sarah Livia Zerbes}
\address[Zerbes]{Department of Mathematics \\
University College London\\
Gower Street, London WC1E 6BT, UK}
\email{s.zerbes@ucl.ac.uk}
\thanks{The authors' research is supported by the following grants: CRM-ISM Postdoctoral Fellowship (Lei); Royal Society University Research Fellowship (Loeffler); EPSRC First Grant EP/J018716/1 (Zerbes).}

\setcounter{tocdepth}{1}
\setcounter{secnumdepth}{4}
\hypersetup{bookmarksdepth=4}

 \maketitle

 \begin{center}
  \emph{Dedicated to Kazuya Kato}
 \end{center}

 \vspace{5ex}

 \tableofcontents

 \section{Outline}

  In \cite{BDR12}, Bertolini, Darmon and Rotger have studied certain canonical global cohomology classes (the ``Beilinson--Flach elements'', obtained from the constructions of \cite{beilinson84} and \cite{flach92}) in the cohomology of the tensor products of the $p$-adic Galois representations of pairs of weight 2 modular forms, and related their image under the Bloch--Kato logarithm maps to the values of $p$-adic Rankin--Selberg $L$-functions. These Beilinson--Flach elements are constructed as the image of elements of the higher Chow group of a product of modular curves.

  In this paper, we construct a form of Euler system -- a compatible system of cohomology classes over cyclotomic fields -- of which the Beilinson--Flach elements are the bottom layer. We first define elements of higher Chow groups of the product of two (affine) modular curves over a cyclotomic field,
  \[ {}_c \Xi_{m, N, j} \in \CH^2(Y_1(N)^2 \otimes \QQ(\mu_m), 1)\]
  for integers $m \ge 1$, $N \ge 5$, and $j \in \ZZ / m\ZZ$ (cf. Definition \ref{def:BFelts}).  These are obtained by considering the images of various maps from higher level modular curves to the surface $Y_1(N)^2$, together with modular units (Siegel units) on these curves. For $m = 1$ our elements reduce to those considered in \cite{BDR12}, and as in \emph{op.cit.}, we show that after tensoring with $\QQ$ we can construct preimages of our elements in the higher Chow group of the self-product of the projective modular curve $X_1(N)$; however, in this paper (as in \cite{kato04}) we shall take the affine versions as the principal objects of study.

  We show two forms of compatibility relation for our generalized Beilinson--Flach elements: firstly, relating ${}_c \Xi_{m, N, j}$ to the pushforward of ${}_c \Xi_{m, Np, j}$, for $p$ prime (Theorem \ref{thm:firstnormrelation}); secondly, relating ${}_c \Xi_{m, N, j}$ to the pushforward (or Galois norm) of ${}_c \Xi_{mp, N, j}$, where $p$ is prime and either $p \mid N$ (Theorem \ref{thm:secondnormbadprime}) or $p \nmid mN$ (Theorem \ref{thm:secondnormrelationprime}).

  We next turn to the relation between our elements and $L$-values. Theorem \ref{thm:beilinsonreg} shows, following an argument due to Beilinson, that the images of the elements ${}_c \Xi_{m, N, j}$ under the Beilinson regulator map into complex de Rham cohomology are related to the derivatives at $s = 1$ of Rankin--Selberg $L$-functions of weight 2 modular forms. Theorem \ref{thm:syntomicreg} is a $p$-adic analogue of this result, generalizing a theorem of Bertolini--Darmon--Rotger \cite{BDR12}; it gives a formula for the image of our element for $m = 1$ under the $p$-adic syntomic regulator, for a prime $p \nmid N$, in terms of Hida's $p$-adic Rankin--Selberg $L$-functions.

  Next we consider the images of our elements in \'etale cohomology. Applying Huber's ``continuous \'etale realization'' functor and the Hochschild--Serre exact sequence, and projecting into the isotypical component corresponding to a pair of eigenforms $(f, g)$ of level $N$, allows us to construct Galois cohomology classes
  \[ {}_c \bfz_m^{(f, g, N)} \in H^1(\QQ(\mu_m), V_f^* \otimes V_g^*)\]
  from the elements ${}_c \Xi_{m, N, j}$; see Definition \ref{def:cohoclasses}. Using the second norm relation in the $p$-adic cyclotomic tower, we can modify these to construct elements of Iwasawa cohomology groups of pairs of modular forms (under a strong ``ordinarity'' hypothesis; this is Theorem \ref{thm:ordinaryEulersystem}), or of the tensor product of Iwasawa cohomology groups with the algebra of distributions (under a weaker ``small slope'' hypothesis; see Theorem \ref{thm:Eulersystem}). These elements satisfy compatibility relations of Euler-system type when additional primes are added to $m$.

  Using the first norm relation, we also obtain variation in Hida families. More specifically, if one of the two forms (say $g$) is ordinary, we may deform our cohomology classes analytically as $g$ varies over a Hida family; cf.~Theorem \ref{thm:onevarhida}. (In the special case $m = 1$, such results have been independently obtained by Bertolini--Darmon--Rotger.) When $f$ and $g$ are both ordinary, we obtain three-variable families which also incorporate cyclotomic twists, cf.~Theorem \ref{thm:threevarhida}.

  As an application of these constructions, we prove (under some technical hypotheses) a finiteness theorem for the strict Selmer group of a product of modular forms (Theorem \ref{thm:ourfiniteSel}) when the associated $p$-adic Rankin--Selberg $L$-function is non-vanishing at $s = 1$, and (under very slightly stronger hypotheses) we give an explicit bound for the order of the strict Selmer group in terms of the $p$-adic $L$-value (Theorem \ref{thm:ourboundedSel}).

  \subsection*{Remark} It is a pleasure to acknowledge the very deep debt this article owes to the magisterial work of Kato \cite{kato04}. We have adopted many aspects of the strategy and methods of Kato's work, reused a number of his results, and even in many cases adopted his notation. It is a pleasure to dedicate this work to Professor Kato, as a humble gift on the occasion of his 60th birthday.

 \subsection*{Acknowledgements} Part of this work was done while the second and the third author were visiting Montr\'eal and Bielefeld in Spring 2012; they would like to thank Henri Darmon and Thomas Zink for their hospitality. We would also like to thank Massimo Bertolini, Francois Brunault, Francesc Castella, John Coates, Henri Darmon, Mathias Flach, Kazuya Kato, Masato Kurihara, Andreas Langer, Jan Nekovar, Ken Ribet, Victor Rotger, Karl Rubin, Tony Scholl, Xin Wan and Andrew Wiles for helpful comments. We are also grateful to the anonymous referee for a number of valuable suggestions and corrections.


 \section{Generalized Beilinson--Flach elements}

  In this section we shall construct elements of motivic cohomology groups of products of modular curves, using the explicit description of the motivic cohomology given by the Gersten complex.


  \subsection{Modular curves}
   \label{sect:modcurves}

   We begin by fixing notation regarding modular curves. We follow the conventions of \cite{kato04} very closely (see in particular \S\S 1, 2, and 5 of \emph{op.cit.}).

   \begin{definition}
    \label{def:yn}
    For $N \ge 3$, let $Y(N)$ denote the smooth affine curve over $\QQ$ which represents the functor on the category of $\QQ$-schemes
    \[ S \mapsto \left\{ \begin{array}{cc}\text{isomorphism classes of triples $(E, e_1, e_2)$,} \\ \text{$E$ an elliptic curve over $S$,} \\ \text{$e_1$, $e_2$ sections of $E/S$ generating $E[N]$.}\end{array}\right\}\]
   \end{definition}

   The variety $Y(N)$ is equipped with a left action of $\GL_2(\ZZ / N\ZZ)$: the element $\tbt a b c d$ maps $(E, e_1, e_2)$ to $(E, e_1', e_2')$ where
   \[ \begin{pmatrix} e_1' \\ e_2' \end{pmatrix} = \tbt a b c d \begin{pmatrix} e_1 \\ e_2 \end{pmatrix}.\]
   In particular, this action factors through $\GL_2(\ZZ / N\ZZ) / \langle \pm 1 \rangle$.

   There is an obvious surjective morphism $Y(N) \to \mu_N^\circ$, where $\mu_N^\circ$ is the scheme of primitive $N$-th roots of unity, which sends $(E, e_1, e_2)$ to $\langle e_1, e_2 \rangle_{E[N]}$, where $\langle-,-\rangle_{E[N]}$ denotes the Weil pairing on $E[N]$. Because the Weil pairing is non-degenerate and alternating, the induced action of $\GL_2(\ZZ / N\ZZ)$ on $\mu_N^\circ$ is given by $\sigma \cdot \zeta = \zeta^{\det \sigma}$; and the fibre of $Y(N)(\CC)$ over the point $e^{2\pi i / N} \in \mu_N^\circ(\CC)$ is canonically and $\SL_2(\ZZ / N\ZZ)$-equivariantly identified with $\Gamma(N) \backslash \cH$, where $\cH$ is the upper half-plane and $\Gamma(N)$ the principal congruence subgroup of level $N$ in $\SL_2(\ZZ)$, via the map
   \[ \tau \mapsto \left( \CC / (\ZZ + \ZZ \tau), \tau/N, 1/N\right).\]

   We shall mainly be working with certain quotients of the curves $Y(N)$, which we now define.

   \begin{definition}
    For $M, N \ge 1$, we shall define $Y(M, N)$ to be the quotient of $Y(L)$, for any $L \ge 3$ divisible by $M$ and $N$, by the group
    \[ \left\{ \tbt a b c d  \in \GL_2(\ZZ / L\ZZ): \begin{array}{ll} a = 1, b = 0 \bmod M, \\ c = 0, d = 1 \bmod{N}\end{array} \right\}.\]
   \end{definition}

   This curve $Y(M, N)$ represents the functor of triples $(E, e_1, e_2)$ where $e_1$ has order $M$, $e_2$ has order $N$, and $e_1, e_2$ generate a subgroup of $E$ of order $MN$.

   \begin{definition}
    We write $Y_1(N)$ for the smooth affine curve over $\QQ$ representing the functor
    \[ S \mapsto \left\{ \begin{array}{cc}\text{isomorphism classes of pairs $(E, P)$,} \\ \text{$E$ an elliptic curve over $S$,} \\ \text{$P$ a section of $E/S$ of exact order $N$.}\end{array}\right\}\]
   \end{definition}

   \begin{remark}
    Note that the cusp $\infty$, which corresponds to the generalized elliptic curve $\left(\mathbb{G}_m / q^{\ZZ}, \zeta_N\right)$, is not defined over $\QQ[[q]]$ but rather over $\QQ(\mu_N)[[q]]$, so the $q$-expansions of elements of $\cO(Y_1(N))$ do not necessarily lie in $\QQ((q))$ but rather in $\QQ(\mu_N)((q))$. See e.g.~\cite[\S 9.3]{diamondim95} for further discussion.
   \end{remark}

   It is clear that $Y_1(N) = Y(1, N)$. More generally, we may use the following proposition to identify $Y_1(N) \times \mu_m^\circ$, for $m, N \ge 1$, with a quotient of a principal modular curve:

   \begin{proposition}
    If $N \ge 3, m \ge 1$, and $L \ge 3$ is divisible by both $N$ and $m$, then the map
    \[\begin{array}{rcl}
     Y(L) & \rTo & Y_1(N) \otimes \mu_m^\circ\\
     (E, e_1, e_2) & \rMapsto & \left[ \left(E, \tfrac L N e_2\right), \left\langle \tfrac{L}{m} e_1, \tfrac{L}{m} e_2 \right\rangle_{E[m]} \right]
    \end{array}\]
    identifies $Y_1(N) \times \mu_m^\circ$ with the quotient of $Y(L)$ by the subgroup of $\GL_2(\ZZ / L \ZZ)$ given by
    \[ \left\{ \tbt a b c d : \begin{array}{ll} c = 0, d = 1 \bmod N, \\ ad-bc = 1 \bmod m\end{array} \right\}.\]
   \end{proposition}

   We shall be most interested in the curves $Y(m, mN)$ for $m \ge 1, N \ge 1$. Note that $Y(m, mN)$ maps naturally to $\mu_m^\circ$, with geometrically connected fibres. It has a left action of the group
   \[ \left\{ \tbt a b c d: c = 0 \bmod N\right\},\]
   compatible with the determinant action on $\mu_m^\circ$: if $x = (E, e_1, e_2)$ is a point of $Y(m, mN)$, so $e_1$ has order $m$ and $e_2$ has order $mN$, and $g = \tbt a b c d \in \GL_2(\ZZ / mN \ZZ)$ with $N | c$, then
   \[ g \cdot x = (E, a e_1 + b N e_2, c/N e_1 + d e_2).\]

   We shall introduce some notation for maps between these curves.

   \begin{definition}
    \label{def:tm}
    Let $m, N \ge 1$.
    \begin{enumerate}
     \item We write $t_{m}$ for the morphism $Y(m, mN) \to Y_1(N) \times \mu_m^\circ$
      given by
     \[ (E, e_1, e_2) \mapsto \left[\left(E/\langle e_1 \rangle, [m e_2]\right), \langle e_1, N e_2\rangle_{E[m]}\right].\]
     \item For $a \ge 1$, we write $\tau_{a}$ for the morphism $Y(am, amN) \to Y(m, mN)$ given by
     \[ (E, e_1, e_2) \mapsto \left(E/C, [e_1], [a e_2]\right)\]
     where $C$ is the cyclic subgroup of order $a$ generated by $m e_1$, and $[e_1]$, $[a e_2]$ denote the images of $e_1$ and $ae_2$ on $E/C$.
    \end{enumerate}
   \end{definition}

   \begin{proposition}
    \label{prop:tmproperties}
    Let $m, N, a$ as above.
    \begin{enumerate}
     \item We have a commutative diagram
     \begin{diagram}
      Y(am, amN) & \rTo^{\tau_a} & Y(m, mN) \\
      \dTo^{t_{am}} & & \dTo^{t_{m}}\\
      Y_1(N) \times \mu_{am}^\circ & \rTo & Y_1(N) \times \mu_m^\circ,
     \end{diagram}
     where the bottom horizontal arrow is the identity map on $Y_1(N)$ and the map $\mu_{am}^\circ \to \mu_m^\circ$ given by $\zeta \mapsto \zeta^a$.
     \item For $b \in (\ZZ / mN\ZZ)^\times$, the map $t_{m}$ intertwines the action of $\stbt b 0 0 1$ with the automorphism $\sigma_b: \zeta \mapsto \zeta^b$ of $\mu_m^\circ$, and $\stbt {b^{-1}} 0 0 b$ with the diamond operator $\langle b \rangle$ on $Y_1(N)$.
    \end{enumerate}
   \end{proposition}

   \begin{proof}
    The first statement is immediate from the definition of the maps and properties of the Weil pairing, and the second is an easy verification (cf.~\cite[5.7.1]{kato04}).
   \end{proof}

   \begin{remark}
    The use of the maps $\tau_a$ is forced on us by the nature of our construction of zeta elements. It would be much more satisfying to use the natural degeneracy maps $Y(am, amN) \to Y(m, mN)$ given by $\tau_a': (E, e_1, e_2) \mapsto (E, ae_1, ae_2)$, but we do not know how to construct elements compatible under these maps; see \S \ref{sect:dreams} below.
   \end{remark}

   \begin{notation}
    For compatibility with \cite{kato04}, we shall use the alternative notation $Y_1(N) \otimes \QQ(\mu_m)$ for $Y_1(N) \times \mu_m^\circ$.
   \end{notation}

   We shall also have to deal with products of two modular curves.

   \begin{definition}
    We shall write $Y(N)^2$ (slightly abusively) for the fibre product \(Y(N) \times_{\mu_N^\circ} Y(N)\).
    This is a subvariety of $Y(N) \times_{\Spec(\QQ)} Y(N)$ preserved by the subgroup
    \[ \left\{ (\sigma, \tau) \in \GL_2(\ZZ / N \ZZ)^2 : \det(\sigma) = \det(\tau)\right\}.\]
    Similarly, we shall write $Y(m, mN)^2$ for \( Y(m, mN) \times_{\mu_m^\circ} Y(m, mN) \), which is acted upon by the group
    \[ G = \left\{ (\sigma, \tau) \in \GL_2(\ZZ / mN \ZZ)^2 : \det(\sigma) = \det(\tau) \bmod m,\quad  \sigma, \tau = \tbt * * 0 * \bmod N.\right\}\]
   \end{definition}

   Evidently, the image of $Y(m, mN) \times_{\mu_m^\circ} Y(m, mN)$ under $t_m \times t_m$ lands in
   \[ \left(Y_1(N) \times \mu_m^\circ\right) \times_{\mu_m^\circ} \left(Y_1(N) \times \mu_m^\circ\right) = Y_1(N)^2 \times \mu_m^\circ,\]
   so we may consider $t_m \times t_m$ as a morphism $Y(m, mN)^2 \to Y_1(N)^2 \times \mu_m^\circ$, which intertwines the action of $\left( \stbt b 0 0 1,\stbt b 0 0 1\right) \in G$ with $\sigma_b$.

   We shall also have to consider, occasionally, some more general classes of modular curves. Here we shall only consider models over $\QQbar$.

   \begin{definition}
    \label{def:qqbarmodel}
    If $\Gamma \subseteq \SL_2(\ZZ)$ is a congruence subgroup, then we shall write $Y(\Gamma)$ for the variety
    \[ \left(\Gamma \backslash Y(L)\right) \otimes_{\QQ(\mu_L)} \QQbar,\]
    where $L$ is any integer $\ge 3$ such that $\Gamma \supseteq \Gamma(L)$; this variety is independent of the choice of $L$.
   \end{definition}

   \begin{remark}
    If $\alpha \in \GL_2^+(\QQ)$, then the isomorphism of Riemann surfaces $Y(\Gamma)(\CC) \cong Y(\alpha \Gamma \alpha^{-1})(\CC)$ mapping $\tau \in \cH$ to $\alpha \tau$ extends to an algebraic isomorphism defined over $\QQbar$; similarly for degeneracy maps $Y(\Gamma) \to Y(\Gamma')$ for $\Gamma \subseteq \Gamma'$.
   \end{remark}

   The above constructions with affine modular curves also have projective analogues where the cusps are taken into account; we shall write $X(-)$ for the compactified version of $Y(-)$ in line with the standard notation.


  \subsection{Siegel units}
   \label{sect:siegelunits}

   \begin{definition}
    For $(\alpha, \beta) \in (\QQ / \ZZ)^2 - \{(0, 0)\}$ of order dividing $N$, and $c > 1$ coprime to $6N$, let ${}_c g_{\alpha, \beta} \in \cO(Y(N))^\times$ denote Kato's Siegel unit, as defined in \cite[\S 1.4]{kato04}.
   \end{definition}

   We identify ${}_c g_{\alpha, \beta}$ with a holomorphic function on the upper half-plane, via the identification of the fibre of $Y(N)(\CC)$ over $e^{2\pi i / N} \in \mu_N^\circ(\CC)$ with $\Gamma(N) \backslash \cH$ given in the previous section. (Note that ${}_c g_{\alpha, \beta}$ is defined over $\QQ$ as a function on $Y(N)$, but in order to interpret it as a holomorphic function on $\cH$ we must make a choice of $N$-th root of unity, and the $q$-expansion coefficients of ${}_c g_{\alpha, \beta}$ are in $\QQ(\mu_N)$.)

   Recall that there is an element $g_{\alpha, \beta} \in \cO(Y(N))^\times \otimes \QQ$ such that ${}_c g_{\alpha, \beta} = c^2 g_{\alpha, \beta} - g_{c\alpha, c\beta}$.

   \begin{proposition}[Distribution relations]
    \label{prop:distrelations}
    Let $m \ge 1$, and let $c$ be a nonzero integer coprime to $6m$ and the orders of $\alpha, \beta$. Then the following relations hold:
    \begin{subequations}
     \begin{equation}
      \label{eq:dist1}
      {}_c g_{\alpha, \beta}(mz) = \prod_{\beta'} {}_c g_{\alpha, \beta'}(z)
     \end{equation}
     where the product is over $\beta' \in \QQ / \ZZ$ such that $m\beta' = \beta$;
     \begin{equation}
      \label{eq:dist2}
      {}_c g_{\alpha, \beta}(z/m) = \prod_{\alpha'} {}_c g_{\alpha', \beta}(z)
     \end{equation}
     where the product is over $\alpha' \in \QQ / \ZZ$ such that $m\alpha' = \alpha$; and
     \begin{equation}
      \label{eq:dist3}
      {}_c g_{\alpha, \beta}(z) = \prod_{\alpha', \beta'} {}_c g_{\alpha', \beta'}(z)
     \end{equation}
     where the product is over pairs $(\alpha', \beta') \in (\QQ / \ZZ)^2$ such that $(m\alpha', m \beta') = (\alpha, \beta)$.
    \end{subequations}
   \end{proposition}

   \begin{proof}
    Formula \eqref{eq:dist1} is Lemma 2.12 of \cite{kato04}. Formula \eqref{eq:dist2} can be proved similarly, or can be deduced directly from \eqref{eq:dist1} using the action of $\begin{pmatrix} 0 & -1 \\ 1 & 0 \end{pmatrix}$. Formula \eqref{eq:dist3}, which is Lemma 1.7(2) of \emph{op.cit.}, is immediate by combining (1) and (2).
   \end{proof}

   \begin{remark}
    The three formulae above admit the following common generalization: let $M$ be a $2 \times 2$ integer matrix with positive determinant $D$. Then we have
    \[ {}_c g_{\alpha, \beta}(M \cdot z) = \prod_{\alpha', \beta'} {}_c g_{\alpha', \beta'}\]
    where the product is over all $(\alpha', \beta')$ such that $(\alpha', \beta') M' = (\alpha, \beta)$, where $M' = (\det M) M^{-1}$ is the adjugate matrix of $M$. Cases (1), (2) and (3) correspond to taking $M = \begin{pmatrix} m & 0 \\ 0 & 1\end{pmatrix}$, $\begin{pmatrix} 1 & 0 \\ 0 & m\end{pmatrix}$ and $\begin{pmatrix} m & 0 \\ 0 & m\end{pmatrix}$ respectively. The case where $M$ is invertible is (part of) Lemma 1.7(1) of \emph{op.cit.}.
   \end{remark}

   We are most interested in the units ${}_c g_{0, 1/N}$, which descend to units on $Y_1(N)$. These have the following compatibility property:

   \begin{theorem}[Kato]\label{thm:siegel-compat}
    If $M, N, N' \ge 1$ are integers with $\Prime(N') = \Prime(N)$, and $\alpha$ is the natural projection $Y(M, N') \to Y(M, N)$ (which induces a norm map $\alpha_*: \cO(Y(M, N'))^\times \to \cO(Y(M, N))^\times$), then we have
    \[ \alpha_*({}_c g_{0, 1/N'}) = {}_c g_{0, 1/N}.\]

    If $N' = N \ell$, where $\ell$ is prime and $\ell \nmid MN$, then we have
    \[ \alpha_*({}_c g_{0, 1/N'}) = {}_cg_{0,1/N} \cdot \left({}_c g_{0,\quot{\ell^{-1}}/N}\right)^{-1}\]
    where $\quot{\ell^{-1}}$ signifies the inverse of $\ell$ modulo $N$.
   \end{theorem}

   \begin{proof}
    These statements are proved in \cite[\S 2.11, 2.13]{kato04} (in the course of proving the norm-compatibility relation for Kato's elements of $K_2$, Propositions 2.3 and 2.4 of \emph{op.cit.}). We reproduce the proofs briefly here.

    Firstly, let us suppose $\Prime(mN') = \Prime(mN)$. Let $a = N' / N$. Since $\Prime(mN') = \Prime(mN)$, for each $(x, y) \in (\ZZ / a\ZZ)^2$ we may choose an element $s_{xy} \in \GL_2(\ZZ / mN' \ZZ)$ of the form
    \[ \begin{pmatrix} 1 & 0 \\ mNx & 1 + mN y \end{pmatrix}.\]
    These elements $s_{xy}$ are coset representatives for the quotient of the two subgroups of $\GL_2(\ZZ / mN'\ZZ)$ corresponding to $Y(m,mN)$ and $Y(m, mN')$, so we have
    \[ \alpha_* \left({}_c g_{0, 1/N'}\right) = \prod_{x, y} s_{xy}^* \left({}_c g_{0, 1/N'}\right).\]

    For any $M \ge 1$, any $u \in \GL_2(\ZZ /M \ZZ)$ and any $\alpha, \beta \in \left(\tfrac1M \ZZ / \ZZ\right)$, we have
    \[ u^* \left({}_c g_{\alpha, \beta}\right) = {}_c g_{\alpha', \beta'}\]
    where
    \[ (\alpha', \beta') = (\alpha, \beta) \cdot u;\]
    applying this to the formula above we deduce that
    \[ \alpha_* \left({}_c g_{0, 1/mN'}\right) = \prod_{x, y \in \ZZ/a\ZZ} \left({}_c g_{x/a, 1/mN' + y/a}\right).\]
    The latter expression is equal to the product of ${}_c g_{\gamma, \delta}$ over all pairs $(\gamma, \delta)$ such that $(a\gamma, a\delta) = (0, 1/mN)$; so using the distribution property of Equation \eqref{eq:dist1}, the product is ${}_c g_{0, 1/mN}$ as required.

    In the second case, where $N' = \ell N$ for $\ell \nmid MN$, we pass via the intermediate modular curves $Y(M, N(\ell))$ and $Y(M(\ell), N)$ described in \cite[\S 2.8]{kato04}. Let $\vp_\ell:Y(M, N(\ell))\to Y(M(\ell), N)$ be the map defined in \emph{op.cit.}, corresponding to $z \mapsto \ell z$ on $\cH$. We factor the projection $\alpha$ as $\alpha_1 \circ \alpha_2$, where $\alpha_1$ and $\alpha_2$ are the natural maps
    \[ Y(M, N\ell) \rTo^{\alpha_2} Y(M, N(\ell)) \rTo^{\alpha_1} Y(M, N).\]

    By \cite[Step 2 of \S2.13 and (2.13.2)]{kato04}, we have
    \begin{align*}
     (\alpha_2)_*\left({}_cg_{0,1/N\ell}\right) &= \vp_\ell^*\left({}_cg_{0,1/N}\right) \cdot \left({}_cg_{0,\quot{\ell^{-1}}/N}\right)^{-1};\\
     (\alpha_1)_*\vp_\ell^*(_cg_{0,1/N}) &= {}_cg_{0,1/N}\cdot({}_c g_{0,\quot{\ell^{-1}}/N})^{\ell};\\
     (\alpha_1)_*\left( {}_cg_{0,\quot{\ell^{-1}}/N}\right) &= \left( {}_cg_{0,\quot{\ell^{-1}}/N}\right)^{\ell + 1}
    \end{align*}
    (the last formula owing to the fact that the degree of $\alpha_1$ is $\ell + 1$). Hence, on combining these three equations, we obtain
    \begin{align*}
     \alpha_*({}_cg_{0,1/N\ell}) &= (\alpha_1)_* (\alpha_2)_* \left({}_cg_{0,1/N\ell}\right) \\
     &= (\alpha_1)_* \left[\vp_\ell^*\left({}_cg_{0,1/N}\right) \cdot \left({}_cg_{0,\quot{\ell^{-1}}/N}\right)^{-1}\right]\\
     &= \Big({}_cg_{0,1/N}\cdot({}_c g_{0,\quot{\ell^{-1}}/N})^{\ell}\Big) \cdot \left(\left({}_c g_{0,\quot{\ell^{-1}}/N}\right)^{-1}\right)^{\ell + 1} \\
     &= {}_cg_{0,1/N} \cdot \left({}_c g_{0,\quot{\ell^{-1}} / N}\right)^{-1}.
    \end{align*}
   \end{proof}

   \begin{remark}
    In the above proposition we excluded from consideration the case when $N' = N\ell$ where $\ell \mid M$ (but $\ell \nmid N$). This case can also be treated using Kato's methods, or deduced directly from Step 1 of \S 2.13 of \emph{op.cit.} by applying the element $\tbt 0 {-1} 1 0$, and one finds that in this case we have
    \[ \alpha_* \left({}_c g_{0, 1/N\ell}\right) = {}_c g_{0, 1/N} \cdot (\vp_\ell^{-1})^* \left({}_c g_{0, \quot{\ell^{-1}} / N}\right).\]
    However we shall solely be working with modular curves of the form $Y(m, mN)$, so we will not need this formula.
   \end{remark}


  \subsection{Integral models of modular curves}

   The following theorem is well-known:

   \begin{theorem}[Igusa]
    There exists a smooth scheme $\cY(N)$ over $\ZZ[\zeta_N, 1/N]$, representing the functor of Definition \ref{def:yn} on the category of $\ZZ[1/N]$-schemes.
   \end{theorem}

   For a sketch of the proof, see e.g.~\cite{delignerapoport73}.

   \begin{proposition}
    The Siegel units ${}_c g_{\alpha, \beta}$, for all $(\alpha, \beta) \in (\tfrac1N \ZZ / \ZZ)^2 - \{(0, 0)\}$, are elements of $\cO(\cY(N))^\times$.
   \end{proposition}

   \begin{proof}
    As shown in \cite[Prop 1.3]{kato04}, given an arbitrary scheme $S$, an elliptic curve $E/S$, and an integer $c > 1$ coprime to 6, there exists a canonical element ${}_c\theta_E \in \cO(E - E[c])^\times$ whose divisor is $c^2 (0) - E[c]$. As noted in \cite[\S 1.3]{scholl98}, if the base $S$ is integral and $E$ has a torsion section $x : S \to E$ of order $N$, where $N > 1$ is coprime to $c$ and either $N$ is invertible on $S$ or $N$ has at least two prime factors, then $x^* {}_c \theta_E \in \cO(S)^\times$. Applying this with $S = \cY(N)$, $E$ the universal elliptic curve over $S$, and $x$ the section $a e_1 + b e_2$ where $(\alpha, \beta) = (a/N, b/N)$, we deduce that ${}_c g_{\alpha, \beta}$ extends from $Y(N)$ to a unit on the integral model $\cY(N)$.
   \end{proof}

   \begin{remark}
    By passage to the quotient we also see that for any $b \in \ZZ/N\ZZ$, $b \ne 0$, the Siegel unit ${}_c g_{0, b/N}$ is a unit on the canonical $\ZZ[1/N]$-model $\cY_1(N)$ of $Y_1(N)$.
   \end{remark}


  \subsection{Hecke correspondences}
   \label{sect:heckecorr}

   We now recall how elements of the Hecke algebra can be interpreted as correspondences between modular curves, or, equivalently, as 1-cycles on a product of two modular curves.

   \begin{lemma}
    \label{lem:birational}
    Let $\alpha \in \GL_2^+(\QQ)$ and $\Gamma_1, \Gamma_2$ finite-index subgroups of $\SL_2(\ZZ)$. Then there is a unique morphism of varieties over $\overline{\QQ}$,
    \[ \sigma: Y(\Gamma_1 \cap \alpha^{-1} \Gamma_2 \alpha) \to Y(\Gamma_1) \times Y(\Gamma_2),\]
    such that the diagram
    \begin{diagram}
     \cH & \rTo^{1 \times \alpha} & \cH \times \cH\\
     \dTo & & \dTo\\
     Y(\Gamma_1 \cap \alpha^{-1} \Gamma_2 \alpha)(\CC) & \rTo^\sigma & (Y(\Gamma_1) \times Y(\Gamma_2))(\CC)
    \end{diagram}
    commutes (where the vertical arrows are the natural projection maps). The image of $\sigma$ is an irreducible closed subvariety of $Y(\Gamma_1) \times Y(\Gamma_2)$, and the map $\sigma$ is a birational equivalence onto its image.
   \end{lemma}

   \begin{proof}
    After Definition \ref{def:qqbarmodel} and the remarks following, the only assertion that needs checking is that $\sigma$ is birational. However, by Proposition \ref{prop:generically injective} in the appendix (applied to the subgroups $\Gamma_1$ and $\alpha^{-1} \Gamma_2 \alpha$) we know that $\sigma$ is injective away from a finite set.
   \end{proof}

   \begin{remark}
    This proposition is well known in the special case $\Gamma = \SL_2(\ZZ)$ and $\alpha = \tbt{p}{0}{0}{1}$ for a prime $p$, where it shows that $Y_0(p)$ is the normalization of the subvariety of $\AA^2$ cut out by the classical modular equation of level $p$; see e.g.~\cite[\S VI.6]{delignerapoport73}.
   \end{remark}

   \begin{lemma}
    \label{lem:distinctcurves}
    Let $\Gamma, \Gamma'$ be as above, let $\alpha_1, \alpha_2 \in \GL_2^+(\QQ)$, and for $i = 1, 2$ let $C_i$ be the curve in $Y(\Gamma) \times Y(\Gamma')$ which is the image of points of the form $(z, \alpha_i z)$. If the double cosets $\Gamma' \alpha_1 \Gamma$ and $\Gamma' \alpha_2 \Gamma$ are distinct as subsets of $\PGL_2^+(\QQ)$, then $C_1 \cap C_2$ is a finite set.
   \end{lemma}

   \begin{proof}
    Suppose $P \in C_1 \cap C_2$. Then $P$ admits liftings to $\cH \times \cH$ of the form $(z_1, \alpha_1 z_1)$ and $(z_2, \alpha_2 z_2)$; and since both of these points are preimages of $P$, we can find $\gamma \in \Gamma$ and $\gamma' \in \Gamma'$ such that $z_1 = \gamma z_2$ and $\alpha_1 z_1 = \gamma' \alpha_2 z_2$. Consequently, $z_2$ is fixed by the element
    \[ \gamma^{-1} \alpha_1^{-1} \gamma' \alpha_2 \in \Gamma \cdot \alpha_1^{-1} \alpha_2 \cdot (\alpha_2^{-1} \Gamma' \alpha_2).\]
    By Lemma \ref{lem:finitefixedpoints}, either $\gamma^{-1} \alpha_1^{-1} \gamma' \alpha_2$ is the identity in $\PGL_2^+(\QQ)$, in which case $\Gamma' \alpha_1 \Gamma$ and $\Gamma' \alpha_2 \Gamma$ have the same projective image; or $z_2$ lies in one of a finite set of orbits under the action of $\Gamma \cap \alpha_2^{-1} \Gamma' \alpha_2$, which implies that $P$ lies in one of a finite set of points of $C_2$, as required.
   \end{proof}

   \begin{lemma}
    \label{lem:doublecosetpullback}
    Let $\Gamma_1, \Gamma_2 \subseteq \SL_2(\ZZ)$, and let $\Gamma_1' \subseteq \Gamma_1$ and $\Gamma_2' \subseteq \Gamma_2$, with all four subgroups having finite index in $\SL_2(\ZZ)$. Let $\alpha \in \GL_2^+(\QQ)$, and suppose $\beta_1, \dots, \beta_h \in \GL_2^+(\QQ)$ are such that we have
    \[ \Gamma_2 \alpha \Gamma_1 = \bigsqcup_{i=1}^h \Gamma_2' \beta_i \Gamma_1'.\]

    Let $C$ be the curve in $Y(\Gamma_1) \times Y(\Gamma_2)$ which is the image of $Y(\Gamma_1 \cap \alpha^{-1} \Gamma_2 \alpha)$ under the map $\sigma$ of Lemma \ref{lem:birational}. Then the preimage of $C$ in $Y(\Gamma_1') \times Y(\Gamma_2')$ is the union of $h$ distinct curves $D_1, \dots, D_h$, where $D_i$ is the image of the map
    \[
     \begin{array}{rcl}
      \sigma_i:Y(\Gamma_1'\cap \beta_i^{-1}  \Gamma_2' \beta_i) &\rTo& Y(\Gamma_1') \times Y(\Gamma_2') \\
      z & \rMapsto & (z, \beta_i z).
     \end{array}
    \]

    Moreover, if for each $i$ we choose some $\gamma_i \in \Gamma_1$ such that $\beta_i \in \Gamma_2 \alpha \gamma_i$, then we have a commutative diagram
    \begin{equation}
     \label{eq:doublecosetpullback}
     \begin{diagram}
      Y(\Gamma_1'\cap \beta_i^{-1}  \Gamma_2' \beta_i) & \rTo^{\sigma_i} & Y(\Gamma_1') \times Y(\Gamma_2') \\
      \dTo^{z \mapsto \gamma_i z} & & \dTo\\
      Y(\Gamma_1 \cap \alpha_i^{-1} \Gamma_2 \alpha_i) & \rTo^{\sigma} & Y(\Gamma_1) \times Y(\Gamma_2)
     \end{diagram}
    \end{equation}
    where the right-hand vertical arrow is the natural projection map.
   \end{lemma}

   \begin{proof}
    The definition of $\gamma_i$ implies that diagram \eqref{eq:doublecosetpullback} commutes, from which it is clear that $D_i$ is a lifting of $C$. By lemma \ref{lem:distinctcurves}, the $D_i$ are distinct.

    It remains only to check that the union of the $D_i$ exhausts the preimage of $C$. Let $P \in C$, and let $\tilde P$ be any lifting of $P$ to $\cH \times \cH$. Then we have $P = (\gamma_1 z, \gamma_2 \alpha z)$ for some $\gamma_1 \in \Gamma_1$ and $\gamma_2 \in \Gamma_2$; so $P = (w, \gamma_2 \alpha \gamma_1^{-1} w)$, where $w = \gamma_1 z \in \cH$. We have $\gamma_2 \alpha \gamma_1^{-1} \in \Gamma_2' \beta_i \Gamma_1'$ for some $i \in \{1, \dots, h\}$, so in particular the image of $\tilde P$ in $Y(\Gamma_1') \times Y(\Gamma_2')$ lies in $D_i$ as required.
   \end{proof}

   \begin{lemma}
    \label{lem:pushpull}
    Let $\Gamma$ be a finite-index subgroup of $\SL_2(\ZZ)$ and let $\Gamma_1, \Gamma_2$ be finite-index subgroups of $\Gamma$ such that $\Gamma_1 \Gamma_2 = \Gamma$. Then, in the diagram of modular curves
    \begin{diagram}
     Y(\Gamma_1 \cap \Gamma_2) & \rTo^\alpha & Y(\Gamma_1) \\
     \dTo^\beta & & \dTo^\gamma \\
     Y(\Gamma_2) & \rTo^\delta & Y(\Gamma)
    \end{diagram}
    where $\alpha, \beta, \gamma, \delta$ are the natural projection maps, the two maps $\cO(Y(\Gamma_1))^\times \to \cO(Y(\Gamma_2))^\times$ given by $\beta_* \circ \alpha^*$ and $\delta^* \circ \gamma_*$ coincide, and similarly the maps $\cO(Y(\Gamma_2))^\times \to \cO(Y(\Gamma_1))^\times$ given by $\alpha_* \circ \beta^*$ and $\gamma^* \circ \delta_*$ coincide.
   \end{lemma}

   \begin{proof}
    Note that the hypotheses are symmetric in $\Gamma_1$ and $\Gamma_2$, so it suffices to show that $\beta_* \circ \alpha^* = \delta^* \circ \gamma_*$. Moreover, since all of the morphisms in the diagram are surjective, the corresponding pullback morphisms are injective, so it suffices to show that
    \[ \beta^* \circ \beta_* \circ \alpha^* = \beta^* \circ \delta^* \circ \gamma_*.\]
    Since the diagram commutes, this is equivalent to
    \[ (\beta^* \circ \beta_*) \circ \alpha^* = \alpha^* \circ (\gamma^* \circ \gamma_*).\]
    However, the map $(\beta^* \circ \beta_*)$ is given by the product over translates by coset representatives for $(\Gamma_1 \cap \Gamma_2) \backslash \Gamma_2$, and the map $(\gamma^* \circ \gamma_*)$ is given by the product over coset representatives for $\Gamma_1 \backslash \Gamma$. However, since $\Gamma_1 \Gamma_2 = \Gamma$, the natural map
    \[ (\Gamma_1 \cap \Gamma_2) \backslash \Gamma_2 \to \Gamma_1 \backslash \Gamma\]
    is surjective. Thus these two quotients admit a common set of coset representatives, so the two maps coincide.
   \end{proof}

   \begin{remark}
    One can interpret this more ``categorically'' as follows: our hypotheses imply that the diagram in the statement of the lemma is Cartesian (in the category of curves and dominant rational maps), so $Y(\Gamma_1 \cap \Gamma_2)$ is birational to the fibre product of $Y(\Gamma_1)$ and $Y(\Gamma_2)$ over $Y(\Gamma)$. The symmetry of pushforward and pullback is then a general property of fibre products.
   \end{remark}


   \subsection{Motivic cohomology, higher Chow groups and the Gersten complex}

    We now recall the definition of the higher Chow group $\CH^2(X, 1)$ of a variety $X$, and how it may be explicitly calculated using the Gersten complex. In this section $k$ may be any field of characteristic $0$. Let $\Var(k)$ be the category of varieties over $k$, by which we mean separated schemes of finite type over $k$. Let $\Sm(k)$ be the full subcategory of smooth varieties. Let $A=\QQ$ or $\ZZ$ be the coefficient ring.

    \begin{definition}[{Voevodsky, cf.~\cite[Definition 3.4]{mazzavoevodskyweibel06}}]
     Let $X\in \Sm(k)$, and $p, q \in \ZZ$ with $q \ge 0$. Define the \emph{motivic cohomology} of $X$ to be
     \[ H^p_{\mathcal{M}}(X,A(q))= \mathbb{H}^p_{\mathrm{Zar}}(X, \ZZ(q) \otimes A),\]
     where $\ZZ(q)$ denotes Voevodsky's motivic complex of sheaves on $X$, and $\mathbb{H}^p_{\mathrm{Zar}}$ denotes hypercohomology (with respect to the Zariski topology).
    \end{definition}

    \begin{remark}{\mbox{~}}
     \begin{enumerate}
      \item We use a slightly different notation than Voevodsky; the notation used in \emph{op.cit.} is $H^{i,j}(X, A)$. Our choice of notation follows \cite{huber00} and \cite{levine04}.
      \item Note that $H^p_{\mathcal{M}}(X,A(q))$ is zero for $p > \inf(2q, q + \dim X)$. It is \emph{not} known to be zero for $p < 0$, since the motivic complex is not bounded below.
     \end{enumerate}
    \end{remark}

    We shall not use the definition of motivic cohomology directly; we shall rather use the fact that these groups are isomorphic to Bloch's higher Chow groups:

    \begin{theorem}
     For any $X \in \Sm(k)$ and any $p,q\geq 0$, there is a natural isomorphsim
     \[ H^{p}(X,\ZZ(q))\cong \CH^q(X,2q-p).\]
     Here, the higher Chow groups are those defined by Bloch.
    \end{theorem}

    \begin{proof}
     See \cite[Corollary 2]{voevodsky02} or \cite[Theorem 1.2]{levine04}.
    \end{proof}

    We also have an alternative description of these groups in terms of Quillen $K$-theory. We will actually be interested in the special case when $p=3$ and $q=2$. Here, we use a result of Landsburg \cite{landsburg91}. For $X$ smooth over a field, $m \ge 0$ and $0 \le p \le m$, he constructs a map
    \[ \Psi_{m,p}:\CH^m(X,m-p)\rTo H^p(X,\mathscr{K}_m),\]
    where $\mathscr{K}_m$ is the sheafification of $U \mapsto K_m(U)$ on $X$. Here, $K_m$ denotes the $m$-th Quillen $K$-group.

    \begin{theorem}
     The map $\Psi_{m,p}$ is an isomorphism for $p=m-1$.
    \end{theorem}

    \begin{proof}
     See \cite[Theorem 2.5]{landsburg91}.
    \end{proof}

    \begin{remark}
     For $p < m-1$ the map $\Psi_{m, p}$ may not be an isomorphism in general. As pointed out to us by Landsburg in a discussion on \url{http://mathoverflow.net/}, if $X = \Spec(k)$, then $\CH^m(X, m)$ is the Milnor $K$-group $K_m^M(k)$ (by a theorem of Nesterenko--Suslin) and the map $\Psi_{m, 0}: K_m^M(k) \to H^0(X, \mathscr{K}_m) = K_m(k)$ is the natural map from Milnor to Quillen $K$-theory, which is not generally an isomorphism for $m > 2$.
    \end{remark}

    Finally, we address the question of how to explicitly describe elements of these groups.

    \begin{proposition}\label{prop:gerstenresolution}
     Suppose that $X$ is a smooth variety of finite type over a field $k$. Then there is a resolution of the sheaf $\mathscr{K}_m$
    \[ 0\rTo \mathscr{K}_m\rTo \coprod_{x\in X^0} (i_x)_* K_m(k(x))\rTo  \coprod_{x\in X^1} (i_x)_* K_{m-1}(k(x))\rTo\dots\]
    \end{proposition}

    \begin{proof}
     See \cite{quillen73}.
    \end{proof}

    \begin{corollary}\label{cor:H1viaGersten}
     The group $H^1(X,\mathscr{K}_2)$ is the first homology group of the ``Gersten complex''
     \begin{equation}\label{Gerstencomplex}
      \Gerst_2(X): \coprod_{x \in X^0} K_2(k(x))\rTo^{d^0} \coprod_{x\in X^1} k(x)^\times \rTo^{d^1} \coprod_{x\in X^2}\ZZ,
     \end{equation}
    where $d^0$ is the tame symbol map, and $d^1$ maps a function to its divisor (c.f. \cite[Section 2]{flach92}).
    \end{corollary}

    Combining the above results, we get the following statement.

    \begin{proposition}
     \label{prop:motiviccohomology}
     Assume that $X$ is a smooth variety of finite type over a field $k$. Then we have isomorphisms
     \[ H^1(\Gerst_2(X)) \cong H^1(X,\mathscr{K}_2) \cong \CH^2(X,1) \cong H^3_{\mathcal{M}}(X, \ZZ(2)).\]
    \end{proposition}

    We shall use these to identify $\CH^2(X,1) $ with $H^1(\Gerst_2(X))$; it is the latter group in which we shall actually construct elements.

    \begin{notation}
     We shall write $Z^2(X, 1)$ to denote the kernel of the boundary map $d^1$ in the Gersten complex $\Gerst_2(X)$, so
    \[ Z^2(X, 1) = \left\{ \sum_i (C_i, \phi_i) : C_i \in X^1, \phi_i \in k(C_i)^\times, \sum_i \operatorname{div}(\phi_i) = 0\right\}.\]
    \end{notation}

    This is a slight abuse of notation, since in Bloch's theory of higher Chow groups $Z^2(X, 1)$ is used to denote something slightly different (a certain subgroup of the codimension 2 cycles on $X \times \AA^1$); but we shall not use Bloch's construction directly in this paper, so this abuse should cause no confusion.

    \begin{remark}
     We shall, in fact, construct an ``Euler system'' in the groups $Z^2(X, 1)$ as $X$ varies over a family of modular surfaces; that is, our compatibility properties will hold at the level of cycles, rather than just after quotienting out by the image of tame symbols. The groups $Z^2(X, 1)$ are much easier to work with, as they have good descent properties: for a finite surjective map $X \to Y$, the pullback $Z^2(Y, 1) \to Z^2(X, 1)$ is injective.

     This is, in a sense, analogous to the fact that in the construction of \cite{kato04} the compatibility properties of the Euler system in $K_2$ of modular curves are proved at the level of $K_1 \otimes K_1$, before quotienting by elements of the form $x \otimes (1 - x)$.
    \end{remark}


  \subsection{Zeta elements on \texorpdfstring{$Y(m, mN)$}{Y(m, mN)}}

   We begin by defining elements of $Z^2(Y(m, mN)^2, 1)$, which we shall call \emph{zeta elements}.

   \begin{definition}
    For $m, N \ge 1$, the curve $\cC_{m, N, j} \subseteq Y(m, mN)^2$ is defined as the subvariety
    \[ \left( u, v: v = \tbt 1 j 0 1 u \right).\]
    For $c > 1$ coprime to $6mN$, we define
    \[ {}_c \cZ_{m, N, j} = \left( \cC_{m, N, j}, \phi\right) \in Z^2(Y(m, mN)^2, 1),\]
    where $\phi \in \cO(C)^\times$ is the pullback of ${}_c g_{0, 1/mN}$ along either of the projections $\cC_{m, N, j} \to Y(m, mN)$.
   \end{definition}

   The first properties of these elements are the following.

   \begin{proposition}
    \label{prop:zetaeltproperties}
    The elements ${}_c \cZ_{m, N, j}$ have the following properties:
    \begin{enumerate}
     \item We have $\rho^* {}_c \cZ_{m, N, j} = {}_c \cZ_{m, N, -j}$, where $\rho$ is the involution of $Y(m, mN)^2$ which interchanges the factors.
     \item For $c, d > 1$ coprime to $6mN$, the element
     \[ \left[d^2 - \left(\stbt d 0 0 d, \stbt d 0 0 d \right)^*\right] \cdot {}_c \cZ_{m, N, j}\]
     is symmetric in $c$ and $d$. In particular, there exists a unique element
     \[ \cZ_{m, N, j} \in Z^2(Y(m, mN)^2, 1) \otimes \QQ\]
     such that ${}_c \cZ_{m, N, j} = \left[c^2 - \left(\stbt c 0 0 c, \stbt c 0 0 c \right)^*\right] \cZ_{m, N, j}$ for any $c$.
     \item We have
     \[ \left( \tbt b 0 0 1, \tbt b 0 0 1\right)^* {}_c \cZ_{m, N, j} = {}_c \cZ_{m, N, b^{-1} j}\]
     for any $b \in (\ZZ / mN\ZZ)^\times$.
    \end{enumerate}
   \end{proposition}

   \begin{proof}
    Part (1) is obvious, and part (2) follows immediately from the fact that the Siegel units ${}_c g_{\alpha, \beta}$ satisfy
    \[ (d^2 {}_c g_{\alpha, \beta} - {}_c g_{d\alpha, d\beta}) = (c^2 {}_d g_{\alpha, \beta} - {}_d g_{c\alpha, c\beta})\]
    for any $\alpha, \beta \in \tfrac1N\ZZ / \ZZ - \{(0, 0)\}$ and any $c, d > 1$ coprime to $6mN$ (cf.~\cite[Proposition 1.3(2)]{kato04}). We may then define $\cZ_{m, N, j} = (c^2 - 1)^{-1} {}_c \cZ_{m, N, j}$ for any $c > 1$ congruent to 1 modulo $mN$.

    Property (3) follows from the identity $\tbt b 0 0 1^{-1} \tbt 1 j 0 1 = \tbt 1 {b^{-1}j} 0 1 \tbt b 0 0 1^{-1}$.
   \end{proof}

  \subsection{Generalized Beilinson--Flach elements}
   \label{section:BFelements}

   The Beilinson--Flach elements of \cite{BDR12} are elements of $\CH^2(Y_1(N)^2, 1)$ defined as $(\Delta, \phi)$, where $\Delta$ is the diagonal and $\phi$ is a suitable modular unit on $\Delta$. Our generalization of this is motivated by the observation that one can recover the twists of a modular form by Dirichlet characters modulo $m$ from the ``shifted'' forms $f(x + a/m)$ for $a \in (\ZZ / m\ZZ)^\times$; this is also the idea underlying the construction of the $p$-adic $L$-function of a single modular form using modular symbols.

   \begin{lemma}
    \label{lem:kappaexists}
    Let $m, N \ge 1$ with $m^2 N \ge 5$, and $j \in \ZZ$. Then there is a unique morphism of algebraic varieties over $\CC$,
    \[ \kappa_j : Y_1(m^2 N) \otimes \CC \to Y_1(N) \otimes \CC, \]
    such that the diagram of morphisms of complex-analytic manifolds
    \begin{diagram}
     \cH & \rTo^{z \mapsto z + j/m} & \cH\\
     \dTo & & \dTo\\
     Y_1(m^2 N)(\CC) & \rTo^{\kappa_j} & Y_1(N)(\CC) \\
    \end{diagram}
    commutes. The morphism $\kappa_j$ is defined over $\QQ(\mu_m)$, and depends only on the residue class of $j \bmod m$.
   \end{lemma}

   \begin{proof}
    The existence of such a map at the level of quotients of $\cH$ follows immediately from the inclusion of matrix groups
    \[ \tbt 1 {\tfrac jm} 0 1 \Gamma_1(m^2 N) \tbt 1 {-\tfrac jm} 0 1 \subseteq \Gamma_1(N).\]
    However, in order to descend to an algebraic morphism over $\QQ(\mu_m)$ we use the canonical models above.

    We first consider the map $Y(m^2 N) \to Y(m, mN)$ which maps $(E, e_1, e_2)$ to $\left(E/\langle m e_2 \rangle, [mNe_1], [e_2]\right)$. This factors through the quotient by the subgroup $\tbt u *0 1 : u = 1 \bmod {m\ZZ}$, which we have identified with $Y_1(m^2 N) \otimes \QQ(\mu_m)$. This map is compatible with $z \mapsto mz$ on $\cH$. We now consider the composition
    \[ Y_1(m^2 N) \times \mu_m \rTo Y(m, mN) \rTo^{\stbt 1 j 0 1} Y(m, mN) \rTo^{t_m} Y_1(N) \times \mu_m,\]
    where $t_m$ is as in Definition \ref{def:tm}.
    All three morphisms are maps of $\QQ(\mu_m)$-varieties (i.e.~they commute with the projection maps to $\mu_m$); and on the fibre over $\zeta_m \in \mu_m(\CC)$ they correspond to $z \mapsto mz$, $z \mapsto z + j$, and $z \mapsto z/m$, so the composition corresponds to $z \mapsto z + j/m$.
   \end{proof}

   \begin{definition}
    For $m, N, j$ as above, let $\iota_{m, N, j}$ be the map
    \[ (1, \kappa_j): Y_1(m^2 N) \times \mu_m \to Y_1(N)^2 \times \mu_m,\]
    and $C_{m, N, j}$ the irreducible curve in $Y_1(N)^2$ that is the image of $\iota_{m, N, j}$.
   \end{definition}

   We shall now use these curves $C_{m, N, j}$ to define a class in $\CH^2(Y_1(N)^2 \times \mu_m, 1)$, using the presentation of the latter group given by the Gersten complex.

   \begin{definition}
    \label{def:BFelts}
    Let $N \ge 5$, $m \ge 1$, $j \in \ZZ / m\ZZ$ as above. Let $c \ge 1$ be coprime to $6mN$ and let $\alpha \in \ZZ / m^2 N \ZZ$. We define the \emph{generalized Beilinson--Flach element}
    \[ {}_c \Xi_{m, N, j, \alpha} \in \CH^2(Y_1(N)^2 \otimes \QQ(\mu_m), 1)\]
    as the class of the pair
    \[ \Big( C_{m, N, j}, (\iota_{m, N, j})_* ({}_c g_{0, \alpha / m^2 N}) \Big) \in Z^2(Y_1(N)^2 \times \mu_m, 1).\]
    When $\alpha = 1$ we drop it from the notation and write simply ${}_c \Xi_{m, N, j}$.
   \end{definition}

   The following proposition shows that these zeta elements are simply the ``$Y_1$-versions'' of those defined in the previous section.

   \begin{proposition}
    \label{prop:pushforwardzeta}
    The generalized Beilinson--Flach element ${}_c \Xi_{m, N, j, \alpha}$ is the pushforward of the element
    \[ \left(\tbt \alpha 0 0 \alpha, \tbt \alpha 0 0 \alpha\right)^* {}_c \cZ_{m, N, j} \in Z^2(Y(m, mN)^2, 1)\]
    along the map $t_m \times t_m: Y(m, mN)^2 \to Y_1(N)^2 \times \mu_m$ introduced in \S \ref{sect:modcurves}.
   \end{proposition}

   \begin{proof}
    It is clear from the construction of the map $\kappa_{m, N, j}$ that $C_{m, N, j}$ is the image of $\cC_{m, N, j}$ under $t_m \times t_m$. So it suffices to show that the pushforward of ${}_c g_{0, 1/m^2 N}$ from $Y_1(m^2 N) \otimes \QQ(\mu_m)$ to $Y(m, mN)$ along the map constructed above is ${}_c g_{0, 1/mN}$.

    Let $U$ be the subgroup of $\GL_2(\ZZ / m^2 N \ZZ)$ consisting of elements $\begin{pmatrix} a & b \\ c & d \end{pmatrix}$ which satisfy $c = 0, d = 1 \bmod m^2 N$ and $a = 1 \bmod m$ (and $b$ arbitrary). This is clearly contained in the subgroup $U'$ of elements satisfying $a = 1 \bmod m$, $c = 0 \bmod m^2 N$ and $d = 1 \bmod mN$, and a set of coset representatives for $U / U'$ is given by the matrices
    \[ \left\{\begin{pmatrix} 1 & 0 \\ 0 & 1 + mNt \end{pmatrix} : 0 \le t < m \right\}.\]
    Hence the pushforward of ${}_c g_{0, 1/m^2 N}$ from $U \backslash Y(m^2 N)$ to $U' \backslash Y(m^2 N)$ is given by
    \[ \prod_{0 \le t < m}\tbt 100{1 + mNt}^* {}_c g_{0, 1/m^2 N} = \prod_{0 \le t < m}{}_c g_{0, 1/m^2 N + t/m}.\]
    By Proposition \ref{prop:distrelations}(2), this is equal to $\vp_m^* \left({}_c g_{0, 1/mN}\right)$. However, conjugation by $\tbt m 0 0 1$ sends $U'$ to the subgroup $U''=\left\{ \tbt a b c d : \begin{array}{ll} a = 1, b = 0 \bmod m, \\ c = 0, d = 1 \bmod mN \end{array}\right\}$, and we have $U''\backslash Y(m^2 N) = Y(m, mN)$.
   \end{proof}

   We now record some properties of the generalized Beilinson--Flach elements.

   \begin{proposition}
    \label{prop:BFeltproperties}
    The elements above have the following properties:
    \begin{enumerate}
     \item The element ${}_c \Xi_{m, N, j, \alpha}$ only depends on the congruence class of $\alpha$ modulo $mN$ (not $m^2 N$).
     \item The involution of $Y_1(N)^2 \otimes \QQ(\mu_m)$ given by switching the two factors interchanges ${}_c \Xi_{m, N, j}$ and ${}_c \Xi_{m, N, -j}$.
     \item For $q \in (\ZZ / m\ZZ)^\times$, we have $\sigma_q^* \left( {}_c\Xi_{m, N, j, \alpha}\right) = {}_c \Xi_{m, N, q^{-1} j, \alpha}$, where $\sigma_q \in \Gal(\QQ(\mu_m) / \QQ)$ is the arithmetic Frobenius at $q$.
     \item For any $r \in (\ZZ / m N\ZZ)^\times$, we have
     \[ {}_c \Xi_{m, N, j, r\alpha} = \langle d \times d \rangle^* {}_c \Xi_{m, N, k, \alpha}\]
     where $k = r^{-2} j \in \ZZ / m\ZZ$, $d$ is the image of $r$ in $(\ZZ / N\ZZ)^\times$, and $\langle d \times d \rangle$ denotes the action on $Y_1(N)^2 \otimes \QQ(\mu_m)$ of the element
     \[ \begin{pmatrix} d^{-1} & 0 \\ 0 & d\end{pmatrix} \times \begin{pmatrix} d^{-1} & 0 \\ 0 & d\end{pmatrix} \in \SL_2(\ZZ / N\ZZ)^2.\]
     \item\label{item:cfactor} For $c, d$ coprime to $6mN$, the expression
     \[ d^2 {}_c \Xi_{m, N, j, \alpha} - {}_c \Xi_{m, N, j, d\alpha} \]
     is symmetric in $c$ and $d$. In particular, there exist well-defined elements $\Xi_{m, N, j, \alpha} \in \CH^2(Y_1(N)^2 \otimes \QQ(\mu_m), 1) \otimes \QQ$ such that we have
     \[ {}_c \Xi_{m, N, j, \alpha} = c^2 \Xi_{m, N, j, \alpha} - \Xi_{m, N, j, c\alpha} = (c^2 - \langle c \times c\rangle^* \sigma_c^2)\, \Xi_{m, N, j, \alpha}.\]
    \end{enumerate}
   \end{proposition}

   \begin{proof}
    After Proposition \ref{prop:pushforwardzeta}, parts (1) and (2) are immediate. The remaining statements follow from Proposition \ref{prop:zetaeltproperties}, together with the fact that $t_m$ intertwines the action of $\stbt d 0 0 1$ on $Y(m, mN)$ with the arithmetic Frobenius $\sigma_d$ on $Y_1(N) \times \mu_m$ (Proposition \ref{prop:tmproperties}(2)).
   \end{proof}


  \subsection{Cuspidal components}
   \label{section:cuspidalcomponents}

   In the preceding sections we have constructed elements of the higher Chow groups of affine surfaces. In order to be able to apply results on regulator maps, it is convenient to have elements of Chow groups of projective surfaces instead. We shall show that this can be achieved, but not in a canonical way, and only at the cost of tensoring with $\QQ$.

   \begin{theorem}
    \label{thm:preimageinChowgp}
    Let $N, m, j$ be as in Definition \ref{def:BFelts}. Then the element $\Xi_{m, N, j}$ of $\CH^2(Y_1(N)^2 \otimes \QQ(\mu_m), 1) \otimes \QQ$ is in the image of the pullback map
    \[ \CH^2(X_1(N)^2 \otimes \QQ(\mu_m), 1) \otimes \QQ \to \CH^2(Y_1(N)^2 \otimes \QQ(\mu_m), 1) \otimes \QQ\]
    induced by the open embedding $Y_1(N) \into X_1(N)$.
   \end{theorem}

   We will actually prove a slightly more precise statement, see Proposition \ref{prop:preimageinGerst21} below.

   Recall that we constructed ${}_c \Xi_{m,N,j}$ as the class in $\CH^2(Y_1(N)^2 \otimes \QQ(\mu_m), 1)$ of an explicit element of $Z^2(Y_1(N)^2 \otimes \QQ(\mu_m), 1)$, which we shall temporarily denote by ${}_c \mathfrak{Y}_{m,N,j}$. It is clear that we may also regard ${}_c \mathfrak{Y}_{m,N,j}$ as an element of $\Gerst_2^1(X_1(N)^2 \otimes \QQ(\mu_m))$, whose divisor is not necessarily trivial, but is supported on the cuspidal locus.

   We will need a preparatory lemma. Let $K$ be any number field.

   \begin{definition}
    We shall call an element of $\Gerst_2^1(X_1(N)^2 \otimes K)$ \emph{negligible} if it is supported on a finite union of curves of the form $\{c\} \times X_1(N)$ or $X_1(N) \times \{d\}$ for points $c,d \in X_1(N) \backslash Y_1(N)$.
   \end{definition}

   \begin{remark}
    Here by ``point'' we mean a 0-dimensional point of $X_1(N) \backslash Y_1(N)$ considered as a $K$-scheme, i.e.~a $\Gal(\overline{K}/K)$-orbit of points in the naive sense. Note that this is slightly more restrictive than the definition of ``negligible'' in \cite{BDR12}.
   \end{remark}

   Before proving the theorem, we will need the following preparatory lemma:

   \begin{lemma}\label{lem:negligibleelement}
    Let $K$ be any number field and let $u,v,x,y$ be cuspidal points of $X_1(N) \otimes K$. Then there exists a negligible element in $\Gerst_2^1(X_1(N)^2 \otimes K) \otimes \QQ$ with divisor $(u,v)-(x,y)$.
   \end{lemma}

   \begin{proof}
    By the Manin--Drinfeld theorem \cite{drinfeld73}, there exist elements $f,g\in \cO(Y_1(N) \otimes K)^\times \otimes \QQ$ whose divisors are $v-y$ and $u-x$, respectively. The the element
    \[ (\{u\}\times X_1(N),f)+(X_1(N)\times\{y\},g)\]
    has the required property.
   \end{proof}

   We can now prove the following proposition:

   \begin{proposition}
    \label{prop:preimageinGerst21}
    Let $N, m, j, c$ be as in Definition \ref{def:BFelts}. Then there exists an integer $r \ge 1$ and a negligible element $\Theta$ such that
    \[ R \cdot {}_c \mathfrak{Y}_{m,N,j} + \Theta \in Z^1(\Gerst_2(X_1(N)^2 \otimes \QQ(\mu_m))).\]
   \end{proposition}

   \begin{proof}
    Recall that $\iota_{m,N,j}$ is the map
    \[ \iota_{m,N,j} = (1, \kappa_j): Y_1(m^2 N)\rightarrow Y_1(N)^2, \]
    where $\kappa_j$ is induced from the map $\cH\rightarrow \cH$ given by $z\mapsto z + \frac{j}{m}$. It follows that, if we regard ${}_c\mathfrak{Y}_{m,N,j}$ as an element of $Z^2(X_1(N)^2 \otimes \QQ(\mu_m), 1)$, then $\Div \big( {}_c\mathfrak{Y}_{m,N,j} \big)$ is a linear combination of divisors of the form $(c_1, c_1 + \frac{j}{m}) - (c_2, c_2 + \frac{j}{m})$. But Lemma \ref{lem:negligibleelement} implies that there exists a negligible element $\Theta \in \Gerst^1_2(X_1(N) \otimes \QQ(\mu_m))$ such that $\Div(\Theta) = \Div\big( {}_c\mathfrak{Y}_{m,N,j} \big)$. Then the element
    \[ {}_c\mathfrak{X}_{m,N,j} \coloneqq {}_c\mathfrak{Y}_{m,N,j} - \Theta \in Z^2(X_1(N)^2 \otimes \QQ(\mu_m), 1)\otimes\QQ\]
    has the required properties.
   \end{proof}

   This clearly implies Theorem \ref{thm:preimageinChowgp}.

   \begin{remark}\mbox{~}
    \begin{enumerate}
     \item Note that the negligible element $\Theta$ is not uniquely determined. However, as we will see below, this will not matter for the evaluation of the element via the Beilinson or the syntomic regulator.
     \item Since $X_1(N)^2$ and $Y_1(N)^2$ have the same rational function field, any element of $\CH^2(X_1(N)^2 \otimes \QQ(\mu_m), 1)\otimes \QQ$ lifting ${}_c \Xi_{m, N, j}$ is necessarily the class of an element of $Z^2(X_1(N)^2 \otimes \QQ(\mu_m), 1) \otimes \QQ$ differing from ${}_c\mathfrak{Y}_{m,N,j}$ by a negligible element.
     \item Since there are only finitely many cusps on $X_1(N)$, the constant $R$ may be chosen to be independent of $c$, $m$ and $j$, although it may of course depend on $N$.
    \end{enumerate}
   \end{remark}

 \subsection{Zeta elements versus generalized Beilinson--Flach elements}
  \label{sect:dreams}

  At the referee's request, we shall briefly clarify the relations between the two classes of elements we have introduced (the zeta elements ${}_c \cZ_{m, N, j}$ and the generalized Beilinson--Flach elements ${}_c \Xi_{m, N, j}$) and how they would relate to a hypothetical ``optimal'' construction. Recall that the element ${}_c \cZ_{m, N, j}$ lies in the group $\CH^2(Y(m, mN)^2, 1)$, and the element ${}_c \Xi_{m, N, j} \in \CH^2(Y_1(N)^2 \otimes \QQ(\mu_m), 1)$ is the pushforward of ${}_c \cZ_{m, N, j}$ via the morphism $t_m \times t_m$ of \S \ref{sect:modcurves}. It is the elements ${}_c \Xi_{m, N, j}$ which will be used in \S\S 4--7 of this paper in order to bound Selmer groups.

  One reason for introducing the elements ${}_c \cZ_{m, N, j}$ is that they are somewhat easier to work with than the ${}_c \Xi_{m, N, j}$. In the next section we shall prove norm-compatibility relations for the ${}_c \cZ_{m, N, j}$, and deduce norm relations for the ${}_c \Xi_{m, N, j}$ as a consequence; given the somewhat opaque map $\iota_{m, N, j}$ entering into the definition of the elements ${}_c \Xi_{m, N, j}$, it seems unlikely that these norm relations could be proved without the introduction of some auxilliary higher-level modular curve.

  A second reason to consider the elements ${}_c \cZ_{m, N, j}$ is the following optimistic idea. Let us fix a prime $p$, and a level $N$ coprime to $p$, and consider the curves $Y(p^r, Np^r)$ for $r \ge 0$, and their self-products $Y(p^r, Np^r)^2$. These form a tower of surfaces with Galois group $\GL_2(\Zp) \times_{\det} \GL_2(\Zp)$. Let us imagine that we could construct a norm-compatible family of elements in the higher Chow groups of this tower, analogous to the compatible family of elements in $K_2$ of the $\GL_2(\Zp)$-tower of modular curves constructed by Kato in \cite{kato04}. Then one could potentially perform a ``nonabelian twisting'' operation analogous to equation (8.4.3) of \emph{op.cit.} in order to obtain classes in the cohomology groups attached to pairs of modular forms of arbitrary weights $k, \ell \ge 2$.

  The elements ${}_c \cZ_{p^r, N, j}$ represent our best attempt to realize this dream. They do indeed live on the surfaces $Y(p^r, Np^r)^2$; but the norm-compatibility relation they satisfy (Theorem \ref{thm:secondnormbadprime1}) involves the ``twisted'' degeneracy map $\tau_p: Y(p^{r+1}, Np^{r+1}) \to Y(p^r, Np^r)$ of Definition \ref{def:tm}, given by $z \mapsto z/p$ on the upper half-plane $\cH$, rather than the natural one corresponding to the identity map on $\cH$. The norm-compatibility relation also involves a Hecke operator at $p$, which does not appear in the setting of \cite{kato04}. Consequently, our methods will only allow us to construct cohomology classes for Rankin--Selberg convolutions of higher weight forms under additional ordinarity assumptions, when we can use Hida's theory of $p$-adic families in order to pass from weight 2 to general weights.


 \section{Norm relations for generalized Beilinson--Flach elements}


\subsection{The first norm relation: varying \texorpdfstring{$N$}{N}}

 We now consider the relation between the zeta elements at different levels $N$ (for fixed $m$ and $j$).

 \begin{theorem}[First norm relation]
  \label{thm:firstnormrelation1}
  Let $\alpha$ be the natural projection $Y(m, mN') \to Y(m, mN)$, where $N$ and $N'$ are positive integers such that $N \mid N'$.
  \begin{enumerate}
   \item If $\Prime(N') \subseteq \Prime(m N)$, the pushforward map
   \[ (\alpha \times \alpha)_* : \CH^2(Y(m, mN')^2, 1) \to \CH^2(Y(m, mN)^2, 1)\]
   maps ${}_c \cZ_{m, N', j}$ to ${}_c \cZ_{m, N, j}$.
   \item If $N' = N\ell$, where $\ell \nmid mN$ is prime, then
   \[ (\alpha \times \alpha)_*\left({}_c \cZ_{m, N\ell, j}\right) = \left[1 - \left(\stbt {\ell^{-1}} 0 0 {\ell^{-1}}, \stbt {\ell^{-1}} 0 0 {\ell^{-1}}\right)^*\right]  {}_c \cZ_{m, N, j},\]
   where $\stbt {\ell^{-1}} 0 0 {\ell^{-1}}$ is considered as an element of $\GL_2(\ZZ / mN \ZZ)$.
  \end{enumerate}
 \end{theorem}

 \begin{proof}
  It is clear that the map $\alpha$ commutes with the action of $\stbt1j01$, so we have $(\alpha \times \alpha)(\cC_{m, N', j}) = \cC_{m, N, j}$; more precisely, we have a commutative diagram
  \begin{diagram}
   Y(m, mN') & \rTo^{(1, \stbt1j01)} & \cC_{m, N', j} \\
   \dTo^\alpha & & \dTo_{\alpha \times \alpha}\\
   Y(m, mN) & \rTo^{(1, \stbt1j01)} & \cC_{m, N, j}.
  \end{diagram}

  From Theorem \ref{thm:siegel-compat}, we know that if $\Prime(mN') = \Prime(mN)$, then $\alpha_* \left({}_c g_{0, 1/m N'}\right) = {}_c g_{0, 1/m N}$, so part (1) of the theorem follows. For part (2), we deduce from the second part of Theorem \ref{thm:siegel-compat} that
  \[ (\alpha \times \alpha)_* \left({}_c \cZ_{m, N', j}\right) = {}_c \cZ_{m, N, j} - {}_c \cZ_{m, N, j, \quot{\ell^{-1}}},\]
  where we write ${}_c \cZ_{m, N, j, a}$ for the element formed with ${}_c g_{0, a/mN}$ in place of ${}_c g_{0, 1/mN}$. However, we have
  \[ {}_c g_{0, \quot{\ell^{-1}}/mN} = \stbt {\ell^{-1}} 0 0 {\ell^{-1}}^* {}_c g_{0, 1/mN}\]
  as elements of $\cO(Y(m, mN))^\times$, and the action of $\stbt {\ell^{-1}} 0 0 {\ell^{-1}}$ evidently commutes with that of $\stbt1j01$.
 \end{proof}

 We now deduce a compatibility relation for zeta elements on $Y_1(N)^2 \otimes \QQ(\mu_m)$.

 \begin{theorem}[First norm relation on $Y_1(N)$]
  \label{thm:firstnormrelation}
  Let $\alpha$ be the natural projection $Y_1(N') \to Y_1(N)$, where $N, N'$ are positive integers such that $N \mid N'$.

  If $\Prime(mN') = \Prime(mN)$ then we have
  \[ (\alpha \times \alpha)_* \left({}_c \Xi_{m, N', j}\right) = {}_c \Xi_{m, N, j}.\]
  If $N' = \ell N$ where $\ell \nmid mN$, then we have
  \[ (\alpha \times \alpha)_* \left({}_c \Xi_{m, N', j}\right) = \left[1 - (\langle \ell^{-1} \rangle, \langle \ell^{-1} \rangle)^* \sigma_\ell^{-2}\right] {}_c \Xi_{m, N, j},\]
  where $\sigma_\ell$ denotes the arithmetic Frobenius at $\ell$.
 \end{theorem}

 \begin{proof}
  This follows immediately from Theorem \ref{thm:firstnormrelation1}, since the map $\pi_{m, N} : Y(m, mN) \to Y_1(N) \times \mu_m$ intertwines $\stbt {\ell^{-1}} 0 0 \ell$ with the diamond operator $\langle \ell \rangle$, and $\stbt \ell 0 0 1$ with the Frobenius $\sigma_\ell$.
 \end{proof}


\subsection{Hecke operators}

 We define Hecke operators, following \cite[\S\S 2.9, 4.9]{kato04}. Let $\ell$ be prime, and $M, N \ge 1$ (we allow $\ell \mid M$ or $\ell \mid N$). We define a correspondence on $Y(M, N)$ as follows. We have a diagram of modular curves
 \begin{diagram}
  Y(M(\ell), N) \\
  \dTo^{\pi_1} & \rdTo^{\pi_2} \\
  Y(M, N) & & Y(M, N),
 \end{diagram}
 where $\pi_1$ is the natural degeneracy map, corresponding to the identity on $\cH$, and $\pi_2$ is the ``twisted'' degeneracy map, corresponding to $z \mapsto z/\ell$ on $\cH$. (In the notation introduced in the proof of Theorem \ref{thm:siegel-compat} above, $\pi_1$ was denoted $\pr_1$, and $\pi_2$ is the composite of $\vp_\ell^{-1} : Y(M(\ell), N) \to Y(M, N(\ell))$ with the natural projection $Y(M, N(\ell)) \to Y(M, N)$).

 We denote the correspondence $(\pi_2)_* (\pi_1)^*$ by $T_\ell'$ if $\ell \nmid MN$, and by $U_\ell'$ if $\ell \mid MN$. We denote the operator $(\pi_1)_* (\pi_2)^*$ by $T_\ell$ (resp.~$U_\ell$); these latter operators $T_\ell$, $U_\ell$ are the familiar Hecke operators of the transcendental theory, but it is the $T_\ell', U_\ell'$ that will concern us most here.


\subsection{The second norm relation for \texorpdfstring{$\ell \mid N$}{l|N}}

 Our goal in this section is to prove the following theorem:

 \begin{theorem}[Second norm relation, $\ell \mid N$ case]
  \label{thm:secondnormbadprime1}
  Let $m \ge 1, N \ge 5$, and $\ell$ a prime dividing $N$. Let $\tau_\ell$ denote the degeneracy map
  \[ Y(m\ell, m\ell N) \to Y(m, mN)\]
  of Definition \ref{def:tm}, compatible with $z \mapsto z/\ell$ on $\cH$. Then for any $j \in (\ZZ / \ell m \ZZ)^\times$, and $c > 1$ coprime to $6\ell mN$, we have
  \[ (\tau_\ell \times \tau_\ell)_* \left({}_c \cZ_{\ell m, N, j}\right) =
   \begin{cases}
    (U_\ell' \times U_\ell')\left({}_c \cZ_{m, N, j}\right)  & \text{if $\ell \mid m$,} \\
    \left(U_\ell' \times U_\ell' - \Delta_\ell^*\right) \left({}_c \cZ_{m, N, j}\right) & \text{if $\ell \nmid m$.}
   \end{cases}
  \]
  where $\Delta_\ell$ denotes the action of any element of $\GL_2(\ZZ /mN\ZZ)^2$ of the form $\left( \tbt x 0 0 1, \tbt x 0 0 1\right)$ with $x = \ell \bmod m$.
 \end{theorem}

 We shall prove Theorem \ref{thm:secondnormbadprime1} below. First, we note that it implies the following property of the generalized Beilinson--Flach elements ${}_c \Xi_{m, N, j}$ on $Y_1(N)$:

 \begin{theorem}[Second norm relation on $Y_1(N)$, $\ell \mid N$ case]
  \label{thm:secondnormbadprime}
  Let $m \ge 1, N \ge 5$, $\ell$ a prime dividing $N$, $j \in (\ZZ / \ell m  \ZZ)^\times$, and $c \in (\ZZ / \ell m  N \ZZ)^\times$. Then we have
  \[ \norm_{m}^{\ell m} \left({}_c \Xi_{\ell m, N, j}\right)=
   \begin{cases}
    (U_\ell' \times U_\ell')\left({}_c \Xi_{m, N, j}\right)  & \text{if $\ell \mid m$,} \\
    (U_\ell' \times U_\ell' - \sigma_\ell) \left({}_c \Xi_{m, N, j}\right) & \text{if $\ell \nmid m$,}
   \end{cases}
  \]
  where $\norm_{m}^{\ell m}$ denotes the Galois norm map, and $\sigma_\ell$, for $\ell \nmid m$, denotes the arithmetic Frobenius at $\ell$ in $\Gal(\QQ(\mu_m) / \QQ)$.
 \end{theorem}

 \begin{proof}[Proof of Theorem \ref{thm:secondnormbadprime} (assuming Theorem \ref{thm:secondnormbadprime1})]
  Let $t_m \times t_m: Y(m, mN)^2 \to Y_1(N)^2 \otimes \QQ(\mu_m)$ be the map of \S \ref{sect:modcurves}. This map commutes with the actions of $(U_\ell', U_\ell')$ on both sides, and intertwines the action of $\Delta_\ell$ with the arithmetic Frobenius $\sigma_\ell$. Since ${}_c \Xi_{m, N, j} = (t_m \times t_m)_* \left({}_c \cZ_{m, N, j}\right)$ by Proposition \ref{prop:pushforwardzeta}, Theorem \ref{thm:secondnormbadprime} follows from Theorem \ref{thm:secondnormbadprime1}.
 \end{proof}

 \begin{proof}[Proof of Theorem \ref{thm:secondnormbadprime1}]
  Since we are assuming $\ell \mid N$, let us write $N' = N / \ell$. We have the following commutative diagram of modular curves:
  \begin{diagram}
   Y(\ell m, \ell m N) & \rTo^{\alpha} & Y(\ell m, m N) & \rTo^{\pr} & Y(m(\ell), m N)\\
   & \rdTo^{\tau_\ell} & & \ldTo^{\pi_2} & \dTo^{\pi_1} \\
   && Y(m, m  N) & & Y(m, mN).
  \end{diagram}
  Here $\alpha$ is the natural projection $Y(\ell m, \ell m N) \to Y(\ell m, \ell m N') = Y(\ell m, m N)$, and $\pr$ is the natural projection map. Consequently, we have a commutative diagram of surfaces
  \begin{diagram}
   Y(\ell m, \ell m N)^2 & \rTo^{\alpha \times \alpha} &Y(\ell m, m N)^2 & \rTo^{\pr \times \pr} & Y(m (\ell), m N)^2\\
   &  \rdTo^{\tau_\ell \times \tau_\ell} && \ldTo^{\pi_2 \times \pi_2} & \dTo_{\pi_1 \times \pi_1} \\
   && Y(m, mN)^2 & & Y(m, mN)^2.
  \end{diagram}

  Applying Theorem \ref{thm:firstnormrelation1}, we see that $(\alpha \times \alpha)_* {}_c \cZ_{\ell m, N, j} = {}_c \cZ_{\ell m, N', j}$. Since $Y(m (\ell), m N)^2$ is the quotient of $Y(\ell m, mN)^2$ by the subgroup
  \[ \left\{ \left( \tbt x 0 0 1, \tbt x 0 0 1 \right) : x \in \ZZ / \ell m \ZZ, = 1 \bmod m\right\}. \]
  Thus we have
  \[ (\pr \times \pr)^* (\pr \times \pr)_* \left({}_c \cZ_{\ell m, N', j}\right) = \sum_{\substack{x \in \ZZ /\ell m \ZZ \\ x = 1 \bmod m}}  {}_c \cZ_{\ell m, N', x j}.\]

  Let us now compute $(\pr \times \pr)^* (\pi_1 \times \pi_1)^* {}_c \cZ_{m, N, j}$. Since $Y(m, mN)^2$ is the quotient of $Y(\ell m, mN)^2$ by the group
  \[ \left\{ \left( \tbt x y 0 1, \tbt x z 0 1 \right) : \begin{array}{ll} x,y,z \in \ZZ /\ell m \ZZ, x = 1 \bmod m, \\ y, z = 0 \bmod m\end{array}\right\},\]
  we see that the preimage of $\cC_{m, N, k}$ is the union of the curves $\cC_{\ell m, N', k}$, for $k \in \ZZ /\ell m \ZZ$ congruent to $j$ modulo $m$, each of which is isomorphic to $Y(\ell m, mN)$. By counting degrees, they must be distinct. The modular units ${}_c g_{0, 1/mN}$ and ${}_c g_{0, 1/\ell m N'}$ coincide, and thus we have
  \[ (\pr \times \pr)^* (\pi_1 \times \pi_1)^* {}_c \cZ_{m, N, j} = \sum_{\substack{k \in \ZZ / \ell m \ZZ \\ k = j \bmod m}}  {}_c \cZ_{\ell m, N', k}.\]
  By hypothesis, $j$ is invertible modulo $\ell m$. Thus if $\ell \mid m$, the sets $\{ xj: x = 1 \bmod m\}$ and $\{ k : k = j \bmod m\}$ coincide, and since $(\pr \times \pr)^*$ is clearly injective, we conclude that
  \[ (\pr \times \pr)_* \left({}_c \cZ_{\ell m, N', j}\right) = (\pi_1 \times \pi_1)^* {}_c \cZ_{m, N, j}.\]
  Applying $(\pi_2 \times \pi_2)_*$ gives the result in this case.

  If $\ell \nmid m$, there is exactly one lifting $j_0$ of $j$ to $\ZZ / \ell m \ZZ$ which is not a unit. The matrix $\tbt 1 {j_0} 0 1$ normalizes the subgroup of $\GL_2(\ZZ / mN \ZZ)$ corresponding to $Y(m(\ell), mN)$, and thus defines a curve $\mathcal{A}$ in $Y(m(\ell), mN)^2$ which is isomorphic to $Y(m(\ell), mN)$, consisting of points $(u, v)$ with $v = \tbt 1 {j_0} 0 1 u$; and we have
  \[ (\pi_1 \times \pi_1)^* {}_c \cZ_{m, N, j}  = (\pr \times \pr)_* \left({}_c \cZ_{\ell m, N', j}\right) + (\cA, {}_c g_{0, 1/mN}).\]
  The image of $\mathcal{A}$ under $\pi_2$ is $\cC_{m, N, \ell^{-1} j_0}$; moreover, we have a diagram
  \begin{diagram}
   Y(m(\ell), mN) & \rTo_\cong & \mathcal{A} \\
   \dTo^{\pi_2} && \dTo^{\pi_2 \times \pi_2} \\
   Y(m, mN) & \rTo_\cong & \cC_{m, N, \ell^{-1} j_0}.
  \end{diagram}
  We claim that $(\pi_2)_*\left({}_c g_{0, 1/mN}\right) = {}_c g_{0, 1/mN}$. However, $(\pi_2)_* {}_c g_{0, 1/mN}$ is the pushforward of $\vp_\ell^*({}_c g_{0, 1/mN}) \in \cO(Y(m, mN(\ell)))^\times$ along the natural projection $\cO(Y(m, mN(\ell)))^\times \to \cO(Y(m, mN))^\times$, and the distribution relation of Equation \eqref{eq:dist2} shows that the pushforward of $\vp_\ell^*({}_c g_{0, 1/mN})$ is ${}_c g_{0, 1/mN}$, as required. Hence $ \in Z^2(Y(m(\ell), mN)^2, 1)$ is
  \[(\pi_2 \times \pi_2)_* (\cA, {}_c g_{0, 1/mN}) = (\cC_{m, N, \ell^{-1} j_0}, {}_c g_{0, 1/mN}) = {}_c \cZ_{m, N, \ell^{-1} j_0} = \Delta_\ell^* \left({}_c \cZ_{m, N, j}\right),\]
  as required.
 \end{proof}


  \subsection{The second norm relation for \texorpdfstring{$p \nmid mN$}{p prime to mN}}

   In this section, we shall assume that $N \ge 5$, $m \ge 1$, $j \in \ZZ / m\ZZ$, and $p$ is a prime such that $p \nmid mN$. Our aim is to prove the following theorem:

   \begin{theorem}
    \label{thm:secondnormrelationprime}
    We have
    \begin{multline*}
     \sum_{\substack{k \in \ZZ / mp\ZZ \\ k = j \bmod m \\ p \nmid k}} {}_c \Xi_{mp, N, k} =
     \Big(-\sigma_p + (T_p', T_p') + \left[(p + 1)(\langle p^{-1}\rangle, \langle p^{-1}\rangle) - (\langle p^{-1}\rangle, T_p'^2) - (T_p'^2, \langle p^{-1}\rangle)\right] \sigma_p^{-1} \\
     + \left(\langle p^{-1}\rangle T_p', \langle p^{-1}\rangle T_p'\right) \sigma_p^{-2}  - p \left(\langle p^{-2}\rangle, \langle p^{-2}\rangle\right) \sigma_p^{-3}\Big) {}_c \Xi_{m, N, j}.
    \end{multline*}
   \end{theorem}

   \begin{remark}
    One can formulate a version of this theorem for the zeta elements ${}_c \cZ_{m, N, j}$, from which Theorem \ref{thm:secondnormrelationprime} would follow in the same way as Theorem \ref{thm:secondnormbadprime} follows from Theorem \ref{thm:secondnormbadprime1}. The argument given below can easily be extended to prove this slightly stronger result; however, we shall not pursue this here, as the above statement suffices for our applications.
   \end{remark}

   We begin the proof of Theorem \ref{thm:secondnormrelationprime} by rewriting the $T_p'^2$ terms using a related Hecke operator $S_p'$.

   \begin{proposition}
    As elements of the Hecke algebra of $\Gamma_1(N)$, we have $T_p'^2 = S_p' + (p+1)\langle p^{-1} \rangle R_p$, where $S_p'$ is the double coset of $\begin{pmatrix} p^2 & 0 \\ 0 & 1 \end{pmatrix}$ and $R_p$ is the double coset of $\begin{pmatrix} p & 0 \\ 0 & p\end{pmatrix}$.
   \end{proposition}

   \begin{proof}
    This is a simple computation from the definition of multiplication in the Hecke algebra.
   \end{proof}

   Since $R_p$ acts trivially on everything in sight, the formula of Theorem \ref{thm:secondnormrelationprime} can be written as
   \begin{multline}
    \label{eq:normrelation2}
    \sum_{\substack{k \in (\ZZ / mp\ZZ)^\times \\ k = j \bmod m}} {}_c \Xi_{mp, N, k} =
    \Big(\left((T_p', T_p') - \sigma_p - p (\langle p^{-1}\rangle, \langle p^{-1} \rangle)\sigma_p^{-1}\right)\left(1 +  \left(\langle p^{-1}\rangle , \langle p^{-1}\rangle \right) \sigma_p^{-2}\right) \\ - \left[(\langle p^{-1}\rangle, S_p') + (S_p', \langle p^{-1}\rangle)\right] \sigma_p^{-1}
     \Big) {}_c \Xi_{m, N, j}.
   \end{multline}

   \subsubsection{Evaluation of \texorpdfstring{$(T_p', T_p') {}_c \Xi_{m, N, j}$}{the (Tp,Tp) term}}

    First we shall make a careful study of the operator $(T_p', T_p')$.

    \begin{proposition}
     If $G = \SL_2(\ZZ / p\ZZ)$ and $B$ is the lower-triangular Borel subgroup, then $B \backslash G / B$ has exactly 2 elements, $B$ and its complement (the ``big Bruhat cell'').
    \end{proposition}

    \begin{proof}
     Well-known.
    \end{proof}

    \begin{corollary}
     Let $\Gamma$ be any congruence subgroup of $\SL_2(\ZZ)$ of level prime to $p$, and let $\alpha \in \SL_2(\QQ)$ be integral at $p$. Then the double coset $\Gamma \alpha \Gamma$ is the union of exactly two double cosets of $\Gamma' = \Gamma \cap \Gamma^0(p)$, corresponding to those elements whose reductions modulo $p$ land in the two double cosets of $B$ in $\SL_2(\ZZ / p\ZZ)$.
    \end{corollary}

    \begin{proof}
     This is a consequence of strong approximation for $\SL_2(\ZZ)$. Since $\Gamma$ has level prime to $p$, it surjects onto $\SL_2(\ZZ / p\ZZ)$. Hence we may assume (by left or right multiplying $\alpha$ by an appropriate element of $\Gamma$) that the reduction of $\alpha$ modulo $p$ is the identity.

     We first note that $\Gamma \cap \alpha^{-1} \Gamma \alpha$ is also a congruence subgroup of level prime to $p$, so (by the strong approximation theorem) we see that $\Gamma \backslash \Gamma'$ admits a set of coset representatives lying in $\Gamma \cap \alpha^{-1} \Gamma \alpha$ and thus $\Gamma \alpha \Gamma = \Gamma \alpha \Gamma'$.

     Now let $x = \gamma \alpha \gamma' \in \Gamma \alpha \Gamma'$. We consider the reduction $\bar x$ of $x$ modulo $p$. If this lies in $B$, then (since $\bar \alpha$ and $\bar \gamma'$ are in $B$) we must have $\bar x \in B$ and hence $\gamma \in \Gamma'$; thus $x \in \Gamma' \alpha \Gamma'$.

     On the other hand, let $\mu$ be any element of $\Gamma$ which is not in $\Gamma^0(p)$. If $\bar \gamma \notin B$, then $\bar \gamma \in B \bar \mu B$; so there is some $\sigma \in \Gamma \cap \alpha^{-1} \Gamma \alpha \cap \Gamma^0(p)$ such that $\bar \gamma \in B \bar \mu \bar \sigma$. So $\bar \gamma \bar \sigma^{-1} \bar \mu^{-1} \in B$ and thus $\gamma \in \Gamma' \mu \sigma$. Hence $x \in \Gamma' \mu \sigma \alpha \Gamma'$; but $\alpha^{-1} \sigma \alpha \in \Gamma'$ (since by hypothesis $\alpha = 1 \bmod p$ and thus conjugation by $\alpha$ fixes $\Gamma^0(p)$) and thus $x \in \Gamma' \mu \alpha \Gamma'$.
    \end{proof}

    \begin{corollary}
     For $k \in \ZZ / mp\ZZ$, let $D_{m, N, k}$ denote the curve in $Y(\Gamma_1(N)\cap\Gamma^0(p))^2$ consisting of points of the form $\left(z, z + \tfrac k m\right)$. Then the preimage $\pi_1^{-1}\left(C_{m, N, j}\right) \subseteq Y(\Gamma_1(N)\cap\Gamma^0(p))^2$ consists of exactly 2 components: one is the curve $D_{m, N, k}$ where $k$ is the unique lifting of $j$ to $\ZZ / mp\ZZ$ which is zero mod $p$, and the other is the curve $D_{m, N, k}$ where $k$ is \emph{any} lifting of $j$ to $\ZZ / mp\ZZ$ which is a unit modulo $p$ (the resulting curve being independent of the choice of lifting).
    \end{corollary}

    \begin{proof}
     The preimage of $C_{m, N, j}$ in $\cH \times \cH$ is exactly the set of $(u, v)$ such that $v = \gamma u$ for some $\gamma$ in the double coset $\Gamma_1(N) \begin{pmatrix} 1 & j/m \\ 0 & 1 \end{pmatrix} \Gamma_1(N)$. The above proposition describes the decomposition of this set into double cosets of $\Gamma_1(N) \cap \Gamma^0(p)$, hence the result.
    \end{proof}

    For any $k$ lifting $j$ (unit or non-unit), we may erect the following diagram of modular curves:
    
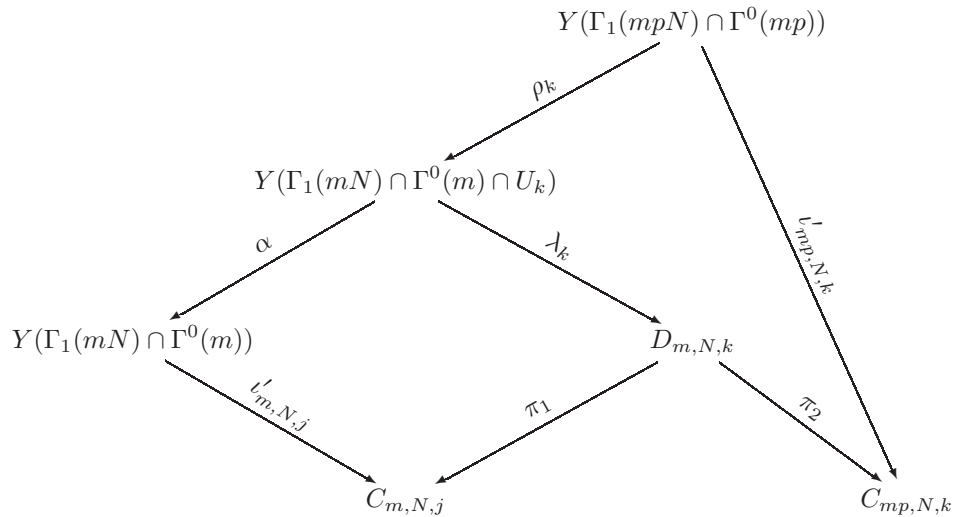
\begin{diagram}
     &&& & Y(\Gamma_1(mpN) \cap \Gamma^0(mp)) \\
     && & \ldTo^{\rho_k} & &  \rdTo(2, 6)^{\iota_{mp, N, k}'}\\
     & & Y(\Gamma_1(mN) \cap \Gamma^0(m) \cap U_k) & &\\
     & \ldTo^\alpha && \rdTo^{\lambda_k} \\
     Y(\Gamma_1(mN) \cap \Gamma^0(m)) & & & & D_{m, N, k} \\
     & \rdTo^{\iota_{m, N, j}'} & & \ldTo^{\pi_1} && \rdTo^{\pi_2}\\
     && C_{m, N, j} & & & & C_{mp, N, k}
    \end{diagram}

    Here $U_k$ is the preimage in $\SL_2(\ZZ)$ of the subgroup
    \[ B \cap \begin{pmatrix} 1 & k \\ 0 & m \end{pmatrix}^{-1} B \begin{pmatrix} 1 &k \\0& m \end{pmatrix} = B \cap \begin{pmatrix} 1 & k \\0 & 1 \end{pmatrix}^{-1} B \begin{pmatrix} 1 &k \\0& 1 \end{pmatrix}.\]
    (The equality follows from the fact that conjugation by $\begin{pmatrix} 1 & 0 \\ 0 & m\end{pmatrix}$ fixes $B$.) This subgroup is just $B$ if $k \in p\ZZ$; otherwise it is the subgroup
    \[ \left\{ \begin{pmatrix} u^{-1} & 0 \\ k^{-1}(u - u^{-1}) & u \end{pmatrix} : u \in (\ZZ / p\ZZ)^\times \right\},\]
    which is a maximal torus in $\SL_2(\ZZ / p\ZZ)$. The square at the bottom left of the diagram is Cartesian. The maps $\alpha$, $\rho_k$ and $\pi_1$ are the natural projection maps, and the remaining maps are defined by
    \[\begin{array}{ccccc}
     \iota_{m, N, j}' & : & z & \mapsto & \left( \frac{z}{m}, \frac{z + j}{m}\right)\\
     \iota_{mp, N, k}' & : & z & \mapsto & \left( \frac{z}{mp}, \frac{z + k}{mp}\right)\\
     \pi_2&:& (u,v) &\mapsto& \left(\tfrac u p,\tfrac v p\right)\\
     \lambda_k &:& z &\mapsto& \left(\frac{z}{m}, \frac{z + k}{m}\right)
    \end{array}\]

    \begin{definition}
     Let $a$ and $b$ be the unique elements of $\ZZ / pm\ZZ$ congruent to $j$ modulo $m$ and such that $a = 0 \bmod p$ and $b = 1 \bmod p$.
    \end{definition}

    An application of Lemma \ref{lem:pushpull} shows that we have:

    \begin{corollary}
     \label{cor:normrelationRHS}
     For any $\alpha \in \ZZ / mN\ZZ$, we have
     \[ (T_p', T_p')({}_c \Xi_{m, N, j, \alpha}) = \left(C_{mp, N, a}, (\pi_2 \circ \lambda_a)_* {}_c g_{0, \alpha/mN}\right) +
      \left(C_{mp, N, b}, (\pi_2 \circ \lambda_b)_* {}_c g_{0, \alpha/mN}\right).
     \]
    \end{corollary}

    (It is convenient to allow $\alpha \ne 1$ here, for reasons that will become clear below.)

    We first consider the term for $a$. Here we have $U_a = \Gamma^0(p)$, so $\Gamma_1(mN) \cap \Gamma^0(m) \cap U_a = \Gamma_1(mN) \cap \Gamma^0(mp)$. Since $p \mid a$, we see that $\pi_2 \circ \lambda_a$ can also be expressed as a composition
    \begin{equation}
     \label{eq:zero term factorization}
     \begin{diagram}
      Y(\Gamma_1(mN) \cap \Gamma^0(mp))\\
      \dTo^{z \mapsto z/p}\\
      Y(\Gamma_1(mN) \cap \Gamma^0(m))\\
      \dTo^{z \mapsto \left(\tfrac z m, \tfrac {z + \text{``$p^{-1}$''} j}{m}\right)}\\
      C_{mp, N, a} = C_{m, N, \text{``$p^{-1}$''} j}
     \end{diagram}
    \end{equation}
    where $\text{``$p^{-1}$''}$ is the inverse of $p$ in $\ZZ / m\ZZ$.

    \begin{proposition}
     The pushforward of ${}_c g_{0, \alpha/mN}$ to $Y(\Gamma_1(mN) \cap \Gamma^0(m))$ along the first map in \eqref{eq:zero term factorization} is ${}_c g_{0, \alpha/mN} \cdot \left( {}_c g_{0, \text{``$p^{-1}$''} \alpha/mN}\right)^p$.
    \end{proposition}

    \begin{proof} See \cite[2.13.2]{kato04}.
    \end{proof}

    The second map in \eqref{eq:zero term factorization} is just $\iota'_{m, N, \text{``$p^{-1}$''} j}$, so we deduce that
    \begin{align*}
     \left(C_{mp, N, a}, (\pi_2 \circ \lambda_a)_* {}_c g_{0, \alpha/mN}\right) &= \Big(C_{m, N, \text{``$p^{-1}$''} a}, (\iota'_{m, N, \text{``$p^{-1}$''} j})_* \left({}_c g_{0, \alpha/mN} \cdot \left( {}_c g_{0, \text{``$p^{-1}$''} \alpha/mN}\right)^p\right)\Big)\\
     &= {}_c \Xi_{m, N, \text{``$p^{-1}$''} j, \alpha} + p\, {}_c \Xi_{m, N, \text{``$p^{-1}$''} j, \text{``$p^{-1}$''} \alpha}\\
     &= (\sigma_p + p (\langle p^{-1}\rangle, \langle p^{-1} \rangle)\sigma_p^{-1}) {}_c \Xi_{m, N, j, \alpha}.
    \end{align*}

    \begin{corollary}
     \label{cor:TpTpterm}
     For any $\alpha \in (\ZZ / mN\ZZ)^\times$, we have
     \[ \left((T_p', T_p') - \sigma_p - p (\langle p^{-1}\rangle, \langle p^{-1} \rangle)\sigma_p^{-1}\right){}_c \Xi_{m, N, j, \alpha} = \Big(C_{mp, N, b}, (\pi_2 \circ \lambda_b)_* \left({}_c g_{0, \alpha/mN}\right)\Big).\]
     In particular,
     \begin{multline*}
      \left((T_p', T_p') - \sigma_p - p (\langle p^{-1}\rangle, \langle p^{-1} \rangle)\sigma_p^{-1}\right)\left(1 +  \left(\langle p^{-1}\rangle , \langle p^{-1}\rangle \right) \sigma_p^{-2}\right) {}_c \Xi_{m, N, j} \\= \Big(C_{mp, N, b}, (\pi_2 \circ \lambda_b)_* \left({}_c g_{0, 1/mN} \cdot {}_c g_{0, \text{``$p^{-1}$''}/mN}\right)\Big).
     \end{multline*}
    \end{corollary}

    \begin{proof}
     The first formula is immediate from \eqref{eq:zero term factorization} and the evaluation of the $C_{mp, N, a}$ term above. The second formula follows by summing the first formula for $\alpha = 1$ and for $\alpha = p^{-1}$.
    \end{proof}

   \subsubsection{Evaluation of the norm term}

    We want to compare the right-hand side of the formula in Corollary \ref{cor:TpTpterm}  with the sum of the ${}_c \Xi_{mp, N, k}$ for all unit liftings $k$ of $j$. To do this, we shall use the fact that all the terms ${}_c \Xi_{mp, N, k}$ may be written as the pushforwards of modular units on the same modular curve $C_{mp, N, b}$. More precisely, if $k$ and $\ell$ are liftings of $j$ to $\ZZ / mp\ZZ$ which are both units modulo $p$, we have a diagram
    \begin{diagram}
     \cH & \rTo^{\alpha_{k\ell}} & \cH \\
     \dTo & & \dTo\\
     Y(\Gamma_1(mN) \cap \Gamma^0(m) \cap U_k) & \rTo^\cong & Y(\Gamma_{1}(mN) \cap \Gamma^0(m) \cap U_\ell)\\
     \dTo^{\lambda_k} & & \dTo^{\lambda_{\ell}} \\
     D_{m, N, \ell} &\rEq & D_{m, N, k}
    \end{diagram}
    where $\alpha_{k\ell}$ is any matrix of the form $\begin{pmatrix} 1 & 0 \\ v & 1 \end{pmatrix}$ with $v \in mN \ZZ$ congruent to $\tfrac{1}{k} - \tfrac{1}{\ell}$ modulo $p$.

    Consequently, we can write ${}_c \Xi_{mp, N, k} = (C_{mp, N, b},  (\pi_2 \circ \lambda_b)_* f_k)$, where
    \[ f_k = \alpha_{bk}^* (\rho_k)_* \left({}_c g_{0, 1/mpN}\right).\]
    We may regard $\cO(Y(\Gamma_1(mpN) \cap \Gamma^0(mp)))^\times$ as a $\SL_2(\FF_p)$-module in the obvious way, since $\Gamma_1(mpN) \cap \Gamma^0(mp)$ is the kernel of the surjective reduction map $\Gamma_1(mN) \cap \Gamma^0(m) \twoheadrightarrow \SL_2(\FF_p)$. With this convention we have
    \[
     (\rho_k)_*\left({}_c g_{0, 1/mpN}\right) = \prod_{u \in \FF_p^\times} \begin{pmatrix} u^{-1} & 0 \\ k^{-1}(u - u^{-1}) & u \end{pmatrix}^* {}_c g_{0, 1/mpN}
    \]
    and thus
    \begin{align*}
     f_k &=
     \prod_{u \in \FF_p^\times} \left[\begin{pmatrix} u^{-1} & 0 \\ k^{-1}(u - u^{-1}) & u \end{pmatrix}\begin{pmatrix} 1 & 0 \\ 1 - k^{-1} & 1 \end{pmatrix} \right]^* {}_c g_{0, 1/mpN}\\
     &= \prod_{u \in \FF_p^\times} \begin{pmatrix} u^{-1}  & 0 \\ u - k^{-1} u^{-1} & u \end{pmatrix}^* {}_c g_{0, 1/mpN}.
    \end{align*}

    Let $K$ be the set of possible values of $k$, i.e.~the set of elements of $\ZZ / mp\ZZ$ congruent to $j$ modulo $m$ and not divisible by $p$. Then as $k$ varies over $K$, for each fixed $u$, the expression $u - k^{-1} u^{-1}$ takes every value in $\FF_p$ exactly once with the exception of $u$, since $k^{-1} u^{-1}$ takes every value except 0. So
    \[ \prod_{k \in K} f_k = \prod_{\substack{u, v \in \FF_p \\ u \ne 0 \\ v \ne u}} {}_c g_{v_0 / mpN, u_1 /mpN}.\]
    Here by $x_1$ and $x_0$ for $x \in \FF_p$ we mean any element of $\ZZ/mpN\ZZ$ congruent to $x$ mod $p$ and to $1$ (resp. 0) modulo $mN$.

    We find that
    \begin{subequations}\begin{align}
     \prod_{u, v \in \FF_p} {}_c g_{v_0 / mpN, u_1 /mpN} &= {}_c g_{0, 1/mN} \\
     \prod_{v \in \FF_p} {}_c g_{v_0 / mpN, 0_1 /mpN} &= \begin{pmatrix} 1 & 0 \\ 0 & p \end{pmatrix}^* {}_c g_{0, \beta / mN} \\
     \prod_{u \in \FF_p} {}_c g_{u_0 / mpN, u_1 / mpN} &= \begin{pmatrix} 1 & 0 \\ 1_0 & 1 \end{pmatrix}^* \prod_{u \in \FF_p} {}_c g_{0, u_1 / mpN}= \begin{pmatrix} 1 & 0 \\ 1_0 & 1 \end{pmatrix}^* \begin{pmatrix} p & 0 \\ 0 & 1 \end{pmatrix}^* {}_c g_{0, 1 / mN}\\
     {}_c g_{0_0 / mpN, 0_1 / mpN} &= {}_c g_{0, \beta / mN}
    \end{align}\end{subequations}
    where $\beta$ is the inverse of $p$ in $\ZZ / mN\ZZ$.

    Combining the above, we have
    \begin{multline}
     \label{eq:normrelationLHS}
     \sum_{\substack{k \in \ZZ / mp\ZZ \\ k = j \bmod m \\ p \nmid k}} {}_c \Xi_{mp, N, k} = \Big(C_{mp, N, b}, (\pi_2 \circ \lambda_b)_* ({}_c g_{0, 1/mN} \cdot {}_c g_{0, \beta / mN})\Big)\\
     - \left(C_{mp, N, b}, (\pi_2 \circ \lambda_b)_*  \begin{pmatrix} 1 & 0 \\ 0 & p \end{pmatrix}^* {}_c g_{0, \beta / mN}\right)\\
     - \left(C_{mp, N, b}, (\pi_2 \circ \lambda_b)_*  \begin{pmatrix} 1 & 0 \\ 1_0 & 1 \end{pmatrix}^* \begin{pmatrix} p & 0 \\ 0 & 1 \end{pmatrix}^* {}_c g_{0, 1 / mN}\right).
    \end{multline}

    Combining the first term on the right-hand side of \eqref{eq:normrelationLHS} with Corollary \ref{cor:normrelationRHS}, we see that Theorem \ref{thm:secondnormrelationprime} is equivalent to

    \begin{proposition}
     \label{prop:normrelation2}
     We have
     \begin{multline*}
      \Big(\left[(\langle p^{-1}\rangle, S_p') + (S_p', \langle p^{-1}\rangle)\right] \sigma_p^{-1}
      \Big) {}_c \Xi_{m, N, j} = \left(C_{mp, N, b}, (\pi_2 \circ \lambda_b)_*  \begin{pmatrix} 1 & 0 \\ 0 & p \end{pmatrix}^* {}_c g_{0, \beta / mN}\right)\\
     + \left(C_{mp, N, b}, (\pi_2 \circ \lambda_b)_*  \begin{pmatrix} 1 & 0 \\ 1_0 & 1 \end{pmatrix}^* \begin{pmatrix} p & 0 \\ 0 & 1 \end{pmatrix}^* {}_c g_{0, 1 / mN}\right).
     \end{multline*}
    \end{proposition}

  \subsubsection{The first term in Proposition \ref{prop:normrelation2}}

    We now calculate how $S_p'$ acts on ${}_c \Xi_{m, N, j}$. We may describe the correspondence $S_p'$ in terms of the subgroup $\Gamma_0^0(p) = \Gamma_0(p) \cap \Gamma^0(p)$; we have $S_p' = (\pi_2')_* (\pi_1')^*$, where $\pi_1'$ and $\pi_2'$ are the two maps from $Y_1(\Gamma_1(N) \cap \Gamma^0_0(p))$ to $Y_1(N)$ given by $z \mapsto pz$ and $z \mapsto z/p$.

    An application of the strong approximation theorem shows (as usual) that the preimage $(\pi_1' \times 1)^{-1} C_{m, N, j} \subseteq Y_1(\Gamma_1(N) \cap \Gamma^0_0(p)) \times Y_1(N)$ is the single curve $F_{m, N, j}$ given by the set of points of the form $\left(\tfrac z p, z + \tfrac j m\right)$.

    Applying Lemma \ref{lem:pushpull} once more, we have a Cartesian square of curves (up to birational equivalence)
    \begin{diagram}
     & & Y(\Gamma_1(mN) \cap \Gamma^0(m) \cap \Gamma^0_0(p)) & & \\
     & \ldTo^{pz \mapsfrom z} & & \rdTo^{z \mapsto \left(\tfrac{z}{m}, \tfrac{pz + j}{m}\right)} & \\
     Y(\Gamma_1(mN) \cap \Gamma^0(m)) & & & & F_{m, N, j}\\
     & \rdTo_{z \mapsto \left(\tfrac{z}{m}, \tfrac{z + j}{m}\right)} & & \ldTo_{(pu, v) \mapsfrom (u, v)} \\
     & & C_{m, N, j}
    \end{diagram}

    The functoriality of pushforward maps gives the following:

    \begin{proposition}
     We have
     \[ (S_p', 1) {}_c \Xi_{m, N, j, \alpha} = \left((\pi_2' \times 1) F_{m, N, j}, \phi_* \left(\begin{pmatrix} p & 0 \\ 0 & 1 \end{pmatrix}^* {}_c g_{0, \alpha/mN}\right)\right)\]
     where $\phi$ is the map
     \begin{align*}
      Y(\Gamma_1(mN) \cap \Gamma^0(m) \cap \Gamma_0^0(p)) &\to (\pi_2' \times 1) F_{m, N, j} \subset Y_1(N)^2 \\
      z &\mapsto \left(\tfrac z {mp}, \tfrac{pz + j}{m}\right).
     \end{align*}
    \end{proposition}

    We first identify the curve $(\pi_2' \times 1) F_{m, N, j}$.

    \begin{proposition}
     \label{prop:gammadoubleprime}
     We have $(\pi_2' \times 1) F_{m, N, j} = (1 \times \langle p\rangle)^{-1} C_{mp, N, k}$, for any integer $k$ congruent to $p^{-1} j$ modulo $m$ and not divisible by $p$.

     More precisely, if $k = 1 \bmod p$, and $\gamma''$ is a suitable element of $\SL_2(\ZZ)$ which we shall construct below, then there is a commutative diagram
     \begin{diagram}
      Y(\Gamma_1(mN) \cap \Gamma^0(m) \cap \Gamma_0^0(p)) & \rTo^{z \mapsto \gamma'' z} & Y(\Gamma_1(mN) \cap \Gamma^0(m) \cap U_k)\\
      \dTo^{z \mapsto \left(\tfrac z {mp}, \tfrac{pz + j}{m}\right)} & & \dTo_{z \mapsto \left(\tfrac z {mp}, \tfrac{z + k}{mp}\right)}\\
      Y_1(N)^2 &\rTo^{(1 \times \langle p\rangle)} & Y_1(N)^2,
     \end{diagram}
     where $U_k$ is the level $p$ congruence subgroup from the previous section.
    \end{proposition}

    \begin{proof}
     We note the following matrix identity, which is easy to verify (although tedious to find): for any elements $p,x,y$ of a field $F$, we have
     \[
      \begin{pmatrix} \tfrac y x & 0 \\ \tfrac 1 x - \tfrac p y & \tfrac x y \end{pmatrix}
      \begin{pmatrix} 1 & \tfrac x p \\ 0 & 1 \end{pmatrix}
      \begin{pmatrix} 1 & 0 \\ \tfrac {p^2} y - \tfrac p x & 1 \end{pmatrix}
      = \begin{pmatrix} p & \tfrac y p \\ 0 & \tfrac 1 p \end{pmatrix}.
     \]
     In particular, taking $F = \Qp$ and $x, y \in \Zp^\times$, we see that the double cosets of $\SL_2(\Zp)$ in $\SL_2(\Qp)$ generated by $\begin{pmatrix} 1 & x/p \\ 0 & 1 \end{pmatrix}$ and $\begin{pmatrix} p & y/p \\ 0 & 1/p \end{pmatrix}$ are equal to each other and independent of $x$ and $y$.

     Since both $\begin{pmatrix} p & y/p \\ 0 & 1/p \end{pmatrix}$ and its inverse have entries in $\tfrac 1 p \Zp$, it follows that
     \[ \begin{pmatrix} 1 & x/p \\ 0 & 1 \end{pmatrix}\gamma \in \SL_2(\Zp) \begin{pmatrix} p & y/p \\ 0 & 1/p \end{pmatrix} \]
     for any $\gamma \in \SL_2(\Zp)$ congruent to $\begin{pmatrix} 1 & 0 \\ -p/x & 1 \end{pmatrix}$ modulo $p^2$. (In fact, one can check that it suffices for the matrix to lie in $\begin{pmatrix} 1 + p\Zp & \Zp \\ -p/x + p^2 \Zp & 1 + p\Zp\end{pmatrix}$.)

     If $x, y$ are in $\Zp^\times \cap \QQ$, and we choose $\gamma$ to be in $\SL_2(\ZZ)$ and congruent to $\begin{pmatrix} 1 & 0 \\ -p/x & 1 \end{pmatrix}$ modulo $p^2$, then the matrix
     \[ \gamma' = \begin{pmatrix} 1 & x/p \\ 0 & 1 \end{pmatrix} \gamma \begin{pmatrix} p & y/p \\ 0 & 1/p \end{pmatrix}^{-1} \]
     will be in $\SL_2(\QQ)$ and will be $p$-adically integral. If we choose $\gamma$ to be $\ell$-adically close to the identity for some prime $\ell \ne p$, then $\gamma'$ will be $\ell$-adically close to $ \begin{pmatrix} 1 & x/p \\ 0 & 1 \end{pmatrix} \gamma \begin{pmatrix} p & y/p \\ 0 & 1/p \end{pmatrix}^{-1} = \begin{pmatrix} 1/p & (px-y)/p \\ 0 & p \end{pmatrix}$.

     So if $x, y \in \QQ$ are units at $p$ and satisfy $y = px \bmod 1$, we may choose $\gamma, \gamma' \in \SL_2(\ZZ)$ such that:
     \begin{itemize}
      \item $\gamma \in \Gamma_1(m^2 N)$,
      \item $\gamma = \begin{pmatrix} 1 & 0 \\ -p/x & 1 \end{pmatrix} \pmod {p^2}$,
      \item $\gamma' = \begin{pmatrix} * & * \\ 0 & 1/p\end{pmatrix} \pmod N$,
      \item the identity $\begin{pmatrix} 1 & x/p \\ 0 & 1 \end{pmatrix} \gamma = \gamma' \begin{pmatrix} p & y/p \\ 0 & p \end{pmatrix}$ holds.
     \end{itemize}
     We now take $y = j/m$, and $x = k/m$ for any $k$ congruent to $p^{-1} j$ modulo $m$ and invertible modulo $p$. Then we obtain a commutative diagram
     \begin{diagram}
      \cH & \rTo^{z \mapsto \gamma z} & \cH \\
      \dTo^{z \mapsto \left(z, p^2 z + \tfrac{j}{m}\right)} & & \dTo_{z \mapsto \left(z, z + \tfrac{k}{mp}\right)} \\
      Y_1(N)^2 & \rTo^{(1 \times \langle p \rangle)} & Y_1(N)^2
     \end{diagram}
     or equivalently
     \begin{diagram}
      \cH & \rTo^{z \mapsto \gamma'' z} & \cH \\
      \dTo^{z \mapsto \left(\tfrac{z}{mp}, \tfrac{pz + j}{m}\right)} & & \dTo_{z \mapsto \left(\tfrac{z}{mp}, \tfrac{z+k}{mp}\right)} \\
      Y_1(N)^2 & \rTo^{(1 \times \langle p \rangle)} & Y_1(N)^2
     \end{diagram}
     where $\gamma'' = \begin{pmatrix} mp & 0 \\ 0 & 1 \end{pmatrix}  \gamma \begin{pmatrix} mp & 0 \\ 0 & 1 \end{pmatrix}^{-1}$. Note that  $\gamma''$ is in $\Gamma_1(mN) \cap \Gamma^0(mp)$, and is congruent modulo $p$ to
     \[ \begin{pmatrix} 1 & 0 \\ -1/mx & 1 \end{pmatrix} = \begin{pmatrix} 1 & 0 \\ -1/k & 1 \end{pmatrix}.\]
     In the preceding diagram, the left vertical map factors through $Y(\Gamma_1(mN) \cap \Gamma^0(m) \cap \Gamma_0^0(p))$, and $\gamma'$ conjugates this onto $\Gamma_1(mN) \cap \Gamma^0(m) \cap U_k$; so we finally obtain the diagram stated in the proposition.
    \end{proof}

    \begin{corollary}
     \label{cor:Sp2term}
     In the notation of the preceding subsection, we have
     \[ (S_p', \langle p^{-1} \rangle) \cdot \sigma_p^{-1} \cdot {}_c \Xi_{m, N, j} =
     \left(C_{mp, N, b}, (\pi_2 \circ \lambda_b)_*  \begin{pmatrix} 1 & 0 \\ 1_0 & 1 \end{pmatrix}^* \begin{pmatrix} p & 0 \\ 0 & 1 \end{pmatrix}^* {}_c g_{0, 1 / mN}\right).\]
    \end{corollary}

    \begin{proof}
     This follows from the previous proposition (and its proof), since the right-hand vertical map in the diagram of the proposition is the same as $\lambda_k$ above, and $(\gamma'')^{-1}$ represents the coset $\begin{pmatrix} 1 & 0 \\ 1_0 & 1 \end{pmatrix}$.
    \end{proof}

  \subsubsection{The second term in Proposition \ref{prop:normrelation2}}

    Now we are left to analyse the operator $(1, S_p')$. To simplify the analysis we shall also consider the operator $S_p$ given by $(\pi_1')_* (\pi_2')^*$ (rather than $S_p' = (\pi_2')_* (\pi_1')^*$); this is the operator associated to the double coset $\begin{pmatrix} 1 & 0 \\ 0 & p^2 \end{pmatrix}$ and is related to $S_p'$ by the formula
    \[ S_p' = \langle p^{-2} \rangle^* S_p.\]

    Again, we find that the preimage $(\pi_2')^{-1} C_{m, N, j}$ in $Y_1(N) \times Y(\Gamma_1(N) \cap \Gamma_0^0(p))$ is a single irreducible curve $F_{m, N, j}$ given by points of the form $(z, p(z + j/m))$.

    \begin{proposition}
     We have
     \[ (1, S_p)\left({}_c \Xi_{m, N, j, \alpha}\right) = \left((1 \times \pi_1')(F_{m, N, j}), \phi_* \begin{pmatrix} 1 & 0 \\ 0 & p \end{pmatrix}^* {}_c g_{0, \alpha/mN}\right),\]
     where the morphism $\phi$ is defined by
     \begin{align*}
      Y(\Gamma_1(mN) \cap \Gamma^0(m) \cap \Gamma^0_0(p)) &\to Y_1(N)^2\\
      z & \mapsto \left( \tfrac z {mp}, \tfrac{pz + p^2 j}{m}\right).
     \end{align*}
    \end{proposition}

    \begin{proof}
     Closely analogous to the previous case.
    \end{proof}

    We also have a matrix identity
    \[ \begin{pmatrix} 1 & \tfrac x p \\ 0 & 1 \end{pmatrix}
       \begin{pmatrix} 1 & 0 \\ -\tfrac px & 1 \end{pmatrix} =
       \begin{pmatrix} 0 & x \\ -\tfrac1x & p + \tfrac{p^2 y}{x} \end{pmatrix}
       \begin{pmatrix} p & py \\ 0 & \tfrac{1}{p}\end{pmatrix}
    \]
    from which we may deduce that if $x, y$ are rational numbers which are units at $p$ and such that $x = py \bmod 1$, there exist $\gamma, \gamma' \in \SL_2(\ZZ)$ such that:
    \begin{itemize}
     \item $\gamma \in \Gamma_1(m^2 N)$,
     \item $\gamma = \begin{pmatrix} 1 & 0 \\ -\tfrac px & 1 \end{pmatrix} \pmod {p^2}$,
     \item the identity
     \[ \begin{pmatrix} 1 & \tfrac x p \\ 0 & 1 \end{pmatrix} \gamma = \gamma' \begin{pmatrix} p & py \\ 0 & \tfrac{1}{p}\end{pmatrix}\]
     holds,
     \item $\gamma'$ is congruent to $\begin{pmatrix} * & * \\ 0 & p \end{pmatrix}$ modulo $N$.
    \end{itemize}
    Thus the diagram
    \begin{diagram}
     \cH & \rTo^{z \mapsto \gamma z} &\cH\\
     \dTo^{z \mapsto (z, p^2(z + y))} & & \dTo_{z \mapsto (z, z + x/p)} \\
     Y_1(N)^2 & \rTo^{(1 \times \langle p\rangle)}& Y_1(N)^2
    \end{diagram}
    commutes. We take $y = j/m$, and $x = k/m$ where $k$ is congruent to $pj$ modulo $m$ and not divisible by $p$. Letting $\gamma'' = \begin{pmatrix} mp & 0 \\ 0 & 1 \end{pmatrix} \gamma \begin{pmatrix} mp & 0 \\ 0 & 1 \end{pmatrix}^{-1}$ as before, we have the diagram
    \begin{diagram}
     \cH & \rTo^{z \mapsto \gamma'' z} &\cH\\
     \dTo^{z \mapsto \left(\tfrac{z}{mp}, \tfrac{pz + p^2 j}{m}\right)} & & \dTo_{z \mapsto \left(\tfrac{z}{mp}, \tfrac{z + k}{mp}\right)} \\
     Y_1(N)^2 & \rTo^{(1 \times \langle p\rangle)}& Y_1(N)^2
    \end{diagram}

    Again, this shows that $(1 \times \pi_1')(F_{m, N, j}) = (1 \times \langle p \rangle)^{-1} C_{mp, N, k}$. If we choose $k$ to be $1$ modulo $p$, then the right vertical map factors through $\Gamma_1(mN) \cap \Gamma^0(m) \cap U_k$, and the isomorphism between the two is given by $\gamma''$, which is in $\Gamma_1(mN) \cap \Gamma^0(p)$ and thus acts trivially on $\begin{pmatrix} 1 & 0 \\ 0 & p \end{pmatrix}^* {}_c g_{0, \alpha/mN}$. Thus we have
    \[ (1, S_p) \sigma_p \cdot  {}_c \Xi_{m, N, j, \alpha} = (1 \times \langle p \rangle)^* \left(C_{mp, N, b}, (\pi_2 \circ \lambda_b)_* \begin{pmatrix} 1 & 0 \\ 0 & p \end{pmatrix}^* {}_c g_{0, \alpha/mN}\right).\]

    Taking $\alpha = \beta$ (the inverse of $p$ modulo $mN$), and using the formula ${}_c \Xi_{m, N, j, t} = (\langle t \rangle , \langle t \rangle) \sigma_t^2 {}_c \Xi_{m, N, j}$, we see that
    \begin{align*}
     \left(C_{mp, N, b}, (\pi_2 \circ \lambda_b)_* \begin{pmatrix} 1 & 0 \\ 0 & p \end{pmatrix}^* {}_c g_{0, \beta/mN}\right) &=
     (1,\langle p^{-1} \rangle) (1, S_p) {}_c \Xi_{m, N, \beta j, \beta}\\
     &= (1,\langle p^{-1} \rangle) (1, \langle p \rangle^{2} S_p') \sigma_p {}_c \Xi_{m, N, j, \beta}\\
     &= (1,\langle p^{-1} \rangle) (1, \langle p \rangle^{2} S_p')\sigma_p (\langle p^{-1} \rangle , \langle p^{-1} \rangle)\sigma_p^{-2} {}_c \Xi_{m, N, j} \\
     &= (\langle p^{-1}\rangle, S_p') \sigma_p^{-1} {}_c \Xi_{m, N, j}
    \end{align*}
    as required, completing the proof of Proposition \ref{prop:normrelation2} and hence of Theorem \ref{thm:secondnormrelationprime}.


  \subsection{The second norm relation: higher powers of p}

   We shall also need to know how to calculate $\norm_{m}^{p^k m} {}_c \Xi_{p^k m, N, j}$ for $k = 2, 3$. This is less central to our theory than the $k = 1$ case, but it will be needed in order to compare the elements we construct for $N$ coprime to $p$ with their ``$p$-stabilized'' versions.

   \begin{theorem}
    \label{thm:stronghighernormrel}
    For $p \nmid m N$ we have
    \begin{multline*}
     \norm_{mp}^{m p^2} \left({}_c \Xi_{m p^2, N, j}\right) = (T_p', T_p') {}_c \Xi_{mp, N, j} \\
     + \Big( p(\langle p^{-1} \rangle, \langle p^{-1} \rangle) - (\langle p \rangle^{-1},  (T_p')^2)- ( (T_p')^2, \langle p \rangle^{-1}) \\
     + 2 (\langle p^{-1} \rangle T_p', \langle p^{-1} \rangle T_p')\sigma_p^{-1} - p (\langle p^{-2} \rangle, \langle p^{-2} \rangle)\sigma_p^{-2}\Big) {}_c \Xi_{m, N, j}
    \end{multline*}
    and
    \begin{multline*}
     \norm_{m p^2}^{m p^3} \left({}_c \Xi_{m p^3, N, j}\right) = (T_p', T_p') {}_c \Xi_{mp^2, N, j}\\
     + \Big(p(\langle p^{-1} \rangle, \langle p^{-1} \rangle) - (\langle p \rangle^{-1},  (T_p')^2)- ( (T_p')^2, \langle p \rangle^{-1})\Big) {}_c \Xi_{mp, N, j}\\
     + \Bigg(  (\langle p^{-1} \rangle, \langle p^{-1} \rangle) \bigg(2 (T_p', T_p') - ( (\langle p \rangle^{-1},  (T_p')^2)+ ( (T_p')^2, \langle p \rangle^{-1})) \sigma_p^{-1} \bigg)\Bigg) {}_c \Xi_{m, N, j}
    \end{multline*}
   \end{theorem}

   Recall the operator $S_p'$ which appeared above, satisfying $S_p' = (T_p')^2 - (p+1) \langle p^{-1} \rangle$. In terms of these operators, the formulae we wish to prove are
    \begin{multline*}
     \norm_{mp}^{m p^2} \left({}_c \Xi_{m p^2, N, j}\right) = (T_p', T_p') {}_c \Xi_{mp, N, j} \\
     + \Big( -(p+2)(\langle p^{-1} \rangle, \langle p^{-1} \rangle) - (\langle p \rangle^{-1},  S_p')- ( S_p', \langle p \rangle^{-1}) \\
     +  (\langle p^{-1} \rangle, \langle p^{-1} \rangle)\sigma_p^{-1} \left(2 (T_p', T_p') - p (\langle p^{-1} \rangle, \langle p^{-1} \rangle)\sigma_p^{-1}\right)\Big) {}_c \Xi_{m, N, j}
    \end{multline*}
    and
    \begin{multline*}
     \norm_{m p^2}^{m p^3} \left({}_c \Xi_{m p^3, N, j}\right) = (T_p', T_p') {}_c \Xi_{mp^2, N, j}\\
     - \Big((p+2)(\langle p^{-1} \rangle, \langle p^{-1} \rangle) + (\langle p \rangle^{-1},  S_p') + ( S_p', \langle p \rangle^{-1})\Big) {}_c \Xi_{mp, N, j}\\
     + \Bigg(  (\langle p^{-1} \rangle, \langle p^{-1} \rangle) \bigg(2 (T_p', T_p') - \left[ (2p + 2)(\langle p^{-1} \rangle, \langle p^{-1} \rangle) + (\langle p \rangle^{-1},  S_p') + (S_p', \langle p \rangle^{-1})\right] \sigma_p^{-1} \bigg)\Bigg) {}_c \Xi_{m, N, j}.
    \end{multline*}

   A routine but unpleasant computation (in which the use of Sage \cite{sage} was found to be invaluable) shows that Theorem \ref{thm:stronghighernormrel}, together with Theorem \ref{thm:secondnormrelationprime}, implies the following formulae for the norms to level prime to $p$:

   \begin{theorem}
    \label{conj:weaknormrel}
    If $p \nmid N$, we have
    \begin{enumerate}
     \item[(a)]
     \[
      \norm_{m}^{p^2 m} \left({}_c \Xi_{p^2 m, N, j}\right) =
      p \sigma_p^2 \Big( (p-1)(1 - (\langle p^{-1} \rangle, \langle p^{-1} \rangle) \sigma_p^{-2})
      - \left((T_p', T_p') \sigma_p^{-1} + (p-1)\right) \underline{P}_p(p^{-1} \sigma_p^{-1}) \Big)
     \]
     \item[(b)]
     \begin{multline*}
      \norm_m^{p^3 m} \left({}_c \Xi_{p^3 m, N, j}\right) = p^2 \sigma_p^3 \Big(
       (p-1)(1 - (\langle p^{-1} \rangle, \langle p^{-1} \rangle) \sigma_p^{-2})\\
       - (p^{-1} \sigma_p^{-2} (T_{p^2}', T_{p^2}') + (p-1)p^{-1} \sigma_p^{-1} (T_{p}', T_{p}') + (p-1))\underline{P}_p(p^{-1} \sigma_p^{-1})\Big)
     \end{multline*}
    \end{enumerate}
    Here $\underline{P}_p$ is the operator-valued Euler factor at $p$ given by
    \begin{multline*} \underline{P}_p(X) = 1 - (T_p', T_p') X + \Big( p( (T_p')^2, \langle p^{-1}\rangle) + p(\langle p^{-1}\rangle,  (T_p')^2 ) - 2p^2(\langle p^{-1}\rangle, \langle p^{-1} \rangle)\Big) X^2 \\- p^2 (\langle p^{-1} \rangle T_p',\langle p^{-1} \rangle T_p')X^3 + p^4 (\langle p^{-2} \rangle, \langle p^{-2} \rangle) X^4,
    \end{multline*}
   \end{theorem}

   \subsubsection{Evaluation of the \texorpdfstring{$(T_p', T_p')$}{(Tp,Tp)} term}

    We begin with a double coset computation in $\SL_2(\Qp)$. We shall write $K = \SL_2(\Zp)$ and $U$ for the lower-triangular Iwahori subgroup $\left\{ \begin{pmatrix} a & b \\ c & d \end{pmatrix} \in K: b \in p\Zp\right\}$.

    \begin{proposition}
     \label{prop:iwahori}
     Let $j \ge 1$. Then the double coset
     \[
      K \begin{pmatrix} p^{-j} & 0 \\ 0 & p^j \end{pmatrix} K
     \]
     decomposes as a disjoint union of exactly four double cosets of the Iwahori $U$, represented by the elements
     \[ \left\{ \begin{pmatrix} p^{-j} & 0 \\ 0 & p^j \end{pmatrix}, \begin{pmatrix} 0 & -p^{-j} \\ p^j & 0 \end{pmatrix}, \begin{pmatrix} p^{j} & 0 \\ 0 & p^{-j} \end{pmatrix}, \begin{pmatrix} 0 & -p^{j} \\ p^{-j} & 0 \end{pmatrix}.\right\}\]
    \end{proposition}

    \begin{proof}
     As shown by Iwahori and Matsumoto \cite[\S 2.2]{iwahorimatsumoto65}, we have a decomposition
     \[ \SL_2(\Qp) = \bigsqcup_{w \in D} U w U,\]
     where $D$ is the set of matrices of the form $\begin{pmatrix} p^j & 0 \\ 0 & p^j\end{pmatrix}$ or $\begin{pmatrix} 0 & -p^{-j} \\ p^j & 0 \end{pmatrix}$ for some $j \in \ZZ$. Comparing this with the well-known Cartan decomposition $\SL_2(\Qp) = \bigsqcup_{j \ge 0}K \begin{pmatrix} p^{-j} & 0 \\ 0 & p^{j} \end{pmatrix} K$ gives the statement above.
    \end{proof}

    \begin{proposition}
     \label{prop:degrees}
     For $\alpha \in \SL_2(\Qp)$, the index of $U \cap \alpha^{-1} U \alpha$ in $U$ is as follows:
     \begin{enumerate}[(a)]
      \item $p^{|2j|}$ for $\alpha \in U \begin{pmatrix} p^j & 0 \\ 0 & p^{-j} \end{pmatrix} U$, $j \in \ZZ$,
      \item $p^{|2j+1|}$ if $\alpha \in U \begin{pmatrix} 0 & -p^{-j} \\ p^j & 0 \end{pmatrix} U$, $j \in \ZZ$.
     \end{enumerate}
    \end{proposition}

    \begin{proof}
     It is clear that the index concerned depends only on the double coset $U \alpha U$, so we may reduce immediately to considering the coset representatives in (a) and (b). In each of these cases we find that the intersection $U \cap \alpha^{-1} U \alpha$ is a subgroup of the form $\left\{ \begin{pmatrix} a & b \\ c & d \end{pmatrix} \in K: p^r \mid b, p^s \mid c\right\}$ for some $r \ge 1$, $s \ge 0$; this clearly has index $p^{r + s - 1}$ in $U$, which gives the above formulae.
    \end{proof}

    \begin{corollary}
     Let $j \in \ZZ$ and $m \ge 1$, neither divisible by $p$, and $k \ge 1$. Then the preimage in $Y(\Gamma_1(N) \cap \Gamma^0(p))^2$ of the curve $C_{mp^k, N, j} \subseteq Y_1(N)^2$ is the union of four distinct curves:
     \begin{enumerate}
      \item the curve $D_1$ given by points of the form
      \[ \left(z, \begin{pmatrix} 1 & \tfrac j {mp^k} \\ 0 & 1\end{pmatrix}z\right),\]
      mapping to $C_{mp^k, N, j}$ with degree $p^2$;
      \item the curve $D_2$ given by points of the form
      \[ \left(z, \gamma \begin{pmatrix} 1 & \tfrac j {mp^k} \\ 0 & 1\end{pmatrix}z \right)\]
      for any $\gamma \in \Gamma_1(N)$ congruent to $\begin{pmatrix} 0 & * \\ * & * \end{pmatrix}$ modulo $p$, again mapping to $C_{mp^k, N, j}$ with degree $p$;
      \item the curve $D_3$ given by points of the form
      \[ \left(z, \begin{pmatrix} 1 & \tfrac j {mp^k} \\ 0 & 1\end{pmatrix} \gamma^{-1} z\right)\]
      where $\gamma$ is as before, mapping to $C_{mp^k, N, j}$ with degree $p$;
      \item the curve $D_4$ given by points of the form
      \[ \left(z, \gamma \begin{pmatrix} 1 & \tfrac j {mp^k} \\ 0 & 1\end{pmatrix} \gamma^{-1} z \right),\]
      mapping isomorphically to $C_{mp^k, N, j}$.
     \end{enumerate}
    \end{corollary}

    \begin{proof}
     All of these curves are evidently in the preimage of $C_{mp^k, N, j}$. One checks that we have
     \begin{align*}
      \begin{pmatrix} 1 & \tfrac j {mp^k} \\ 0 & 1\end{pmatrix} & \in U \begin{pmatrix} 0 & -p^{-k} \\ p^k & 0 \end{pmatrix}U \\
      \gamma \begin{pmatrix} 1 & \tfrac j {mp^k} \\ 0 & 1\end{pmatrix} &\in U \begin{pmatrix} p^k & 0 \\ 0 & p^{-k} \end{pmatrix} U \\
      \begin{pmatrix} 1 & \tfrac j {mp^k} \\ 0 & 1\end{pmatrix} \gamma^{-1} &\in U \begin{pmatrix} p^{-k} & 0 \\ 0 & p^k \end{pmatrix} U \\
      \gamma \begin{pmatrix} 1 & \tfrac j {mp^k} \\ 0 & 1\end{pmatrix} &\gamma^{-1} \in U \begin{pmatrix} 0 & -p^k \\ p^{-k} & 0 \end{pmatrix} U.
     \end{align*}
     Hence the curves $D_i$ exhaust the preimage of $C_{mp^k, N, j}$, by Proposition \ref{prop:iwahori}. The calculation of the degrees of the maps down follows from Proposition \ref{prop:degrees}; and since the total degree is $(p + 1)^2$, they must be distinct.
    \end{proof}

    We set
    \[ \alpha_1 = \begin{pmatrix} 1 & \tfrac j {mp^k} \\ 0 & 1\end{pmatrix},\hspace{3ex} \alpha_2 = \gamma \alpha_1,\hspace{3ex} \alpha_3 = \alpha_1 \gamma^{-1},\hspace{3ex} \alpha_4 = \gamma \alpha_1 \gamma^{-1}, \]
    so $D_i$ is the locus of points of the form $(z, \alpha_i z)$. Let us define
    \[ \Delta_i \coloneqq \left(\pi_2(D_i), (\pi_2)_* (\pi_1)^* (\iota_{mp^k, N, j}')_* {}_c g_{0, 1/mp^k N}\right) \in Z^2(Y_1(N)^2 \otimes \QQ(\mu_{mp^k}, 1).\]

    Then we evidently have
    \[ (T_p', T_p') {}_c \Xi_{mp^k, N, j} = \Delta_1 + \Delta_2 + \Delta_3 + \Delta_4.\]
    We shall evaluate each of these in turn, showing that $D_1$ is the norm of ${}_c \Xi_{mp^{k+1}, N, j}$ and the remaining $\Delta_i$ can be calculated in terms of Hecke operators acting on ${}_c \Xi_{mp^r, N, j}$ for $r < k$.

   \subsubsection{Evaluation of \texorpdfstring{$\Delta_1$}{Delta1}}

    \begin{corollary}
     \label{cor:pushpullhighernormrel}
     Pushforward and pullback commute in each of the following four diagrams:
     \begin{subequations}
      \begin{equation}
       \label{eq:D1x}
       \begin{diagram}
        Y\left(\Gamma_1(mp^k N) \cap \Gamma^0(mp^{k+1}) \cap U\right)
        & \rTo^{z \mapsto \left(\tfrac{z}{mp^k}, \tfrac{z+j}{mp^k}\right)} & D_1\\
        \dTo & & \dTo^{\pi_1}\\
        Y(\Gamma_1(mp^k N) \cap \Gamma^0(mp^k)) & \rTo^{\iota'_{mp^k, N, j}} & C_{mp^k, N, j}
       \end{diagram}
      \end{equation}
      where $U$ is the subgroup of $\Gamma(p^k)$ consisting of matrices whose reduction modulo $p^{k+1}$ lies in the subgroup $\begin{pmatrix} 1 & \tfrac jm \\ 0 & 1 \end{pmatrix}^{-1} U \begin{pmatrix} 1 & \tfrac jm \\ 0 & 1 \end{pmatrix}$, and both vertical arrows have degree $p^2$;
      \begin{equation}
       \label{eq:D2x}
       \begin{diagram}
        Y\left(\Gamma_1(mp^k N) \cap \Gamma^0(m p^{k+1})\right)
        & \rTo^{z \mapsto \left(\tfrac{z}{mp^k}, \gamma \cdot \tfrac{z + j}{mp^k}\right)} & D_2 \\
        \dTo & & \dTo \\
        Y(\Gamma_1(mp^k N) \cap \Gamma^0(m p^{k})) & \rTo^{\iota'_{mp^k, N, j}} & C_{mp^k, N, j} \\
       \end{diagram}
      \end{equation}
      where both vertical arrows have degree $p$;
      \begin{equation}
       \label{eq:D3x}
       \begin{diagram}
        Y\left(\Gamma_1(mp^k N) \cap \Gamma^0(mp^k) \cap U\right)
        & \rTo^{z \mapsto \left(\gamma \cdot \tfrac{z}{mp^k}, \tfrac{z + j}{mp^k}\right)} & D_3 \\
        \dTo & & \dTo \\
        Y(\Gamma_1(mp^k N) \cap \Gamma^0(m p^{k})) & \rTo^{\iota'_{mp^k, N, j}} &
        C_{mp^k, N, j}
       \end{diagram}
      \end{equation}
      where both vertical arrows again have degree $p$; and
      \begin{equation}
       \label{eq:D4x}
       \begin{diagram}
        Y\left(\Gamma_1(mp^k N) \cap \Gamma^0(mp^k) \right) & \rTo^{z \mapsto \left(\gamma \cdot \tfrac{z}{mp^k}, \gamma \cdot \tfrac{z + j}{mp^k}\right)} &
        D_4\\
        \dTo & & \dTo \\
        Y(\Gamma_1(mp^k N) \cap \Gamma^0(m p^{k})) & \rTo^{\iota'_{mp^k, N, j}} &
        C_{mp^k, N, j}
       \end{diagram}
      \end{equation}
      where both vertical arrows are isomorphisms.
     \end{subequations}
    \end{corollary}

    \begin{proof}
     Up to conjugation (and identifying $C_{mp^{k}, N, j}$ and the $D_i$ with their normalizations) each diagram takes the form
     \begin{diagram}
      Y(\Gamma_1 \cap \Gamma_2) &\rTo& Y(\Gamma_1)\\ \dTo & & \dTo \\ Y(\Gamma_2) & \rTo & Y(\Gamma)
     \end{diagram}
     for subgroups $\Gamma_1, \Gamma_2 \subseteq \Gamma$. So it suffices to check in each case that $\Gamma_1 \Gamma_2 = \Gamma$, or equivalently that $[\Gamma: \Gamma_1] = [\Gamma_2 : \Gamma_1 \cap \Gamma_2]$; that is, that the degrees of the two vertical arrows in each diagram are the same. In each case this reduces to an elementary local computation at $p$.
    \end{proof}

    \begin{proposition}[Evaluation of $\Delta_1$]
     \label{prop:delta1term}
     We have
     \[ \Delta_1 = \sum_{\substack{j' \in (\ZZ / mp^{k+1}\ZZ)^\times \\ j' = j \bmod{p^k}}} {}_c \Xi_{mp, N, j'}.\]
    \end{proposition}

    \begin{proof}
     This follows by exactly the same argument as in the case $k = 0$ considered above. From Corollary \ref{cor:pushpullhighernormrel} we know that the modular unit $(\pi_1)^* (\iota_{mp^k, N, j}')_* {}_c g_{0, 1/m N}$ on $D_1$ is equal to the pushforward of ${}_c g_{0, 1/mN}$ along the top horizontal arrow in diagram \eqref{eq:D1x}. The subgroups $U$ for all $j$ in a congruence class modulo $p^{k}$ are conjugate, and by exactly the same argument as in Proposition \ref{prop:normrelation2} we deduce the result.
    \end{proof}

   \subsubsection{Evaluation of \texorpdfstring{$\Delta_2$}{Delta2} and \texorpdfstring{$\Delta_3$}{Delta3}}

    We now turn our attention to $\Delta_2$. Evidently $\pi_2(D_2)$ is the image of $Y(\Gamma_1(m p^k N) \cap \Gamma^0(mp^{k+1}))$ in $Y_1(N)^2$ under the map
      \[ z \mapsto \left(\begin{pmatrix} 1 & 0 \\ 0 & mp^{k+1}\end{pmatrix} z, \begin{pmatrix} 1 & 0 \\ 0 & p \end{pmatrix} \gamma \begin{pmatrix} 1 & j \\ 0 & mp^k \end{pmatrix} z\right).\]

    Let us write $\gamma = \begin{pmatrix} pa & b \\ Nc & d \end{pmatrix}$, where $d = 1 \bmod N$. Then we find that
    \[ \begin{pmatrix} 1 & 0 \\ 0 & p\end{pmatrix} \gamma \begin{pmatrix} 1 & j \\ 0 & mp^k\end{pmatrix} = \begin{pmatrix} p & 0 \\ 0 & p \end{pmatrix} \begin{pmatrix} a & b \\ Nc & pd \end{pmatrix} \begin{pmatrix} 1 & j \\ 0 & mp^{k-1} \end{pmatrix}.\]
    Since scalar matrices act trivially, and $\begin{pmatrix} a & b \\ Nc & pd \end{pmatrix}$ acts on $Y_1(N)$ as the diamond operator $\langle p \rangle$, we see that $\pi_2(D_2)$ can be written as the image of
    \[\begin{array}{rcl}
     Y(\Gamma_1(m p^k N) \cap \Gamma^0(mp^{k+1})) & \rTo & Y_1(N)^2 \\
     z & \rMapsto & \left( \frac{z}{mp^{k+1}}, \langle p \rangle \cdot \frac{z + j}{mp^{k-1}}\right).
    \end{array}\]

    This map factors through the natural projection
    \[ \lambda: Y(\Gamma_1(m p^k N) \cap \Gamma^0(mp^{k+1})) \to Y(\Gamma_1(m p^{k-1} N) \cap \Gamma^0(mp^{k+1}))\]
    (indeed, the first component obviously factors through $\Gamma_1(N) \cap \Gamma^0(mp^{k+1})$, and the second component factors through $Y(\Gamma_1(m p^{k-1} N) \cap \Gamma^0(mp^{k-1}))$ as the map $\iota_{mp^{k-1}, N, j}'$ constructed above composed with the automorphism $\langle p \rangle$).

    \begin{proposition}
     We have
     \[ \lambda_*\left({}_c g_{0, 1/mp^{k} N}\right) =
      \begin{cases}
       {}_c g_{0, 1/mp^{k-1} N} & \text{if $k \ge 2$,}\\
       {}_c g_{0, 1/m N} \cdot \left( \begin{pmatrix} 1 & 0 \\ 0 & p\end{pmatrix}^* {}_c g_{0, \text{``}p^{-1}\text{''}/mN}\right)^{-1} & \text{if $k = 1$.\qed}
      \end{cases}
     \]
    \end{proposition}

    We thus have:

    \begin{proposition}
     \label{prop:delta2term1}
     We have $\Delta_2 = (C, \phi)$ where
     \begin{itemize}
      \item $C$ is the image of $Y(\Gamma_1(m p^{k-1} N) \cap \Gamma^0(mp^{k+1}))$ in $Y_1(N)^2$ under the map
      \[ \beta: z \mapsto \left( \frac{z}{mp^{k+1}}, \langle p \rangle \cdot \frac{z + j}{mp^{k-1}}\right),\]
      \item $\phi$ is the pushforward of ${}_c g_{0, 1/mp^{k-1}}$ (resp. of ${}_c g_{0, 1/m N} \cdot \left( \begin{pmatrix} 1 & 0 \\ 0 & p\end{pmatrix}^* {}_c g_{0, \text{``}p^{-1}\text{''}/mN}\right)^{-1}$) along this map if $k \ge 2$ (resp. if $k = 1$).
     \end{itemize}
    \end{proposition}

   \subsubsection{Evaluation of \texorpdfstring{$\Delta_4$}{Delta4}}

    The last, and easiest, term is $\Delta_4$.

    \begin{proposition}[Evaluation of $\Delta_4$]
     \label{prop:delta4term}
     We have
     \[ \Delta_4 =
      p (\langle p^{-1} \rangle, \langle p^{-1} \rangle) \cdot
      \begin{cases}
       {}_c \Xi_{mp^{k-1}, N, j} & \text{if $k \ge 2$,} \\
       \left(1 -  (\langle p^{-1} \rangle, \langle p^{-1} \rangle) \sigma_p^{-2}\right) {}_c \Xi_{m, N, j}&  \text{if $k = 1$.}
      \end{cases}
     \]
    \end{proposition}

    \begin{proof}
     Contemplating diagram \eqref{eq:D4x} we know that $\Delta_4$ is equal to the pushforward of ${}_c g_{0, 1/mp^{k-1} N}$ from $Y(\Gamma_1(mp^k N) \cap \Gamma^0(mp^k))$ to $Y_1(N)^2$ along the map
     \[ z \mapsto \left( \begin{pmatrix} 1 & 0 \\ 0 & p \end{pmatrix} \gamma \begin{pmatrix} 1 & 0 \\ 0 & mp^{k} \end{pmatrix} z, \begin{pmatrix} 1 & 0 \\ 0 & p \end{pmatrix} \gamma \begin{pmatrix} 1 & j \\ 0 & mp^{k} \end{pmatrix} z\right).\]
     Since $\begin{pmatrix} 1 & 0 \\ 0 & p \end{pmatrix} \gamma = \langle p \rangle \begin{pmatrix} p & 0 \\ 0 & 1 \end{pmatrix}$, this is simply
     \[ z \mapsto (\langle p \rangle, \langle p \rangle) \cdot \left( \begin{pmatrix} 1 & 0 \\ 0 & mp^{k-1} \end{pmatrix} z, \begin{pmatrix} 1 & j \\ 0 & mp^{k-1} \end{pmatrix} z\right).\]
     This evidently factors as the projection
     \[ \lambda: Y(\Gamma_1(mp^k N) \cap \Gamma^0(mp^k)) \to Y(\Gamma_1(mp^{k-1} N) \cap \Gamma^0(mp^{k-1})\]
     composed with the map $(\langle p \rangle, \langle p \rangle) \circ \iota'_{mp^{k-1}, N, j}$.

     On the other hand, the pushforward of ${}_c g_{0, 1/mp^{k} N}$ from $Y(\Gamma_1(mp^k N) \cap \Gamma^0(mp^k))$ to $Y(\Gamma_1(mp^k N) \cap \Gamma^0(mp^{k-1}))$ is clearly $\left({}_c g_{0, 1/mp^{k}N}\right)^p$, since the degree of the map is $p$; and the pushforward from $Y(\Gamma_1(mp^k N) \cap \Gamma^0(mp^{k-1}))$ to $Y(\Gamma_1(mp^{k-1} N) \cap \Gamma^0(mp^{k-1}))$ maps ${}_c g_{0, 1/mp^k N}$ to ${}_c g_{0, 1/mp^k N}$ if $k \ge 2$ and to ${}_c g_{0, 1/mN} \cdot \left( {}_c g_{0, \text{``}p^{-1}\text{''}/mN}\right)^{-1}$ otherwise.
    \end{proof}

   \subsubsection{Evaluation of \texorpdfstring{$(S_p', \langle p^{-1} \rangle) {}_c \Xi_{mp^{k-1}, N, j}$}{the Sp term}}

    We now compute the image of ${}_c \Xi_{mp^{k-1}, N, j}$ under the Hecke operator $(S_p', \langle p \rangle)$.

    \begin{proposition}
     We have the following coset decompositions in $\SL_2(\Qp)$:
     \[ K \begin{pmatrix} 1 & j/m \\ 0 & 1 \end{pmatrix} K = K = K\begin{pmatrix} 1 & j/m \\ 0 & 1 \end{pmatrix} U^0(p^2)\]
     (where $K = \SL_2(\Zp)$); and
     \begin{multline*}
      K \begin{pmatrix} 1 & j/mp \\ 0 & 1 \end{pmatrix} K = K \begin{pmatrix} p^{-1} & 0 \\ 0 & p \end{pmatrix} K =
      K \begin{pmatrix} 1 & j/mp \\ 0 & 1 \end{pmatrix} U^0(p^2) \sqcup K \begin{pmatrix} p^{-1} & 1 \\ 0 & p \end{pmatrix} U^0(p^2)  \\
      \sqcup  K \begin{pmatrix} p^{-1} & \xi \\ 0 & p \end{pmatrix} U^0(p^2)\sqcup K \begin{pmatrix} p^{-1} & 0 \\ 0 & p \end{pmatrix}U^0(p^2).
     \end{multline*}
     where $\xi$ is any quadratic non-residue in $\Zp^\times$.
    \end{proposition}

    Geometrically this is expressed as follows:

    \begin{proposition}
     The preimage in $Y(\Gamma_1(N) \cap \Gamma^0(p^2)) \times Y_1(N)$ of $C_{mp^{k-1}, N, j}$ is:
     \begin{itemize}
      \item if $k = 1$, the single curve $E$ consisting of points of the form $(z, z + j/m)$, with degree $p(p + 1)$ over $C_{m, N, j}$;
      \item if $k = 2$, the union of four distinct curves $E_1, E_2, E_3, E_4$, with degrees $(p, \tfrac{p-1}{2}, \tfrac{p-1}{2}, 1)$ respectively over $C_{mp, N, j}$, where
      \begin{itemize}
       \item $E_1$ is the curve consisting of points of the form $(z, z + j/mp)$,
       \item the curve $E_2$ is the locus of points of the form $(\delta_2 z, z + j/mp)$ where $\delta_2$ is any matrix in $\Gamma_1(N)$ of the form $\begin{pmatrix} a & b \\ c & d \end{pmatrix}$ where $p \mid a$ and $\tfrac{a}{pb} - \tfrac{m}{j}$ is a quadratic residue mod $p$;
       \item the curve $E_3$ is the locus of points of the form $(\delta_3 z,  z + j/mp)$ where $\delta_3$ is any matrix in $\Gamma_1(N)$ of the form $\begin{pmatrix} a & b \\ c & d \end{pmatrix}$ where $p \mid a$ and $\tfrac{a}{pb} - \tfrac{m}{j}$ is a quadratic nonresidue mod $p$;
       \item the curve $E_4$ is the locus of points of the form $(\delta_4 z,  z + j/mp)$ where $\delta_4$ is any matrix in $\Gamma_1(N)$ of the form $\begin{pmatrix} a & b \\ c & d \end{pmatrix}$ where $p \mid a$ and $\tfrac{a}{pb} - \tfrac{m}{j} = 0 \bmod p$.
      \end{itemize}
     \end{itemize}
    \end{proposition}

    \begin{proof}
     This follows from Lemma \ref{lem:distinctcurves} and the previous proposition, noting the identity
     \[ \tbt 1 {\tfrac{j}{mp}} 0 1 =
      \tbt 1 0 {\tfrac{pm}{j}} 1
      \tbt {\tfrac{1}{p}} \xi 0 p
      \tbt {p + \tfrac{p\xi m}{j}} {\tfrac jm} {-\tfrac m j} 0.
     \]
    \end{proof}

    \begin{proposition}
     In the case $k = 1$, the image of $E$ in $Y_1(N)^2$ under $(u, v) \mapsto (u/p^2, \langle p \rangle v)$ is the curve $C$ of Proposition \ref{prop:delta2term} above.

     In the case $k = 2$, the image of $E_1$ under this map is the curve $C$; the images of $E_2$ and $E_3$ are both $(\langle p \rangle, \langle p \rangle) C_{mp, N, j}$; and the image of $E_4$ is $(\langle p \rangle, \langle p \rangle) C_{m, N, \quot{p^{-1}}j}$.
    \end{proposition}

    \begin{proof}
     The $k = 1$ case is clear, as is the assertion for $E_1$ in case $k = 2$. The remaining statements are a fiddly double coset computation.
    \end{proof}

    \begin{proposition}
     In the case $k = 1$, we have
     \[ (S_p', \langle p^{-1} \rangle) \cdot {}_c \Xi_{m, N, j} = (C, \beta_* {}_c g_{0, 1/m N}),\]
     in the notation of Proposition \ref{prop:delta2term}.

     In the case $k = 2$, we have
     \[ (S_p', \langle p^{-1} \rangle) \cdot {}_c \Xi_{pm, N, j} = \sum_{i = 1}^4 \Theta_i,\]
     where $\Theta_i$ is the term corresponding to the curve $E_i$ of the previous proposition, and
     \[ \Theta_1 = (C, \beta_* {}_c g_{0, 1/mpN}).\]
    \end{proposition}

    \begin{proof}
     An argument using Lemma \ref{lem:pushpull} in a familiar manner shows that the pullback of $(\iota_{mp^{k-1}, N, j}')({}_c g_{0, 1/mp^{k-1}})$ to $E$ coincides with the pushforward of ${}_c g_{0, 1/mp^{k-1}}$ along the map
     \[ \begin{array}{rcl}Y(\Gamma_1(mp^{k-1} N) \cap \Gamma^0(mp^{k+1})) &\rTo& E\\
         z & \rMapsto & \left(\tfrac{z}{mp^{k-1}}, \tfrac{z + j}{mp^{k-1}}\right)
        \end{array}
     \]
     So the pushforward of this along the map $E \to Y_1(N)^2$ given by $(u, v) \mapsto (u/p^2, \langle p \rangle v)$ is $\left(C, \beta_*\, {}_c g_{0, 1/mN}\right)$, since the composition of these two maps is $\beta$.
    \end{proof}

    Combining the preceding proposition with Proposition \ref{prop:delta2term}, we have $\Delta_2 = (S_p', \langle p^{-1} \rangle) \cdot {}_c \Xi_{m, N, j} - \Delta_2'$, where
    \[ \Delta_2' =
     \begin{cases}
      \left(C, \beta_* \begin{pmatrix} 1 & 0 \\ 0 & p\end{pmatrix}^* {}_c g_{0, \text{``}p^{-1}\text{''}/mN}\right) & \text{if $k = 1$,}\\
      \Theta_2 + \Theta_3 + \Theta_4 & \text{if $k = 2$.}
     \end{cases}
    \]
    We may express the $k = 1$ case equivalently as
    \[ \Delta_2' = \left(C, \beta'_* {}_c g_{0, \text{``}p^{-1}\text{''}/mN}\right)\]
    where $\beta'$ is the map $Y(\Gamma_1(mN) \cap \Gamma_0(p) \cap \Gamma^0(mp)) \to C$ given by $z \mapsto \left(\tfrac{z}{mp}, \langle p \rangle \cdot \tfrac{pz + j}{m}\right)$.

    We have seen this map before: we showed above in proposition \ref{prop:gammadoubleprime} that there was a commutative diagram
    \begin{diagram}
     Y(\Gamma_1(mN) \cap \Gamma_0(p) \cap \Gamma^0(mp)) & \rTo^{\gamma''} & Y(\Gamma_1(mN) \cap \Gamma^0(m) \cap U_{j'}) \\
     \dTo^{z \mapsto \left(\tfrac{z}{mp}, \tfrac{pz + j}{m}\right)} &&
     \dTo_{z \mapsto \left(\tfrac{z}{mp}, \tfrac{z + j'}{mp}\right)}\\
     Y_1(N)^2 & \rTo^{1 \times \langle p \rangle} & Y_1(N)^2
    \end{diagram}
    where $\gamma''$ is a suitable element of $\Gamma_1(mN)$ (in fact of $\Gamma_1(mN) \cap \Gamma^0(mp)$, although we do not need this), $j' = p^{-1}j \bmod m$ is invertible modulo $p$, and $U_{j'}$ is the preimage in a conjugate of the diagonal torus in $\SL_2(\FF_p)$.

    \begin{proposition}[Evaluation of $\Delta_2'$]
     \label{prop:delta2term}
     For $k = 1$ we have
     \[ \Delta_2' = (\langle p^{-1} \rangle, \langle p^{-1} \rangle) \sigma_p^{-1} \left( (T_p', T_p') - \sigma_p - p (\langle p^{-1} \rangle, \langle p^{-1} \rangle)\sigma_p^{-1}\right) {}_c \Xi_{m, N, j}, \]
     and consequently
     \[ \Delta_2 = \left[(S_p', \langle p^{-1} \rangle)  - (\langle p^{-1} \rangle, \langle p^{-1} \rangle) \sigma_p^{-1} \left( (T_p', T_p') - \sigma_p - p (\langle p^{-1} \rangle, \langle p^{-1} \rangle)\sigma_p^{-1}\right)\right] \cdot {}_c \Xi_{m, N, j} .\]
    \end{proposition}

    We now consider $\Delta_3$. By applying the automorphism of $Y_1(N)^2$ which switches the two factors, and running through essentially the same argument as above, we see that:

    \begin{proposition}[Evaluation of $\Delta_3$]
     \label{prop:delta3term}
     For $k = 1$ we have
     \[ \Delta_3 = \left[(\langle p^{-1}\rangle, S_p') - (\langle p^{-1} \rangle, \langle p^{-1} \rangle) \sigma_p^{-1} \left( (T_p', T_p') - \sigma_p - p (\langle p^{-1} \rangle, \langle p^{-1} \rangle)\sigma_p^{-1}\right)\right] \cdot {}_c \Xi_{m, N, j} .\]
    \end{proposition}

    We now have all the ingredients necessary for the proof in the case $k = 1$, which will be carried out in \S \ref{sect:endofproof} below. However, for $k=2$ there are a few more ingredients we will need.

   \subsubsection{Study of \texorpdfstring{$\Theta_4$}{Theta4}}

    Let us now consider the term $\Theta_4$ that arises for $k = 2$. Recall that $\delta_4$ was any element of $\Gamma_1(N)$ satisfying a certain congruence modulo $p$; we may use strong approximation to make additional congruence assumptions modulo primes away from $p$, so we shall assume that $\delta_4 = \tbt{pa}{b}{mNc}{d}$ with $ja = mb \bmod p$.

    For brevity, we shall write $\Gamma(M, N)$ for the group
    \[ \left\{ \tbt a b c d \in \SL_2(\ZZ): \begin{array}{c} a = 1, b = 0 \bmod M \\ c = 0, d = 1 \bmod N \end{array}\right\}.\]

    \begin{proposition}
     There is a commutative diagram
     \begin{diagram}
      Y(\Gamma(mp, mpN)) & \rTo^{z \mapsto\left(\delta_4 \tbt100{mp}z, \tbt1j0{mp}z\right)} & E_4 \\
      \dTo^{z \mapsto \varepsilon z} & & \dTo_{(u, v) \mapsto \left(\tfrac{z}{p^2}, \langle p \rangle z\right)}\\
      Y(\Gamma(mp, mpN)) & \rTo^{(\langle p \rangle, \langle p \rangle) \iota'_{mp, N, j_0}} & Y_1(N)^2
     \end{diagram}
     where $\varepsilon$ is a suitably chosen element of $\Gamma_1(N)$ and $j_0$ is the unique integer congruent to $j$ modulo $m$ and to 0 modulo $p$.
    \end{proposition}

    \begin{proof}
     Firstly, we note that $\Gamma_1(N)$ normalizes $\Gamma(mp, mpN)$, so the left-hand vertical arrow is well-defined. More subtly, the well-definedness of the top horizontal arrow follows from the inclusion
     \[ \delta_4 \tbt100{mp}\Gamma(mp, mpN) \tbt100{mp}^{-1}\delta_4^{-1} \subseteq \Gamma^0(p^2);\]
     indeed $\delta_4 \tbt100{mp} = \tbt p00m \delta_4'$ where $\delta_4' = \tbt a{mb}{Nc}d \in \Gamma_0(N)$ normalizes $\Gamma(mp, mpN)$, so
     \[ \delta_4 \tbt100{mp}\Gamma(mp, mpN) \tbt100{mp}^{-1}\delta_4^{-1} = \tbt m00p \Gamma(mp, mpN) \tbt m00p^{-1} \subseteq \Gamma_0(m^2N) \cap \Gamma^0(p^2).\]

     It remains to show that $\varepsilon$ may be chosen so that the diagram commutes. We need to choose $\varepsilon$ so that we have
     \[ \Gamma_1(N) \langle p \rangle \tbt 1 0 0 {mp} \varepsilon = \Gamma_1(N) \tbt{p^{-1}}00p \delta_4 \tbt 1 0 0 {mp} = \Gamma_1(N) \tbt100{mp} \delta_4',\]
     and so that
     \[ \Gamma_1(N) \tbt 1 {j_0} 0 {mp} \varepsilon = \Gamma_1(N) \tbt 1 j 0 {mp}\]
     where $j_0$ is the unique integer congruent to $j$ mod $m$ and $0$ mod $p$.

     These conditions are both satisfied if we take $\varepsilon$ to be congruent to 1 modulo $mN$, and to satisfy the same congruence modulo $p$ as $\delta_4'$, so $\varepsilon = \tbt{x}{jx}{*}{*} \bmod p$ for some $x$.
    \end{proof}

    \begin{corollary}
     We have
     \[ \Theta_4 = (\langle p^{-1} \rangle, \langle p^{-1} \rangle)^* \left( C_{m, N, \quot{p^{-1}}j}, (\iota'_{mp, N, j_0})_* (\varepsilon^{-1})^* {}_c g_{0, 1/mpN}\right).\]
    \end{corollary}

    Now we shall calculate the pushforward of $(\varepsilon^{-1})^* {}_c g_{0, 1/mpN}$ from $Y(\Gamma(mp, mpN))$ to $Y(\Gamma(m, mN))$.

    \begin{proposition}
     \label{prop:theta4}
     Let $\alpha, \beta \in \ZZ$ be such that $\alpha = 0, \beta = 1 \bmod mN$ and $\beta \ne 0 \bmod p$. Then the pushforward of ${}_c g_{\alpha/mpN, \beta/mpN}$ from $Y(\Gamma(mp, mpN))$ to $Y(\Gamma(m, mN))$ along the map $z \mapsto z/p$ is ${}_c g_{0, 1/mN} \cdot \left({}_c g_{0, \quot{p^{-1}}/mN}\right)^{-1}$, and hence
     \[ \Theta_4 = (\langle p^{-1}\rangle, \langle p^{-1} \rangle)\sigma_p (1 - (\langle p^{-1}\rangle, \langle p^{-1} \rangle) \sigma_p^{-2}) {}_c \Xi_{m, N, j}.\]
    \end{proposition}

    \begin{proof}
     A calculation using Theorem \ref{thm:siegel-compat} shows that pushing forward to $Y(\Gamma_1(mpN) \cap \Gamma^0(m))$ gives ${}_c g_{\alpha/mN, \beta/mpN} = {}_c g_{0, \beta/mpN}$, and we are now in familiar territory.
    \end{proof}

   \subsubsection{Study of \texorpdfstring{$\Theta_2$ and $\Theta_3$}{Theta2 and Theta3}}

    Let $\delta$ be any element of $\Gamma_1(N) \cap \Gamma_0(mN)$ whose top left entry is divisible by $p$, so $\delta = \tbt{pa}{b}{mNc}{d}$ with $pa = d = 1 \bmod N$. Let $\delta' = \tbt{a}{mb}{Nc}{pd}$, so $\delta' \in \langle p \rangle \Gamma_1(N)$ and we have $\delta \tbt100{mp} = \tbt p00m \delta'$.

    Let $E_\delta$ be the locus of points in $Y(\Gamma_1(N) \cap \Gamma^0(p^2)) \times Y_1(N)$ of the form $(\delta z, z + j/{mp})$; this clearly maps to $C_{mp, N, j}$ under the natural projection to $Y_1(N)^2$. We then build the following (rather unwieldy) diagram of modular curves:

    \begin{diagram}
     Y(\Gamma(mp, mpN)) & \rTo^{\pr} &
     Y\left(\begin{array}{c}  (\delta')^{-1} \stbt100m^{-1} (\Gamma_1(N) \cap \Gamma_0^0(p)) \stbt 100m \delta' \\
             \cap \stbt1j0{mp}^{-1} \Gamma_1(N) \stbt1j0{mp}
            \end{array}\right)
     & \rTo^{\stbt p00m \delta', \stbt 1j0{mp}}_{\cong} & E_\delta\\
     \dTo_\cong^{\id} & & \dTo^{\pr} & & \dTo^{\pr} \\
     Y(\Gamma(mp, mpN)) & \rTo^{\pr} &
     Y\left(\begin{array}{c} \stbt100{mp}^{-1} \Gamma_1(N) \stbt 100{mp}\\
             \cap \stbt1j0{mp}^{-1} \Gamma_1(N) \stbt 1j0{mp}
            \end{array}\right)
     & \rTo^{\stbt100{mp}, \stbt1j0{mp}}_\cong & C_{mp, N, j}.
    \end{diagram}

    \begin{proposition}
     Suppose that $\delta' = \tbt {a'} {b'} {c'} {pd}$ with $aj - b \ne 0 \bmod p$. Then the intersection
     \[ \left[ (\delta')^{-1} \stbt100m^{-1} (\Gamma_1(N) \cap \Gamma_0^0(p)) \stbt 100m \delta' \cap \stbt1j0{mp}^{-1} \Gamma_1(N) \stbt1j0{mp}\right] \cap \Gamma(m, mN) \]
     consists precisely of those matrices in $\Gamma(m, mN)$ which are congruent to $\pm 1$ modulo $p$.
    \end{proposition}

    \begin{proof}
     It suffices to show that
     \[ \Gamma_1(N) \cap \Gamma_0^0(p) \cap \stbt 100m \delta'\stbt1j0{mp}^{-1} \Gamma_1(N) \stbt1j0{mp}(\delta')^{-1}\stbt100m^{-1}\]
     consists of matrices that are $\pm 1$ modulo $p$, since such matrices are clearly preserved under conjugation by $\stbt 100m \delta'$ (which is integral at $p$).

     Let $\gamma \in \Gamma_1(N) \cap \Gamma_0^0(p)$. Then $\gamma$ is congruent modulo $p$ to $\tbt x 0 0 {x^{-1}}$ for some $x \in (\ZZ/p\ZZ)^\times$. We require that
     \[ \tbt 1 j 0 {mp} (\delta')^{-1} \tbt 1 00m^{-1} \gamma \tbt 1 0 0 m \delta' \tbt 1j0{mp}^{-1} \in \SL_2(\Zp)\]
     or, equivalently, that
     \[ \tbt1j01 (\delta')^{-1} \tbt x 0 0 {x^{-1}} \delta' \tbt 1j01^{-1} \cong \tbt{*}{0}{*}{*} \pmod p.\]
     Substituting the entries of $\delta'$, we find that the top right-hand entry of the product on the left is congruent modulo $p$ to $(aj - b)cj (x - x^{-1})$. So if $aj - b$ is not divisible by $p$, then we must have $x - x^{-1} = 0 \bmod p$, i.e.~$x = \pm 1$, as required.
    \end{proof}

    \begin{remark}
     Conceptually, what is going on here is that we have calculated the intersection of \emph{three} Borel subgroups of $\SL_2(\mathbb{F}_p)$ in general position relative to each other, which is simply the centre of the group.
    \end{remark}

    \begin{corollary}
     The pullback to $E_\delta$ of $(\iota_{mp, N, j}')_* {}_c g_{0, 1/mpN}$ is equal to the pushforward along the top row of the above diagram of the modular unit
     \[ \prod_{\gamma \in U_j / \{\pm 1\}} \gamma^* {}_c g_{0, 1/mpN} \in \cO(Y(mp, mpN)^\times),\]
     where $U_j$ is (as above) the torus in $\SL_2(\FF_p)$ whose preimage is
     \[ \tbt 1 0 0 p^{-1} K \tbt 10 0p \cap \tbt 1 j 0 p^{-1} K \tbt 1j0p,\]
     and we choose a lifting of each element of $U / \{\pm 1\}$ to an element of $\Gamma(m, mN)$.
    \end{corollary}

    Note that this depends only rather weakly on $\delta$. We calculated $U_j$ explicitly above: it consists of all matrices of the form $\tbt{u^{-1}}0{j^{-1}(u-u^{-1})}u$ with $u \in \FF_p^\times$.

    We now consider the pushforward of this to $Y_1(N)^2$ along the map $(u, v) \mapsto (\langle p \rangle^{-1} \tbt{p^{-1}}{0}{0}{p} u, v)$, so the image of $E_\delta$ is one of the components of the image of $C_{mp, N, j}$ under the Hecke operator $(\langle p \rangle S_p', 1)$.

    \begin{proposition}
     The image of $E_\delta$ under this map is $C_{mp, N, j}$ itself. More specifically, we may find $\varepsilon \in \Gamma(m, mN)$ such that there is a commutative diagram
     \begin{diagram}
      Y(\Gamma(mp, mpN)) &\rTo^{\left( \langle p \rangle^{-1} \tbt100{mp} \delta', \tbt 1 j 0 {mp}\right)} & Y_1(N)^2 \\
      \dTo^\varepsilon_\cong & & \dTo^\id_{\cong}\\
      Y(\Gamma(mp, mpN)) &\rTo^{\iota'_{mp, N, j}} & Y_1(N)^2
     \end{diagram}
    \end{proposition}

    \begin{proof}
     We must show that $\varepsilon$ can be found in such a way that
     \[ \tbt 100{mp} \delta' \in \Gamma_1(N) \langle p \rangle \tbt 100{mp} \varepsilon \]
     and
     \[ \tbt 1j0{mp} \in \Gamma_1(N) \tbt 1j0{mp} \varepsilon.\]

     For any $\varepsilon \in \Gamma(m, mN)$, the matrices
     \[ \langle p \rangle^{-1} \tbt 1 0 0 {mp} \delta' \varepsilon^{-1} \tbt100p^{-1}\]
     and
     \[ \tbt 1j0{mp} \varepsilon^{-1} \tbt1j0{mp}\]
     are integral away from $p$ and have bottom right entry congruent to 1 modulo $N$; so we need only show that $\varepsilon$ may be chosen such that both are integral at $p$. So we must show that we can find $\varepsilon$ in the intersection
     \[ \varepsilon \in \tbt 1 j 0 1^{-1} U^0(p) \tbt 1 j 0 1 \cap U^0(p) \delta'.\]
     The non-emptiness of this intersection is equivalent to the equality of the double cosets
     \[ U^0(p) \tbt 1j01 U^0(p)\quad \text{and} \quad U^0(p) \tbt 1 j 0 1 (\delta')^{-1} U^0(p).\]
     However, as we have seen before, there is only one double $U^0(p)$ coset in $K$ other than $U^0(p)$ itself, so this equality is equivalent to $\tbt 1 j 0 1 (\delta')^{-1} \notin U^0(p)$, which is equivalent to our hypothesis $b \ne ja \bmod p$.
    \end{proof}

    It remains to be shown that we can choose $\delta$ and $\varepsilon$ in some reasonable fashion. Let $\xi \in \FF_p^\times$. Then we can take $\varepsilon = \tbt {j^{-1}(1 - \xi)} {-\xi} {1 - j^{-2}(1 - \xi)}{ j + \xi j^{-1}}$, and $\delta' = \tbt {j^{-1}(1 - \xi)}{-\xi}{\xi^{-1}}{0}$. A routine verification shows that $\tbt 1 j 0 1 \varepsilon \tbt 1 {-j} 0 1 = \tbt j 0 {1 - j^{-2}(1 - \xi)}{j^{-1}}$ is lower-triangular, and that if we take $\xi$ to be a quadratic residue or a nonresidue $\delta'$ satisfies the congruences stated above, so it suffices to take $\xi = 1$ and one non-square $\xi$.

    Let us write $\Gamma(mp, mpN)^{\pm}$ for the subgroup of $\Gamma(m, mN)$ consisting of matrices that are congruent to $\pm 1$ modulo $p$. Then we have a diagram
    \begin{diagram}[width=2cm,height=2cm]
     Y(\Gamma(mp, mpN)) &\rTo^\sigma & Y(\Gamma(mp, mpN)^\pm) & \rTo_\cong^\varepsilon & Y(\Gamma(mp, mpN)^\pm)\\
     & \ldTo^{\stbt100p, \stbt1j0p}_{\mu_1} & & \rdTo^{\langle p \rangle^{-1} \stbt100p \delta', \stbt1j0p}_{\mu_2} & \dTo_{\stbt100p, \stbt1j0p}^{\mu_3}\\
     Y(\Gamma(m, mN))^2 & & & & Y(\Gamma(m, mN))^2
    \end{diagram}
    Here $\sigma$ is the natural pushforward map.

    The images of $\mu_1$ and $\mu_2$ are both given by the curve $C$ of points of the form $(z, z + j/p)$ in $Y(\Gamma(m, mN))^2$, which maps to $C_{mp, N, j}$ under the map $(u, v) \mapsto (u/m, v/m)$ to $Y_1(N)^2$. We find that
    \[ \mu_1^*\, (\sigma \circ \mu_1)_* {}_c g_{0, 1/mpN} = \prod_{u \in \FF_p^\times} \tbt{u^{-1}}0{j^{-1}(u - u^{-1})}u^* {}_c g_{0, 1/mpN}\]
    and hence
    \begin{align*}
     (\mu_2)_*\, \mu_1^*\, (\sigma \circ \mu_1)_*\, {}_c g_{0, 1/mpN} &= (\mu_3)_* (\varepsilon^{-1})^*\, \mu_1^*\, (\sigma \circ \mu_1)_*\, {}_c g_{0, 1/mpN}\\
     &= (\mu_3)_* (\varepsilon^{-1})^* \prod_{u \in \FF_p^\times} \tbt{u^{-1}}0{j^{-1}(u - u^{-1})}u^* {}_c g_{0, 1/mpN}\\
     &= \prod_{v \in \FF_p^\times / {\pm 1}} \prod_{u \in \FF_p^\times} \left[ \tbt{u^{-1}}0{j^{-1}(u - u^{-1})}u \varepsilon^{-1} \tbt{v^{-1}}0{j^{-1}(v - v^{-1})}v\right]^* {}_c g_{0, 1/mpN}
    \end{align*}

    Conjugating by $\tbt10{j^{-1}}1$ maps the torus $U_j$ onto the diagonal torus and maps $\varepsilon^{-1}$ onto the matrix $\tbt j \xi 0 {j^{-1}}$, and the above expression becomes
    \begin{gather*}
     \prod_{v \in \FF_p^\times / {\pm 1}} \prod_{u \in \FF_p^\times}
     \left[ \tbt10{-j^{-1}} 1 \tbt{u^{-1}}00u \tbt j \xi 0 j^{-1} \tbt{v^{-1}}00 v\tbt 1 0 {j^{-1}} 1\right]^* {}_c g_{0, 1/mpN}\\
     = \prod_{v \in \FF_p^\times / {\pm 1}} \prod_{u \in \FF_p^\times}
     \left[ \tbt10{-j^{-1}} 1 \tbt {u^{-1}v^{-1} j} {u^{-1} v \xi} 0 {u v j^{-1}}\tbt 1 0 {j^{-1}} 1\right]^* {}_c g_{0, 1/mpN}.
    \end{gather*}

    We may change variables by letting $a = uv$ and $b = u^{-2}$. Then the product becomes
    \begin{gather*}
     \prod_{a \in \FF_p^\times} \prod_{b \in \FF_p^{\times 2}\cdot \xi} \left[ \tbt10{-j^{-1}} 1  \tbt j b 0 {j^{-1}} \tbt {a^{-1}}00a \tbt 1 0 {j^{-1}} 1\right]^* {}_c g_{0, 1/mpN}\\
     = (\mu_1 \circ \sigma)_* \prod_{b \in \FF_p^{\times 2} \cdot \xi} \left[ \tbt10{-j^{-1}} 1  \tbt j b 0 {j^{-1}} \tbt 1 0 {j^{-1}} 1\right]^* {}_c g_{0, 1/mpN}\\
     = (\mu_1 \circ \sigma)_* \prod_{b \in \FF_p^{\times 2} \cdot \xi} \left[ \tbt j b {-1}{j^{-1}(1-b)} \tbt 1 0 {j^{-1}} 1\right]^* {}_c g_{0, 1/mpN}\\
     = (\mu_1 \circ \sigma)_* \tbt 1 0 {j^{-1}} 1^* \prod_{b \in \FF_p^{\times 2} \cdot \xi} {}_c g_{(-1)_0 / mpN, (j^{-1}(1-b))_1/{mpN}}
    \end{gather*}

    Considering $\Theta_2$ and $\Theta_3$ together corresponds to letting $b$ vary over all of $\FF_p^\times$. If we were to extend the product over all $b \in \FF_p$ (residue, nonresidue, or zero), then we would get
    \begin{equation}
     \label{eq:referee}
     (\mu_1 \circ \sigma)_* \tbt 1 0 {j^{-1}} 1^* \tbt p 0 0 1^* {}_c g_{\alpha/p, 1/mN},
    \end{equation}
    where $\alpha$ is the image of $-1/mN$ in $(\ZZ / p\ZZ)^\times$ (thus $\alpha/p = (-1)_0 / mN$).

    The term for $b = 0$ is just
    \[ (\mu_1 \circ \sigma)_* \tbt j 0 {-1 + j^{-2}}{j^{-1}}^* {}_c g_{0, 1/mpN}  = (\mu_1 \circ \sigma)_* {}_c g_{0, 1/mpN},\]
    since $\tbt j 0 {-1 + j^{-2}}{j^{-1}} \in U_j$. This is what we want: it is the definition of ${}_c \Xi_{mp, N, j}$.

    What can we say about the expression in \eqref{eq:referee}? Writing the pushforward in terms of coset representatives gives us
    \begin{gather*}
     \prod_{u \in \FF_p^\times} \left[ {\tbt p 0 0 1} {\tbt 1 0 {j^{-1}} 1} {\tbt 1 0 {j^{-1}} 1}^{-1} \tbt {u^{-1}} 0 0 u {\tbt 1 0 {j^{-1}} 1}\right]^* {}_c g_{\alpha/p, 1/mN}\\
     = \prod_{u \in \FF_p^\times} \left[ \tbt {u^{-1}} 0 0 u \tbt p 0 0 1 \tbt 1 0 {j^{-1}} 1 \right]^* {}_c g_{\alpha/p, 1/mN}\\
     = \tbt 1 0 {j^{-1}} 1^* \tbt p 0 0 1^* \prod_{u \in \FF_p^\times} {}_c g_{u/p, 1/mN} \\
     = \tbt 1 0 {j^{-1}} 1^* \left( {}_c g_{0, 1/mN} \cdot \left( \tbt p 0 0 1^* {}_c g_{0, 1/mN}\right)^{-1}\right) \\
     = {}_c g_{0, 1/mN} \cdot \left( \tbt 1 0 {j^{-1}} 1^*  \tbt p 0 0 1^* {}_c g_{0, 1/mN}\right)^{-1}
    \end{gather*}
    The last line is justified by the fact that $\tbt 1 0 {j^{-1}} 1^*$ denotes the action of a matrix congruent to $\tbt 1 0 {j^{-1}} 1$ modulo $p$ but to the identity modulo $mN$, and such a matrix will act trivially on ${}_c g_{0, 1/mN}$.

    We have seen both of these terms before: the class in $\CH^2(Y_1(N)^2 \otimes \QQ(\mu_m), 1)$ defined by ($C_{mp, N, j}$, pushforward of ${}_c g_{0, 1/mN}$) is $( (T_p', T_p') - \sigma_p - p \langle p \times p \rangle^{-1} \sigma_p^{-1}) {}_c \Xi_{m, N, j}$, by Corollary \ref{cor:TpTpterm}; and the term corresponding to ($C_{mp, N, j}$, pushforward of $ \tbt 1 0 {j^{-1}} 1^*  \tbt p 0 0 1^* {}_c g_{0, 1/mN}$) is $(\langle p \rangle^{-1}, S_p') {}_c \Xi_{m, N, j}$, by Corollary \ref{cor:Sp2term}.

   \subsubsection{Conclusion of the proof}
    \label{sect:endofproof}

    We can now complete the proof of Theorem \ref{thm:stronghighernormrel} for $k = 1$.

    We know that
    \[ (T_p', T_p') {}_c \Xi_{mp, N, j} = \Delta_1 + \Delta_2 + \Delta_3 + \Delta_4,\]
    and we have shown that:
    \[ \Delta_1 = \norm_{mp}^{mp^2}\left( {}_c \Xi_{mp^2, N, j}\right) \quad \text{(Proposition \ref{prop:delta1term});}
    \]
    \begin{multline*}
     \Delta_2 = \left[(S_p', \langle p^{-1} \rangle) - (\langle p^{-1} \rangle, \langle p^{-1} \rangle) \sigma_p^{-1} \left( (T_p', T_p') - \sigma_p - p (\langle p^{-1} \rangle, \langle p^{-1} \rangle)\sigma_p^{-1}\right)\right] \cdot {}_c \Xi_{m, N, j}\\
     \text{(Proposition \ref{prop:delta2term});}
    \end{multline*}
    \begin{multline*}
     \Delta_3 = \left[(\langle p^{-1}\rangle, S_p') - (\langle p^{-1} \rangle, \langle p^{-1} \rangle) \sigma_p^{-1} \left( (T_p', T_p') - \sigma_p - p (\langle p^{-1} \rangle, \langle p^{-1} \rangle)\sigma_p^{-1}\right)\right] \cdot {}_c \Xi_{m, N, j}\\
     \text{(Proposition \ref{prop:delta3term});}
    \end{multline*}
    and
    \[ \Delta_4 =
      p (\langle p^{-1} \rangle, \langle p^{-1} \rangle) \left(1 -  (\langle p^{-1} \rangle, \langle p^{-1} \rangle) \sigma_p^{-2}\right) {}_c \Xi_{m, N, j} \quad \text{(Proposition \ref{prop:delta4term}).}
    \]

    Combining these statements gives the $k = 1$ case of Theorem \ref{thm:stronghighernormrel}.

    In the case $k = 2$ we have again $(T_p', T_p') {}_c \Xi_{mp^2, N, j} = \Delta_1 + \Delta_2 + \Delta_3 + \Delta_4$ where
    \[ \Delta_1 = \norm_{mp^2}^{mp^3}\left( {}_c \Xi_{mp^3, N, j}\right) \quad \text{(Proposition \ref{prop:delta1term});}
    \]
    \[ \Delta_2 = (S_p', \langle p^{-1} \rangle)\cdot {}_c \Xi_{mp, N, j} - \Theta_2 - \Theta_3 - \Theta_4;
    \]
    \[ \Delta_3 = (\langle p^{-1} \rangle, S_p') \cdot {}_c \Xi_{mp, N, j} - \Theta_2' - \Theta_3' - \Theta_4';
    \]
    and
    \[ \Delta_4 = p (\langle p^{-1} \rangle, \langle p^{-1} \rangle) \cdot {}_c \Xi_{mp, N, j}\quad \text{(Proposition \ref{prop:delta4term}).}\]

    Moreover, we have
    \[ \Theta_4 = (\langle p^{-1}\rangle, \langle p^{-1} \rangle)\sigma_p (1 - (\langle p^{-1}\rangle, \langle p^{-1} \rangle) \sigma_p^{-2}) {}_c \Xi_{m, N, j} \text{(Proposition \ref{prop:theta4}),}\]
   and
   \[ \Theta_2 + \Theta_3 = (\langle p^{-1} \rangle, \langle p^{-1} \rangle) \cdot {}_c \Xi_{mp, N, j} - ( (T_p', T_p') - \sigma_p - p \langle p \times p \rangle^{-1} \sigma_p^{-1}) {}_c \Xi_{m, N, j} +  (\langle p \rangle^{-1}, S_p') {}_c \Xi_{m, N, j}.\]

   The obvious involution of $Y_1(N)^2 \otimes \QQ(\mu_m)$ given by swapping the two factors maps ${}_c \Xi_{p^k m, N, j}$ to ${}_c \Xi_{p^k m, N, -j}$ for each $j$; and it interchanges $\Theta_i$ with $\Theta_i'$ for $i = 1, \dots, 4$, so we obtain formulae for these terms which are identical with the non-primed versions except $(\langle p ^{-1} \rangle, S_p')$ is interchanged with $(S_p', \langle p^{-1} \rangle)$. Collecting terms gives Theorem \ref{thm:stronghighernormrel} for $k = 2$.


 \section{Relation to complex L-values}

\subsection{Definition of Rankin--Selberg L-functions}

 We recall the definition of Rankin--Selberg $L$-functions of pairs of modular forms.

 \begin{definition}
  Let $f, g$ be cuspidal new modular eigenforms (of possibly distinct weights $k, \ell$ and levels $N_f, N_g$), $L$ a number field containing the coefficients of $f$ and $g$, and $p$ a prime. We define the local Euler factor
  \[ P_p(f, g, X) = \det\left(1 - X \Frob_p^{-1} | (V_{L_\lambda}(f) \otimes V_{L_\lambda}(g))^{I_p}\right)\]
  where $\lambda$ is an arbitrary place of $L$ of residue characteristic distinct from $p$, $V_{L_\lambda}(f)$ is the $L_\lambda$-linear representation of $G_{\QQ}$ attached to $f$ (and similarly for $g$) -- see \S \ref{sect:galreps} below -- and $\Frob_p$ denotes the arithmetic Frobenius at $p$.
 \end{definition}

 This Euler factor may be defined in purely automorphic terms (cf.~\cite[Theorem 14.8]{jacquet72}), but the above definition is convenient for our purposes. The following is an elementary calculation:

 \begin{proposition}
  \label{prop:explicitEulerfactor}
  If $p \nmid N_f N_g$, then
  \[ P_p(f, g, X) = (1 - \alpha \gamma X)(1 - \alpha \delta X)(1 - \beta \gamma X)(1 - \beta \delta X)\]
  where $\alpha, \beta$ are the roots of $X^2 - a_p(f) X + p^{k-1} \varepsilon_p(f)$ and similarly $\gamma, \delta$ are the roots of $X^2 - a_p(g) X + p^{\ell-1} \varepsilon_p(g)$. Completely explicitly, this becomes
  \begin{multline*}
   P_p(f, g, X) = 1 - a_p(f) a_p(g) X + \Big( p^{\ell-1} a_p(f)^2 \varepsilon_p(g) + p^{k-1} \varepsilon_p(f) a_p(g)^2 - 2p^{k+\ell-2} \varepsilon_p(f) \varepsilon_p(g)\Big) X^2 \\- p^{k+\ell-2} \varepsilon_p(f) a_p(f) \varepsilon_p(g) a_p(g) X^3 + p^{2k+2\ell-4} \varepsilon_p(f)^2 \varepsilon_p(g)^2 X^4.
  \end{multline*}
 \end{proposition}

 \begin{proposition}
  We may write
  \[ P_p(f, g, X) = \prod_{i = 1}^4 (1 - \lambda_i X),\]
  where each $\lambda_i$ is either 0, or a $p$-Weil number of weight $\le (k + \ell-2)$. In particular, all poles of the meromorphic function $P_p(f, g, p^{-s})^{-1}$ have real part at most $\frac{k + \ell - 2}{2}$.
 \end{proposition}

 \begin{proof}
  This is clear from Proposition \ref{prop:explicitEulerfactor} if $p$ does not divide the levels of $f$ and $g$. The remaining cases follow from an explicit computation of the possible local components of $f$ and $g$, using the Galois-theoretic definition adopted above (since the Weil--Deligne representations attached to $f$ and $g$ must fall into a finite list of possible types).
 \end{proof}

 We now define global Rankin--Selberg $L$-functions as a product of local terms in the usual way.

 \begin{definition}
  We let
  \[ L(f, g, s) = \prod_{\text{$p$ prime}} P_p(f, g, p^{-s})^{-1}\]
  and for $N \ge 1$ we let
  \[ L_{(N)}(f, g, s) = \prod_{\substack{\text{$p$ prime}\\ p \nmid N}} P_p(f, g, p^{-s})^{-1}.\]
 \end{definition}

 \begin{proposition}
  \label{prop:completedL}
  Suppose $k \ge \ell$, and write $\Gamma_{\CC}(s) = (2\pi)^{-s} \Gamma(s)$. Then the completed $L$-function
  \[ \Lambda(f, g, s) = \Gamma_\CC(s) \Gamma_\CC(s - \ell + 1) L(f, g, s)\]
  has analytic continuation to all $s \in \CC$, except for a simple pole at $s = k$ if $\ell = k$ and $f = \overline{g}$; and it satisfies a functional equation of the form
  \[ \Lambda(f, g, k + \ell - 1 - s) = \varepsilon(s) \cdot \Lambda(\overline{f} \otimes \overline{g}, s)\]
  where $\varepsilon$ is a function of the form $Ae^{Bs}$ for constants $A, B$.
 \end{proposition}

 \begin{remark}
  The function $\varepsilon(s)$ is, as the notation suggests, a global $\varepsilon$-factor, but we shall not use this interpretation here.
 \end{remark}

 In particular, if $k = \ell = 2$ and $s = 1$, the value $L(f, g, 1)$ vanishes (because $\Gamma_{\CC}(s - 1)$ has a simple pole) and we have
 \begin{equation}
  \label{eq:Lfcnzero}
  L'(f, g, 1) = 2\pi \Lambda(f, g, 1).
 \end{equation}


\subsection{Real-analytic Eisenstein series}

 We now express the Rankin--Selberg $L$-function in terms of the Petersson product with a non-holomorphic Eisenstein series, the original example of the Rankin--Selberg method.

 \begin{definition} Let $k \ge 0 \in \ZZ$, and $\alpha \in \QQ / \ZZ$.
  \begin{enumerate}
   \item For $\tau \in \cH, s \in \CC$ with $k + 2 \Re(s) > 2$, we define
   \[ E^{(k)}_{\alpha}(\tau, s) = (-2\pi i)^{-k} \pi^{-s} \Gamma(s + k) \sideset{}{'}\sum_{(m, n) \in \ZZ^2} \frac{\Im(\tau)^s}{\left(m\tau + n + \alpha \right)^{k} \left|m\tau + n + \alpha\right|^{2s}},\]
   where the prime denotes that the term $(m, n) = (0, 0)$ is omitted if $\alpha = 0$ (but not otherwise).
   \item For $\tau, s$ as above, define
   \[ F^{(k)}_{\alpha}(\tau, s) = (-2\pi i)^{-k} \Gamma(s + k) \pi^{-s} \sideset{}{''}\sum_{(m, n) \in \ZZ^2} \frac{e^{2 \pi i \alpha m} \Im(\tau)^s}{\left(m\tau + n\right)^{k} \left|m\tau + n \right|^{2s}}\]
   where the double prime denotes that the term $(m, n)=(0,0)$ is omitted (always).
  \end{enumerate}
 \end{definition}

 \begin{proposition}
  The above series have the following properties:
  \begin{enumerate}[(i)]
   \item (Automorpy) If $N\alpha = 0$, then for fixed $\tau$ both $E^{(k)}_{\alpha}$ and $F^{(k)}_{\alpha}$ are preserved by the weight $k$ action of $\Gamma_1(N)$, and moreover the diamond operators act on $\alpha$ by multiplication in the obvious way.
   \item (Action of Atkin--Lehner involutions) If $N\alpha = 0$, then we have
   \begin{align*}
    F^{(k)}_{\alpha}(\tau, s) &= N^{-k-s} \sum_{x \in \ZZ / N\ZZ} e^{2\pi i \alpha x} \tau^{-k} E^{(k)}_{x/N}\left(\tfrac{-1}{N\tau}, s\right). \\
   \end{align*}
   \item (Differential operators) The Maass--Shimura weight-raising differential operator
   \[ \delta_k \coloneqq \frac{1}{2\pi i}\left( \frac{\d}{\d\tau} + \frac{k}{\tau - \overline{\tau}}\right)\]
   (cf.~\cite[Equation (2.8)]{shimura76}) acts on the Eisenstein series via
   \[ \delta_k E^{(k)}_{\alpha}(\tau, s) = E^{(k+2)}_{\alpha}(\tau, s-1)\]
   and similarly for $F^{(k)}_\alpha$.
   \item (Analytic continuation and functional equation) For fixed $k, \tau, \alpha$, both functions $E^{(k)}_{\alpha}(\tau, s)$ and $F^{(k)}_{\alpha}(\tau, s)$ have meromorphic continuations to the whole $s$-plane, which are holomorphic everywhere if $k \ne 0$; and we have
   \[ E^{(k)}_{\alpha}(\tau, s) = F^{(k)}_{\alpha}(\tau, 1-k-s).\]
   \item (Relation to Siegel units) We have
   \[ E^{(0)}_{\alpha}(\tau, 0) = 2 \log |g_{0, \alpha}(\tau)|\]
   where $g_{0,\alpha}$ is the Siegel unit of \S\ref{sect:siegelunits}.
  \end{enumerate}
 \end{proposition}

 \begin{proof}
  Parts (i)--(iii) are easy explicit computations. Part (iv) is a standard application of the Poisson summation formula, and (v) is formula (3.8.4)(iii) of \cite{kato04}.
 \end{proof}

 Now let $f, g$ be any two newforms of levels $N_f, N_g$ dividing $N$, and weights $k, \ell$, with $k > \ell$. Let $\breve f \in S_k(\Gamma_1(N))[\pi_f]$ and $\breve g \in S_\ell(\Gamma_1(N))[\pi_g]$ be forms in the oldspaces at level $N$ attached to $f$ and $g$ (which we shall think of as ``test vectors''). For any $\alpha \in \tfrac{1}{N} \ZZ / \ZZ$, set
 \begin{align*}
  \mathcal{D}(\breve f, \breve g, x, s) &= \int_{\Gamma_1(N) \backslash \cH} \breve f(-\overline{\tau})\, \breve g(\tau)\, E^{(k - \ell)}_{\alpha}(\tau, s - k + 1) \Im(\tau)^{k-2}\d x \d y\\
  &= \left\langle \breve{f}^*(\tau), \breve g(\tau) \cdot E^{(k - \ell)}_{\alpha}(\tau, s - k + 1)\right\rangle_{\Gamma_1(N)}.
 \end{align*}

 The next theorem shows that the function $\mathcal{D}(\breve f, \breve g, 1/N, s)$ is an ``approximation'' to the completed $L$-function $\Lambda(f, g, s)$ of the previous section, differing from it only by possible bad Euler factors at primes $\ell \mid N$.

 \begin{theorem}[Rankin--Selberg, Shimura]
  We have
  \[
   \mathcal{D}(\breve f, \breve g, 1/N, s) = 2^{1-k} i^{k-\ell} N^{2s + 2 - k - \ell} \Lambda(f, g, s) C(\breve f, \breve g, s),
  \]
  where
  \[ C(\breve f, \breve g, s) \coloneqq \left(\prod_{p \mid N} P_p(f, g, p^{-s}) \right) \sum_{n \in S(N)} a_n(\breve f) a_n(\breve g) n^{-s}\]
  is a polynomial in the variables $p^{-s}$ for $p \mid N$; in particular, it is holomorphic for all $s \in \CC$. Here $S(N)$ is the set of integers all of whose prime factors divide $N$.
 \end{theorem}

 \begin{proof}
  See \cite[Proposition 7.1]{kato04}; our $E^{(j)}_{1/N}(\tau, s)$ corresponds to
  \[ (-2\pi i)^{-j} \Gamma(s + j) \pi^{-s} \Im(\tau)^s E(j, \tau, 1/N, 2s)\]
  in Kato's notation, where $j = k - \ell$. To see that $C(\breve f, \breve g, s)$ is a polynomial, it suffices to consider the case when $\breve f = f(az)$ and $\breve g = g(bz)$ for integers $a \mid N/N_f, b \mid N/N_g$, in which case the result is clear.
 \end{proof}

 In particular, for $s = 1$ and $k = \ell = 2$, using Equation \eqref{eq:Lfcnzero} and the above proposition gives
 \begin{equation}
  \label{eq:Lvalueat1}
  \mathcal{D}(\breve f, \breve g, 1/N, 1) = (4 \pi)^{-1} L'(f, g, 1) C(\breve f, \breve g, 1).
 \end{equation}

 \begin{remark}
  If $f$, $g$ have coprime levels $N_f, N_g$ with $N_f N_g = N$, and we take $\breve f = f$ and $\breve g = g$ to be the normalized newforms, then $C(\breve f, \breve g, s)$ is identically 1, so in this case $\mathcal{D}(\breve f, \breve g, 1/N, s)$ is $N^{2s} \Lambda(f, g, s)$ up to constants.
 \end{remark}

 From the functional equation for the real-analytic Eisenstein series, and the action of Atkin--Lehner involutions, we have
 \begin{equation}
 \begin{aligned}
 \mathcal{D}(\breve f, \breve g, x, k + \ell - 1 - s) &= \left\langle \breve f^*(\tau), \breve g(\tau) \cdot F^{(k - \ell)}_{x/N}(\tau, s - k + 1)\right\rangle_{\Gamma_1(N)}\\
 &= N^{1 - s} \sum_{y \in \ZZ / N\ZZ} e^{2\pi i x y} \mathcal{D}( w_N \breve f, w_N \breve g, y, s).
 \end{aligned}
 \end{equation}
 Here $w_N \breve f$ is the function $\tau \mapsto N^{-1} \tau^{-k} \breve f(-1/(N\tau))$; that is, we have chosen our normalizations so that $w_N$ is an involution in weight 2 (but not in more general weights).


\subsection{The Beilinson regulator}
 \label{sect:beilinson regulator}

 For any smooth variety $X$ over a subfield of $\CC$ there is a canonical map, the \emph{Beilinson regulator}, from $H^3_{\mathcal{M}}(X, \ZZ(2))$ into complex-analytic Deligne--Beilinson cohomology. These maps were introduced in \cite{beilinson84}. We shall only need these maps for $H^3_{\mathcal{M}}(X, \ZZ(2))$ where $X$ is a projective surface, in which case the target group can be identified with de Rham cohomology:

 \begin{theorem}[{Beilinson, cf.~\cite[p.~45]{jannsen88b}}]
  \label{thm:beilinson regulator formula}
  Let $X$ be a smooth projective surface over $\CC$ (or a subfield of $\CC$). There is a homomorphism
  \[ \reg_{\CC}: \CH^2(X, 1) \to H^2_{\dR}(X / \CC) / \Fil^2 = \left(\Fil^1 H^2_{\dR}(X / \CC)\right)^\vee\]
  which sends the class of $\sum_j (Z_j, g_j) \in Z^2(X, 1)$ to the linear functional
  \begin{equation}
   \label{eq:beilinsonreg}
   \omega \mapsto \frac{1}{2\pi i} \sum_j \int_{Z_j - Z_j^{\mathrm{sing}}} \omega \log |g_j|.
  \end{equation}
 \end{theorem}

 We now show that the images of the generalized Beilinson--Flach elements $\Xi_{m, N, j}$ under $\reg_\CC$, paired with differentials corresponding to weight 2 modular forms $f, g$, are related to the derivatives of Rankin--Selberg $L$-functions at the point $s = 1$. More precisely, we shall apply $\reg_{\CC}$ to a lifting of $\Xi_{m, N, j}$ to $\CH^2(X_1(N)^2 \otimes \QQ(\mu_m), 1) \otimes \QQ$; the result will turn out to be independent of the choice of lifting.

 \begin{definition}
  If $f \in S_2(\Gamma_1(N))$, we let $f^* \in S_2(\Gamma_1(N))$ be the form obtained by applying complex conjugation to the Fourier coefficients of $f$.

  We let $\omega_f$ denote the holomorphic differential on $X_1(N)$ whose pullback to $\cH$ is $2 \pi i f(z)\d z$, and $\eta_f^{\ah}$ the anti-holomorphic differential $\overline{\omega_{f^*}}$, whose pullback is $-2\pi i f(-\overline{z}) \d\overline{z}$.
 \end{definition}

 \begin{remark}
  \begin{enumerate}
   \item The factor $2 \pi i$ is convenient since $\frac{\d q}{q} = 2\pi i \d z$.
   \item The map $f \to \eta_f^{\ah}$ is $\CC$-linear and Hecke-equivariant (whereas the more obvious map $f \mapsto \overline{\omega_f}$ has neither of these desirable properties).
  \end{enumerate}
 \end{remark}

 \begin{theorem}[{Beilinson, cf.~\cite[Proposition 4.1]{BDR12}}]
  \label{thm:beilinson1}
  Let $\widetilde{\Xi}_{N}$ be any element of $\CH^2(X_1(N)^2, 1)$ lifting $\Xi_N \coloneqq \Xi_{1, N, 1} \in \CH^2(X_1(N)^2, 1)$, and let $p_1, p_2$ be the projections of $X_1(N)^2$ onto its two factors. Then for $\breve f, \breve g$ as above we have
  \[ \left\langle \reg_{\CC} \left(\widetilde{\Xi}_{N}\right), p_1^*(\eta_{\breve f}^{\ah}) \wedge p_2^*(\omega_{\breve g})\right\rangle = 2\pi \mathcal{D}(\breve f, \breve g, 1/N, 1) = \tfrac{1}{2} L'(f, g, 1) C(\breve f, \breve g, 1). \]
 \end{theorem}

 \begin{proof}
  We have
  \begin{align*}
   \mathcal{D}'(f, g, 1/N, 1) &= \left(  \int_{\Gamma_1(N) \backslash \cH} f(-\bar{\tau}) g(\tau) E^{(0)}_{1/N}(\tau, 0)\, \mathrm{d}x \wedge \mathrm{d}y \right)\\
   &= 2 \int_{\Gamma_1(N) \backslash \cH} f(-\overline{\tau}) g(\tau) \log \left|g_{0, 1/N}(\tau)\right| \, \mathrm{d}x \wedge \mathrm{d}y \\
   &= 2 \int_{\Gamma_1(N) \backslash \cH} f(-\overline{\tau}) g(\tau) \log \left|g_{0, 1/N}(\tau)\right| \, \frac{(-2\pi i\, \mathrm{d}\bar{z}) \wedge (2\pi i\, \mathrm{d}z)}{8\pi^2 i}\\
   &= \frac{1}{2\pi} \left( \frac{1}{2 \pi i} \int_{Y_1(N)(\CC)} \log \left|g_{0, 1/N}\right|\, \eta_f^{\ah} \wedge \eta_g\right).
  \end{align*}
  We compare this with Beilinson's formula for the regulator on $\CH^2(X_1(N)^2, 1)$ (Theorem \ref{thm:beilinson regulator formula}). We know that $\widetilde{\Xi}_{N}$ can be written as the class of $(\Delta, g_{0, 1/N})$ (where $\Delta = C_{1, N, 1}$ is the diagonal in $X_1(N)^2$) plus a linear combination of elements supported on cuspidal components. It is clear that $p_1^*(\eta_f^{\ah}) \wedge p_2^*(\omega_g)$ restricts to 0 on any horizontal or vertical component, and to $\eta_f^{\ah} \wedge \eta_g$ on $\Delta$; so we obtain
  \[ \left\langle \reg_{\CC} \left(\widetilde{\Xi}_{N}\right), \eta_f^{\ah} \wedge \eta_g\right\rangle = \frac{1}{2\pi i} \int_{Y_1(N)}  \log \left|g_{0, 1/N}\right|\, \eta_f^{\ah} \wedge \eta_g = 2\pi D'(f, g, 1)\]
  as required. The final equality follows from Equation \eqref{eq:Lvalueat1}.
 \end{proof}

 We are interested in a version of Theorem \ref{thm:beilinson1} for $m \ge 1$, incorporating twists by Dirichlet characters. This relation becomes easier to state if we introduce ``equivariant'' versions of some of our objects, as follows:

 \begin{definition}
  For $N \ge 5, m \ge 1$ as above, and cusp forms $f, g$ of level $N$ which are eigenforms for the Hecke operators away from $N$, we define the following elements:
  \[\begin{aligned}
   \bfg_m &= \sum_{a \in (\ZZ / m\ZZ)^\times} [a]^{-1} \otimes g\left(z + \tfrac{a}{m}\right) &\in \CC[(\ZZ/m\ZZ)^\times] &\otimes_{\CC} S_2(\Gamma_1(m^2 N), \CC)
  \end{aligned}\]
  and the $\CC[(\ZZ/m\ZZ)^\times]$-valued Dirichlet series
  \[
   L_{(mN)}(f, g, (\ZZ/m\ZZ)^\times, s) = \prod_{\ell \nmid mN} P_\ell(f, g, [\ell] \ell^{-s})^{-1}.
  \]
 \end{definition}

 (There is no obvious way to define an equivariant Euler factor at the primes dividing $m$.)

 \begin{proposition}
  \begin{enumerate}[(a)]
   \item We have
   \[ a_n(\bfg_m) = a_n(g) \tau(n, m),\]
   where $\tau(n, m)$ is the ``universal Gauss sum'' $\sum_{a \in (\ZZ / m\ZZ)^\times} [a]^{-1} e^{2\pi i n a / m} \in \CC[(\ZZ/m\ZZ)^\times]$.
   \item If we extend the Hecke operators on $S_2(\Gamma_1(N))$ linearly to $\CC[(\ZZ/m\ZZ)^\times] \otimes_{\CC} S_2(\Gamma_1(m^2 N))$, then we have
   \begin{align*}
    T_n(\bfg_m) &= [n] t_g(n) \bfg_m\\
    \langle n \rangle (\bfg_m) &= [n]^2 \varepsilon_g(m) \bfg_m
   \end{align*}
   for all $n$ such that $(n, mN) = 1$, where $t_g(n)$ and $\varepsilon_g(n)$ are the eigenvalues of $g$ for the $T_n$ and $\langle n \rangle$ operators respectively.
  \end{enumerate}
 \end{proposition}

 (We can interpret (b) above as stating that $\bfg_m$ transforms under the Hecke operators away from $mN$ as ``$g$ twisted by the universal character of level $m$''.)

 \begin{proof}
  Part (a) is immediate by a $q$-expansion computation. For part (b), we note that the statement regarding the diamond operators can be verified directly -- by essentially the same computation as Proposition \ref{prop:BFeltproperties}(4) -- and the statement for the $T_n$'s now follows immediately from the standard formulae for the action of $T_n$ on $q$-expansions, together with the easily verified fact that $\tau(nn', m) = [n] \tau(n', m)$ if $(n, m) = 1$.
 \end{proof}

 We extend the Beilinson regulator to a homomorphism
 \[ \reg_{\CC[(\ZZ / m\ZZ)^\times]} : \CH^2(X_1(N) \otimes \QQ(\mu_m), 1) \to \CC[(\ZZ/m\ZZ)^\times] \otimes_{\CC} \left(\Fil^1 H^2_{\dR}(X / \CC)\right)^\vee\]
 by mapping $\delta$ to $\sum_{a \in (\ZZ / N\ZZ)^\times} [a] \otimes \reg_{\CC}(\sigma_a \cdot \delta)$. (Note that this is \emph{not} a homomorphism of modules over the group ring $\CC[(\ZZ / m\ZZ)^\times]$; the Poincar\'e duality pairing interchanges the natural action of $\CC[(\ZZ / m\ZZ)^\times]$ with its inverse.)

 \begin{theorem}
  \label{thm:beilinsonreg}
  Let $\breve f, \breve g$ be as above, and let $\widetilde{\Xi}_{m, N, 1}$ be any lifting of $\Xi_{m, N, 1}$ to $X_1(N)^2$. Then as elements of $\CC[(\ZZ / m\ZZ)^\times]$ we have
  \[ \left\langle \reg_{\CC[(\ZZ / m\ZZ)^\times]}(\widetilde{\Xi}_{m, N, 1}), p_1^*(\eta_{\breve f}^{\ah}) \wedge p_2^*(\omega_{\breve g}) \right\rangle =  \tfrac{1}{2} L_{(mN)}'(f, g, (\ZZ/m\ZZ)^\times, 1)A(\breve f, \breve g, m, 1),\]
  where we define
  \[ A(\breve f, \breve g, m, s) = \sum_{a \in (\ZZ / m\ZZ)^\times} [a]^{-1} \sum_{n \in S(mN)} a_n(f) a_n(g) e^{2\pi i a n / m} n^{-s}.\]
 \end{theorem}

 \begin{proof}
  As we showed in the previous section, $\widetilde{\Xi}_{m, N,j}$ may be represented as the class of an element in $Z^2(X_1(N)^2 \otimes \QQ(\mu_m), 1) \otimes \QQ[ (\ZZ/m\ZZ)^\times]$ which differs by negiligible elements from \[\left(C_{m, N, j}, (\iota_{m, N, j})_*(g_{0, 1/m^2N})\right).\]
  As in the case $m = 1$ considered above, these negligible elements pair to 0 with the differential $p_1^*(\eta_f^{\ah}) \wedge p_2^*(\omega_g)$. Hence we have
  \begin{multline*}
   \langle \reg_{\CC}(\widetilde{\Xi_{m, N}}), p_1^*(\eta_f^{\ah}) \wedge p_2^*(\omega_g) \rangle =
   \sum_{j \in (\ZZ / m\ZZ)^\times} [j]^{-1} \int_{C_{m, N, j}} \log \left| (\iota_{m, N, j})_* (g_{0, 1/m^2N}) \right| \cdot p_1^*(\eta_f^{\ah}) \wedge p_2^*(\omega_g)\\
   = \sum_{j \in (\ZZ / m\ZZ)^\times} [j]^{-1} \int_{X_1(m^2 N)} \log\left|g_{0, 1/m^2N}\right|\cdot (p_1 \circ \iota_{m, N, j})^*(\eta_f^{\ah}) \wedge (p_2 \circ \iota_{m, N, j})^*(\omega_g).
  \end{multline*}
  By construction $p_1 \circ \iota_{m, N, j}$ is just the natural projection map $X_1(m^2 N) \to X_1(N)$, so the pullback of $\eta_f^{\ah}$ along this map is just $\eta_f^{\ah}$ again (where now we consider $f$ as a modular form of level $m^2 N$). On the other hand, $p_2 \circ \iota_{m, N, j}$ corresponds to the map $z \mapsto z + \tfrac j m$ on the upper half-plane, so $(p_2 \circ \iota_{m, N, j})^*(\omega_g)$ is the differential whose pullback to $\cH$ is $2\pi i g\left(z + \tfrac{j}{m}\right) \d z$, and hence we have
  \[ \sum_{j} [j]^{-1} (p_2 \circ \iota_{m, N, j})^*(\omega_g) = 2\pi i \bfg_m(z)\d z\]
  as elements of $\CC[(\ZZ / m\ZZ)^\times] \otimes \Omega^1_{\mathrm{hol}}(X_1(m^2 N))$. Hence, by exactly the same computation as above,
  \[ \langle \reg_{\CC}(\widetilde{\Xi_{m, N}}), p_1^*(\eta_f^{\ah}) \wedge p_2^*(\omega_g) \rangle = 4 \pi \int_{X_1(m^2 N)} f(-\bar{\tau}) \bfg_m(\tau) \log\left|g_{0,1/m^2N}\right| \d x \wedge \d y.\]

  As remarked above, $\bfg_m$ is an eigenform for the Hecke operators away from $mN$; so we may now apply exactly the same formal manipulations as in the proof of \cite[Proposition 7.1]{kato04}, but with group ring coefficients rather than $\CC$ coefficients, and the result follows in this case also.
 \end{proof}


\subsection{A non-vanishing result}

 In this section, we shall use the results of the previous section, together with a deep theorem of Shahidi on the non-vanishing of Rankin--Selberg $L$-values, to show that the elements $\Xi_{m, N, j}$ are not all zero (which is in no way obvious from their construction).

 \begin{theorem}[{Shahidi, \cite[Theorem 5.2]{shahidi81}}]
  \label{thm:shahidi} Let $f, g$ be any two newforms of weight 2. Then the completed $L$-function $\Lambda(f, g, s)$ is holomorphic and nonvanishing on the line $\Re(s) = 2$, unless $f = g^*$, in which case it has a simple pole at $s = 2$.
 \end{theorem}

 \begin{remark}
  We have stated only a special case of Shahidi's very general theorem, which applies to automorphic forms on $\GL_n \times \GL_m$ over an arbitrary number field. Note also that Shahidi's normalizations are slightly different from ours (he normalizes the $L$-function so that the abcissa of symmetry is $s = \tfrac{1}{2}$, independently of the weights of $f$ and $g$, while we normalize it to be at $s = \frac{k + \ell - 1}{2} = \frac{3}{2}$).
 \end{remark}

 \begin{corollary}
  If $\Sigma$ is a finite set of primes, then the function $L_{\Sigma}(f, g, s)$ has a zero at $s = 1$ of order $r_1 + r_2$, where
  \[ r_1 =
   \begin{cases}
    1 & \text{if $f^* \ne g$}\\
    0 & \text{if $f^* = g$}
   \end{cases}
  \]
  and $r_2$ is the sum of the orders of the poles at $s = 1$ of the Euler factors $L_p(f, g, s)$ for primes $p \in \Sigma$.
 \end{corollary}

 \begin{proof}
  Applying the functional equation for the completed $L$-function, which switches $s$ with $3-s$, we deduce from Shahidi's result that that $\Lambda(f, g, s)$ is holomorphic and nonvanishing (resp. has a simple pole) at $s = 1$ if $f \ne g^*$ (resp. if $f = g^*$).

  However, the $L$-factor at $\infty$, $L_\infty(f, g, s) = \Gamma_\CC(s) \Gamma_{\CC}(s - 1)$, has a simple pole at $s = 1$, so the order of vanishing of $L(f, g, s)$ is $r_1$ as defined above. Since $L_{\Sigma}(f, g, s)$ is $L(f, g, s)$ divided by the product of the $L$-factors at primes in $\Sigma$, the result clearly follows.
 \end{proof}

 \begin{remark}
  Note that if $f^* = g$, then the local $L$-factor vanishes at $s = 1$ for \emph{every} prime, so $r_2$ will tend to be rather large in this case.
 \end{remark}

 \begin{corollary}
  Let $f, g$ be any two newforms, $\Sigma$ any set of primes, and $p$ any prime in $\Sigma$. Then for all but finitely many Dirichlet characters $\chi$ of $p$-power conductor, $\prod_{\ell \in \Sigma} L_\ell(f, g \otimes \chi, s)$ is holomorphic and nonzero at $s = 1$.
 \end{corollary}

 \begin{proof}
  If $\chi$ has sufficiently large $p$-power conductor, then the local $L$-factor of $f \otimes g \otimes \chi$ at $p$ is identically 1; so it suffices to consider the $L$-factors at primes $\ell \ne p$. However, since $\chi$ has conductor prime to $\ell$, $L_\ell(f \otimes g \otimes \chi, s) = P_\ell(\chi(\ell) \ell^{-s})^{-1}$, so it suffices to arrange that $\chi(\ell) \ell^{-1}$ does not lie in the finite set of zeroes of the polynomial $L_\ell(f \otimes g, X)$. It is clear that this may also be achieved by ensuring that the conductor of $\chi$ is sufficiently big.
 \end{proof}

 \begin{corollary}
  Given any two forms $f, g$ of level $N$ that are eigenvectors for all Hecke operators, and $p$ any prime, there is $k \ge 0$ such that the projection of $\Xi_{mp^k, N, 1}$ to the $(f, g)$-isotypical quotient of $\CH^2(Y_1(N)^2 \otimes \QQ(\mu_{mp^k}), 1)$ is nonzero.
 \end{corollary}

 \begin{proof}
  Immediate from the previous corollary and Theorem \ref{thm:beilinsonreg}.
 \end{proof}


\section{Relation to \texorpdfstring{$p$-adic $L$-values}{p-adic L-values}}

 In this section we develop an analogue of the $m = 1$ case of Theorem \ref{thm:beilinsonreg} in the $p$-adic setting. This is essentially a variant of the main theorem of \cite{BDR12}.

\subsection{Holomorphic Eisenstein series}

  We begin by constructing some holomorphic Eisenstein series which may be defined over a number field. We follow chapter 3 of \cite{kato04} closely, but we work on $Y_1(N)$ rather than $Y(N)$. Our purpose is to define, for $\alpha \in \QQ / \ZZ$, the following modular forms:
 \begin{itemize}
  \item $E^{(k)}_{\alpha} \in M_k(\Gamma_1(N))$, where $k \ge 1$, $k \ne 2$;
  \item $\widetilde E^{(2)}_{\alpha} \in M_2(\Gamma_1(N))$;
  \item $F^{(k)}_{\alpha} \in M_k(\Gamma_1(N))$, for $k \ge 1$, with $\alpha \ne 0$ if $k = 2$.
 \end{itemize}

 We set
 \[ E^{(k)}_{\alpha}(\tau) =  E^{(k)}_{\alpha}(\tau, 0),\]
 and similarly for $F^{(k)}$.

 \begin{proposition}
  If $k \ge 1$, $k \ne 2$, then $E^{(k)}_{\alpha}, F^{(k)}_\alpha \in M_k(\Gamma_1(N))$ for any $\alpha \in \tfrac{1}{N} \ZZ / \ZZ$.

  For $k = 2$, we have $F^{(2)}_{\alpha} \in M_2(\Gamma_1(N))$ for any $\alpha \ne 0$, and $\widetilde E^{(2)}_{\alpha} \coloneqq E^{(2)}_{\alpha} - E^{(2)}_{0} \in M_2(\Gamma_1(N))$ (for any $\alpha$). The function $F^{(2)}_{0} = E^{(2)}_0$ is a $C^\infty$ function on $\cH$ invariant under the weight 2 action of $\Gamma_1(N)$, with slow growth at the cusps, but is not holomorphic.
 \end{proposition}

 \begin{proof}
  See \cite[\S 3.8]{kato04}; our $E^{(k)}_{\alpha}$ is Kato's $E^{(k)}_{0, \alpha}$.
 \end{proof}

 We have $q$-expansion formulae for both families. Let $\alpha \in \QQ / \ZZ$. For $\Re(s) > 1$, define
 \[
  \zeta(\alpha, s) = \sum_{\substack{n \in \QQ, n > 0 \\ n = \alpha \bmod \ZZ}} n^{-s}
  \quad\text{and}\quad
  \zeta^*(\alpha, s) = \sum_{n=1}^\infty e^{2\pi i \alpha n} n^{-s}
 \]
 as in \cite[\S 3.9]{kato04}. Then both $\zeta(\alpha, s)$ and $\zeta^*(\alpha, s)$ have meromorphic continuation to all $s \in \CC$, and satisfy
 \[
  \zeta^*(\alpha, 1-s) = \frac{\Gamma(s)}{(2\pi)^s}\left( e^{-i\pi s/2}\zeta(-\alpha, s) + e^{i\pi s/2}\zeta(\alpha, s)\right),
 \]
 a version of the standard functional equation for the Hurwitz zeta function.

 \begin{proposition} Let $k \ge 1$, $\alpha \in \QQ / \ZZ$.
  \label{prop:eisqexp}
  \begin{enumerate}
   \item Assume $k \ne 2$. Then we have
   \[ E^{(k)}_{\alpha} = a_0 + \sum_{n \ge 1} \left( \sum_{d \mid n} d^{k-1}(e^{2\pi i \alpha d} + (-1)^k e^{-2\pi i \alpha d}) \right) q^n,\]
   where
   \[ a_0 =
    \begin{cases}
     \zeta^*\left(\alpha, 1-k\right)  & \text{if $k \ge 3$} \\
     \frac{1}{2} \left(\zeta^*\left(\alpha, 0\right) - \zeta^*\left(-\alpha, 0\right)\right) & \text{if $k = 1$}.
    \end{cases}.
   \]
   \item We have
   \[ \widetilde E^{(2)}_{\alpha} = a_0 + \sum_{n \ge 1} \left( \sum_{d \mid n} d (e^{2\pi i \alpha d} + e^{-2\pi i \alpha d} - 2) \right) q^n,\]
   where $a_0 = \zeta^*\left(\alpha, -1\right) + \tfrac1{12}$.
   \item Assume $\alpha \ne 0$ in the case $k = 2$. Then
   \[ F^{(k)}_{\alpha} = a_0 + \sum_{n \ge 1} \left( \sum_{d \mid n} \left(\tfrac{n}{d}\right)^{k-1} (e^{2 \pi i \alpha d} + (-1)^k e^{-2 \pi i \alpha d}) \right) q^n,\]
   where
   \[ a_0 =
    \begin{cases}
     \zeta(1-k)  & \text{if $k \ge 2$} \\
     \frac{1}{2} \left(\zeta^*\left(\alpha, 0\right) - \zeta^*\left(-\alpha, 0\right)\right) & \text{if $k = 1$}.
    \end{cases}.
   \]
  \end{enumerate}
 \end{proposition}

 \begin{proof}
  This is \cite[Proposition 3.10]{kato04}. Note that there is a typographical error in the statement of the proposition \emph{loc.cit.}; there is an extra star in the formula for $\sum a_n n^{-s}$ in case (1), and the formula should read
  \[ \sum_{n \in \QQ, n > 0} a_n n^{-s} = \zeta(\alpha, s) \zeta^*(\beta, s-k + 1) + (-1)^k \zeta(-\alpha, s) \zeta^*(-\beta, s-k + 1).\]
 \end{proof}

\subsection{Nearly holomorphic modular forms}

 For $k \ge 0$, we define (following e.g.~\cite{shimura86, shimura00}) the space of \emph{nearly holomorphic} modular forms $M_k^{\nh}(\Gamma_1(N), \CC)$. This is the space of $C^\infty$ slowly-increasing functions on $\cH$ which are invariant under the weight $k$ action of $\Gamma_1(N)$ and are annihilated by some power of the Maass--Shimura weight-lowering differential operator
 \[ \varepsilon_k = \frac{-1}{2\pi i} \Im(\tau)^2 \frac{\d}{\d \overline{\tau}}.\]
 Any such function is in fact annhilated by $\varepsilon^{[k/2] + 1}$, and can be expanded as a finite sum
 \begin{equation}
  \label{eq:nearlyhol}
  f(\tau) = \sum_{j = 0}^{[k/2]} f_j(\tau) \left(\pi \Im(\tau) \right)^{-j},
 \end{equation}
 where the $f_j$ are holomorphic functions. (In particular, any nearly holomorphic form of weight 0 or 1 is in fact a holomorphic form.) For $K$ a number field containing the $N$-th roots of unity\footnote{This is for compatibility with our notation for classical modular forms, since in our model of $Y_1(N)$, the cusp $\infty$ is not rational.} we shall say that $f \in M_k^{\nh}(\Gamma_1(N), K)$ is defined over $K$ if the Fourier coefficients of the holomorphic functions $f_j$ are in $K$, and write $M_k^\nh(\Gamma_1(N), K)$ for the space of such functions.

 We let $S_k^\nh(\Gamma_1(N), K)$ be the subspace of rapidly decreasing functions in $M_k^\nh(\Gamma_1(N), K)$. If $k \ge 2$, then $S_k^\nh(\Gamma_1(N), K)$ coincides with the space defined algebraically in \cite[\S 2.4]{darmonrotger12} using the ``Hodge splitting'' of the de Rham cohomology.

 \begin{corollary}
  Let $k, j$ be integers with $k \ge 1$ and $j \in [0, k-1]$. Then for any $\alpha \in \tfrac{1}{N}\ZZ / \ZZ$, the function $\tau \mapsto E^{(k)}_\alpha(\tau, -j)$ lies in $M_k^{\nh}(\Gamma_1(N), \QQ(\mu_N))$.
 \end{corollary}

 \begin{proof}
  We first note that $E^{(k)}_{\alpha}, F^{(k)}_\alpha \in M^\nh_k(\Gamma_1(N), K)$ for all $k \ge 1$. This is clear for $k = 1$ or $k \ge 3$ (since holomorphic forms are certainly nearly holomorphic), and for $k = 2$ it suffices to check that $E^{(2)}_0 = F^{(2)}_0$ is nearly holomorphic, which is clear from the formula
  \[ E^{(2)}_0(\tau) = \frac{1}{4\pi \Im(\tau)} - \tfrac{1}{12} + 2\sum_{n \ge 1} \left(\sum_{d \mid n} d\right) q^n.\]

  With this in hand, we obtain the near-holomorphy of $E^{(k)}_\alpha(\tau, -j)$ for $0 \le j \le \tfrac{k-1}{2}$ by applying $\delta^j$ to $E^{(k-2j)}_\alpha = E^{(k-2j)}_\alpha(\tau, 0)$. Applying the same argument to $F^{(k-2j)}_\alpha$ shows that $F^{(k)}_\alpha(\tau, -j)$ is nearly holomorphic for $j$ in the same range; but $F^{(k)}_\alpha(\tau, -j) = E^{(k)}_\alpha(\tau, 1-k + j)$ by the functional equation, as required, so we obtain the result for all $j \in [0, k-1]$.
 \end{proof}

 We define the \emph{$q$-expansion} of a nearly-holomorphic modular form $f$ to be the Fourier expansion of the holomorphic $\ZZ$-periodic function $f_0$, when we write $f$ in the form \eqref{eq:nearlyhol}. Then $\delta$ corresponds to $q \tfrac{\mathrm{d}}{\mathrm{d}q}$ on $q$-expansions, and the following is clear from Proposition \ref{prop:eisqexp} and the proof of the previous proposition:

 \begin{proposition}
  For any $j \in [0, k-1]$, the $q$-expansion of the nearly-holomorphic form $E^{(k)}_\alpha(\tau, -j)$ is
  \[ a_0 + \sum_{n \ge 1} \left(\sum_{d \mid n} d^{k-1-j} \left(\tfrac{n}{d}\right)^j(e^{2\pi i \alpha d} + (-1)^k e^{-2\pi i \alpha d})\right) q^n,\]
  where $a_0 = 0$ unless $j \in \{0, k-1\}$.
 \end{proposition}


\subsection{P-adic families of Eisenstein series}

 We now use the $q$-expansion formulae above as motivation for defining a two-parameter family of $p$-adic Eisenstein series, which can be regarded as an analogue of the $E^{(k)}(-, s)$, with two continuous $p$-adic parameters $\phi_1, \phi_2$ replacing the discrete parameter $k$ and the continuous real-analytic parameter $s$.

 \begin{definition}
  Choose some (sufficiently large) finite extension $L / \Qp$ and let $\Lambda = \cO_L[[\Zp^\times]]$, the Iwasawa algebra of $\Zp^\times$. Let $\Omega = \operatorname{Spf} \Lambda$, the \emph{weight space} classifying continuous characters $\phi: \Zp^\times \to \CC_p$; and consider the formal power series with coefficients in $\Lambda \htimes \Lambda$ given by
 \[ \mathcal{E}_\alpha(\phi_1, \phi_2) = \sum_{\substack{n \ge 1 \\ p \nmid n}} \left( \sum_{d \mid n} \phi_1(d) \phi_2\left(\tfrac n d\right) \left[ e^{2 \pi i \alpha d} + \varepsilon e^{-2\pi i \alpha d}\right] \right)q^n \in \cO_L[[q]],\]
 where $\varepsilon = -\phi_1(-1) \phi_2(-1)$.
 \end{definition}

 \begin{note}
  We consider $\ZZ \times \ZZ$ as a subset of $\Omega \times \Omega$ in the natural way. Then for $k \ge 1, k \ne 2$, we have $\mathcal{E}_\alpha(k-1, 0) = (E^{(k)}_{\alpha})^{[p]}$ and $\mathcal{E}_\alpha(0, k-1) = (F^{(k)}_{\alpha})^{[p]}$, where $(-)^{[p]}$ denotes the ``$p$-depletion'' operator.
 \end{note}

 We now fix a newform $g \in S_\ell(\Gamma_1(N_g))$, for some $\ell$ and some $N_g \mid N$. (Although the weight $\ell$ will be fixed in our discussion, it is convenient to keep it in the notation as a parameter, since this will make our notation more consistent with \cite{BDR12}.) We will write $\breve g$ for any element of $S_\ell(N)[\pi_g]$.

 \begin{definition}
  For integers $k, j$, we define
  \[ \Xi(k, \ell, \alpha, j)^{\ord, p} = \eord\left[ \mathcal{E}_\alpha(j - \ell, k - 1 - j) \cdot \breve g \right].\]
 \end{definition}

 This is an ordinary $\Lambda$-adic family of modular forms, parametrized by $k$ and $j$ (we are taking $\ell$ to be fixed here, in order to avoid the need to make any ordinarity hypotheses on $g$). For any $(k, j)$, $\Xi(k, \ell, \alpha, j)^{\ord, p}$ is a $p$-adic modular form of weight $k$.

 We now compare this with the complex-analytic theory. It is clear that $\mathcal{E}_\alpha(j - \ell, k - 1 - j)$ is the $p$-depletion of the nearly-overconvergent form $\tau \mapsto E^{(k-\ell)}(\tau, -k + j + 1)$.

 \begin{definition}
  For $\ell \le j \le k-1$, let $\Xi(k, \ell, \alpha, j)$ denote the nearly-holomorphic modular form of weight $k-\ell$ given by
  \[ \Xi(k, \ell, \alpha, j)(\tau) = E^{(k-\ell)}_\alpha(\tau, -k + j + 1) \cdot \breve g,\]
  and $\Xi(k, \ell, \alpha, j)^{\mathrm{hol}}$ its image under the holomorphic projector.
 \end{definition}

 \begin{notation}
  Let $H^1(Y_1(N), \mathcal{L}_{k-2}, \nabla)$ denote the de Rham cohomology of $Y_1(N)$ with coefficients in the \mbox{$(k-2)$-nd} symmetric power of the relative de Rham cohomology sheaf of the universal elliptic curve over $Y_1(N)$, endowed with its Gauss--Manin connection.
  \end{notation}

 \begin{proposition}
  \label{prop:rankinselbergpstab}
  Let $k, \ell$ be fixed, with $k > \ell$. Let $f$ be a newform in $S_k(\Gamma_1(N_f))$, for some $N_f \mid N$, and let $e_{f^*}$ be the projection to the $f^*$-isotypic component in the Hecke algebra acting on $S_k(\Gamma_1(N))$. Assume $f$ is ordinary at $p$, and let $j \in [\ell, k-1]$.
  Then we have the following relation in $H^1(Y_1(N), \mathcal{L}_{k-2}, \nabla)$:
  \[ e_{f^*} \Xi(k, \ell, \alpha, j)^{\ord, p} = \frac{\mathcal{E}(f, g, j)}{\mathcal{E}(f)} e_{f^*} \eord \Xi(k, \ell, \alpha, j)^{\mathrm{hol}},\]
  where
  \[ \mathcal{E}(f) = 1 - p^{-1}\beta_p(f) \alpha_p(f)^{-1}\]
  and
  \begin{multline*}
   \mathcal{E}(f, g, j) = (1 - p^{-j}\beta_p(f) \alpha_p(g))(1 - p^{-j}\beta_p(f) \beta_p(g))\\
   \times (1 - p^{j-1}\alpha_p(f)^{-1} \alpha_p(g)^{-1})(1 - p^{j-1}\alpha_p(f)^{-1} \beta_p(g)^{-1}).
  \end{multline*}
  Here $\alpha_f, \beta_f$ are the roots of the Hecke polynomial of $f$ at $p$, and similarly for $g$.
 \end{proposition}

 \begin{proof}
  This follows from Proposition 4.15 of \cite{darmonrotger12} with the $f,g,h$ of the theorem taken to be $f$, $E^{(k-\ell)}_\alpha(-, -k+j+1)$ and $g$. (Note that the special case $j \ge \frac{k + \ell - 1}{2}$ is \cite[Proposition 2.7]{BDR12}.)
 \end{proof}


\subsection{Interpolation in Hida families}

 We now interpolate the left-hand side of Proposition \ref{prop:rankinselbergpstab} in Hida families.

 \begin{notation}
  Let $f$ be a newform (of some level $N_f \mid N$) and let $\bff$ be the Hida family through $f$ (with coefficients in some finite flat $\Lambda$-algebra $\Lambda_f$). Then we define the space $\mathbf{S}^{\ord}(N; \Lambda_f)[\pi_f]$ for the $\Lambda_f$-module of families of oldforms at level $N$ corresponding to $f$, which is simply the space of formal $q$-expansions spanned over $\Lambda_f$ by $\bff(q^d)$ for $d \mid N/N_f$. We write $\breve\bff$ for a generic element of $\mathbf{S}^{\ord}(N; \Lambda_f)[\pi_f]$, which we shall think of as a ``test vector'' associated to $\bff$.
 \end{notation}

 We shall continue to write $g$ for a newform in $S_\ell(N_g)$ for some $N_g \mid N$, and $\breve g$ for a generic element of $S_\ell(N; K)[\pi_g]$.

 \begin{proposition}
  For any $\breve\bff \in \mathbf{S}^{\ord}(N; \Lambda_f)[\pi_f]$ and $\breve g \in S_\ell(N; K)[\pi_g]$ as above, and any $\alpha \in \frac{1}{N}\ZZ / \ZZ$, there exists an element
  \[ \mathcal{D}_p(\breve\bff, \breve g, \alpha) \in \operatorname{Frac}(\Lambda_f) \htimes \Lambda\]
  such that for all integers $k, j$ with $k \ge 2$, we have
  \[ \mathcal{D}_p(\breve\bff, \breve g, \alpha)(k, j) = \frac{ \left\langle \breve f_k^*, \Xi(k, \ell, \alpha, j)^{\ord, p}\right\rangle}{\mathcal{E}^*(f_k) \left\langle f_k, f_k \right\rangle},\]
  where $f_k$ and $\breve f_k$ are the eigenforms at level $N$ whose ordinary $p$-stabilizations are the weight $k$ specializations of $\bff$ and $\breve\bff$, and
  \[ \mathcal{E}^*(f_k) \coloneqq 1 - \beta_p(f_k) \alpha_p(f_k)^{-1}.\]
 \end{proposition}

 Combining this with the previous proposition, we have

 \begin{proposition}
  For integers $k, j$ with $\ell \le j \le k-1$, we have
  \[ \mathcal{D}_p(\breve\bff, \breve g, \alpha)(k, j) = \frac{ \mathcal{E}(f_k, g, j)}{\mathcal{E}(f_k) \cdot \mathcal{E}^*(f_k) \cdot \left\langle f_k, f_k \right\rangle} \mathcal{D}(\breve f_k, \breve g, \alpha, j).\]
 \end{proposition}

\begin{note}
  We know that the Atkin--Lehner operator gives an isomorphism
 \[ S_\ell(N; K)[\pi_g] \rTo^{w_N}_{\cong} S_\ell(N; K)[\pi_{g^*}].\]
 Less obviously, there is also an operator
 \[ \mathbf{S}^{\ord}(N; \Lambda_f)[\pi_f] \rTo^{w_N}_{\cong} \mathbf{S}^{\ord}(N; \Lambda_f)[\pi_{f^*}]\]
 interpolating the action of the Atkin--Lehner operators on the weight $k$ specializations. To see this, it suffices to note that the inclusions $S_k(N) \into S_k(Np) \into S_k(Np^2) \into \dots$ commute with the action of $w_{N}$ and this operator is continuous with respect to the $p$-adic norm (by the $q$-expansion principle); the resulting operator on the completion $S_k(Np^\infty)$ commutes with $U_p$, and hence preserves $\eord S_k(Np^\infty)$.
 \end{note}

 \begin{proposition}
  For any $k, j \in \Omega_f \times \Omega$, we have
  \[ \mathcal{D}_p(\breve\bff, \breve g, \alpha)(k, k + \ell - 1 - j) = N^{1-j} \cdot \sum_{y \in \ZZ / N\ZZ} e^{2\pi i \alpha x / N} \mathcal{D}_p(w_N \breve\bff, w_N \breve g, x/N)(k, j).\]
 \end{proposition}

 \begin{proof}
  It suffices to check this result for all pairs of integers $k, j$ with $k \ge j$, since these points are Zariski-dense in $\Omega_f \times \Omega$. By the classical functional equation, we find that for such $k, j$ we have
  \[ \mathcal{D}_p(\breve\bff, \breve g, \alpha)(k, k + \ell - 1 - j) = A \cdot N^{1-j} \sum_{y \in \ZZ / N\ZZ} e^{2\pi i \alpha x / N}\mathcal{D}_p(w_N \breve\bff, w_N \breve g, \alpha)(k, j)\]
  where the quantity $A$ is defined by
  \[ A = \frac{\cE(f_k, g, k + \ell - 1 - j)}{\cE(f_k) \cE^*(f_k) \langle f_k, f_k\rangle} \cdot \left( \frac{\cE(f_k^*, g^*, j)}{\cE(f_k^*) \cE^*(f_k^*) \langle f_k^*, f_k^*\rangle}\right)^{-1}.\]

  We obviously have $\langle f_k^*, f_k^*\rangle = \langle f_k, f_k\rangle$. More subtly, we have $\alpha_p(f^*) = p^{k-1} / \beta_p(f)$ and $\beta_p(f^*) = p^{k-1} / \alpha_p(f)$; similarly, we have $\{ \alpha_p(g^*), \beta_p(g^*)\} = \{p^{\ell-1}/\alpha_p(g), p^{\ell-1}/\beta_p(g)\}$. From these relations, it is clear that $\cE(f_k^*) = \cE(f_k)$, $\cE^*(f_k^*) = \cE(f_k)$, and $\cE(f_k, g, k + \ell - 1 - j) = \cE(f_k^*, g^*, j)$. So the ratio $A$ is identically 1.
 \end{proof}

\begin{notation}
  We write $\mathcal{D}_p(\breve f, \breve g, \alpha)$ for the restriction of $\mathcal{D}_p(\breve \bff, \breve g, \alpha)$ to $k \times \Omega \subset \Omega_f \times \Omega$, where $\breve\bff$ is any element of $\mathbf{S}^{\ord}(N; \Lambda_f)[\pi_f]$ whose specialization in weight $k$ is $\breve f$. (This is independent of the choice of family $\breve\bff$).
\end{notation}

\begin{note}
 If $k > \ell$, then the $L$-function $\mathcal{D}_p(\breve f, \breve g, \alpha)$ interpolates the critical values $\mathcal{D}(f, g, \alpha, s)$; but when $k = \ell$, there are no such critical values.
\end{note}


\subsection{The syntomic regulator}
 \label{sect:syntomic}

 Let $p$ be prime and $K$ be a finite extension of $\Qp$ with ring of integers $\cO_K$, and $\cX$ a smooth proper scheme over $\cO_K$ with generic fibre $X$. Then there exists a map, the \emph{syntomic regulator},
 \[ r_\syn: \CH^2(\mathcal{X}, 1) \to H^2_{\dR}(X / K) / \Fil^2 = \left(\Fil^1 H^2_{\dR}(X / K)\right)^\vee, \]
 with the property that the diagram
 \[
   \begin{diagram}
    \CH^2(\mathcal{X}, 1) & \rTo & \CH^2(X, 1) \\
    \dTo^{r_\syn} & & \dTo^{r_\et} \\
    H^2_{\dR}(X / K) / \Fil^2 & \rInto^{\exp} & H^1(K, H^2_{\et}(\overline{X}, \Qp)(2))
   \end{diagram}
 \]
 commutes (c.f. \cite{besser00}). Here $\exp$ denotes the Bloch--Kato exponential map constructed in \cite{blochkato90} for the crystalline $G_K$-representation $V = H^2_\et(\overline{X}, \Qp)(2)$.

 \begin{remark}
  Note that since $\cX$ has good reduction, all eigenvalues of Frobenius on $\Dcris(V)$ are Weil numbers of weight $-2$; thus $\Dcris(V)^{\vp = 1} = 0$, implying that $\exp$ is injective. This also implies that $H^1_e(K, V) = H^1_f(K, V)$. Note that we do \emph{not} necessarily have $H^1_f(K, V) = H^1_g(K, V)$, since $V^*(1)$ has all weights equal to $0$. It is conjectured that the image of $r_{\et}$ is precisely $H^1_g(K, V)$ but this is only known in a few special cases, cf.~\cite[Fact 1.1]{saitosato10}.
 \end{remark}

\subsection{Generalization of a theorem of Bertolini--Darmon--Rotger}

 Note that there is a map
 \[ \operatorname{dlog}: \cO(Y_1(N))^\times \otimes \QQ \to M_2(\Gamma_1(N)),\]
 which corresponds to $F(\tau) \mapsto \frac{F'(\tau)}{F(\tau)}$ as functions on $\cH$; and this commutes with the Atkin--Lehner involutions.

 \begin{proposition}
  For any $\alpha \ne 0 \in \QQ / \ZZ$, we have
  \[ \operatorname{dlog} g_{0, \alpha} = -F^{(2)}_{\alpha}.\]
 \end{proposition}

 \begin{proof}
  Immediate from comparing the $q$-expansion of $F^{(2)}_{\alpha}$ with that of $g_{0, \alpha}$, which is given in \cite[\S 1.9]{kato04}.
 \end{proof}

 We recall the following result, which is a slight reformulation and extension of the main theorem of \cite{BDR12}.

 \begin{theorem}
  Let $u_\alpha$ be the modular unit on $\cY_1(N) \otimes \ZZ(\mu_N)$ such that
  \[ \operatorname{dlog} u_\alpha = \widetilde E^{(2)}_{\alpha},\]
  and let $\Delta_{u_\alpha}$ be any element of $\CH^2(\cX_1(N) \otimes \ZZ(\mu_N), 1)$ whose pullback to $\CH^2(\cY_1(N) \otimes \ZZ(\mu_N), 1)$ is the class of $(\Delta, u_\alpha)$ where $\Delta$ is the diagonal subvariety.

  Let $f, g$ be any two newforms of weight 2 and levels $N_f, N_g$ dividing $N$, with $f$ ordinary at $p$, and let $\breve f, \breve g$ be test vectors attached to $f, g$ as before. Then we have
  \[ \left( \mathcal{D}_p(\breve f, \breve g, \alpha)(2) - \mathcal{D}_p(\breve f, \breve g, 0)(2) \right) = \frac{\mathcal{E}(f, g, 2)}{\mathcal{E}(f) \cdot \mathcal{E}^*(f)} \left\langle r_\syn(\Delta_{u_\alpha}), \pr_1^*(\eta_{\breve f}^{\ur}) \wedge \pr_2^*(\omega_{\breve g})\right\rangle.\]
 \end{theorem}

 \begin{proof}
  By Fourier inversion on the multiplicative group $(\ZZ / N\ZZ)^\times$, which acts on both sides of the claimed formula, it suffices to show that for each Dirichlet character $\psi$ modulo $N$ we have
  \begin{multline}
   \label{eq:BDRgenclaim}
    \frac{\mathcal{E}(f, g, 2)}{\mathcal{E}(f) \cdot \mathcal{E}^*(f)} \left\langle r_\syn(\Delta_{u_\psi}), \pr_1^*(\eta_{\breve f}^{\ur}) \wedge \pr_2^*(\omega_{\breve g})\right\rangle = \\
    \begin{cases}
     \sum_{d \in (\ZZ / N\ZZ)^\times} \psi(d)^{-1} \mathcal{D}_p(\breve f, \breve g, d\alpha)(2) & \text{if $\psi \ne 1$},\\
     \sum_{d \in (\ZZ / N\ZZ)^\times} \left(\mathcal{D}_p(\breve f, \breve g, d\alpha)(2) - \mathcal{D}_p(\breve f, \breve g, 0)(2) \right) & \text{if $\psi = 1$}.
    \end{cases}
  \end{multline}
  where $u_\psi = \sum_{d} \psi(d)^{-1} \otimes u_{d\alpha} \in \ZZ(\chi) \otimes_{\ZZ} \cO(Y_1(N)^\times)$. However, it is clear that both sides of Equation \eqref{eq:BDRgenclaim} are zero unless $\psi = \chi \coloneqq \chi_f^{-1} \chi_g^{-1}$, so we may assume $\psi = \chi$.

  If $\chi \ne 1$ and $\alpha$ has exact order $N$, then we can assume without loss of generality that $\alpha = 1/N$, and we are in the case studied in \cite{BDR12}. In the remaining cases, the argument goes through essentially identically.
 \end{proof}

 \begin{remark}
  If in fact $\chi$ is primitive modulo $N$, then both sides are zero unless $\alpha$ has exact order $N$, so we may reduce to precisely the case covered by \cite{BDR12}.
 \end{remark}

 We can now deduce our main theorem of this section.

 \begin{theorem}
  \label{thm:syntomicreg}
  Let $f, g, \breve f, \breve g$ be as above. Then we have
  \[ \mathcal{D}_p(\breve f, \breve g, 1/N)(1) = -\frac{\mathcal{E}(f, g, 1)}{\mathcal{E}(f) \cdot \mathcal{E}^*(f)} \left\langle r_\syn(\Xi_{1, N, 0}), \pr_1^*(\eta_{\breve f}^{\ur}) \wedge \pr_2^*(\omega_{\breve g})\right\rangle.\]
 \end{theorem}

 \begin{proof}
  Applying the previous theorem to $w_N \breve f$ and $w_N \breve g$, we have
  \[ \left( \mathcal{D}_p(w_N \breve f, w_N \breve g, x/N)(2) - \mathcal{D}_p(w_N \breve f, w_N \breve g, 0)(2) \right) = \frac{\mathcal{E}(f, g, 1)}{\mathcal{E}(f) \cdot \mathcal{E}^*(f)} \left\langle r_\syn(\Delta_{u_{x/N}}), \pr_1^*(\eta_{w_N \breve f}^{\ur}) \wedge \pr_2^*(\omega_{w_N \breve g})\right\rangle.\]
  We multiply by $e^{2\pi i x / N}$ and sum over $x \in \ZZ / N\ZZ$. The left-hand side becomes
  \[ \sum_{x \in \ZZ / N\ZZ} e^{2\pi i x / N} \mathcal{D}_p(w_N \breve f, w_N \breve g, x/N)(2) = N \mathcal{D}_p(\breve f, \breve g, x/N)(1)\]
  by the $p$-adic functional equation.

  Meanwhile, the right-hand side is
  \[ \sum_{x \in x \in \ZZ / N\ZZ} e^{2\pi i x / N} \frac{\mathcal{E}(f, g, 1)}{\mathcal{E}(f) \cdot \mathcal{E}^*(f)} \left\langle r_\syn(\Delta_{u_{x/N}}), \pr_1^*(\eta_{w_N \breve f}^{\ur}) \wedge \pr_2^*(\omega_{w_N \breve g})\right\rangle.\]
  By the functoriality of the syntomic regulator, we have
  \[ \left\langle r_\syn(\Delta_{u_{x/N}}), \pr_1^*(\eta_{w_N \breve f}^{\ur}) \wedge \pr_2^*(\omega_{w_N \breve g})\right\rangle = \left\langle r_\syn(\Delta_{(w_N^* u_{x/N})}), \pr_1^*(\eta_{\breve f}^{\ur}) \wedge \pr_2^*(\omega_{\breve g})\right\rangle.\]
  As elements of $\QQ(\mu_N) \otimes_{\ZZ} \cO(Y_1(N))^\times$, we have
  \[ \sum_{x \in \ZZ / N\ZZ}  e^{2\pi i x / N} \otimes w_N^*(u_{x/N}) = -N \otimes g_{0, 1/N},\]
  and the result follows.
 \end{proof}

 \begin{remark}
  One could also prove this statement directly (without the extended detour via Atkin--Lehner involutions and functional equations) by generalizing some of the calculations of \cite{BDR12} to use the weight 2 Eisenstein series $F^{(2)}_{\chi} = \sum_{x \in (\ZZ / N\ZZ)^\times} \chi(x)^{-1} F^{(2)}_{x/N}$ in place of $E^{(2)}_{\chi} = \sum_{x \in (\ZZ / N\ZZ)^\times} \chi(x)^{-1} E^{(2)}_{x/N}$. (Note that $F^{(2)}_{\chi}$ is always a holomorphic Eisenstein series if $N > 1$, while $E^{(2)}_\chi$ becomes non-holomorphic if $\chi$ is the trivial character.)
 \end{remark}

 \begin{remark}
  If we impose slightly more restrictive hypotheses we can avoid the need for any generalization of the main theorem of \cite{BDR12}. If $\chi$ is primitive, then it suffices to check that the main theorem of \cite{BDR12} holds without the assumption that $f, g$ are eigenforms for the $U_\ell$ with $\ell \mid N$; but this assumption is not used anywhere in the paper, except in order to explicitly evaluate the Euler factors at $\ell \mid N$. If $N_f = N_g = N$ then we can dispense with this assumption as well.
 \end{remark}


 \section{Families of cohomology classes}

In this section, we will construct \'etale cohomology classes from the generalized Beilinson--Flach elements in motivic cohomology defined above, and investigate their properties.


\subsection{The \'etale regulator}

 In \cite{huber00}, Huber constructs a $p$-adic regulator map from motivic cohomology into Jannsen's continuous \'etale cohomology:

 \begin{proposition}
  \label{prop:etaleregulator}
  Assume that $X$ is a smooth variety over a characteristic 0 field $k$. Then there is a regulator map
  \begin{equation}
   \label{eq:etaleregulator}
   r_{\et}:\CH^2(X,1)\rTo H^3_{\cont}(X,\Zp(2))
  \end{equation}
 \end{proposition}

 \begin{proof}
  See the second example on \cite[p. 772]{huber00}.
 \end{proof}

 \begin{proposition}\label{prop:HochschildSerre}
  If $X$ is any smooth variety over $k$, we have a Hochschild--Serre spectral sequence
  \[ H^p(k, H^q_{\et}(\overline{X}, \Zp(2))) \Rightarrow H^{p + q}_{\cont}(X, \Zp(2)),\]
  where $H^*(k, -)$ denotes continuous Galois cohomology.
 \end{proposition}

 \begin{proof}
  See \cite[Remark 3.5]{jannsen88}.
 \end{proof}

 \begin{corollary}
  Suppose that $X$ is a smooth affine surface over $k$. Then we have an edge map
  \begin{equation}\label{eq:edgemap}
   H^3_{\cont}(X, \Zp(2)) \to H^1(k, H^2_\et(\overline{X}, \Zp(2))).
  \end{equation}
 \end{corollary}

 \begin{proof}
  The fact that $\overline{X}$ is defined over an algebraically closed field implies that $H^q_{\et}(\overline{X},\Zp(2)) = 0$ for $q > 2$, as a $d$-dimensional affine variety over an algebraically closed field has \'etale cohomological dimension $d$ \cite[Arcata IV.6.4]{deligne77}. Consequently, we have $H^3_{\cont}(\overline{X},\Zp(2))=0$, and we obtain the required edge map by Proposition \ref{prop:HochschildSerre}.
 \end{proof}

 \begin{corollary}
  If $X$ is a smooth affine surface over $k$, the \'etale regulator induces a map (which we also denote by $r_{\et}$ by abuse of notation)
  \begin{equation}
   \label{eq:regulatorintoH1}
   r_{\et}:\CH^2(X,1)\rTo H^1(k, H^2_\et(\overline{X}, \Zp(2))).
  \end{equation}
 \end{corollary}

 \begin{proof}
  Compose $r_{\et}$ with the edge map \eqref{eq:edgemap}.
 \end{proof}

 The regulator maps have the following functoriality property:

 \begin{proposition}\label{prop:regulatorfunctorial}
  The regulator maps \eqref{eq:regulatorintoH1} are compatible with pullback along flat morphisms of surfaces $X \to Y$ over $k$, and pushforward along finite morphisms. In particular, they are compatible with the Galois restriction maps for arbitrary extensions $k'/k$, and with the corestriction maps for finite ones.
 \end{proposition}

 \begin{proof}
  This is true essentially by construction for Huber's regulator into continuous cohomology, since it arises from a realization functor on Voevodsky's category \(D\mathcal{M}_{gm}\) of geometrical motives, which in turn is built up from the category (denoted by $\operatorname{SmCor}$ in \cite{huber00}) whose objects are smooth varieties over $k$ and whose morphisms are finite correspondences $X \rightrightarrows Y$. It remains only to check that the Hochschild--Serre exact sequence \eqref{eq:edgemap} has the required functoriality property, which is standard.
 \end{proof}


  \subsection{The K\"unneth formula}
   \label{sect:kunneth}

   We also recall the K\"unneth formula for \'etale cohomology (cf.~\cite[Theorem 22.4]{milneLEC}): if $U$ and $V$ are varieties of finite type over an algebraically closed field of characteristic 0, then we have an exact sequence

   \begin{multline*}
    0 \rTo \sum_{r + s = m} H^r_\et(U, \Zp)\otimes_{\Zp}H^s_\et(V, \Zp)  \rTo H^m_\et(U \times V, \Zp)  \\ \rTo \sum_{r + s = m + 1} \operatorname{Tor}_1^{\Zp}(H^r_\et(U, \Zp), H^s_\et(V, \Zp)) \rTo 0.
   \end{multline*}

   We are interested in the case when $m = 2$, and $U$ and $V$ are smooth curves. If $U, V$ are affine, then they have \'etale cohomological dimension 1; so the third term vanishes, as do two of the three summands in the first term, and we have the following result:

   \begin{lemma}\label{lem:kunneth}
    For affine curves $U, V$, the K\"unneth formula gives an isomorphism
    \[ H^1_{\et}(U, \Zp) \otimes_{\Zp} H^1_{\et}(V, \Zp) \rTo^\cong H^2_{\et}(U \times V, \Zp),\]
    functorial in $U$ and $V$ and compatible with the Galois action.
   \end{lemma}

   We shall also need to consider the case when $U$ and $V$ are projective (and connected). In this case, we shall assume the ground field $k$ is $\overline{\QQ}$. By the compatibility of \'etale cohomology with Betti cohomology after base extension to $\CC$, we find that in this case the \'etale cohomology is $\Zp$ in degree 0 or 2, and $\Zp^{2g}$ in degree 1, where $g$ is the genus. Hence all the $\operatorname{Tor}$ terms vanish, since the cohomology groups are free $\Zp$-modules; and we conclude that $H^2_{\et}(U \times V, \Zp)$ is the direct sum of $H^1_\et(U, \Zp) \otimes_{\Zp} H^1_\et(V, \Zp)$ and two other summands which are both isomorphic (as Galois representations) to $\Zp(-1)$.


  \subsection{Galois representations attached to modular forms}
   \label{sect:galreps}

   We recall the construction of the Galois representations attached to cuspidal modular forms of weight 2, using the cohomology of the affine modular curves $Y_1(N)$.

   \begin{notation}
    Let $f$ be a cuspidal modular form of weight 2 and level $N$. We assume that $f$ is a normalized eigenform for all the Hecke operators $T_v$ (for $v \nmid N$) and $U_v$ (for $v \mid N$). (We do \emph{not} assume that $f$ is new of level $N$.) As usual, we write $a_v(f)$ for the $v$-th Fourier coefficient of $f$, which is its eigenvalue for $T_v$ if $v \nmid N$ and for $U_v$ if $v \mid N$; we also write $\varepsilon_d(f)$ for the eigenvalue of $f$ for the $\langle d \rangle$ operator for $d \in (\ZZ / N\ZZ)^\times$.
   \end{notation}

   By \cite[Proposition 4]{ashstevens86}, the compactly-supported cohomology $H^1_{c, \Betti}(Y_1(N)(\CC), \CC)$ is isomorphic to the space of \emph{modular symbols} of level $\Gamma_1(N)$ with coefficients in $\CC$. This contains a unique two-dimensional $\CC$-linear subspace $V_{\CC}(f)$ on which the Hecke operators $T_v, U_v$ act as multiplication by the Fourier coefficients $a_v(f)$; and the period isomorphism relating Betti and de Rham cohomology allows us to regard $f$ as an element of $V_{\CC}(f)$. Moreover, if $L$ is any finite extension of $\QQ$ containing the Fourier coefficients of $f$, $V_{\CC}(f)$ is the base-extension of a two-dimensional $L$-subspace $V_L(f) \subseteq H^1_{\Betti, c}(Y_1(N), L)$.

   Let $p$ be a prime. Invoking the comparison theorem between (compactly-supported) $p$-adic and Betti cohomology, we can regard $\Qp \otimes_{\QQ} V_L(f)$ as a subspace of $L \otimes_{\QQ} H^1_{\et, c}(\overline{Y_1(N)}, \Qp)$. Both of these are free modules of rank 2 over $L \otimes_{\QQ} \Qp = \prod_{\fp \mid p} L_\fp$, where the product is over primes of $L$ above $p$; so we obtain for each $\fp$ a two-dimensional $L_\fp$-linear subspace $V_{L_\fp}(f) \subseteq H^1_{\et, c}(\overline{Y_1(N)}, L_\fp)$.

   The following proposition is well known:

   \begin{proposition}
    The Galois representation $V_{L_\fp}(f)$ is ``the'' irreducible $L_\fp$-linear Galois representation attached to $f$. That is, for each prime $v \nmid Np$, the representation $V_{L_\fp}(f)$ is unramified at $v$ and we have
    \[ \trace_{L_\fp} \left(\Frob_v^{-1}\, \middle|\, V_{L_\fp}(f)\right) = a_v(f)\]
    where $\Frob_v$ is the arithmetic Frobenius.
   \end{proposition}

   We note that under Poincar\'e duality, the dual space $V_{L}(f)^*$ is identified with the maximal quotient of $H^1_{\Betti}(Y_1(N)(\CC), L)$ on which the transposes $T_v'$ and $U_v'$ of $T_v$ and $U_v$ act as multiplication by $a_v(f)$. Tensoring with $\Qp$, and noting that Poincar\'e duality holds in \'etale cohomology with a twist by the cyclotomic character, we obtain an identification of $V_{L_\fp}(f)^*$ with a quotient of $H^1_{\et}(\overline{Y_1(N)}, L_\fp)(1)$.

   \begin{definition}
    Let $\cO_\fp$ be the ring of integers of $L_\fp$. We define $T_{\cO_\fp}(f)^*$ as the $\cO_\fp$-submodule of $V_{L_\fp}(f)^*$ generated by the image of $H^1_{\et}(\overline{Y_1(N)}, \ZZ_p)(1)$, which is a $G_{\QQ}$-stable $\cO_\fp$-lattice in $V_{L_\fp}(f)^*$.
   \end{definition}

   \begin{remark}
    Note that our conventions are somewhat different from those of \cite[\S\S 6.3, 8.4]{kato04}: we define $V_{L_\fp}(f)$ as a subspace of compactly-supported cohomology of a modular curve, while Kato uses the same symbol to denote a quotient of the non-compactly-supported cohomology. If $f$ is new of level $N$, then our $V_{\cO_\fp}(f)^*$ coincides with the space Kato would denote by $V_{L_\fp}(\overline{f})(1)$ where $\overline{f}$ is the complex conjugate of $N$, and similarly for the integral lattices (our $T_{\cO_\fp}(f)^*$ is Kato's $V_{\cO_\fp}(\overline{f})(1)$).
   \end{remark}

   \begin{remark}
    \label{remark:tildelattice}
    One can also define a lattice in $V_{L_\fp}(f)^*$ using the cohomology of the projective modular curve. The inclusion $\overline{Y_1(N)} \into \overline{X_1(N)}$ induces a pullback map $H^1(\overline{X_1(N)}, \Zp) \to H^1(\overline{Y_1(N)}, \Zp)$, which is injective with cokernel isomorphic to $\Zp^{r - 1}$ where $r$ is the number of cusps. The action of the Hecke algebra on the boundary term $\Zp^{r-1}$ is Eisenstein, so the map $H^1(\overline{X_1(N)}, \Qp) \to H^1(\overline{Y_1(N)}, \Qp)$ is an isomorphism on the $f$-isotypical component. We define $\widetilde T_{\cO_\fp}(f)^*$ as the image of $H^1(\overline{X_1(N)}, \Zp) \otimes \cO_\fp$ in $V_{L_\fp}(f)^*$. Note that $\widetilde T_{\cO_\fp}(f)^* \subseteq T_{\cO_\fp}(f)^*$, and equality holds if $f$ is not congruent modulo $\fp$ to an Eisenstein series.
   \end{remark}


  \subsection{Generalized Beilinson--Flach classes}

   Let $N \ge 5$. Observe that $Y_1(N)^2 \otimes \QQ(\mu_m)$ is a smooth variety over $\QQ(\mu_{m})$, for any $m$. By \eqref{eq:regulatorintoH1}, for any prime $p$ we therefore have an \'etale regulator
   \[ r_{\et, \QQ(\mu_m)}: \CH^2(Y_1(N)^2 \otimes \QQ(\mu_m), 1) \rTo H^1\left(\QQ(\mu_{m}), H_{\et}^2(\overline{Y_1(N)^2}, \Zp(2)) \right).\]

   \begin{definition}
   \label{def:choosingMF}
    Let $f, g$ be modular forms of level $N$ which are normalized eigenforms for all the Hecke operators $T_\ell$ (for $\ell \nmid N$) and $U_\ell$ (for $\ell \mid N$), $L$ a number field containing the Fourier coefficients of $f$ and $g$, and $\fp$ a place of $L$ above the rational prime $p$.
   \end{definition}

   \begin{remark}
    In the situation of definition \ref{def:choosingMF}, we can use the K\"unneth formula (Lemma \ref{lem:kunneth}) to regard
    \[ T_{\cO_\fp}(f, g)^* \coloneqq T_{\cO_\fp}(f)^* \otimes_{\cO_\fp} T_{\cO_\fp}(g)^*\]
    as a quotient of $\cO_\fp \otimes_{\Zp} H^2_{\et}(\overline{Y_1(N)^2}, \Zp)(2)$.
   \end{remark}

   \begin{definition}
    Define the map
    \[ \kappa_{f, g, \QQ(\mu_m)}: \CH^2(Y_1(N)^2 \otimes \QQ(\mu_m), 1) \rTo H^1(\QQ(\mu_{m}), T_{\cO_\fp}(f, g)^*)\]
    to be the composition of $r_{\et,\QQ(\mu_m)}$ with the map on Galois cohomology induced by the projection
    \[ H^2_{\et}(\overline{Y_1(N)^2}, \Zp)(2) \rTo T_{\cO_\fp}(f, g)^*.\]
   \end{definition}

   \begin{definition}
    \label{def:cohoclasses}
    We define the \emph{generalized Beilinson--Flach class}
    \[ {}_c\bfz^{(f, g, N)}_m \coloneqq \kappa_{f, g, \QQ(\mu_m)} ({}_c \Xi_{m, N, 1}) \in H^1(\QQ(\mu_{m}), T_{\cO_\fp}(f, g)^*),\]
    and its non-integral version
    \[ \bfz^{(f, g, N)}_m \coloneqq \kappa_{f, g, \QQ(\mu_m)} (\Xi_{m, N, 1}) \in H^1(\QQ(\mu_{m}), V_{\cO_\fp}(f, g)^*).\]
   \end{definition}

   The compatibility relations we have shown for the generalized Beilinson--Flach elements for varying $m$ carry over to the cohomology classes:

   \begin{corollary}
    \label{cor:nearlycompatible}
    For any integers $m \ge 1, N \ge 1$, and $\ell$ a prime such that $\ell \mid N$, we have
    \[ \cores_{m}^{\ell m} \big({}_c\bfz^{(f, g, N)}_{\ell m}\big) =
     \begin{cases}
      (\alpha_f \alpha_g) \cdot {}_c\bfz^{(f, g, N)}_m & \text{if $\ell \mid m$,}\\
      (\alpha_f \alpha_g - \sigma_{\ell}) \cdot {}_c\bfz^{(f, g, N)}_m & \text{if $\ell \nmid m$,}
     \end{cases}
    \]
    where $\alpha_f, \alpha_g$ are the $U_\ell$-eigenvalues of $f$ and $g$, and in the latter case $\sigma_\ell$ is the arithmetic Frobenius element at $\ell$ in $\Gal(\QQ(\mu_m) / \QQ)$.

    If $\ell$ is a prime not dividing $mN$, then
    \[\cores^{\ell m}_m \big({}_c\bfz^{(f, g, N)}_{\ell m}\big)= \sigma_\ell\left( (\ell - 1)(1 - \varepsilon_f(\ell)\varepsilon_g(\ell) \sigma_\ell^{-2}) - \ell P_\ell(f, g, \ell^{-1} \sigma_\ell^{-1}) \right){}_c\bfz^{(f, g, N)}_{m},\]
    where $P_\ell(f, g, X)$ is the local Euler factor of $f$ and $g$ at $\ell$ (cf.~Proposition \ref{prop:explicitEulerfactor} above).
   \end{corollary}

   \begin{proof}
    Immediate from Theorems \ref{thm:secondnormbadprime}, \ref{thm:secondnormrelationprime} and the compatiblity of the regulator map with corestriction (Proposition \ref{prop:regulatorfunctorial}).
   \end{proof}

   The dependence of ${}_c\bfz^{(f, g, N)}_m$ on $c$ is as follows:

   \begin{proposition}
    There exist classes
    \[ \bfz^{(f, g, N)}_m \in H^1(\QQ(\mu_{m}), V_{L_\fp}(f, g)^*)\]
    such that the relation
    \begin{equation}
     \label{eq:cfactor}
     {}_c \bfz^{(f, g, N)}_m = (c^2 - \varepsilon_f(c)^{-1} \varepsilon_g(c)^{-1} [c]^2) \bfz^{(f, g, N)}_m
    \end{equation}
    holds for any $c > 1$ coprime to $6mN$.
   \end{proposition}

   \begin{proof}
    Immediate from Proposition~\ref{prop:BFeltproperties} (\ref{item:cfactor}).
   \end{proof}

   \begin{proposition}
    If there exists $d \ge 1$ coprime to $6mN$ such that $d^2 - \varepsilon_f(d)^{-1} \varepsilon_g(d)^{-1} [d]^2$ is invertible in $\cO_\fp[(\ZZ / m\ZZ)^\times]$, then there exists $\bfz^{(f, g, N)}_m \in H^1(\QQ(\mu_{m}), T_{\cO_\fp}(f, g)^*)$ such that Equation \eqref{eq:cfactor} holds in $H^1(\QQ(\mu_{m}), T_{\cO_\fp}(f, g)^*)$ (not just modulo torsion).

    In particular, this holds if the conductor of the reduction modulo $\fp$ of $\varepsilon_f \varepsilon_g$ is divisible by some prime which does not divide $mp$.
   \end{proposition}

   \begin{proof}
    Clear, since if such a $d$ exists we may define
    \[ \bfz^{(f, g, N)}_m \coloneqq (d^2 - \varepsilon_f(d)^{-1} \varepsilon_g(d)^{-1} [d]^2)^{-1} {}_d \bfz^{(f, g, N)}_m.\]
   \end{proof}

  \subsection{Local properties of the generalized Beilinson--Flach classes (I)}

   We now study the local properties of the Beilinson--Flach classes. We shall first recall some standard definitions.

   \begin{definition}
    If $K$ is a local field and $M$ is a topological $G_K$-module, we define $H^1_{nr}(K, M)$ to be the image of the inflation map
    \[ H^1(K^{nr} / K, M^{I_K}) \to H^1(K, M),\]
    where $I_K$ is the inertia subgroup of $G_K$ and $K^{nr}$ the maximal unramifed extension of $K$.

    If $V$ is a finite-dimensional $\Qp$-vector space, and $\ell$ is the residue characteristic of $K$, we define
    \[ H^1_f(K, V) =
     \begin{cases}
      H^1_{nr}(K, V) & \text{if $\ell \ne p$,} \\
      \ker(H^1(K, V) \to H^1(K, V \otimes \BB_\cris) & \text{if $\ell = p$ .}
     \end{cases}
    \]
    If $T$ is a $\Zp$-lattice in $V$ stable under $G_K$, we write $H^1_f(K, T)$ for the preimage of $H^1_f(K, V)$ in $H^1(K, T)$. (Cf.~\cite{blochkato90}.)
   \end{definition}

   \begin{proposition}
    If $T$ is a finite-rank free $\Zp$-module with a continuous action of $G_K$ which is trivial on $I_K$, and $\ell \ne p$, then
    \[ H^1_f(K, T) = H^1_{nr}(K, T).\]
   \end{proposition}

   \begin{proof}
    We have an inflation-restriction exact sequence
    \[0 \rTo H^1_{nr}(K, T) \rTo H^1(K, T) \rTo H^0(K^\mathrm{nr} / K, H^1(I_K, T)) \rTo 0,\]
    and a corresponding sequence for $V$ in place of $T$. Suppose $x \in H^1_f(K, T)$. Then the image of $x$ in $H^1(I_K, V)$ is zero, so the image of $x$ in $H^1(I_K, T)$ is torsion. However, $H^1(I_K, T) = \Hom(I_K, T)$ is torsion-free, since $T$ is; thus the image of $x$ in $H^1(I_K, T)$ is zero, and hence $x \in H^1_{nr}(K, T)$.
   \end{proof}

   \begin{definition}
    If $K$ is a number field and $M$ is a topological $G_K$-module, and $v$ is a prime of $K$, we say that $x \in H^1(K, M)$ is \emph{unramified at $v$} if its image in $H^1(K_v, M)$ lies in $H^1_{nr}(K, M)$. If $M$ is a finite-rank $\Zp$-module or $\Qp$-vector space, and $v$ is a prime above $p$, we say $x$ is \emph{crystalline at $v$} if its image in $H^1(K_v, M)$ lies in $H^1_f(K_v, M)$.
   \end{definition}

   \begin{proposition}
    The generalized Beilinson--Flach class ${}_c \bfz^{(f, g, N)}_m$ is unramified outside the primes dividing $m N p$. If $p \nmid mN$, it is crystalline at the primes above $p$.
   \end{proposition}

   \begin{proof}
    By the preceding proposition, it suffices to check this result after inverting $p$.

    Let us choose a prime $\ell \nmid mNp$. The compactified modular curve $X(m, mN)$ associated to $Y(m, mN)$ admits a smooth proper model $\cX(m, mN)$ over $\ZZ[1/mN]$; hence it has such a model over $\ZZ_\ell$. It is clear that the class ${}_c \cZ_{m, N, 1}$ lies in the higher Chow group $Z^2(\cY(m, mN), 1)$ of the integral model of $Y(m, mN)$, and we can choose the ``negligible elements'' of Theorem \ref{prop:preimageinGerst21} in order to obtain a lifting of ${}_c \cZ_{m, N, 1}$ to $\CH^2(\cX(m, mN), 1) \otimes \QQ$.

    For proper smooth schemes $\cS$ over $\ZZ_\ell$, with $\ell \ne p$, there is a regulator map
    \[ \CH^2(\cS, 1) \otimes \Qp \to H^3_{\et}(\cS, \Qp(2))\]
    (see e.g. \cite{flach92}) compatible with the regulator map $r_{\et}$ on the generic fibre $S$. Moreover, the \'etale cohomology $H^2_\et(\overline{S}, \Qp(2))$ is unramified as a representation of $G_{\QQ_\ell}$, by the proper base change theorem; and the Hochschild--Serre spectral sequence maps $H^3_{\et}(\cS, \Qp(2))$ to $H^1(\QQ_\ell^{nr} / \QQ_\ell, H^2_\et(\overline{S}, \Qp(2)) \subset H^1(\QQ_\ell, H^2_\et(\overline{S}, \Qp(2))$, where $\overline{S} = \cS \otimes \overline{\QQ}_\ell$ (cf.~\cite[Lemma 2.3]{flach92}). Hence the class ${}_c \bfz^{(f, g, N)}_m$ is unramified at the primes above $\ell$, as required.

    Similarly, if $p \nmid mN$, we can lift ${}_c\cZ_{m, N, 1}$ to a class in $\CH^2(\cX(m, mN), 1) \otimes \QQ$ where $\cX(m, mN)$ is proper and smooth over $\Zp$. However, the regulator $r_{\et}$ for proper smooth $\Zp$-schemes takes values in $H^1_f$, as a consequence of the commutative diagram of \S \ref{sect:syntomic} above relating $r_{\et}$ to the syntomic regulator $r_{\syn}$; so we are done.
   \end{proof}

   \begin{remark}
    I believe it is known that the regulator map $r_{\et}$ for arbitrary varieties over $p$-adic fields takes values in $H^1_g$, as remarked in \S \ref{sect:syntomic}, which would imply that the localization at $p$ of the Beilinson--Flach classes always lies in this subspace.
   \end{remark}

  \subsection{Local properties of the generalized Beilinson--Flach classes (II)}

   In order to control the local properties of the generalized Beilinson--Flach classes at the ``bad'' primes, we shall make use of the compatibility in the $p$-adic cyclotomic tower, under mild additional hypotheses.

   \begin{assumption}
    \label{assump:smallslope}
    The level $N$ is divisible by $p$, and the $U_p$-eigenvalues $\alpha_f, \alpha_g$ of $f$ and $g$ satisfy
    \[ v_{\fp}(\alpha_f \alpha_g) < 1.\]
   \end{assumption}

   \begin{proposition}
    \label{prop:unram outside p}
    Suppose Assumption \ref{assump:smallslope} holds. Then for any $m \ge 1$ and any prime $v \nmid p$ of $\QQ(\mu_m)$, the cohomology class ${}_c\bfz^{(f, g, N)}_m$ lies in $H^1_f(\QQ(\mu_m)_v, T_{\cO_\fp}(f, g)^*)$.

    If $v_{\fp}(\alpha_f \alpha_g) = 0$, then it lies in $H^1_{nr}(\QQ(\mu_m)_v, T_{\cO_\fp}(f, g)^*)$.
   \end{proposition}

   \begin{proof}
    To lighten the notation, we write $K = \QQ(\mu_m)_v$, and $M = T_{\cO_\fp}(f, g)^*$. We write $K_i = \QQ(\mu_{mp^i})_v$ (after choosing one of the finitely many primes of $\QQ(\mu_{mp^\infty})$ above $v$). Each $K_i$ is contained in $K^\mathrm{nr}$, since $v \nmid p$.

    For each $i$, there is an inflation-restriction exact sequence
    \[ 0 \rTo H^1(K^{\mathrm{nr}} / K_i, M^{I_v}) \rTo H^1(K_i, M) \rTo H^0(K^\mathrm{nr} / K_i, H^1(I_v, M)) \rTo 0,\]
    and the corestriction maps $H^1(K_{i+1}, M) \to H^1(K_i, M)$ correspond to the trace maps
    \[ H^0(K^\mathrm{nr} / K_{i+1}, H^1(I_v, M)) \to H^0(K^\mathrm{nr} / K_i, H^1(I_v, M)).\]
    Since $M$ is a finitely-generated $\Zp$-module, $H^1(I_v, M)$ is finitely generated over $\Zp$, by \cite[Proposition B.2.7(iii)]{rubin00}; thus the sequence of modules $M_i = H^0(K^\mathrm{nr} / K_i, H^1(I_v, M))$ stabilizes at some $i_0 \gg 0$. So for $i \ge i_0$, the trace maps $M_{i+1} \to M_i$ are simply multiplication by $p$ on $M_{i+1} = M_i = M_{\infty}$.  Let $z_i$ be the image of ${}_c\bfz^{(f, g, N)}_{mp^i}$ in $M_i$. It then follows that for $i \ge i_0$ we have
    \[ (\alpha_f \alpha_g)^i z_0 = p^{i - i_0} \cores_{i_0}^0(z_{i}).\]
    If $\alpha_f \alpha_g$ is a $p$-adic unit, then this immediately implies that $z_0 = 0$, since it is divisible by arbitrarily high powers of $p$. Thus ${}_c\bfz^{(f, g, N)}_{m}$ is unramified at $v$.

    Otherwise, we can only deduce that
    \[ z_0 \in \left(\frac{p^{i - i_0}}{(\alpha_f \alpha_g)^i}\right) M_0 + (M_0)_{\mathrm{tors}}\]
    for all $i \gg 0$, which implies that $z_0 \in (M_0)_{\mathrm{tors}}$ as $v_p(\alpha_f\alpha_g) < 1$. Hence the image of $z_0$ in $M_0 \otimes \Qp$ is zero, so the image of ${}_c\bfz^{(f, g, N)}_{m}$ in $H^1(\QQ(\mu_m)_v, M \otimes \Qp)$ is unramified. Thus ${}_c\bfz^{(f, g, N)}_{m} \in H^1_f(\QQ(\mu_m)_v, M)$.
   \end{proof}


  \subsection{Relation between p-stabilized and non-p-stabilized classes}
   \label{sect:pstab}

   For the arguments of the previous section, we assumed throughout that $p \mid N$. If we are given forms of levels prime to $p$, then we can obtain forms of level divisible by $p$ via ``$p$-stabilization'' (choosing old eigenforms of level divisible by $p$ with the same Hecke eigenvalues at all other primes). In this section, we shall investigate the relations between the classes obtained for the $p$-stabilized and non-$p$-stabilized forms.

   Let $f$ be a normalized eigenform of weight 2 and level $N$, and let $p$ be a prime such that $p \nmid N$. Then there are two eigenforms $f_\alpha$, $f_\beta$ at level $Np$ in the oldspace attached to $f$, whose $U_p$-eigenvalues are the roots $\alpha$, $\beta$ of the Hecke polynomial $X^2 - a_p(f) X + p \epsilon_p(f)$. (We assume, by enlarging the field if necessary, that these lie in our coefficient field $L$.)

   Then there are projection maps
   \begin{align*}
    \pr_{f_\alpha} &: H^1_\et(\overline{Y_1(Np)}, L_\fp) \to V_{L_\fp}(f_\alpha)^*\\
    \pr_{f_\beta} &: H^1_\et(\overline{Y_1(Np)}, L_\fp) \to V_{L_\fp}(f_\beta)^*\\
    \pr_{f} &: H^1_\et(\overline{Y_1(N)}, L_\fp) \to V_{L_\fp}(f)^*\\
   \end{align*}
   and a pushforward map $\pi: H^1_\et(\overline{Y_1(Np)}, L_\fp) \to H^1_\et(\overline{Y_1(N)}, L_\fp)$.

   \begin{proposition}
    In the above situation, there is a nonzero, $G_{\QQ}$-equivariant map
    \[ \pi^{(\alpha)} : V_{L_\fp}(f_\alpha)^* \to V_{L_\fp}(f)^* \]
    and similarly $\pi^{(\beta)}$, with the property that
    \[ \pi^{(\alpha)} \circ \pr_{f_\alpha} + \pi^{(\beta)} \circ \pr_{f_\beta} = \pr_f \circ \pi\]
    as maps $H^1_\et(\overline{Y_1(Np)}, L_\fp) \to  V_{f}^*$.
   \end{proposition}

   \begin{proof}
    Let $H^1_\et(\overline{Y_1(Np)}, L_\fp)_{[f]}$ denote the maximal quotient of $H^1_\et(\overline{Y_1(Np)}, L_\fp)$ where the operators $T_v'$ for $v \nmid Np$ and $U_v'$ for $v \mid N$ act via $a_v(f)$. Then, by comparison with modular symbols, we see that $H^1_\et(\overline{Y_1(Np)}, L_\fp)_{[f]}$ is 4-dimensional, and the $U_p'$ operator on this space is annihilated by the Hecke polynomial.

    By \cite[Theorem 2.1]{colemanedixhoven98}, the roots $\alpha$ and $\beta$ are distinct, so we may write $H^1_\et(\overline{Y_1(Np)}, L_\fp)_{[f]}$ as a direct sum of $G_{\QQ}$-stable eigenspaces, which map isomorphically onto the quotients $V_{L_\fp}(f_\alpha)$ and $V_{L_\fp}(f_\beta)$. This gives a lifting of $V_{L_\fp}(f_\alpha)$ to a subspace of $H^1_\et(\overline{Y_1(Np)}, L_\fp)_{[f]}$, and the map $\pr_f \circ \pi$ clearly factors through $H^1_\et(\overline{Y_1(Np)}, L_\fp)_{[f]}$ as stated.
   \end{proof}

   Now let us suppose we have two normalized weight 2 eigenforms $f, g$, of level $N$ prime to $p$ as before. Let $\alpha, \beta$ be the roots of the Hecke polynomial of $f$ at $p$, and similarly $\gamma, \delta$ for $g$. By the Coleman--Edixhoven theorem cited above, we have $\alpha \ne \beta$ and $\gamma \ne \delta$.

   A choice of root of each polynomial gives $p$-stabilized eigenforms $f_{\alpha}$, $g_{\gamma}$ of level $Np$. Then for each $m$ we have
   \begin{itemize}
    \item a class $\bfz_m^{(f, g, N)}$ in the cohomology of $V_{L_\fp}(f, g)^*$, which is a quotient of $H^2_\et(\overline{Y_1(N)^2}, L_\fp)(2)$;
    \item an element $\bfz_m^{(f_\alpha, g_\gamma, Np)}$ living in the cohomology of the representation $V_{L_\fp}(f_\alpha, g_\gamma)^*$, which is a quotient of $H^2_\et(\overline{Y_1(Np)^2}, L_\fp)(2)$.
   \end{itemize}

   These two representations are isomorphic as abstract Galois representations, but are realized differently as quotients of \'etale cohomology. We can regard both as quotients of the following space:

   \begin{definition}
    Let $H^2_\et(\overline{Y_1(Np)^2}, L_\fp)_{f, g}$ denote the maximal $L_\fp$-linear quotient of $H^2_\et(\overline{Y_1(Np)^2}, L_\fp)$ on which the operators $(T_v', 1)$ (for $v \nmid Np$) and $(U_v', 1)$ (for $v \mid N$) act via the Fourier coefficients of $f$, and similarly for $g$.
   \end{definition}

   \begin{note}
    Using the K\"unneth formula and a modular symbol calculation, we see that $H^2_\et(\overline{Y_1(Np)^2}, L_\fp)_{f, g}$ has dimension 16, and can be viewed as a direct sum of four simultaneous eigenspaces for the two operators $(U_p',1)$ and $(1, U_p')$, corresponding to the stabilizations $(\alpha, \gamma)$, $(\alpha, \delta)$, $(\beta, \gamma)$ and $(\beta, \delta)$. Each of these is a 4-dimensional $\Gal(\overline{\QQ} / \QQ)$-stable $L_\fp$-linear subspace.
   \end{note}

   For the remainder of this section, we shall assume the following:

   \begin{assumption}
    \label{assumption:neq}
    We have $\alpha \gamma \ne \beta \delta$.
   \end{assumption}

   \begin{remark}
    \label{rmk:neq}
    Assumption \ref{assumption:neq} is a consequence of Assumption \ref{assump:smallslope}, since $v_\fp(\alpha \beta \gamma \delta) = v_\fp(p^2 \varepsilon_f(p) \varepsilon_g(p)) = 2$, so if $v_\fp(\alpha \gamma) < 1$, then $v_\fp(\beta\delta) > 1$.
   \end{remark}

   \begin{proposition}
    If Assumption \ref{assumption:neq} is satisfied, then the operator
    \[ J_{\alpha, \gamma} \coloneqq \frac{(\mathcal{U}  - \alpha \delta)(\mathcal{U} - \beta \gamma)(\mathcal{U} - \beta \delta)} {(\alpha\gamma  - \alpha \delta)(\alpha\gamma - \beta \gamma)(\alpha\gamma - \beta \delta)},\]
    where $\mathcal{U} = (U_p', U_p')$, is an idempotent in $\End_{L_\fp} H^2_\et(\overline{Y_1(Np)^2}, L_\fp)_{f, g}$; it is equal to the identity on the $(\alpha, \gamma)$ eigenspace and zero on the other three eigenspaces.
   \end{proposition}

   \begin{proof}
    We know that $\alpha \gamma \ne \alpha \delta$ and $\alpha\gamma \ne \beta \gamma$ by the Coleman--Edixhoven theorem, so if $\alpha\gamma \ne \beta\delta$, the $\alpha\gamma$ eigenspace for the operator $\mathcal{U}$ coincides with the $(\alpha, \gamma)$ simultaneous eigenspace for $(U_p', 1)$ and $(1, U_p')$. We may thus define a projection onto this eigenspace by applying to $\mathcal{U}$ a polynomial that is 1 at $\alpha\gamma$ and zero at the other three eigenvalues.
   \end{proof}

   \begin{proposition}
    There is a $\Gal(\overline{\QQ}/\QQ)$-equivariant $L_\fp$-linear isomorphism
    \[ \pi^{(\alpha, \gamma)} : V_{L_\fp}(f_\alpha, g_\gamma)^* \rTo^\cong V_{L_\fp}(f, g)^*\]
    with the property that
    \[ \pi^{(\alpha, \gamma)} \circ \pr_{(\alpha, \gamma)} = \pi \circ J_{\alpha, \gamma}\]
    as maps $H^2_\et(\overline{Y_1(Np)^2}, L_\fp)_{f, g}(2) \to V_{L_\fp}(f, g)^*$, where
    \[ \pi: H^2_\et(\overline{Y_1(Np)^2}, L_\fp)_{f, g} \to V_{L_\fp}(f, g)^*\]
    is the natural map induced by the pushforward map $Y_1(Np)^2 \to Y_1(N)^2$.
   \end{proposition}

   \begin{proof}
    We define $\pi^{(\alpha, \gamma)}$ as $\pi \circ \iota^{(\alpha, \gamma)}$, where $\iota^{(\alpha, \gamma)}$ is the section of $\pr_{\alpha, \gamma}$ identifying $V_{L_\fp}(f_\alpha, g_\gamma)$ with the $(\alpha, \gamma)$-eigenspace of $H^2_\et(\overline{Y_1(Np)^2}, L_\fp)_{f, g}$. The composition $\iota^{(\alpha, \gamma)} \circ \pr_{\alpha, \gamma}$ is therefore equal to the projection operator $J_{\alpha, \gamma}$ above, and the proposition follows.
   \end{proof}

   \begin{corollary}
    \label{cor:pstabformula}
    For $p \nmid m$ we have
    \[ \pi^{(\alpha, \gamma)} \left( {}_c z_m^{(f_\alpha, g_\gamma, Np)} \right) =
    \frac{\alpha \gamma \left(1 - \tfrac{\beta \delta}{p} \sigma_p^{-1}\right)\left(1 - \tfrac{\alpha \delta}{p} \sigma_p^{-1}\right) \left(1 -\tfrac{\beta \gamma}{p} \sigma_p^{-1}\right)} {(\gamma - \delta) (\alpha - \beta)} \cdot {}_c z_m^{(f, g, N)}.
    \]
   \end{corollary}

   \begin{proof}
    We shall prove this by a slightly roundabout argument, using the second norm relation ``in reverse'' to understand how $\mathcal{U}$ acts on the zeta elements. Let us write the polynomial
    \[ \frac{(X  - \alpha \delta)(X - \beta \gamma)(X - \beta \delta)} {(\alpha\gamma  - \alpha \delta)(\alpha\gamma - \beta \gamma)(\alpha\gamma - \beta \delta)} \in L_\fp[X]\]
    as $j_0 + j_1 X + j_2 X^2 + j_3 X^3$, and let ${}_c z_m^{(f, g, Np)}$ be the image of $\reg_{\et}({}_c \Xi_{m, Np, j})$ in $H^2_\et(\overline{Y_1(Np)^2}, L_\fp)_{f, g}$.

    Essentially by definition, we have ${}_c z_m^{(f_\alpha, g_\gamma, Np)} = \pr_{\alpha, \gamma} {}_c z_m^{(f, g, Np)}$, and hence we may apply the preceding proposition to obtain
    \begin{align*}
     \pi^{(\alpha, \gamma)}\left({}_c z_m^{(f_\alpha, g_\gamma, Np)}\right) &= \pi\left( J_{\alpha, \gamma} \cdot {}_c z_m^{(f, g, Np)}\right)\\
     &= \pi \left( \left(j_0 + j_1 \mathcal{U} + j_2 \mathcal{U}^2 + j_3 \mathcal{U}^3\right) {}_c z_m^{(f_\alpha, g_\gamma, Np)}\right).
    \end{align*}
    By the second norm relation for $p \mid N$ (Theorem \ref{thm:secondnormbadprime}) and induction on $r$, we see that for $r \ge 1$ we have
    \[ \mathcal{U}^r\left( {}_c z_m^{(f, g, Np)}\right) = \norm^{p^r m}_{m}\left({}_c z_{p^r m}^{(f, g, Np)}\right) + \sigma_p \cdot \norm^{p^{r-1} m}_{m}\left({}_c z_{p^{r-1} m}^{(f, g, Np)}\right) + \dots + \sigma_p^r \cdot {}_c z_{m}^{(f, g, Np)}.\]
    On the other hand, by the first norm relation (Theorem \ref{thm:firstnormrelation}) we know that for $r \ge 1$ we have
    \[ \pi\left(\norm^{p^r m}_{m}({}_c z_{p^r m}^{(f, g, Np)})\right) = \norm^{p^r m}{m} \left(\pi({}_c z_{p^r m}^{(f, g, Np)})\right) = \norm^{p^r m}_m {}_c z_{m}^{f, g, N},\]
    while for $r = 0$ we have
    \[ \pi\left({}_c z_{m}^{(f, g, Np)}\right) = (1 - \varepsilon_p(f) \varepsilon_p(g) \sigma_p^{-2}) {}_c z_{m}^{(f, g, N)}.\]
    Combining these statements we have
    \begin{multline*}
     \pi^{(\alpha, \gamma)}\left({}_c z_m^{(f_\alpha, g_\gamma, Np)}\right) =
     (j_0 + j_1 \sigma_p + j_2 \sigma_p^2 + j_3 \sigma_p^3)(1 - \varepsilon_p(f) \varepsilon_p(g) \sigma_p^{-2}){}_c z_{m}^{(f, g, N)} \\
     + (j_1 + j_2 \sigma_p + j_3 \sigma_p^2) \norm^{pm}_m\left({}_c z_{pm}^{(f, g, N)}\right) \\
     + (j_2 + \sigma_p j_3) \norm^{p^2 m}_m\left({}_c z_{p^2 m}^{(f, g, N)}\right)
     + j_3  \norm^{p^3 m}_m\left({}_c z_{p^3 m}^{(f, g, N)}\right).
    \end{multline*}
    The prime-to-$p$ case of the second norm relation (Theorem \ref{thm:secondnormrelationprime}) gives a formula for the second term, and Theorem \ref{conj:weaknormrel} extends this to the remaining two terms. Substituting these in, the entirety of the right-hand side simplifies to a linear combination of terms each of which is ${}_c z_{m}^{(f, g, N)}$ acted on by a polynomial in $\sigma_p, \sigma_p^{-1}$ with coefficients given as rational functions in $\alpha, \beta, \gamma, \delta$. After a computation (which was carried out using Sage, \cite{sage}), one finds the polynomial simplifies to the product of Euler-type factors stated above.
   \end{proof}

   This extremely laborious computation allows us to prove the following theorem, which will be crucial to the Iwasawa-theoretic applications of our Euler system:

   \begin{corollary}
    \label{cor:h1f}
    Suppose $f, g$ admit $p$-stabilizations $f_\alpha, g_\gamma$ such that $v_\fp(\alpha \gamma) < 1$. Suppose $m$ is coprime to $p$ and neither of the quantities $\alpha \delta / p$, $\beta \gamma / p$ is an $r$-th root of unity, where $r$ is the order of $p$ in $(\ZZ / m\ZZ)^\times$.

    Then for every prime $v \nmid p$ of $\QQ(\mu_m)$, the localization of ${}_c \bfz_{m}^{f, g, N}$ at $v$ lies in $H^1_f$.
   \end{corollary}

   \begin{proof}
    We know from Proposition \ref{prop:unram outside p} above that the class ${}_c \bfz_{m}^{f_\alpha, g_\beta, Np}$ is in $H^1_f$ at all primes away from $p$. Since $\alpha \gamma \ne \beta \delta$ by Remark \ref{rmk:neq}, the formula of the previous corollary applies.

    We note that for $\lambda \in L_\fp$, the element $1 - \lambda \sigma_p^{-1}$ is invertible in $L_\fp[(\ZZ / m\ZZ)^\times]$ if and only if $\lambda^r \ne 1$, where $r$ is the order of $\sigma_p$ as above. It is clear that $\tfrac{\beta \delta}{p}$ cannot be a root of unity of any order, as its $\fp$-adic valuation is strictly positive. By assumption neither $\alpha \delta / p$ nor $\beta \gamma / p$ is an $r$-th root of unity; so the quantity
    \[ \frac{\alpha \gamma \left(1 - \tfrac{\beta \delta}{p} \sigma_p^{-1}\right)\left(1 - \tfrac{\alpha \delta}{p} \sigma_p^{-1}\right) \left(1 -\tfrac{\beta \gamma}{p} \sigma_p^{-1}\right)} {(\gamma - \delta) (\alpha - \beta)}\]
    is invertible in $L_\fp[(\ZZ / m\ZZ)^\times]$. Hence ${}_c z_{m}^{(f, g, N)}$ is also in $H^1_f$.
   \end{proof}

   \begin{remark}
    Note that the conclusion of Corollary \ref{cor:h1f} does not explicitly mention the choice of $p$-stabilization $(\alpha, \gamma)$; we use only the fact that one exists. We conjecture that the conclusion holds much more generally.
   \end{remark}


  \subsection{Iwasawa cohomology classes}

   \begin{notation}
    We now let $S$ be a finite set of places of $\QQ$ containing $p$, $\infty$, and all primes whose inertia groups act nontrivially on $T_{\cO_\fp}(f, g)^*$ (which can only happen for primes dividing $N$). Let $\QQ^S$ be the maximal extension of $\QQ$ unramified outside $S$.
   \end{notation}

   \begin{definition}
    For $K$ a finite extension of $\QQ$ contained in $\QQ^S$, $i \ge 0$, and $T$ a topological $\Zp[G_K]$-module unramified outside $S$, define
    \[ H^i_S(K, T) = H^i(\QQ^S / K, T). \]
    If $T$ is also a finitely-generated $\Zp$-module, and $K_\infty$ is a $p$-adic Lie extension of $K$ unramified outside $S$, define
    \[ H^i_{\Iw, S}(K_\infty, T) = \varprojlim_{L} H^i_S(L, T)\]
    where $L$ varies over the set of finite extensions of $K$ contained in $K_\infty$ and the inverse limit is with respect to the corestriction maps.
   \end{definition}

   \begin{remark}
    If $K_\infty$ contains finite extensions of $K$ of degree divisible by arbitrarily large powers of $p$ -- for instance, if $K_\infty / K$ is Galois and its Galois group is a $p$-adic Lie group of positive dimension -- then $H^0_{\Iw, S}(K_\infty, V)$ is zero, and $H^1_{\Iw, S}(K_\infty, T)$ is in fact independent of $S$, as long as $S$ contains the set $S_0$ consisting of all primes above $p$ or $\infty$, all primes ramifying in $K_\infty / K$ and all primes at which $T$ is ramified.
   \end{remark}

   We now let $f, g$ be eigenforms of level $N$, with coefficients in a field $L$, as in Definition \ref{def:choosingMF}.

   \begin{proposition}
    Suppose $m \ge 1$ and there is no Dirichlet character $\psi$ of conductor dividing $mp^\infty$ such that $f \sim \bar{g} \otimes \psi$, where $\sim$ signifies that these two eigenforms have the same Hecke eigenvalues away from their levels (i.e. correspond to the same newform). Then
    \[ H^0(\QQ(\mu_{mp^\infty}), V_{L_\fp}(f, g)^*) = 0.\]
   \end{proposition}

   \begin{proof}
    The space $H^0(\QQ(\mu_{mp^\infty}), V_{L_\fp}(f, g)^*$ is preserved by the residual action of the abelian group $\Gal(\QQ(\mu_{mp^\infty}) / \QQ)$, so if it is nonzero, it contains a subspace on which $\Gal(\QQ(\mu_{mp^\infty}) / \QQ)$ acts by some character $\lambda$ (possibly after a finite extension of the field $L$). This gives a nonzero $\Gal(\QQbar /\QQ)$-equivariant homomorphism
    \[ V_{L_\fp}(f) \to V_{L_\fp}(g)^*(\lambda) = V_{L_\fp}(\bar g)(\chi \lambda) \]
    where $\chi$ is the cyclotomic character. Since both sides are irreducible representations of $\Gal(\QQbar / \QQ)$, this map must be an isomorphism; consequently $\psi = \chi \lambda$ has finite order, and $f \sim \bar{g} \otimes \psi$.
   \end{proof}

   We can now prove the main result of this section, which shows that the elements $ (\alpha_f \alpha_g)^{-i}{}_c\bfz^{(f, g, N)}_{mp^i}$ for $i\geq 0$ can be glued together into a (possibly \emph{unbounded}) Euler system:

   \begin{theorem}
    \label{thm:Eulersystem}
    Let $m, N\geq 1$ with $(m, p) = 1$, and let $p$ be a prime dividing $N$. Let $f, g$ be modular forms of level $N$ which are eigenforms for all the Hecke operators, with $U_p$-eigenvalues $\alpha_f$ and $\alpha_g$ such that $h \coloneqq v_p(\alpha_f \alpha_g) < 1$. Suppose that $f \not\sim \bar{g} \otimes \psi$ for all Dirichlet characters $\psi$ of conductor dividing $mp^\infty$. Then for any $r$ such that $h \le r < 1$, there is a unique element
    \[ {}_c\mathfrak{z}^{(f,g,N)}_{m,r}\in \cH_r(\Gamma) \otimes_{\Lambda(\Gamma)} H^1_{\Iw,S}(\QQ(\mu_{mp^\infty}), T_{\cO_\fp}(f, g)^*) \]
    whose projection to $H_S^1(\QQ(\mu_{mp^i}), V_{L_\fp}(f, g)^*)$ is equal to
    \[ (\alpha_f \alpha_g)^{-i}{}_c\bfz^{(f, g, N)}_{mp^i}\]
    if $i \ge 1$, and to
    \[ (1 - (\alpha_f \alpha_g)^{-1} \sigma_p) {}_c\bfz^{(f, g, N)}_{m}\]
    if $i = 0$.

    Moreover, if $\ell$ is a prime not dividing $mN$, the corestriction map sends ${}_c\mathfrak{z}^{(f,g,N)}_{\ell m,r}$ to
    \[\sigma_\ell\left( (\ell - 1)(1 - \varepsilon_f(\ell)\varepsilon_g(\ell) \sigma_\ell^{-2}) - \ell P_\ell(\ell^{-1} \sigma_\ell^{-1}) \right) {}_c\mathfrak{z}^{(f, g, N)}_{m, r}.\]
   \end{theorem}

   \begin{proof}
    The existence of ${}_c\mathfrak{z}^{(f,g,N)}_{m,r}$ satisfying the projection formula for $i \ge 1$ is immediate from Corollary \ref{cor:nearlycompatible} and Proposition \ref{prop:lifttoIwasawacohomology}. The projection formula for $i = 0$ follows from the $i = 0$ case of Corollary \ref{cor:nearlycompatible}.
   \end{proof}

   \begin{note}
    The elements ${}_c\mathfrak{z}^{(f,g,N)}_{m,r}$ are in fact independent of $r \in [h, 1)$, in the sense that if $v_p(\alpha_f\alpha_g)\le r<r'<1$, then $ {}_c\mathfrak{z}^{(f,g,N)}_{m,r'}$ is the image of $ {}_c\mathfrak{z}^{(f,g,N)}_{m,r}$  under the natural map
    \[ \cH_r(\Gamma) \otimes_{\Lambda(\Gamma)} H^1_{\Iw,S}(\QQ(\mu_{mp^\infty}), T_{\cO_\fp}(f, g)^*) \rTo \cH_{r'}(\Gamma) \otimes_{\Lambda(\Gamma)} H^1_{\Iw,S}(\QQ(\mu_{mp^\infty}), T_{\cO_\fp}(f, g)^*)\]
    induced by the inclusion $\cH_r(\Gamma)\hookrightarrow \cH_{r'}(\Gamma)$.
    \end{note}

   In the case when $v_p(\alpha_f\alpha_g)=0$, we can prove a stronger result; in this case we can dispense with the assumption that $f$ is not a twist of $\bar{g}$, and we even get integral coefficients.

   \begin{theorem}
    \label{thm:ordinaryEulersystem}
    Assume that $f$ and $g$ are eigenforms of level dividing $N$, and such that $v_p(\alpha_f\alpha_g)=0$. Then there is a unique element ${}_c\mathfrak{z}^{(f,g,N)}_m\in H^1_{\Iw,S}(\QQ(\mu_{mp^\infty}), T_{\cO_\fp}(f, g)^*)$
    whose projection to $H_S^1(\QQ(\mu_{mp^i}), T_{\cO_\fp}(f, g)^*)$ is equal to
    \[
     \begin{cases}
      (\alpha_f \alpha_g)^{-i}{}_c\bfz^{(f, g, N)}_{mp^i} & \text{if $i \ge 1$}\\
      \left(1 - (\alpha_f \alpha_g)^{-1}\sigma_p\right){}_c\bfz^{(f, g, N)}_{m} & \text{if $i = 0$.}
     \end{cases}
    \]
   \end{theorem}

   \begin{proof}
     For $i\geq 1$, let
     \[ {}_c\mathfrak{z}^{(f,g,N)}_{m,i} = (\alpha_f \alpha_g)^{-i} {}_c\bfz^{(f, g, N)}_{mp^i}.\]
     As $v_p(\alpha_f\alpha_g)=0$, we have ${}_c\mathfrak{z}^{(f,g,N)}_{m,i}\in H^1_S(\QQ(\mu_{mp^i}), T_{\cO_\fp}(f, g)^*)$. Moreover, it is clear from Corollary \ref{cor:nearlycompatible} that
     \[ \cores_{i/i-1}({}_c\mathfrak{z}^{(f,g,N)}_{m,i} )=  {}_c\mathfrak{z}^{(f,g,N)}_{m,i-1},\]
     i.e. ${}_c\mathfrak{z}^{(f,g,N)}_m = \big({}_c\mathfrak{z}^{(f,g,N)}_{m,i}\big)_{i\geq 1}$ defines an element in $H^1_{\Iw,S}(\QQ(\mu_{mp^\infty}), T_{\cO_\fp}(f, g)^*)$. The projection formula for $i=0$ follows as before.
   \end{proof}

   \begin{remark}
    The difficulties arising when $v_p(\alpha_f \alpha_g) \ge 1$ are somewhat reminiscent of the ``critical slope'' case in the Iwasawa theory of a single modular form over the cyclotomic tower, cf.~\cite{pollackstevens11,pollackstevenspreprint,loefflerzerbes10}. However, in our situation the conditions we must impose are more restrictive; in particular, given forms $f', g'$ of level prime to $p$, at most two of the four possible choices of $p$-stablizations $f,g$ of $f', g'$ will be possible, and in many cases (e.g. if $a_p(f') = a_p(g') = 0$) there are no valid choices at all.

    Our methods using higher Chow groups can perhaps be thought of as an ``algebraic avatar'' of the modular symbol computations of \cite{amicevelu75}. It is interesting to speculate whether the overconvergent modular symbols of \cite{pollackstevens11} also admit such an algebraic analogue, which could conceivably be applied in the critical-slope cases.
   \end{remark}


  \subsubsection{Dispensing with \texorpdfstring{$c$}{c}}

   We now investigate the extent to which the ``smoothing factor'' $c$ may be removed.

   \begin{notation}
    If $R$ is a integral domain, we write $Q(R)$ for its field of fractions. For integers $c$ and $m$ such that $(c,m)=1$, we write $[c]$ for the image of $\sigma_c$ in $\Zp[\Gal(\QQ(\mu_m)/\QQ)]$.
   \end{notation}

    Under the hypotheses of Theorem \ref{thm:ordinaryEulersystem}, there exists an element
    \[
     \mathfrak{z}^{(f,g,N)}_m\in H^1_{\Iw,S}(\QQ(\mu_{mp^\infty}), T_{\cO_\fp}(f, g)^*) \otimes_{\Lambda(\Gamma_1)} Q(\Lambda(\Gamma_1))
    \]
    such that
    \begin{equation}
     \label{eq:crelation2}
     {}_c\mathfrak{z}^{(f,g,N)}_m=(c^2-\varepsilon_f(c)^{-1}\varepsilon_g(c)^{-1}[c]^2)\mathfrak{z}^{(f,g,N)}_m.
    \end{equation}

    \begin{remark}
     Note that we are identifying $\Gamma_1$ with $\Gal(\QQ(\mu_{mp^\infty}) / \QQ(\mu_{mp}))$. We may define $\mathfrak{z}^{(f,g,N)}_m$ as
     \[(d^2 - \varepsilon_f(d)^{-1}\varepsilon_g(d)^{-1}[d]^2)^{-1} {}_d\mathfrak{z}^{(f,g,N)}_m\]
     for any $d > 1$ coprime to $6N$ and congruent to $1 \bmod mp$; then $[d]$ lies in $\Gamma_1$, so the expression is well-defined, and it is evidently independent of the choice of $d$.
    \end{remark}

    \begin{notation}
     Write $\Gamma^{(m)} = \Gal(\QQ(\mu_{mp^\infty}) / \QQ) \cong (\ZZ / m\ZZ)^\times \times \Gamma$.
    \end{notation}

    \begin{lemma}
     \label{lem:choosec}
     Let $\fp$ be a prime ideal of $\Lambda(\Gamma^{(m)})$ of height $1$ which does not contain $p$. If the conductor of the Dirichlet character $\varepsilon_f \varepsilon_g$ does not divide $mp^\infty$, then there exists an integer $c > 1$ coprime to $6mpN$ such that $c^2 - \varepsilon_f(c)^{-1}\varepsilon_g(c)^{-1}[c]^2 \notin \fp$.
    \end{lemma}

    \begin{proof}
     Since $\fp$ does not contain $p$, it corresponds to a Galois orbit of continuous characters $\Gamma^{(m)} \to \overline{\QQ}_p$. Let $\kappa_\fp$ be a representative of this orbit. Define $\tilde{h}: \Gamma^{(m)} \to\overline{\QQ}_p^\times$ by $\tilde{h}(x)=\kappa_p(x)^2 / \chi(x)^2$ where $\chi: \Gamma^{(m)} \to \Zp^\times$ is the $p$-adic cyclotomic character. We need to show that there is an integer $c \ge 1$ coprime to $6mpN$ such that $\tilde{h}([c]) \ne \varepsilon_f(c) \varepsilon_g(c)$.

     However, if no such integer existed, then $\varepsilon_f \varepsilon_g$ would have to factor through the natural map $\widehat{\ZZ}^\times \to \Gamma^{(m)}$, i.e.~would have to have conductor dividing $mp^\infty$, contrary to our hypotheses.
    \end{proof}

    \begin{corollary}\label{cor:integrality}
     If $\varepsilon_f \varepsilon_g$ does not have conductor dividing $mp^\infty$, then
     \[ \mathfrak{z}^{(f,g,N)}_m\in H^1_{\Iw,S}(\QQ(\mu_{mp^\infty}), T_{\cO_\fp}(f, g)^*)\otimes\QQ.\]
    \end{corollary}

    \begin{proof}
     Let $Z_0$ be the $\Lambda(\Gamma^{(m)})$-module generated by ${}_c\mathfrak{z}^{(f,g,N)}_m$ for all possible $c$ and let $Z_1$ be the $\Lambda(\Gamma^{(m)})$-module generated by $\mathfrak{z}^{(f,g,N)}_m$. By \eqref{eq:crelation2}, $Z_0\subset Z_1$ and there exists $\mu\in\Lambda(\Gamma^{(m)})$ such that $\mu Z_1\subset Z_0$ and $Z_0/\mu Z_1$ is $p$-torsion free. Hence, it is enough to show that $Z_{0,\fp}=Z_{1,\fp}$ for any prime ideal $\fp$ of height $1$ which does not contain $p$. Fix such a $\fp$. By Lemma~\ref{lem:choosec}, there exists $c$ such that $c^2 - \varepsilon_f(c)^{-1}\varepsilon_g(c)^{-1}[c]^2 \notin \fp$, so
     \[
      \mathfrak{z}^{(f,g,N)}_m\in Z_{0,\fp}
     \]
     by \eqref{eq:crelation2}, as required.
    \end{proof}

    \begin{remark}
     Note that if the mod $p$ reduction of the Dirichlet character $\varepsilon_f \varepsilon_g$ does not have conductor dividing $mp^\infty$, then we can even deduce that
     \[ \mathfrak{z}^{(f,g,N)}_m \in H^1_{\Iw,S}(\QQ(\mu_{mp^\infty}), T_{\cO_\fp}(f, g)^*).\]
    \end{remark}

    \begin{lemma}\label{lem:free}
     If the residual representation of $T_{\cO_\fp}(f, g)^*$ restricted to $G_{\QQ_{\mu_m}}$ is irreducible, then $H^1_{\Iw,S}(\QQ(\mu_{mp^\infty}), T_{\cO_\fp}(f, g)^*)$ is a free $\Lambda(\Gamma^{(m)})$-module.
    \end{lemma}

    \begin{proof}
     This follows from the argument of \cite[\S13.8]{kato04}, which we briefly outline here. Let $T=T_{\cO_\fp}(f, g)^*$, $\Lambda=\Lambda(\Gamma^{(m)})$ and $H^1(T)=H^1_{\Iw,S}(\QQ(\mu_{mp^\infty}), T_{\cO_\fp}(f, g)^*)$ for simplicity. It is enough to show that if $(x,y)$ is a maximal ideal of $\Lambda$, the two maps
     \[
      \alpha:H^1(T)\rTo^x H^1(T)\quad\text{and}\quad \beta:H^1(T)/xH^1(T)\rTo^yH^1(T)/xH^1(T)
     \]
     are injective.

     If $x=p$, the injectivity of $\alpha$ follows from the fact that $\varprojlim_n H^0(\QQ(\mu_{mp^n}),T/p)=0$, which is a consequence of the finiteness of $T/p$. If $x$ is such that $\Lambda/x\Lambda$ is $p$-torsion free, it is enough to show that $H^0(\QQ(\mu_m),T\otimes\Lambda/x\Lambda)=0$, where the action of $\sigma\in G_{\QQ(\mu_m)}$ on $\Lambda$ is given by multiplication by $\bar{\sigma}^{-1}$, where $\bar{\sigma}$ denotes the image of $\sigma$ in $\Gamma^{(m)}$. But $T$ is irreducible, so non-abelian. This implies that $H^0(\QQ(\mu_m),T\otimes\Lambda/x\Lambda)=0$.

     To show that $\beta$ is injective, it is enough to show that $H^0(\ZZ[1/mp],T\otimes\Lambda/(x,y))=0$. But $\Lambda/(x,y)\cong\mathcal{O}_\fp/\mathfrak{M}_\fp(r)$ for some $r$, so we are done by the irreducibility of $T/\mathfrak{M}_\fp$.
    \end{proof}

    \begin{corollary}
     If $\varepsilon_f \varepsilon_g$ does not have $p$-power conductor and the residual representation of $T_{\cO_\fp}(f, g)^*$ restricted to $G_{\QQ_{\mu_m}}$ is irreducible, then
     \[ \mathfrak{z}^{(f,g,N)}_m\in H^1_{\Iw,S}(\QQ(\mu_{mp^\infty}), T_{\cO_\fp}(f, g)^*).\]
    \end{corollary}

    \begin{proof}
     This follows from the argument in \cite[\S13.14]{kato04}. Let $Z_1$ be the $\Lambda(\Gamma^{(m)})$-module generated by $\mathfrak{z}^{(f,g,N)}_m$. By the proof of Corollary~\ref{cor:integrality}, $Z_{1,\fp}\subset H^1_{\Iw,S}(\QQ(\mu_{mp^\infty})_\fp$ for any prime ideal $\fp$ of $\Lambda(\Gamma^{(m)})$ height one. But $H^1_{\Iw,S}(\QQ(\mu_{mp^\infty})$ is a free $\Lambda(\Gamma^{(m)})$-module by Lemma~\ref{lem:free}, hence the result.
    \end{proof}


  \subsection{Variation in Hida families}

   We now make use of the first norm relation (Theorem \ref{thm:firstnormrelation}) to build elements in the cohomology of \emph{towers} of modular curves, under additional ordinarity hypotheses. Let us begin by recalling some of Ohta's results in \cite{ohta99,ohta00} concerning the structure of the module
   \[ GES_p(N, \Zp) \coloneqq \varprojlim_{n \ge 1} H^1_\et(\overline{Y_1(Np^n)}, \Zp),\]
   which can be roughly summarized by the statement that one can build a Hida theory for this module after replacing the usual Hecke operators with their transposes.

   More precisely, we note that the full Hecke algebra does not act on $GES_p(N, \Zp)$, since the operator $U_p = \begin{pmatrix} 1 & 0 \\ 0 & p \end{pmatrix}$ does not commute with the trace maps. Rather, we obtain an action of the Hecke operator $U_p'$, corresponding to $\begin{pmatrix} p & 0 \\ 0 & 1 \end{pmatrix}$. Ohta shows that one may use the operator $U_p'$ to define an ``anti-ordinary projector'' $\esord = \lim_{n \to \infty} (U_p')^{n!}$, analogous to the usual Hida ordinary projector $\eord = \lim_{n \to \infty} (U_p)^{n!}$.

   Ohta proves the following control theorem for the anti-ordinary part of $GES_p(N)_{\Zp}$, which is naturally a module over the Iwasawa algebra of $\Gamma_1 = (1 + p\Zp)^\times$ via the diamond operators:

   \begin{proposition}[{\cite[1.3, 1.4]{ohta99}}]
    \label{thm:ohtacontrol}
    The module $\esord GES_p(N)_{\Zp}$ is free of finite rank over $\Lambda(\Gamma_1)$, and for each $r \ge 1$ there is a $G_{\QQ}$-equivariant isomorphism
    \[ \esord GES_p(N, \Zp) / \omega_r \cong \esord H^1(\overline{Y_1(Np^r)}, \Zp),\]
    where $\omega_r$ is the kernel of the natural map $\Lambda(\Gamma_1) \to \Zp[(\ZZ / p^r \ZZ)^\times]$.
   \end{proposition}

   From this isomorphism, we deduce that $\esord GES_p(N, \Zp)$ has an action of the Hecke algebra
   \[ \esord \mathcal{H}'(N, \Zp) = \varprojlim_{r \ge 1}\esord \mathcal{H}'(\Gamma_1(Np^r), \Zp),\]
   where $\mathcal{H}'(\Gamma_1(Np^r), \Zp)$ is the $\Zp$-subalgebra of $\End_{\Qp} M_2(\Gamma_1(Np^r), \Qp)$ generated by the Hecke operators $T'(n)$ for $n \ge 1$ and $\langle q \rangle$ for $q \in (\ZZ / N\ZZ)^\times$. Here $T'(\ell)$, for $\ell$ prime, corresponds to the double coset $\begin{pmatrix} \ell & 0 \\ 0 & 1\end{pmatrix}$, so in particular for $(m, Np) = 1$ we have $T'(m) = \langle m \rangle^{-1} T(m)$ (the adjoint of $T(m)$ with respect to the Petersson product).

   The algebra $\esord \mathcal{H}'(N, \Zp)$ algebra is finite and free as a $\Lambda(\Gamma_1)$-module. (Note that it is \emph{not} generally free as a $\Lambda(\Gamma)$-module, although it is evidently projective).

   \begin{proposition}[{\cite[2.2]{ohta99}}]
    The Hecke algebra $\esord \mathcal{H}'(N, \Zp)$ is isomorphic to the Hecke algebra $\mathcal{H}_{\ord}(N, \Zp) \coloneqq \eord \mathcal{H}(N, \Zp)$ acting on the module $\eord M_2(N, \Lambda(\Gamma_1))$ of ordinary $\Lambda$-adic modular forms (not necessarily cuspidal), via the map sending $T(n)'$ to $T(n)$ and $\langle n \rangle$ to $\langle n \rangle^{-1}$.
   \end{proposition}

   \begin{definition}
    In the above situation, by a \emph{Hida family} of tame level $N$, we mean a maximal ideal of the ring $\mathcal{H}_{\ord}(N, \Zp)$. For each Hida family $\bfg$, we define
    \[ T(\bfg)^* = \left(\esord GES_p(N, \Zp)\right)_{\bfg}(1).\]
   \end{definition}

   \begin{corollary}
    Let $g$ be an ordinary weight 2 Hecke eigenform of level $Np^s$, with coefficients in some finite extension $L_\fp / \Qp$ with ring of integers $\cO_\fp$. Then we have an isomorphism of $\cO_\fp$-linear Galois representations
    \[ \cO_\fp \otimes_{\cH_{\ord}(N, \Zp)} T(\bfg)^* \cong T_{\cO_\fp}(g)^*, \]
    where $T_{\cO_\fp}(g)^*$ is the representation defined in \ref{sect:galreps} above.
   \end{corollary}

   \begin{proof}
    Clear from the definition of $T_{\cO_\fp}(g)^*$ and the control theorem (Theorem \ref{thm:ohtacontrol}).
   \end{proof}

   \begin{theorem}
    \label{thm:onevarhida}
    Let $N \ge 1$ be prime to $p$. If $\bfg$ is a Hida family of tame level $N$, and $f$ is any eigenform of level $Np^k$ for $k \ge 1$ whose $U_p$-eigenvalue $\alpha_f$ satisfies $v_p(\alpha_f) < 1$, then for each integer $m \ge 1$ there is a cohomology class
    \[ {}_c\bfz^{(f, \bfg)}_m \in H^1_S(\QQ(\mu_m), T_{\cO_\fp}(f)^* \otimes_{\Zp} T(\bfg)^*),\]
    such that for each classical weight 2 specialization $g$ of $\bfg$ with coefficients in $L$, the image of ${}_c\bfz^{(f, \mathbf{g})}_m$ in
    \[ H^1(\QQ(\mu_m), T_{\cO_\fp}(f)^* \otimes_{\cO_\fp} T_{\cO_\fp}(g)^*) = H^1(\QQ(\mu_m), T_{\cO_\fp}(f, g)^*)\]
    is the generalized Beilinson--Flach element ${}_c \bfz^{(f, g, N')}_m$, where $N'$ is the greatest common divisor of the levels of $f$ and $g$.
   \end{theorem}

   \begin{proof}
    We know that the elements ${}_c \Xi_{m, Np^s, 1}$ for $s \ge 1$ are unramified outside $S$, and are compatible under pushforward via the natural projection maps. Hence the sequence of elements defined by pushing forward ${}_c \Xi_{m, Np^s, 1}$ to $\CH^2(Y_1(Np^r) \times Y_1(Np^s) \times \QQ(\mu_m), 1)$, for $s \ge r$, are compatible under pushforward maps in the $Y_1(Np^s)$ factor alone. Applying the \'etale regulator, we obtain elements of the module
    \[ \varprojlim_{s \ge r} H^1(\QQ(\mu_m), H^2_{\et}(\overline{Y_1(Np^r) \times Y_1(Np^s)}, \Zp)(2)).\]
    For each $s$, using the K\"unneth formula we may decompose $H^2_{\et}(\overline{Y_1(Np^r) \times Y_1(Np^s)}, \Zp)$ as the tensor product of the $H^1$'s of the two factors. Projecting to the quotient $T_{\cO_\fp}(f)$ of $H^1_\et(\overline{Y_1(Np^r)}, \Zp)(1)$, and applying the anti-ordinary projector $\esord$ to $H^1_\et(\overline{Y_1(Np^s)}, \Zp)(1)$, we may argue exactly as in Proposition \ref{prop:unram outside p} above to deduce that the elements we obtain are unramified outside $Np$.

    Since the restricted-ramification cohomology groups $H^i_S(\QQ(\mu_m), -)$ commute with inverse limits, we obtain an element of
    \[ H^1_S(\QQ(\mu_m), T_{\cO_\fp}(f)^* \otimes \esord GES_p(N)_{\Zp}(1)).\]
    Pushing forward along the canonical map $\esord GES_p(N)_{\Zp}(1) \to T(\bfg)^*$, we obtain the required elements.
   \end{proof}

   We also obtain a corresponding result for the product of two Hida families, whose proof is essentially identical to the above:

   \begin{theorem}
    \label{thm:twovarhida}
    Let $N \ge 1$ be prime to $p$. If $\bff$, $\bfg$ are Hida families of tame level $N$, then for each integer $m \ge 1$ there is a cohomology class
    \[ {}_c\bfz^{(\bff, \bfg)}_m \in H^1_S(\QQ(\mu_m), T(\bff)^*\mathop{\hat\otimes}_{\Zp} T(\bfg)^*),\]
    such that for classical weight 2 specializations $f$, $g$ of $\bff$, $\bfg$ with coefficients in $L$, the image of ${}_c\bfz^{(\bff, \bfg)}_m$ in
    \[ H^1(\QQ(\mu_m), T_{\cO_\fp}(f)^* \otimes_{\cO_\fp} T_{\cO_\fp}(g)^*) = H^1(\QQ(\mu_m), T_{\cO_\fp}(f, g)^*)\]
    is the generalized Beilinson--Flach element ${}_c \bfz^{(f, g, N')}_m$, where $N'$ is the greatest common divisor of the levels of $f$ and $g$.
   \end{theorem}

   \begin{remark}
    We do not know if one can formulate a result analogous to Theorem \ref{thm:Eulersystem} incorporating Hida-family variation in $g$, since we do not know whether the results of Appendix \ref{sect:unbounded cohomology} apply for ``big'' Galois representations; but if $f$ is ordinary there are no such issues.
   \end{remark}

   \begin{theorem}
    \label{thm:threevarhida}
    In the situation of Theorem \ref{thm:twovarhida}, for each $m$ prime to $p$ there exists a cohomology class
    \[ {}_c\mathfrak{z}^{(\bff, \bfg)}_m \in  H^1_{\Iw, S}(\QQ(\mu_{mp^\infty}), T(\bff)^*\mathop{\hat\otimes}_{\Zp} T(\bfg)^*)\]
    whose image in
    \( H^1_{\Iw, S}(\QQ(\mu_{mp^i}), T(\bff)^*\mathop{\hat\otimes}_{\Zp} T(\bfg)^*)\)
    for each $i \ge 1$ is equal to
    \( (\alpha_f \alpha_g)^{-i} \cdot {}_c\bfz^{(\bff, \bfg)}_{mp^i}\).
   \end{theorem}


 \subsection{Integrality of the Poincar\'e pairing}

  In this section, we prove a technical lemma that will be needed in our applications to bounding Selmer groups. We assume that $p > 3$, $N \ge 5$, and $p \nmid N$.

  Recall that $X_1(N)$ admits a canonical smooth proper model over $\ZZ[1/N]$ (\cite{delignerapoport73}) and hence over $\Zp$. By \cite{fontainemessing87}, the integral de Rham cohomology $H^1(X_1(N), \Omega^\bullet_{X_1(N) / \Zp})$ is a filtered Dieudonn\'e module over $\Zp$, and
  \[ T(H^1(X_1(N), \Omega^\bullet_{X_1(N) / \Zp})) = H^1_{\et}(\overline{X_1(N)}, \Zp),\]
  where $T(-)$ is the Fontaine--Laffaille functor.

  We define versions of these in the $f$-isotypical component by projection. As in Remark \ref{remark:tildelattice} above, we define $\widetilde{T}_{\cO_\fp}(f)^*$ to be the image of $H^1_{\et}(\overline{X_1(N)}, \Zp) \otimes \cO_\fp$ in $V_{L_\fp}(f)^*$, and similarly for $g$. We define $\Dcris(\widetilde{T}_{\cO_\fp}(f)^*)$ as the image of $H^1(X_1(N), \Omega^\bullet_{X_1(N) / \Zp}) \otimes_{\Zp} \cO_\fp$ in $\Dcris(V_{L_\fp}(f)^*)$; then $\Dcris(\widetilde{T}_{\cO_\fp}(f)^*)$ is a strongly divisible $\cO_\fp$-lattice, and its image under $T(-)$ is $\widetilde{T}_{\cO_\fp}(f)^*$.

  Let us recall here the definition of $\eta_f^{\ur}$.

   \begin{definition}
    Let $X=X_1(N)$ and
   \[
    \eta_f^{\ah}=\frac{\bar{f}^*(z)d\bar{z}}{\langle f^*,f^* \rangle_{k,N}}\in H^1_{\dR}(X_\CC).
   \]
   We denote $\eta_f$ its image in $H^1(X / \CC,\cO_{X / \CC})$, which lies in $H^1(X / \Qp, \cO_{X / \Qp})$; then $\eta_f^{\ur}$ is defined to be the lift of $\eta_f$ to the unit root subspace of $H^1_{\dR}(X_{\CC_p})^{f,\ur}$.
  \end{definition}

  Our aim is to investigate the denominator of the class $\eta_f^{\ur}$ relative to the sublattice
  \[ \cO_L \otimes_{\Zp} H^1(X_1(N), \Omega^\bullet_{X_1(N) / \Zp}) \subseteq L_\fp \otimes_{\Qp} H^1(X_1(N), \Omega^\bullet_{X_1(N) / \Qp}).\]
  Since the unit root lifting is obviously integral, it suffices to show that $\eta_f \in \cO_L \otimes_{\Zp} H^1(X_1(N), \cO_{X_1(N) / \Zp})$.

  \begin{proposition}
   \label{prop:serreduality}
   An element of
   \[ L_\fp \otimes H^1(X_1(N) / \Qp, \cO_{X_1(N) / \Qp})\]
   lies in the sublattice
   \[\cO_\fp \otimes H^1(X_1(N), \cO_{X_1(N) / \Zp})\]
   if and only if it pairs to an element of $\cO_\fp$ with all elements of $\cO_\fp \otimes H^0(X_1(N) / \Zp, \Omega^1_{X_1(N) / \Zp})$.
  \end{proposition}

  \begin{proof}
   Since the pairing between $H^1(X_1(N), \cO_{X_1(N) / \Zp})$ and $H^0(X_1(N) / \Zp, \Omega^1_{X_1(N) / \Zp})$ is defined over $\Zp$, it suffices to assume $\cO_{\fp} = \Zp$. But Serre duality shows that this pairing is perfect, i.e.~identifies $H^0(X_1(N) / \Zp, \Omega^1_{X_1(N) / \Zp})$ with the $\Zp$-dual of $H^1(X_1(N) / \Zp, \cO_{X_1(N) / \Zp})$.
  \end{proof}

  \begin{lemma}\label{lem:integral}
   Let $\phi\in S_2(N; L_\fp)$. Then the element $\omega_\phi$ of $H^0(X/L, \Omega^1_{X/L})$ lies in $H^0(X/\cO_L, \Omega^1_{X/ \cO_L})$ if and only if $\phi \in S_2(N; \cO_\fp)$.
  \end{lemma}

  \begin{proof}
   We have by definition $\omega_\phi(q)=\phi(q)dq/q$, which is defined over $\cO_L$ if $\phi$ is.
  \end{proof}

  \begin{definition}
   \label{def:idealIf}
   If $f \in S_2(\Gamma_1(N), L)$, let $I_f$ denote the ideal in $\cO$ such that
   \[ \left\{ \frac{ \langle f^*, \phi \rangle}{\langle f^*, f^* \rangle} : \phi \in S_2(N, \cO)\right\} = I_f^{-1}. \]
  \end{definition}

  \begin{remark}
   Note that $I_f^{-1}$ contains $\cO$, so $I_f$ is an integral ideal (rather than a fractional ideal). The ideal $I_f$ essentially measures the extent to which $f$ is congruent to other eigenforms in $S_2(N, \cO)$.
  \end{remark}

  \begin{corollary}
   \label{cor:integraletaf}
   For any prime $\fp \nmid N$, we have
   \[ \eta_{f} \in I_f^{-1} \cdot \cO_\fp \otimes_{\Zp} H^1(X_1(N), \cO_{X_1(N) / \Zp}).\]
  \end{corollary}

  \begin{proof}
   By the construction of the class $\eta_{f}$, for any $\vp \in S_2(\Gamma_1(N), \cO)$ we have
   \[ \langle \eta_f, \omega_{\phi} \rangle = \frac{\langle f^*, \phi\rangle}{\langle f^*, f^* \rangle} \in I_f^{-1} \cO,\]
   so the result follows by Lemma \ref{lem:integral} and Proposition \ref{prop:serreduality}.
  \end{proof}

  \begin{corollary}
   \label{cor:integralpoincare}
   The linear functional
   \[ \Dcris(V_{L_\fp}(f, g)^*) \to L_\fp\]
   given by pairing with $\eta_f^{\ur} \otimes \omega_g$ maps the submodule
   \[ \Dcris(\widetilde{T}_{\cO_\fp}(f)^*) \otimes \Dcris(\widetilde{T}_{\cO_\fp}(g)^*)\]
   into $I_f^{-1} \cO_\fp$.
  \end{corollary}

  \begin{proposition}
   Let $z \in H^1_f\left(\Qp, \left[\widetilde{T}_{\cO_\fp}(f) \otimes \widetilde{T}_{\cO_\fp}(g)\right]^*\right)$. Then
   \[ \langle \log(z), \eta_f^{\ur} \otimes \omega_g \rangle \in I_f^{-1} \cdot (1 - \alpha^{-1} \gamma^{-1})^{-1}(1 - \alpha^{-1} \delta^{-1})^{-1} \cdot \cO_\fp,\]
   where $\alpha$ is the unit root of the Hecke polynomial of $f$ and $\beta, \delta$ are the roots of the Hecke polynomial of $g$.
  \end{proposition}

  \begin{proof}
   By Fontaine--Laffaille theory, for any crystalline $\cO_\fp$-linear $G_{\Qp}$-representation $V$ whose Hodge filtration has length $< p$ and such that $\Dcris(V)^{\vp = 1} = 0$, the map
   \[ \log_{\Qp, V}: H^1_f(\Qp, V) \rTo^\cong \frac{\Dcris(V)}{\Fil^0 \Dcris(V)}\]
   induces an isomorphism of $\cO_\fp$-modules
   \[ \frac{H^1_f(\Qp, T)}{\text{torsion}} \rTo^\cong \frac{(1 - \vp)^{-1} D}{(1 - \vp)^{-1} D \cap \Fil^0 \Dcris(V)}\]
   for any $G_{\Qp}$-stable lattice $T \subseteq V$ with corresponding strongly divisible lattice $D \subseteq \Dcris(V)$; cf.~Theorem 4.1 and Lemma 4.5 of \cite{blochkato90}.

   In our case we may take $V = W \otimes V_{L_\fp}(g)^*$ where $W$ is the 1-dimensional unramified quotient of $V_{L_\fp}(f)^*$, since the linear functional given by pairing with $\eta_f^{\ur} \otimes \omega_g$ factors through this quotient. Let us suppose that $g$ is ordinary; using the explicit description of the strongly divisible lattices in $\Dcris(V_{L_\fp}(g)^*)$ given in \cite[\S 5]{loefflerzerbes10} one checks that
   \[ \frac{(1 - \vp)^{-1} D}{(1 - \vp)^{-1} D \cap \Fil^0 \Dcris(V)} \subseteq p^{-k} \cdot \frac{D}{D \cap \Fil^0 \Dcris(V)},\]
   where $k = v_p\left[ (1 - \alpha^{-1} \gamma^{-1})(1 - \alpha^{-1} \delta^{-1})\right]$. In the non-ordinary case one reasons similarly using the description of the Wach module of the (unique up to scaling) lattice in $\Dcris(V_{L_\fp}(g)^*)$ given in \cite{bergerlizhu04}. Combining this with Corollary \ref{cor:integralpoincare} gives the result.
  \end{proof}


 \section{Bounding strict Selmer groups}

 Let $f$, $g$ be newforms of weight $2$, level $N$ and characters $\chi_f$ and $\chi_g$, respectively. Let $L$ be the subfield of $\overline{\QQ}$ generated by the coefficients of $f$ and $g$. For a prime $\fp$ of $L$, denote by $V_{L_{\fp}}(f)$ and $V_{L_{\fp}}(g)$ the $L_{\fp}$-representations of $G_{\QQ}$ attached to $f$ and $g$, respectively. The aim of this section is to apply Theorem \ref{thm:boundedSel} below to the representation $V_{L_{\fp}}(f)\otimes V_{L_{\fp}}(g)$.

\subsection{The method of Euler systems}
\label{Eulermethod}

 We recall some definitions and results from \cite{rubin00}. Let $\cO$ be the ring of integers of a finite extension $E / \Qp$, and let $T$ be a free $\cO$-module of finite rank with a continuous action of $G_{\QQ}$ which is unramified at almost all primes. Let $V = T \otimes_{\cO} E$ and $W = V/T = T \otimes_{\cO} E / \cO$.

 Let $\Sigma$ be a finite set of primes containing $p$ and all prime numbers at which the action of $G_{\QQ}$ on $T$ ramifies. Let $A$ be a set of integers such that
 \begin{itemize}
  \item if $m \in A$, then all divisors of $m$ are in $A$;
  \item if $r, s \in A$, then $LCM(r, s) \in A$;
  \item $\ell \in A$ for all primes $\ell \notin \Sigma$.
 \end{itemize}

 For a prime $\ell\not\in \Sigma$, define
 \[ p_\ell(X)=\sideset{}{_E}\det \left(1 - \Frob_\ell^{-1} X \middle| V^*(1) \right)\in \Zp[X],\]
 where $\Frob_\ell$ is the arithmetic Frobenius at $\ell$.

 \begin{definition}[C.f. {\cite[Definition 2.1.1]{rubin00}}]
  \label{def:Eulersystem}
  An Euler system for $(T, A, \Sigma)$ is a system of elements
  \[ c_m \in H^1(\QQ(\mu_m), T) \text{ for all } m \in A,\]
  such that if $\ell$ is a prime such that $m, m\ell \in A$ and $\ell \ne \Sigma$, then the corestriction map
  \[ H^1(\QQ(\mu_{\ell m}), T) \rightarrow H^1(\QQ(\mu_m),T)\]
  sends $c_{\ell m}$ to
  \[
   \begin{cases}
    p_\ell(\sigma_\ell^{-1}) c_m & \text{if $\ell \nmid m$ and $\ell \ne \Sigma$,}\\
    c_m & \text{if $\ell \mid m$ or $\ell \in \Sigma$.}
   \end{cases}
  \]
  Here, $\sigma_\ell$ denotes the arithmetic Frobenius of $\ell$ in $\Gal(\QQ(\mu_{m}) / \QQ)$.
 \end{definition}

 \begin{remark}
  Our notations differ slightly from those of \cite{rubin00}. Firstly, Rubin writes $T^*$ for the ``Tate dual'' $\Hom(T, \Zp(1))$, while we write this as $T^*(1)$. More significantly, Rubin considers an infinite abelian extension $\mathcal{K}$ and a class $c_F$ for every finite subextension $F$ of $\mathcal{K}$; in our case $\mathcal{K}$ is the extension $\mathcal{K}_A = \QQ(\mu_r : r \in A)$, and it suffices to specify a class for each subextension of the form $\QQ(\mu_m)$, which is our $c_m$, and to fill in the remainder via corestriction.
 \end{remark}

 \begin{definition}
  For each prime $\ell$, let $H^1_f(\QQ_\ell, W)$ be the image of $H^1_f(\QQ_\ell, V)$ in $H^1(\QQ_\ell, W)$.

  Define
  \[ \mathcal{S}^{\{p\}}(\QQ,W)=\ker\big(H^1(\QQ,W)\rTo \bigoplus_{\ell\neq p} H^1(\QQ_\ell,W)/H^1_f(\QQ_\ell,W)\big),\]
  and define the \emph{strict Selmer group} of $W$ over $\QQ$ as
  \[ \mathcal{S}_{\{p\}}(\QQ,W)=\ker\big(\mathcal{S}^{\{p\}}(\QQ,W)\rTo H^1(\QQ_p,W)\big).\]
 \end{definition}

 (Thus $\mathcal{S}_{\{p\}}$ is the Selmer group with local conditions given by the Bloch--Kato condition at primes away from $p$ and the zero local condition at $p$.)

 We define $\mathcal{S}^{\{p\}}(\QQ, T)$ similarly, and also $\mathcal{S}^{\{p\}}(K, T)$ similarly, for any number field $K$. (We shall only need this when $K = \QQ(\mu_m)$, see Hypothesis $\Hyp(\cS^{(p)}, V)$ below.)

 In order to state the main theorem, we introduce the following sets of hypotheses. Note that $\Hyp(\QQ, T)$ is strictly stronger than $\Hyp(\QQ, V)$, but $\Hyp(p, A)$ and $\Hyp(\cS^{\{p\}}, V)$ are independent of each other.

 \begin{hypothesis}[$\Hyp(\QQ, T)$]
  $T \otimes \mathbf{k}$ is an irreducible $\mathbf{k}[G_{\QQ}]$-module, where $\mathbf{k}$ is the residue field of $\cO$; and there exists an element $\tau\in G_{\QQ}$ which satisfies the following conditions:
  \begin{enumerate}[(i)]
   \item $\tau$ acts trivially on $\mu_{p^\infty}$;
   \item $T/(\tau-1)T$ is free of rank 1 over $\cO$.
  \end{enumerate}
 \end{hypothesis}

 \begin{hypothesis}[$\Hyp(\QQ, V)$]
  $V$ is an irreducible $E[G_{\QQ}]$-module; and there exists an element $\tau\in G_{\QQ}$ which satisfies the following conditions:
  \begin{enumerate}[(i)]
   \item $\tau$ acts trivially on $\mu_{p^\infty}$;
   \item $\dim_{\Qp}(V/(\tau-1)V)=1$.
  \end{enumerate}
 \end{hypothesis}

 \begin{hypothesis}[$\Hyp(p, A)$]
  The set $A$ contains all powers of $p$.
 \end{hypothesis}

 \begin{hypothesis}[$\Hyp(\cS^{\{p\}}, V)$]
  The following three conditions hold:
  \begin{enumerate}[(i)]
   \item $T^{G_{\QQ}} = 0$;
   \item $c_m \in \cS^{\{p\}}(\QQ(\mu_m), T)$ for all $m \in A$;
   \item there exists an element $\gamma \in G_{\QQ}$ such that
  \begin{itemize}
   \item $\gamma$ acts trivially on $\mu_{p^\infty}$,
   \item $\gamma - 1$ is injective on $T$.
  \end{itemize}
  \end{enumerate}
 \end{hypothesis}

 \begin{theorem}
  \label{thm:finiteSel}
  Assume that $V$ is not the trivial representation, and that Hypothesis $\Hyp(\QQ, V)$ and at least one of hypotheses $\Hyp(p)$ and $\Hyp(\cS^{(p)}, V)$ are satisfied. If $\mathbf{c} = (c_m)_{m \in A}$ is an Euler system for $(T, A, \Sigma)$, and the image of $c_1$ in $H^1(\QQ, T)$ is not contained in $H^1(\QQ,T )_{\text{tors}}$, then $\cS_{\{p\}}(\QQ, W^*(1))$ is finite.
 \end{theorem}

 \begin{theorem}
  \label{thm:boundedSel}
  Assume that $p > 2$ and that Hypothesis $\Hyp(\QQ, T)$ and at least one of hypotheses $\Hyp(p)$ and $\Hyp(S^{(p)}, V)$ are satisfied. If $\mathbf{c} = (c_m)_{m \in A}$ is an Euler system for $(T, A, \Sigma)$, then
  \[ \operatorname{length}_{\cO}(\cS_{\{p\}}(\QQ, W^*(1))) \le \operatorname{ind}_{\cO}(\mathbf{c}) + \mathfrak{n}_W + \mathfrak{n}^*_W\]
  where $\operatorname{ind}_{\cO}(\mathbf{c})$ is the largest power of the maximal ideal by which $c_1$ can be divided in $H^1(\QQ, T) / $ torsion, and the quantities $\mathfrak{n}_W$ and $\mathfrak{n}^*_{W}$ are as defined in Theorem 2.2.2 of \cite{rubin00}.
 \end{theorem}

 \begin{proof}[Proofs]
  If $\Hyp(p, A)$ holds, then Theorem \ref{thm:finiteSel} and Theorem \ref{thm:boundedSel} are Theorem 2.2.3 and Theorem 2.2.2 of \cite{rubin00} respectively. If instead $\Hyp(\cS^{(p)}, V)$ holds, then the necessary modifications to the proofs are outlined in \S 9.1 of \emph{op.cit.}.
 \end{proof}


\subsection{Verifying the hypotheses on T}

 The main result of this section is Proposition \ref{prop:specialelement} below, which implies that under some mild technical assumptions there is a large supply of primes where the condition $\Hyp(\QQ, T)$ is satisfied.


 \subsubsection{Big image results for one modular form}

 We begin by some results from \cite{momose81} and \cite{ribet85} regarding the image of the Galois representations attached to a modular form. Let $f=\sum_{n \geq 1}a_n q^n$ be a new eigenform of weight $k \ge 2$, level $N$ and character $\epsilon$, not of CM type. Let $L = \QQ(a_n: n\geq 1)$ be its coefficient field, with ring of integers $\cO_L$.

 Recall that an \emph{extra twist} of $f$ is an element $\gamma \in \Gal(L / \QQ)$ such that $\gamma(f)$ is equal to the twist of $f$ by some Dirichlet character $\chi_\gamma$. We let $\Gamma \subseteq \Gal(L / \QQ)$ be the group of such $\gamma$, and $F \subseteq L$ the fixed field of $\Gamma$; and we let $H \subseteq \Gal(\QQbar / \QQ)$ be the absolute Galois group of the finite abelian extension $K$ cut out by the Dirichlet characters $\chi_\gamma$.

 For each prime $\lambda$ of $L$, it is clear that the trace of the Galois representation $\rho_{L_\lambda}(f) |_{H}$ takes values in $F_\mu$, where $\mu$ is the prime of $F$ below $\lambda$.

 \begin{theorem}[Momose--Ribet; see {\cite[Theorem 3.1]{ribet85}}]
  \label{thm:momoseribet}
  For all but finitely many $\lambda$, the image of the Galois representation $\rho_{L_\lambda}(f) |_{H}$ is a conjugate of the group
  \[ \{ g \in \GL_2(\cO_{F, \mu}) : \det(g) \in \ZZ_\ell^\times\}\]
  where $\mu$ and $\ell$ are the primes of $F$ and $\QQ$ below $\lambda$.
 \end{theorem}

 \begin{remark}
  For a ``generic'' modular form $f$, there will be no extra twists if the character $f$ is trivial, but there will always be at least one if $f$ has nontrivial character, since the complex conjugate $f^*$ is a twist of $f$.
 \end{remark}

 We will need the following slight strengthening:

 \begin{proposition}
  \label{prop:bigsubgroup1form}
  Let $K'$ be any finite extension of $K$ which is abelian over $\QQ$, and let $H' \subseteq H$ be its absolute Galois group. Then for all but finitely many $\lambda$, the image of $\rho_{L_\lambda}(f) |_{H'}$ is a conjugate of the group $\{ g \in \GL_2(\cO_{F, \mu}) : \det(g) \in \ZZ_\ell^\times\}$ above.
 \end{proposition}

 \begin{proof}
  If $\lambda$ is a prime satisfying the conclusion of the theorem, then the image of $\rho_{L_\lambda}(f)|_{H'}$ contains $\SL_2(\cO_{F, \mu})$, since $\SL_2(\cO_{F, \mu})$ is equal to its own commutator subgroup. But for all but finitely many primes $\ell$, the field $K'$ is linearly disjoint from $\QQ(\mu_{\ell^\infty})$ and thus the cyclotomic character is a surjection $H' \to \ZZ_\ell$.
 \end{proof}


\subsubsection{Big image results for pairs of modular forms}

 We recall the following result from group theory:

 \begin{proposition}[Goursat's Lemma, cf.~{\cite[Exercise I.5]{lang02}}]
  Let $G_1, G_2$ be groups and $H$ a subgroup of $G=G_1 \times G_2$ such that the projections $\pi_i: H \to G_i$ are surjective. Let $N_1 = H \cap \left(G_1 \times \{e_2\}\right)$ and $N_2 = H \cap \left(\{e_1\} \times G_2\right)$, which we identify with subgroups of $G_1, G_2$ in the obvious manner. Then the $N_i$ are normal in $G_i$, and $H$ is the graph of an isomorphism $G_1/N_1 \cong G_2/N_2$.
 \end{proposition}

 \begin{corollary}
  Let $\FF, \FF'$ be finite fields of the same characteristic, both of order $\ge 4$. Let $H$ be a subgroup of $\SL_2(\FF) \times \SL_2(\FF')$ surjecting onto both factors. Then either $H$ is the whole of $\SL_2(\FF) \times \SL_2(\FF')$, or $\FF = \FF'$ and $H$ is conjugate in $\GL_2(\FF) \times \GL_2(\FF)$ to one of the following subgroups:
  \begin{enumerate}[(i)]
   \item the diagonal subgroup $\{ (x, \vp^j(x)) : x \in G\}$, for some $0 \le j < k$, where $\FF = \FF_{p^k}$ and $\vp$ is the $p$-power Frobenius of $\FF$;
   \item the subgroup $\{ (x, y) : y = \pm \vp^j(x)\}$, for some $0 \le j < k$.
  \end{enumerate}
 \end{corollary}

 \begin{proof}
  This follows immediately from Goursat's lemma and a case-by-case check, given that the groups $\operatorname{PSL}_2(\FF)$ for fields $\FF$ of order $\ge 4$ are pairwise non-isomorphic simple groups, and the automorphism groups of $\SL_2(\FF)$ and $\operatorname{PSL}_2(\FF)$ are both isomorphic to $\PGL_2(\FF) \rtimes \langle \vp \rangle$.
 \end{proof}

 \begin{proposition}
  Let $\cO, \cO'$ be the rings of integers of any two unramified extensions of $\Qp$, where $p$ is a prime $\ge 5$, with residue fields $\FF, \FF'$. Then any closed subgroup $H \subseteq \SL_2(\cO) \times \SL_2(\cO')$ which surjects onto $\SL_2(\FF) \times \SL_2(\FF')$ must be the whole of $\SL_2(\cO) \times \SL_2(\cO')$.
 \end{proposition}

 \begin{proof}
  We follow the argument given for $\SL_2(\Zp)$ by Swinnerton-Dyer in \cite{swinnertondyer72}. It suffices to show that for each $n \ge 2$, the image $H_n$ of $H$ in $\SL_2(\cO / p^n)\times \SL_2(\cO'/p^n)$ contains the subgroups $K_n \times 1$ and $1 \times K_n'$, where $K_n$ is the kernel of $\SL_2(\cO / p^n) \to \SL_2(\cO / p^{n-1})$ and similarly for $K_n'$. Note that for each $n$, the group $K_n$ is abelian, and is isomorphic (via $m \mapsto 1 + p^{n-1} m$) to the group of trace zero matrices in $M_2(\FF)$, which is generated by matrices $u$ such that $u^2 = 0$.

  We now proceed by induction on $n$. Let $u\in M_2(\FF)$ satisfy $u^2=0$. By assumption, we may then find $h \in H$ congruent to $(1 + u, 1)$ modulo $p$; and, as shown in \textit{op.cit.}, we have $h^p = (1 + p u, 1) \bmod p^2$. Thus $(1 + pu, 1) \in H_2$, and thus $H_2 \supseteq K_2 \times 1$. Similarly, $H_2$ contains $1 \times K_2'$, so in fact $H_2$ is the whole of $\SL_2(\cO / p^2) \times \SL_2(\cO' / p^2)$.

  Suppose $n \ge 3$ and $H_{n-1}$ is everything. We claim $H_n$ contains $K_n \times 1$. Again, $K_n$ consists of matrices of the form $(1 + p^{n-1} u, 1)$, and by the induction assumption we can find $h \in H$ congruent to $(1 + p^{n-2} u, 1)$ modulo $p^{n-1}$. Then $h^p$ is congruent to $(1 + p^{n-1} u, 1)$ modulo $p^n$, so $(1 + p^{n-1} u, 1) \in H_n$. Thus $H_n \supseteq K_n \times 1$ and similarly $H_n \supseteq 1 \times K_n'$, so we are done.
 \end{proof}

 As a corollary, we obtain the following result.

 \begin{proposition}
  \label{prop:bigsubgroup}
  Let $\cO, \cO'$ be as above, with characteristic $\ge 5$, and let $H$ be a subgroup of $\SL_2(\cO) \times \SL_2(\cO')$ which surjects onto both factors. Then either $H = \SL_2(\cO) \times \SL_2(\cO')$, or $\cO' = \cO$  and $H$ is contained in the subgroup
  \[ \{ (x, y) \in \SL_2(\cO) : x = \pm \vp^j y \bmod p\}\]
  for some $j$.
 \end{proposition}

 We shall now boost this to a statement about $\GL_2$. For $\cO, \cO'$ as before, let $G$ denote the group
 \[ \{ (x, y) \in \GL_2(\cO) \times \GL_2(\cO') : \det(x) = \det(y) \in \Zp^\times\}.\]
 We can regard this as a fibre product $G_1 \times_{\Zp^\times} G_2$, where $G_1 = \{ x \in \GL_2(\cO) : \det(x) \in \Zp^\times\}$ and similarly for $G_2$.

 \begin{proposition}
  Let $H$ be a subgroup of $G$ which surjects onto $G_1$ and $G_2$. Then either $H = G$, or we have $\cO = \cO'$ and $H$ is contained in the subgroup of $G$ given by $\{ (x, y) : x = \pm \vp^j y \bmod p\}$ for some $j$.
 \end{proposition}

 \begin{proof}
  Let $G^\circ = \SL_2(\cO) \times \SL_2(\cO')$ and let $H^\circ = H \cap G^\circ$. Then $H^\circ$ has full image in each of $\SL_2(\cO)$ and $\SL_2(\cO')$, so either $H^\circ = G^\circ$, or $\cO' = \cO$ and the image of $H^\circ$ modulo $p$ is contained in the subgroup $\{ (x, y): x = \pm \vp^j y \bmod p\}$ for some $j$.

  Suppose first that $H^\circ = G^\circ$. Then we must have $H = G$, since for each $g \in G$, there is some $h \in H$ with $\det(h) = \det(g)$, and then $h^{-1} g$ lies in $G^\circ$ so by assumption it must be in $H$.

  In the remaining case, by replacing $H$ with its image under the automorphism $\vp^j \times 1$, we may assume without loss of generality that $j = 0$. Then any $h \in H^\circ$ is of the form $(x, y)$ with $x = \pm y$ modulo $p$. Let $(x, y)$ be any element of $H$, and consider the class of $t = x^{-1} y$ in $\operatorname{PSL}_2(\FF)$; then for any $(u, v) \in H^\circ$, we have
  \[ [u^{-1}t u] = [u^{-1} x^{-1} y u] = [x^{-1}][ (x u x^{-1})^{-1}(y v y^{-1}) ][y][v^{-1} u] = [x^{-1} y] = [t],\]
  since $(x u x^{-1}, y v y^{-1}) \in H^\circ$. Thus the classes $[t]$ and $[u]$ commute in $\operatorname{PSL}_2(\FF)$. However, since $H^\circ$ surjects onto $\SL_2(\cO)$, this forces $[t]$ to be in the centre of $\operatorname{PSL}_2(\FF)$, which is trivial (since it is a simple group). Thus $x = \pm y \bmod p$ for all $(x, y) \in H$, as claimed.
 \end{proof}

 Assume now that we have two newforms $f$ and $g$, and let $L$ be the subfield of $\QQbar$ generated by the coefficients of $f$ and $g$. For each prime $\fp$ of $L$, we may consider the image of the Galois representation $\rho_{f, \fp} \times \rho_{g, \fp}: \Gal(\QQbar / \QQ) \to \GL_2(L_\fp) \times \GL_2(L_\fp)$.

 Let $H$ be the subgroup of $\Gal(\QQbar / \QQ)$ cut out by the Dirichlet characters corresponding to the ``extra twists'' of $f$ and $g$, and let $K$ be its fixed field (an abelian extension of $\QQ$). Let $F, F'$ be the subfields of $L$ fixed by the extra twists. By Proposition \ref{prop:bigsubgroup1form}, we know that for all but finitely many $\fp$, the image of $\rho_{f, \fp}|_H$ is the group $\{ x \in \GL_2(\cO) : \det(x) \in \Zp^\times\}$, where $\cO$ is the completion of $F$ at the prime below $\fp$ and $p$ is the residue characteristic of $\fp$; similarly, the image of $\rho_{g, \fp}|_H$ will be the group $\{ x \in \GL_2(\cO') : \det(x) \in \Zp^\times\}$ where $\cO'$ is the completion of $F'$ at $\fp$. Then the image of the Galois representation $\rho_{f, \fp} \times \rho_{g, \fp}$ is a subgroup of the group $G_\fp = G$ defined above, which surjects onto either factor.

 \begin{proposition}
  In the above situation, either the image of $H$ under $\rho_{f, \fp} \times \rho_{g, \fp}$ is $G_\fp$, or $\cO' = \cO$ and there is an element $\gamma \in \Gal(L / \QQ)$ and a quadratic character $\chi: H \to \{\pm 1\}$ such that the equality
  \begin{equation}
   \label{eq:congruence}
   \rho_{f,\fp}(\sigma) = \pm \rho_{f,\fp}(\sigma)^\gamma \bmod \fp
  \end{equation}
  holds for all $\sigma \in H$.
 \end{proposition}

 \begin{proof}
  We know from above that if the image of $\rho_{f, \fp} \times \rho_{g, \fp}$ is not $G$, then $\cO = \cO'$ and $\rho_{f, \fp}(\sigma) = \pm \vp^j \rho_{g, \fp}(\sigma) \bmod \fp$ for all $\sigma \in H$, where $\vp^j$ is the mod $\fp$ Frobenius.

  Now we may take $\gamma$ to be any element of the decomposition group of $\fp$ in $\Gal(L / \QQ)$ reducing to $\vp^j$ modulo $\fp$. (Of course, there will almost always be only be one such element, since only finitely many primes ramify in $L / \QQ$.)
 \end{proof}

 We now lift to characteristic 0. Let $w$ be a prime of the field $K$; we define $a_w(f) = \tr \rho_{f, \lambda}(\sigma_w^{-1})$, where $\sigma_w$ is the arithmetic Frobenius at $w$ in $H$ and $\lambda$ is some prime of $L$; if $w$ is a degree 1 prime, then this is just $a_v(f)$ where $v$ is the rational prime below $w$, and for general primes $w$ it may be expressed as a polynomial in $a_v(f)$ and $\chi_v(f)$. In any case it is obviously independent of the choice of auxilliary prime $\lambda$, and (since $K$ is abelian over $\QQ$) it depends only on the prime $v$ of $\QQ$ below $w$. We define $a_w(g)$ similarly.\footnote{Of course, we can define the quantity $a_w(f)$ intrinsically in ``automorphic'' terms, as (up to normalizations) it is the trace of the $d$-th power of the conjugacy class in $\GL_2(\CC)$ which is the Satake parameter of $\pi_{f, v}$, where $d$ is the degree of the unramified extension $[K_w: \QQ_v]$; this makes the independence of $\lambda$ selfevident.}

 \begin{definition}
  Let us say a prime $\fp$ of residue characteristic $\ge 5$ is a \emph{good prime} for the pair $(f, g)$ if the image of $H$ under $\rho_{f, \fp} \times \rho_{g, \fp}$ is the whole of $G_\fp$. If the image is a proper subgroup, but has full projection to either factor, we say $\fp$ is a \emph{bad prime}.
 \end{definition}

 \begin{remark}
  If $\fp$ divides 2 or 3, or is such that $\rho_{f, \fp}$ or $\rho_{g, \fp}$ has small image, we consider $\fp$ to be neutral, neither good nor bad. By the theorem of Momose--Ribet (Theorem \ref{thm:momoseribet}), there are only finitely many neutral primes.
 \end{remark}

 \begin{corollary}
  If there are infinitely many bad primes for $(f, g)$, then there is $\gamma \in \Gal(L / \QQ)$ such that the equality
  \[ a_w(f) = \pm \gamma(a_w(g))\]
  for all primes $w$ of $K$.
 \end{corollary}

 \begin{proof}
  For each bad prime $\fp$, there exists a $\gamma \in \Gal(L / \QQ)$ such that the congruence \eqref{eq:congruence} holds, and in particular (by taking $\sigma = \sigma_w^{-1}$) we have \( a_w(f)^2 = \gamma(a_w(g)^2) \bmod \fp\) for all primes $w$ of $K$.

  Since $\Gal(L / \QQ)$ is finite, there exists some $\gamma$ such that the congruence \eqref{eq:congruence} of the proposition holds for all $\fp$ in an infinite set $B$. In this case, we have
  \[ a_w(f)^2 = \gamma(a_w(g)^2) \bmod \fp\]
  for infinitely many $\fp$. So we must have an equality $a_w(f)^2 = \gamma(a_w(g)^2)$, since a nonzero element of a number field cannot be divisible by infinitely many primes.
 \end{proof}

 \begin{corollary}
  If there are infinitely many bad primes for $(f, g)$, there exists a quadratic Groessencharacter $\kappa$ of $K$ (equivalently, a continuous quadratic character of $H$) such that
  \[ a_w(f) = \kappa(w) a_w(g)\]
  for all primes $w$ of $K$.
 \end{corollary}

 \begin{proof}
  This follows from the strong multiplicity one theorem for $\SL_2 / K$, cf.~\cite{ramakrishnan00}: the Satake parameters of the base-change representations $BC(\pi_f)$ and $BC(\pi_g^\gamma)$ of $\GL_2(\AA_K)$ agree up to sign at any prime $w$, and Ramakrishnan's result guarantees that the sign relating the two is given by a quadratic character.
 \end{proof}

 \begin{remark}
  Frustratingly it does not seem to be possible to show the existence of $\kappa$ without such heavy automorphic machinery, even though we know that for infinitely many primes $\fp$ the sign relating $\rho_f$ and $\rho_g$ modulo $\fp$ is given by a character.
 \end{remark}

 \begin{theorem}\label{thm:badprimes}
  Suppose there are infinitely many bad primes for $(f, g)$. Then $f$ is Galois-conjugate to some twist of $g$.
 \end{theorem}

 \begin{proof}
  From Ramakrishan's theorem, we know that we have
  \[ \rho_{f, \lambda}|_H \cong \rho_{g, \lambda}^\gamma|_H \otimes \tau\]
  for a quadratic character $\tau$ of $H$, and any choice of prime $\lambda$. Inducing up from $H$ to $H' = \Gal(\QQbar / \QQ)$, we have
  \[ \rho_{f, \lambda} \otimes \Ind_{H}^{H'}(1_H) = \rho_{g, \lambda} \otimes \Ind_{H}^{H'}(\tau).\]
  But the left-hand side contains $\rho_{f, \lambda}$ as a direct summand, while the right-hand side is a direct sum of representations of the form $\rho_{g, \lambda}^\gamma \otimes \mu$ where $\mu$ is an irreducible Artin representation. Hence there must be at least one $\mu$ which is one-dimensional and such that $\rho_{f, \lambda} \cong \rho_{g, \lambda}^\gamma \otimes \mu$, in which case we must have $f = g^\gamma \otimes \mu$.
 \end{proof}


\subsubsection{Existence of the special element}

 As in the previous section, let $f$ and $g$ be two newforms, and let $L$ be the subfield of $\QQbar$ generated by the coefficients of $f$ and $g$. We assume that $f$ is not Galois conjugate to a twist of $g$, so by Theorem \ref{thm:badprimes} there are only finitely many bad primes for $(f,g)$. We retain the notation of the previous section.

 Let $\fp$ be a good prime which does not divide the levels of $f$ and $g$, and $p$ the rational prime below $\fp$.

 We make the following crucial assumption:

 \begin{assumption}
  The character $\chi = (\chi_f \chi_g)^{-1}$ is nontrivial, and its conductor is not a power of $p$.
 \end{assumption}

  For simplification, we also make the following assumption:

  \begin{assumption}
   We have $L_{f, \fp} = L_{g,\fp}= \Qp$, so after a suitable choice of basis, we may assume that the image of $\rho_{f, \fp} \times \rho_{g, \fp}$ is contained in $\GL_2(\Zp) \times \GL_2(\Zp)$.
  \end{assumption}

  We can now prove the main result of this section.

  \begin{proposition}\label{prop:specialelement}
   There exists an element $\tau\in G_{\QQ(\mu_{p^\infty})}$ such that if $V / (\tau - 1) V$ is 1-dimensional, where $V = V_{L_\fp}(f, g)^*$.

   If $\chi$ is not congruent modulo $\fp$ to any character of $p$-power conductor, then there exists $\tau$ such that $T / (\tau - 1) T$ is free of rank 1, for any $G_{\QQ}$-stable lattice $T$ in $V$.
  \end{proposition}

  \begin{proof}
   Choose some $\alpha\in\Gal(\overline{\QQ}/\QQ)$ such that $\chi(\alpha) \ne 1$, but $\alpha$ is in the kernel of the $p$-adic cyclotomic character. Note that such $\alpha$ do exist, since the conductor of $\chi$ not a power of $p$. Consider the coset $\alpha \cdot (H \cap G_{\QQ(\mu_{p^\infty})})$. Since $\fp$ is a good prime, under $\rho_{f, \fp} \times \rho_{g, \fp}$, the coset $\alpha \cdot (H \cap G_{\QQ(\mu_{p^\infty})})$ is mapped to
   \[ (\rho_{f, \fp}(\alpha), \rho_{g, \fp}(\alpha) ) \cdot \SL_2(\Zp)^2, \]
   which consists of \emph{all} pairs $(u, v)$ of matrices such that $\det(u) = \chi_f(\alpha)$ and $\det(v) = \chi_g(\alpha)$. In particular, it contains the pair
   \[ \tbt x 0 0 {x^{-1}\chi_f(\alpha)} \tbt {x^{-1}} 0 0 {x \chi_g(\alpha)}\]
   for any $x \in \Zp^\times$. The image of this pair under the tensor product homomorphism $\GL_2 \times \GL_2 \to \GL_4$ is the diagonal matrix with entries
   \[ \left[1, x^{-2}\chi_f(\alpha), x^2 \chi_g(\alpha), \chi_f(\alpha) \chi_g(\alpha)\right].\]
   By choosing $x$ appropriately, we can arrange that neither $x^{-2}\chi_f(\alpha)$ nor $x^2 \chi_g(\alpha)$ is equal to 1. Thus 1 is an eigenvalue of $\tau$ on $V_{L_\fp}(f) \otimes V_{L_\fp}(g)$ with multiplicitly exactly 1.

   If we assume the stronger condition on $\chi$ in the statement, then we can assume that $\chi_f(\alpha) \chi_g(\alpha)$ is not 1 modulo $\fp$. By choosing $x$ appropriately we can assume that $x^{-2}\chi_f(\alpha)$ and $x^2 \chi_g(\alpha)$ are also non-congruent to 1, so it follows that $T / (\tau - 1)T$ is free of rank 1 as required (for any $\tau$-stable $\cO_L$-lattice in $V$, and in particular any $G_{\QQ}$-stable lattice).
  \end{proof}

 \subsubsection{The quantities $\mathfrak{n}_W$ and $\mathfrak{n}^*_W$}

  We recall the definitions of the quantities $\mathfrak{n}_W$ and $\mathfrak{n}^*_W$ in Rubin's theory. Let $T$ be a finite-rank free $\cO$-module with a continuous action of $G_{\QQ}$. As usual, write $V = T \otimes E$ and $W = V / T$.

  \begin{definition}
   Define $\Omega$ to be the smallest extension of $\QQ$ whose Galois group acts trivially on $W$ and on $\mu_{p^\infty}$, and define
   \begin{align*}
    \mathfrak{n}_W &= \ell_\cO\left( H^1(\Omega / \QQ, W) \cap \cS^{\{p\}}(K, W)\right)\\
    \mathfrak{n}^*_W &= \ell_\cO\left( H^1(\Omega / \QQ, W^*(1)) \cap \cS_{\{p\}}(K, W^*(1))\right).
   \end{align*}
  \end{definition}

  We now give conditions under which these quantities are zero.

  \begin{proposition}
   \label{prop:nw}
   Suppose the centre of $\Omega$ acts on each of $T \otimes \mathbf{k}$ and $T^*(1) \otimes \mathbf{k}$ via a nontrivial character. Then $n_W = n^*_W = 0$.
  \end{proposition}

  \begin{proof}
   We shall show that the hypotheses imply that $H^1(\Omega / \QQ, W) = H^1(\Omega / \QQ, W^*(1)) = 0$. We give the argument for $W$; the proof for $W^*(1)$ is similar.

   Clearly we have $H^1(\Omega / \QQ, V) = 0$, and hence $H^1(\Omega / \QQ, W)$ is finite. So it suffices to show that $H^1(\Omega / \QQ, W)[\varpi] = 0$ where $\varpi$ is a uniformizer. But we have a surjection $H^1(\Omega / \QQ, W[\varpi]) \twoheadrightarrow H^1(\Omega / \QQ, W)[\varpi]$, so we are reduced to showing that $H^1(\Omega / \QQ, W[\varpi]) = H^1(\Omega / \QQ, T \otimes \mathbf{k})$ is zero. However, this is immediate since any representation of nontrivial central character cannot have a nontrivial extension by the trivial representation.
  \end{proof}


\subsection{The Euler system}

 As above, let $f$ and $g$ be newforms of weight $2$, level $N$ and characters $\chi_f$ and $\chi_g$, respectively. Let $p$ be a prime not dividing $N$. Let $L$ be a number field containing the coefficients of $f_\alpha$ and $g_\beta$, and let $\fp$ be a prime of $L$ above $p$. Let $E = L_\fp$ and $\cO$ its ring of integers. We write $T = T_{\cO}(f, g)^*$ and $p_\ell(X) = \det(1 - \Frob_q^{-1} X | T^*(1)) = P_\ell(f, g, \ell^{-1} X) \in \cO_L[X]$. We assume that the following conditions are satisfied:

 \begin{assumption}
  \begin{enumerate}[(i)]
   \item the character $\chi = \chi_f \chi_g$ is not trivial, and moreover is not trivial modulo $\fp$;
   \item there exist $p$-stabilizations $f_\alpha, g_\gamma$ of $f$ and $g$ with $U_p$-eigenvalues $\alpha, \gamma$ respectively such that
   \begin{itemize}
    \item $v_{\fp}(\alpha \gamma) < 1$,
    \item $\alpha / \gamma$ is not a root of unity.
   \end{itemize}
  \end{enumerate}
 \end{assumption}

 Fix $c \ge 1$ coprime to $6N$, and let $A$ be the set of square-free integers prime to $Npc$. By Corollary \ref{cor:nearlycompatible}, we have, for every integer $m \in A$, a cohomology class ${}_c \bfz_m^{f,g,N} \in H^1(\QQ_{m},T)$ which satisfy the following compatibility property: if $m \in A$ and $\ell$ is a prime comprime to $mNpc$, then the image of $\bfz_{\ell m}^{f,g,N}$ under the corestriction map $H^1(\QQ_{\ell m},T)\rightarrow H^1(\QQ_{m},T)$ is
 \[ -\sigma_\ell A_\ell(\sigma_\ell^{-1})\bfz_m^{f,g,N}, \]
 where $A_\ell(X)$ is a polynomial in $\cO_L[X]$ congruent modulo $\ell - 1$ to $p_\ell(X)$.

 \begin{lemma}
  There exists a system of cohomology classes
  \[ \{ {}_c \tilde\bfz_m^{f,g,N} \in H^1(\QQ_{m},T) : m \in A\}\]
  such that
  \[ {}_c \tilde\bfz_1^{f,g,N} = {}_c \bfz_1^{f,g,N}\]
  and if $m \in A$ and $\ell$ is a prime such that $m\ell \in A$, then the image of ${}_c \tilde\bfz_{\ell m}^{f,g,N}$ under the corestriction map $H^1(\QQ_{\ell m},T)\rightarrow H^1(\QQ_{m},T)$ is
 \[A_\ell(\sigma_\ell^{-1}) {}_c\tilde\bfz_m^{f,g,N}.\]
 \end{lemma}

 \begin{proof}
  By induction on the number of prime factors of $m$, we can choose (non-canonically) a system of elements $\gamma_m \in (\ZZ / m\ZZ)^\times$ for every $m \in A$ such that $\gamma_{m\ell} = \ell^{-1} \gamma_m \bmod m$. Identify $\gamma_m$ with an element of $\Gal(\QQ(\mu_m/\QQ)$ via the inverse of the cyclotomic character, and define ${}_c \tilde\bfz_m^{f,g,N}=(-1)^{s(m)} \gamma_m \cdot {}_c \bfz_m^{f,g,N}$, where $s(m)$ is the number of prime factors of $m$. It is clear by construction that the elements have the required property.
 \end{proof}

 \begin{note}
  By Corollary \ref{cor:h1f}, the classes ${}_c \bfz_m^{f,g,N}$ are in the Selmer group $\cS^{\{p\}}(\QQ(\mu_m), T_{\cO_\fp}(f, g)^*)$. As $\cS^{\{p\}}(\QQ(\mu_m),T_{\cO_\fp}(f, g)^*)$ is invariant under the action of $\Gal(\QQ(\mu_m)/\QQ)$, it follows that the same is true for the modified classes ${}_c \tilde\bfz_m^{f,g,N}$.
 \end{note}

 We now show that we can convert the classes $({}_c \bfz_m^{f,g,N})_{m\in A}$ into an Euler system. Let $T$, $\Sigma$ and $A$ be as defined  at the beginning of Section \ref{Eulermethod}. Then we have the following result (c.f. \cite[Lemma 9.6.1]{rubin00}):

 \begin{lemma}
  \label{lem:modifyclasses}
  Suppose that for all primes $\ell \in A$ we have polynomials $r_\ell(X),s_\ell(X)\in \cO[X]$ such that
  \[ r_\ell(X)\equiv s_\ell(X)\pmod{\ell-1},\]
  and suppose that we have a collection of cohomology classes $\big\{\tilde{c}_m\in H^1(\QQ(\mu_m),T): m\in A\big\}$ such that if $\ell\in A$ is coprime to $m$, then
  \[ \cores_{\QQ(\mu_{\ell m})/\QQ(\mu_m)}(\tilde{c}_{\ell m})=\begin{cases}
   r_\ell(\sigma_\ell^{-1})\tilde{c}_m & \text{if $\ell\nmid m$}\\
   \tilde{c}_m & \text{if $\ell \mid m$}
  \end{cases}.\]
  Then there exists a collection of classes $\big\{c_m\in H^1(\QQ(\mu_m),T): m\in A\big\}$ with the following properties:
  \begin{enumerate}[(i)]
   \item For all $m$,
   \[ c_m \in \cO[(\ZZ / m\ZZ)^\times] \cdot \tilde{c}_m.\]
   \item if $\ell$ is a prime such that $m,m\ell\in A$, then
   \[ \cores_{\QQ(\mu_{\ell m})/\QQ(\mu_m)}(c_{\ell m})=
   \begin{cases}
    s_\ell(\sigma_\ell^{-1})c_m & \text{if $\ell\nmid m$}\\
    c_m & \text{if $\ell\mid m$}
   \end{cases} \]
   \item if $m\in A$ and $\chi$ is a character of $\Gal(\QQ(\mu_m)/\QQ)$ of conductor $k$ such that $\Prime(m)\subset\Prime(k)\cup\Sigma$, then
   \[ \sum_{\gamma\in \Gal(\QQ(\mu_m)/\QQ)}\chi(\gamma)\gamma(c_m)=\sum_{\gamma\in \Gal(\QQ(\mu_m)/\QQ)}\chi(\gamma)\gamma(\tilde{c}_m).\]
  \end{enumerate}
 \end{lemma}

 \begin{definition}
  \label{def:ourEulersystem}
  Define $\{ {}_c \hat\bfz_m^{f,g,N}\in H^1(\QQ(\mu_m),T_{\cO_\fp}(f, g)^*)): m\in A\}$ to be the classes obtained by applying Lemma \ref{lem:modifyclasses} to our classes ${}_c \tilde\bfz_m^{f,g,N}$, where we take $r_\ell(X)=A_\ell(X)$ and $s_\ell(X)=p_\ell(X) = P_\ell(f, g, \ell^{-1}X)$, and as above $A$ is the set of square-free integers coprime to $Npc$.
 \end{definition}

 \begin{note}
  By construction, the classes $\{ {}_c \hat\bfz_m^{f,g,N}: m\in A\}$ are an Euler system for $(T, \Sigma, A)$ in the sense of Definition \ref{def:Eulersystem}, where $\Sigma$ is the set of primes dividing $Npc$. Moreover, because of (i), we have ${}_c \hat\bfz_m \in \cS^{\{p\}}(\QQ(\mu_m), T)$ for all $m$.
 \end{note}

\subsection{Finiteness of the strict Selmer group}

 We now combine the above results to prove a finiteness theorem for the strict Selmer group. For the convenience of the reader, we shall recapitulate all of the assumptions we have made on $f$ and $g$.

 \begin{assumption}
  \label{assump:everything}
  Assume that $f$ and $g$ are weight 2 newforms with coefficients in a number field $L$, and $\fp$ a prime of $L$ above the rational prime $p$, with the following properties:
  \begin{enumerate}[(i)]
   \item Neither $f$ nor $g$ is of CM type.
   \item $f$ is not a twist of $g$.
   \item The character $\varepsilon_f \varepsilon_g$ is non-trivial.
   \item $p \ge 5$.
   \item $p$ does not divide the levels of $f$ and $g$.
   \item $\fp$ is totally split in the field $L$, so $L_\fp = \Qp$.
   \item The $\fp$-adic Galois representations of $f$ and $g$ are surjective onto $\GL_2(\Zp)$.
   \item There exists some prime $v$ such that $\chi(v) = 1$ for all inner twists $\chi$ of $f$ or $g$, and $a_v(f) \ne \pm a_v(g) \bmod \fp$.
   \item $f$ is ordinary at $\fp$.
   \item There exists a root $\gamma$ of the Hecke polynomial of $g$ at $p$ such that $v_\fp(\gamma) < 1$ and $\alpha / \gamma$ is not a root of unity, where $\alpha$ is the unit root of the Hecke polynomial of $f$.
  \end{enumerate}
 \end{assumption}

 If we assume hypotheses (i)--(iii) (which do not depend on $\fp$), then there will be many $\fp$ such that the remaining hypotheses hold.

 \begin{theorem}
  \label{thm:ourfiniteSel}
  Suppose Assumption \ref{assump:everything} is satisfied, and the $p$-adic Rankin--Selberg $L$-function $\mathcal{D}_{\fp}(f, g, 1/N)$ does not vanish at $1$, where $N$ is some integer divisible by the levels of $f$ and $g$. Then
  \[ \# \cS_{\{p\}}\left(\QQ, \frac{V_{L_\fp}(f, g)}{T_{\cO_\fp}(f, g)}(1)\right) < \infty.\]
 \end{theorem}

 \begin{proof}
  It suffices to show that the hypotheses of Theorem \ref{thm:finiteSel} are satisfied for $T = T_{\cO_\fp}(f, g)^*$. By Proposition \ref{prop:specialelement}, the element $\tau$ required by Hypothesis $\Hyp(\QQ, V)$ exists; and the Euler system of Definition \ref{def:ourEulersystem} satisfies Hypothesis $\Hyp(\cS^{\{p\}}, V)$(ii). Since $T$ is nontrivial and irreducible, $T^{G_K} = 0$; and the element $\gamma$ in Hypothesis $\Hyp(\cS^{\{p\}}, V)$(iii) clearly exists.

  By Theorem \ref{thm:syntomicreg}, if $\mathcal{D}_{\fp}(f, g, 1/N)(1) \ne 0$, the image of $\reg_p \Xi_{1, N, 1}$ in the $(f, g)$-isotypical quotient  of $H^2_{\dR}(X_1(N) / \Qp) / \Fil^2$ is nonzero. Hence, by the diagram of \S \ref{sect:syntomic}, the localization of the Galois cohomology class $ \bfz_{1}^{f, g, N}$ at $p$ is nonzero, so in particular ${}_c \bfz_{1}^{f, g, N}$ is non-torsion as an element of $H^1(\QQ, T_{L_\fp}(f, g)^*)$ for any $c > 1$. Thus we may apply Theorem \ref{thm:finiteSel} to the Euler system $({}_c \hat\bfz_{m}^{f, g, N})_{m \in A}$ of Definition \ref{def:ourEulersystem} to obtain the finiteness of the strict Selmer group.
 \end{proof}

\subsection{The order of the strict Selmer group}

 Theorem \ref{thm:boundedSel} gives a bound for the order of the strict Selmer group, under slightly stronger hypotheses than Theorem \ref{thm:ourfiniteSel}.

 \begin{theorem}
  \label{thm:ourboundedSel}
  Suppose Assumption \ref{assump:everything} is satisfied, and in addition the mod $\fp$ reduction of $\varepsilon_f \varepsilon_g$ is not trivial. Then we have
  \[ \operatorname{length}_{\Zp} \cS_{\{p\}}\left(\QQ, \frac{V_{L_\fp}(f, g)}{T_{\cO_\fp}(f, g)}(1)\right) \le v_p\left( \frac{(1 - p^{-1} \beta \alpha)}{(1 - p^{-1} \beta \gamma)(1 - p^{-1} \beta \delta)} \mathcal{D}_\fp(f, g, 1/N)(1) \right) + \lambda\]
  where $\lambda$ is the $\fp$-adic valuation of the ideal $I_f$ of Definition \ref{def:idealIf} above.
 \end{theorem}

 \begin{proof}
  Our condition on the mod $\fp$ reduction of $\varepsilon_f \varepsilon_g$ implies that Hypothesis $\Hyp(\QQ, T)$ is satisfied (again by Proposition \ref{prop:specialelement}; note that the mod $\fp$ reduction cannot be a nontrivial character of $p$-power conductor as $p$ does not divide the levels of $f$ and $g$). The condition also assures that the quantities $n_W$ and $n_W^*$ appearing in Theorem \ref{thm:boundedSel} are zero (Proposition \ref{prop:nw}).

  We consider the linear functional $\alpha$ on $H^1(\QQ, V_{L_\fp}(f, g)^*)$ given by $x \mapsto \langle \log_{\Qp}(x), \eta_{f}^{\ur} \otimes \omega_g\rangle$. On the lattice $\widetilde T_{\cO_\fp}(f, g)^*$ this takes values in $I_f^{-1} \cdot (1 - \alpha^{-1} \gamma^{-1})^{-1}(1 - \alpha^{-1} \delta^{-1})^{-1} \cO_{\fp}$, by Corollary \ref{cor:integralpoincare}; but since the Galois representations of $f$ and $g$ are assumed to have big image, we have $\widetilde T_{\cO_\fp}(f, g)^* = T_{\cO_{\fp}}(f, g)^*$.

  Theorem \ref{thm:syntomicreg} shows that $\tau$ maps the class $\bfz_{f, g, 1}$ to
  \[ \frac{\mathcal{E}(f)\mathcal{E}^*(f)} {\mathcal{E}(f, g, 1)} \mathcal{D}_\fp(f, g, 1/N)(1).\]
  Hence the index of divisibility of $\bfz_1^{f, g, N}$ is bounded above by
  \[
   v_p \left( \frac{\cE(f) \cE^*(f) (1 - \alpha^{-1} \gamma^{-1})(1 - \alpha^{-1} \delta^{-1})}{\cE(f, g, 1)} \mathcal{D}_\fp(f, g, 1/N)(1) \right) + \lambda.
  \]
  We can ignore the factor $\mathcal{E}^*(f) \coloneqq 1 - \beta \alpha^{-1}$, since $\alpha_f$ is a unit and $\beta_f$ is a non-unit so $\mathcal{E}^*(f) \in \cO_{\fp}^\times$. Substituting the definitions of $\cE(f)$ and $\cE(f, g, 1)$, we have
  \[ \frac{\cE(f) (1 - \alpha^{-1} \gamma^{-1})(1 - \alpha^{-1} \delta^{-1})}{\cE(f, g, 1)} = \frac{(1 - p^{-1} \beta \alpha)}{(1 - p^{-1} \beta \gamma)(1 - p^{-1} \beta \delta)}.\]
 \end{proof}

\subsection{An example}

 It may seem slightly unclear whether the long list of conditions in Assumption \ref{assump:everything} may be simultaenously satisfied, so we present the following explicit example (computed using Sage \cite{sage}).

 Let $f$ be the unique weight 2 newform of level 11 (corresponding to the elliptic curve $E: y^2 + y = x^3 - x$); and let $g$ be the unique newform of weight 2, level 26, and character $\chi \coloneqq \left(\frac{\bullet}{13}\right)$ with $a_2(g) = i$, so the $q$-expansions of $f$ and $g$ are
 \begin{align*}
  f &= q - 2q^{2} - q^{3} + 2q^{4} + q^{5} + O(q^{6}),\\
  g &= q + i q^{2} - q^{3} - q^{4} - 3 i q^{5} + O(q^{6}).
 \end{align*}

 Note that $\chi$ has conductor 13, so the local component of $\pi_g$ at 2 is an unramified twist of the Steinberg representation; on the other hand $f$ is unramified principal series at 2 and Steinberg at 11. So $f$ cannot be a twist of $g$, and neither $f$ nor $g$ is of CM type (since CM forms cannot be Steinberg at any prime). The form $f$ has no inner twists (since it is non-CM and has coefficients in $\QQ$); as for $g$, its Galois orbit consists of $g$ and $\bar{g}$, so its only inner twist is $\bar g$.

 To calculate the image of the Galois representations of $f$ and $g$, we note that Sage \cite{sage} has a facility to compute all the exceptional primes for the Galois representation attached to an elliptic curve (i.e.~those primes for which the image of the Galois representation is not $\GL_2(\Zp)$). This speedily tells us that $\rho_{f, p}$ is surjective for all $p \ne 5$.

 The form $g$ does not correspond to an elliptic curve, but there is a Dirichlet character $\psi: (\ZZ / 13\ZZ)^\times \to \QQ(i)^\times$ such that $g \otimes \psi$ corresponds to an elliptic curve $E'$ of conductor $2 \times 13^2 = 338$ (the curve with Cremona label 338d, given by $y^2 + xy = x^3 + x^2 + 504x - 13112$), and the only exceptional primes for $E'$ are $\{3, 5\}$. Letting $H = G_{\QQ(\sqrt{13})}$, the kernel of the character $\chi$, we see that for all primes $p \notin \{2, 3, 5, 13\}$ the image of $H$ under $\rho_{g, \fp}$, for any prime $\fp$ of $\QQ(i)$ above $p$, is $\GL_2(\Zp)$.

 Moreover, the only prime such that $a_v(f) = a_v(g)$ for all $v$ split in $\QQ(\sqrt{13})$ is $p = 5$. We deduce that for any $p$ congruent to 1 mod 4 and not in $\{5, 13\}$, and any prime $\fp$ of $\QQ(i)$ above $p$, the hypotheses (i) -- (viii) are satisfied.

 We check that both $f$ and $g$ are ordinary at the primes above 17 (it doesn't matter which prime we take, since $a_{17}(g) \in \ZZ$); and for any choice of roots $\alpha, \gamma$ of roots of the Hecke polynomials of $f$ and $g$, the minimal polynomial of $\alpha / \gamma$ over $\QQ$ is $x^{4} + \frac{6}{17} x^{3} - \frac{21}{17} x^{2} + \frac{6}{17} x + 1$, so in particular $\alpha / \gamma$ is not a root of unity. Thus hypotheses (ix) and (x) are satisfied if $\fp$ is either of the primes above 17.


 \section{Conjectures on higher-rank Euler systems}

  We now explain how the cohomology classes constructed in the previous section may be reconciled with the general conjectural setup of cyclotomic Iwasawa theory for motivic Galois representations formulated by Perrin-Riou, and its extension to the two-variable situation as formulated by the second and third authors in \cite{loefflerzerbes11}.


  \subsection{Euler systems: rank 1 and higher rank}

   Let us place ourselves again in the general setting of \S \ref{Eulermethod} above, so $T$ is a free $\cO$-module with a continuous action of $G_{\QQ}$ unramified outside a finite set $\Sigma \ni p$, and $A$ is a set of integers satisfying the conditions \emph{loc.cit.}. Suppose that all integers in $A$ are coprime to $p$.

   Perrin-Riou's conjectures, as formulated in \cite{perrinriou98} (cf. also \cite[\S 8.5]{rubin00}), discuss the following class of objects:

   \begin{conjecture}
    An \emph{Euler--Iwasawa system} of rank $r \ge 1$ consists of the data of, for each $m \in A$, a class
    \[ c_m \in \bigwedge^r_{\Lambda(\Gamma_m)} H^1_{\Iw}(\QQ(\mu_{m p^\infty}), V) \]
    with the property that if $\ell$ is prime and $\ell, m\ell \in A$, we have
    \[ \cores_{\QQ(\mu_{mp^\infty})}^{\QQ(\mu_{m\ell p^\infty})} c_{m\ell} =
     \begin{cases}
      p_\ell(\sigma_\ell^{-1})c_{m} & \text{if $\ell \nmid m\Sigma$,}\\
      c_m & \text{if $\ell \mid m\Sigma$.}
     \end{cases}
    \]
   \end{conjecture}

   Note that a rank 1 Euler--Iwasawa system is equivalent to the data of an Euler system for $(T, A_p, \Sigma)$ in the previous sense, where $A_p = \{ p^k m : m \in A, k \ge 0\}$.

   As noted in \cite[\S 1.2.3]{perrinriou98}, a higher-rank Euler system can be used to construct rank 1 Euler systems, by pairing with appropriate ``rank $r-1$'' elements. We make the following definition:

   \begin{definition}
    We define a \emph{Perrin-Riou functional} to be the data of, for each squarefree $m$ prime to $S$ as above, an element
    \[ \Phi_m \in \bigwedge^{r-1} \Hom_{\Lambda}\left(H^1_{\Iw, S}(\QQ(\mu_{mp^\infty}), T), \Lambda\right)^\iota, \]
    with the property that for each $\ell \nmid mS$, we have
    \[ \Phi_{m} = \Phi_{m\ell} \circ \res_{\QQ(\mu_{mp^\infty})}^{\QQ(\mu_{m\ell p^\infty})}.\]
   \end{definition}

   Lemma 1.2.3 of op.cit.~shows that if $\left(c_p(m)\right)$ is an Euler system of rank $r$, and $\left(\Phi_m\right)$ is a Perrin-Riou functional, then the elements
   \[ \Phi_m(c_p(m)) \in H^1_{\Iw, S}(\QQ(\mu_{mp^\infty}), T)\]
   define an Euler system of rank 1. (Here, as explained loc.cit., we interpret $\Phi_m$ as a map
   \[ \bigwedge^{r} H^1_{\Iw, S}(\QQ(\mu_{mp^\infty}), T) \to H^1_{\Iw, S}(\QQ(\mu_{mp^\infty}), T)\]
   which we also denote by $\Phi_m$.)

   Appendix B of \emph{op.cit.} shows that (under mild hypotheses on $T$) there is a plentiful supply of Perrin-Riou functionals, although there is no obvious canonical choice. More specifically, given any $m$ and any $\Phi_m$, there exists a Perrin-Riou functional extending $\Phi_m$. Hence, given as a starting point a rank $r$ Euler system, one may construct a rank 1 Euler system (indeed many such systems) and obtain Iwasawa-theoretic results from this rank 1 system; but these rank 1 Euler systems are noncanonical, and in particular there is no reason to expect that they should have any relation to $L$-values.

   \begin{remark}
    An alternative approach to bounding Selmer groups in the $r > 1$ case by directly utilizing a notion of ``higher-rank Kolyvagin systems'', rather than by constructing rank 1 Euler systems, has been initiated by Mazur and Rubin (unpublished).
   \end{remark}


  \subsection{Otsuki's functionals}

   We now explain a construction due to Otsuki \cite{otsuki09}, who has shown how to construct \emph{canonical} linear functionals on cohomology groups by composing the dual exponential map with an appropriate ``weighted trace''. These maps do not satisfy the compatibility properties of a Perrin-Riou functional, and thus give rise to systems of elements of group rings satisfying a modified compatibility property; we shall show that this modification is consistent with the results we have shown for our generalized Beilinson--Flach classes.

   For technical reasons we shall work in the limit over the cyclotomic extension, rather than directly over $\QQ(\mu_m)$; this avoids problems caused by zeroes of local Euler factors (cf.~the discussion at the start of \S 9.1 of \cite{rubin00}).

   Choose a system of roots of unity $\zeta_m \in \overline{\QQ}$ for all $m \ge 1$ which satisfy $\zeta_{mn}^n = \zeta_m$ for all integers $m, n$. Let $G_m = \Gal(\QQ(\mu_m) / \QQ)$ and $\Gamma_m = \Gal(\QQ(\mu_{mp^\infty}) / \QQ)$; we identify $\Gamma_m$ with $G_m \times \Gamma$ in the obvious way.

   Let $V$ be an $E$-linear $p$-adic representation of $G_{\QQ}$, where $E / \Qp$ is a finite extension, which is crystalline at $p$ with non-negative Hodge--Tate weights and such that no eigenvalue of Frobenius on $\Dcris(V)$ is a root of unity. Then for all $m \ge 1$ the $p$-adic regulator map
   \[ \mathcal{L}^\Gamma_{\QQ(\mu_m), V} : H^1_{\Iw}(\QQ(\mu_{mp^\infty}), V) \rTo \QQ(\mu_m) \otimes_{\QQ} \mathcal{H}_{E}(\Gamma) \otimes_{E} \Dcris(V)\]
   is well-defined (as the sum of the local regulator maps at the primes of $\QQ(\mu_m)$ above $p$).

   Let $D = \Dcris(V^*) = \Dcris(V)^*$, and let $D_{m p^\infty} = \Lambda_E(\Gamma) \otimes_{E} D \otimes_{\QQ} \QQ(\mu_m)$. We regard $D_{mp^\infty}$ as a $\Gamma_m$-module, via the usual action of $G_m$ on $\QQ(\mu_m)$ and of $\Gamma$ on $\Lambda_E(\Gamma)$. Following \cite{kurihara02} and \cite{otsuki09}, we make the following definition:

   \begin{definition}
    Define a pairing
    \[ t_m: D_{mp^\infty} \times H^1_{\Iw}(\QQ(\mu_{mp^\infty}), V) \rTo \cH_E[\Gamma_m] \]
    by
    \[ t_m(x, z) = \sum_{\sigma \in G_m} [\sigma] \trace_{\QQ(\mu_m) / \QQ} \left\langle \sigma x, \mathcal{L}^\Gamma_{\QQ(\mu_m), V}(z)\right\rangle_{\cris}.\]
   \end{definition}

   Here we extend $\langle, \rangle_{\cris}$ to be $\Gamma$-linear in the second variable and $\Gamma$-antilinear in the first. One checks that
    \[ t_m(\sigma x, \tau z) = [\sigma^{-1} \tau] \cdot t_m(x, z)\]
   for all $\sigma, \tau \in \Gamma_m$ (not just in $\Gamma$).

   Now fix two families $F_\ell, G_\ell$ of polynomials in $E[X]$, indexed by primes $\ell \notin \Sigma$, such that $F_\ell, G_\ell \in 1 + X E[X]$ for all $\ell$.

   Let $A$ be the set of square-free integers prime to $\Sigma$. For each prime $\ell \notin \Sigma$ and each $m \in A$, consider the $\Lambda(\Gamma_m)$-linear endomorphism of $\Lambda_E(\Gamma) \otimes_{\QQ} \QQ(\mu_m)$ given by $\hat\sigma_{\ell}(x \otimes \zeta) = \tau_\ell x \otimes \zeta^\ell$, for all roots of unity $\zeta \in \mu_m$, where $\tau_\ell$ is the arithmetic Frobenius at $\ell$ in $\Gamma$. Thus $\hat\sigma_{\ell}$ is the action of the Frobenius at $\ell$ in $\Gamma_m$ if $\ell \nmid m$, and is a possibly non-invertible endomorphism if $\ell \mid m$; and the $\hat\sigma_\ell$ all commute with each other.

   \begin{proposition}
    The endomorphism $F_\ell(\hat\sigma_\ell)$ is invertible in $\End_{Q} \left(Q \otimes_{\QQ} \QQ(\mu_m)\right)$, where $Q = \operatorname{Frac} \Lambda_E(\Gamma)$.
   \end{proposition}

   \begin{proof}
    Clear, since the roots of the characteristic polynomial of $\hat\sigma_\ell$ on $\QQ(\mu_m)$ are scalars.
   \end{proof}

   For each $m \in A$, let us define an element $x'_m \in Q \otimes_{\QQ} \QQ(\mu_m)$ by
   \[ x'_m = \left(\prod_{\ell \mid m} F_\ell(\hat\sigma_\ell)^{-1} G_\ell(\hat\sigma_\ell) \right)\cdot (1 \otimes \zeta_m).\]

   \begin{proposition}
    If $\ell \nmid m$, then we have
    \[ \tr^{m\ell}_m(x_{m\ell}') = \sigma_\ell^{-1} F_\ell(\sigma_\ell)^{-1}\left((\ell - 1) G_\ell(\sigma_\ell) - \ell F_\ell(\sigma_\ell) \right) x_m'.\]
   \end{proposition}

   \begin{proof}
    This is a straightforward generalization of (one case of) Proposition 2.5 of \cite{otsuki09}.
    We define $H_\ell(X) = \frac{G_\ell(X) - F_\ell(X)}{X}$, so we have
    \[ F_\ell(\hat\sigma_\ell)^{-1} G_\ell(\hat \sigma_\ell) = 1 + H_\ell(\hat\sigma_\ell) F_\ell(\hat\sigma_\ell)^{-1} \hat\sigma_\ell\]
    in $\End_{Q} Q \otimes_{\QQ} \QQ(\mu_{m\ell})$. The operator $\hat\sigma_v$ commutes with $\tr^{m\ell}_m$ whenever $v \ne \ell$. Hence
    \begin{align*}
     \tr^{m\ell}_m (x'_{m\ell}) &= \tr^{m\ell}_m \left( \left(\prod_{v \mid m\ell}F_v(\hat\sigma_v)^{-1} G_v(\sigma_v) \right) \zeta_{m\ell}\right)\\
     &= \left(\prod_{v \mid m}F_v(\hat\sigma_v)^{-1} G_v(\hat\sigma_v)\right) \tr^{m\ell}_m \left(F_\ell(\hat\sigma_\ell)^{-1} G_\ell(\sigma_\ell)\zeta_{m\ell}\right)\\
     &= \left(\prod_{v \mid m}F_v(\hat\sigma_v)^{-1} G_v(\hat\sigma_v)\right)\tr^{m\ell}_m \left(\left(1 + H_\ell(\hat\sigma_\ell) F_\ell(\hat\sigma_\ell)^{-1} \hat\sigma_\ell\right) \zeta_{m\ell}\right)\\
     &= \left(\prod_{v \mid m}F_v(\hat\sigma_v)^{-1}G_v(\hat\sigma_v)\right)\left[\tr^{m\ell}_m(\zeta_{m\ell}) + \tr^{m\ell}_m \left(H_\ell(\hat\sigma_\ell)F_\ell(\hat\sigma_\ell)^{-1} \zeta_{m}\right)\right] \\
     &= \left(\prod_{v \mid m}F_v(\hat\sigma_v)^{-1}G_v(\hat\sigma_v)\right)\left[\left(-\sigma_{\ell}^{-1} \zeta_m\right) + (\ell - 1) \left(H_\ell(\hat\sigma_\ell)F_\ell(\hat\sigma_\ell)^{-1} \zeta_{m}\right)\right]\\
     &= \left(-\sigma_\ell^{-1} + (\ell - 1)H_\ell(\sigma_\ell)F_\ell(\sigma_\ell)^{-1}\right) x_m'
    \end{align*}
    (where we have dropped the hats, since $\hat\sigma_\ell$ acts on $Q \otimes \QQ(\mu_m)$ as the usual Frobenius $\sigma_\ell$). Since $\sigma_\ell H_\ell(\sigma_\ell) = G_\ell(\sigma_\ell) - F_\ell(\sigma_\ell)$, we have
    \begin{align*}
     -\sigma_\ell^{-1} + (\ell - 1)H_\ell(\sigma_\ell)F_\ell(\sigma_\ell)^{-1}
     &= \sigma_\ell^{-1} F_\ell(\sigma_\ell)^{-1} \left(-F_\ell(\sigma_\ell) + (\ell - 1) \sigma_\ell H_\ell(\sigma_\ell)\right)\\
     &= \sigma_\ell^{-1} F_\ell(\sigma_\ell)^{-1} \left(-F_\ell(\sigma_\ell) + (\ell - 1) (G_\ell(\sigma_\ell) - F_\ell(\sigma_\ell))\right)\\
     &= \sigma_\ell^{-1} F_\ell(\sigma_\ell)^{-1} \left((\ell-1)G_\ell(\sigma_\ell) - \ell F_\ell(\sigma_\ell)\right).
    \end{align*}
    which gives the formula stated above.
   \end{proof}

   \begin{corollary}
    If we are given, for each $m \in A$, an element
    \[ z_m \in Q \otimes_{\Lambda(\Gamma)} H^1_{\Iw}(\QQ(\mu_{mp^\infty}), V)\]
    satisfying
    \[ \cores^{m\ell}_m(z_{m\ell}) = F_\ell(\sigma_\ell^{-1}) z_m\]
    for each $m$ and each prime $\ell \nmid m$, $\ell \notin \Sigma$, and we define $x_m = x_m' v \in D_{mp^\infty}$ for some fixed $v \in D$, then we have the relation
    \[ \pr_{m}^{m\ell} t_{m\ell}(x_{m\ell}, z_{m\ell}) = \sigma_\ell\left((\ell-1)G_\ell(\sigma_\ell^{-1}) - \ell F_\ell(\sigma_\ell^{-1})\right) t_m\left( x_{m},   z_{m}\right).\]
   \end{corollary}

   By base extension we may regard $t_m(x_m, -)$ as a map $\bigwedge^2_{\Lambda(\Gamma_m)} M_m \to M_m$, where
   \[  M_m = Q \otimes_{\Lambda(\Gamma)} H^1_{\Iw}(\QQ(\mu_{mp^\infty}), V), \]
   so it makes sense to evaluate $t_m(x_m, -)$ against a rank 2 Euler--Iwasawa system.

   We now specialize to the case where $V = V_{L_\lambda}(f, g)^*$, for some weight 2 eigenforms $(f, g)$ of levels  divisible only by primes in $\Sigma - \{p\}$. We take $G_\ell(X) = 1 - \varepsilon_\ell(f) \varepsilon_\ell(g) X^2$ and $F_\ell(X) = P_\ell(\ell^{-1}X)$ as before. Choose a $p$-stabilization $(\alpha, \gamma)$ of $f$ and $g$, and let $v = v_\alpha \otimes v_\gamma$ be the obvious $\vp$-eigenvector in $\Dcris(V^*)$ of eigenvalue $\alpha \gamma$.

   \begin{proposition}
    Let $(w_m)_{m \ge 1}$ be an Euler--Iwasawa system of rank 2 for $(T, A, \Sigma)$, for some lattice $T$ in $V$, and let
    \[ v_m = t_m(x_m, w_m).\]
    Then we have
    \[ v_m \in \cH_h(\Gamma) \otimes_{\Lambda(\Gamma)} Q \otimes_{\Lambda(\Gamma)} H^1(\QQ(\mu_{mp^\infty}), V),\]
    where $h = v_p(\alpha \gamma)$; and the elements $v_m$ satisfy the compatibility relation
    \[ \cores_m^{m\ell} v_{m\ell} = \sigma_\ell\left((\ell-1)G_\ell(\sigma_\ell^{-1}) - \ell F_\ell(\sigma_\ell^{-1})\right) v_m.\]
   \end{proposition}

   Note that the growth condition $\cH_h(\Gamma)$ is consistent with what we have seen for the elements $\mathfrak{z}_m^{f_\alpha, g_\gamma, Np}$ (cf. Theorem \ref{thm:Eulersystem}) and the compatibility condition between levels $m$ and $m\ell$ is consistent with Theorem \ref{thm:secondnormrelationprime}. This suggests the following conjecture:

   \begin{conjecture}
    Then there exists a rank 2 Euler--Iwasawa system $(w_m)$ for $(T_{\cO_{\lambda}}(f, g)^*, A, \Sigma)$ with the property that for all $m \in A$, and all choices of $p$-stabilizations $(\alpha, \gamma)$ of $(f, g)$, the Iwasawa cohomology class $\mathfrak{z}_m^{f_\alpha, g_\gamma, Np}$ of Theorem \ref{thm:Eulersystem} is given by
    \[ \fz_m^{f_\alpha, g_\gamma, Np} = t_m(x_m, w_m) \]
    in the notation above.
   \end{conjecture}

   This gives a conceptual explanation for the (somewhat surprising) growth and compatibility properties of the generalized Beilinson--Flach elements in the context of Perrin-Riou's theory of higher-rank Euler systems. The authors would like to express their cautious hope that similar rank 1 ``shadows'' of higher rank Euler systems might also exist in other contexts.

   \begin{remark}
    Note that it is implicit in this conjecture that the elements $t_m(x_m, w_m)$ have no poles (except possibly at the trivial character), so the singularities of $t_m$, at the characters where one of the $F_\ell(\hat\sigma_\ell)$ for $\ell \mid m$ fails to be invertible, must be ``cancelled out'' by zeroes of $w_m$.
   \end{remark}

\appendix

 \section{Ancillary results}

  \subsection{Fixed points of double cosets}

  Here we shall prove a result that is used in the proof of Theorem \ref{thm:secondnormrelationprime} above.

  \newcommand{\Pslt}{\operatorname{PSL}_2(\RR)}

  Let $\Gamma$ be a discrete subgroup of $\Pslt$. Recall that a \emph{fundamental domain} for $\Gamma$ is a closed subset $D$ of $\cH$ such that
  \begin{itemize}
   \item $D$ is equal to the closure of its interior $D^\circ$,
   \item $\bigcup_{\gamma \in \Gamma} \gamma D = \cH$,
   \item $\gamma D^\circ \cap D^\circ = \varnothing$ for all non-identity elements $\gamma \in \Gamma$.
  \end{itemize}

  We assume henceforth that $\Gamma$ is a \emph{Fuchsian group of the first kind}, i.e.~that $\Gamma$ admits a fundamental domain $D$ with finite hyperbolic area. We shall say that a fundamental domain $D$ is \emph{polygonal} if $D$ is the region bounded by a finite number of geodesic arcs in $\cH$; it is known that every $\Gamma$ admits a polygonal fundamental domain.

  \begin{lemma}
   \label{lemma:comparable}
   Let $D$ be a Dirichlet domain for $\Gamma$, and let $E = \alpha D$ where $\alpha$ lies in the commensurator $\operatorname{Comm}(\Gamma)$. Then there are only finitely many $\gamma \in \Gamma$ such that $\alpha D \cap \gamma D \ne \varnothing$.
  \end{lemma}

  \begin{proof}
   It is clear that $\alpha D$ is a Dirichlet domain for $\alpha \Gamma \alpha^{-1}$. In particular, it is polygonal. Hence it can be decomposed as the union of a compact set $M$ and a finite number $N_i$ of ``cusp neighbourhoods'', which are subsets bounded by two geodesics intersecting at a vertex at infinity, which is a parabolic point $x_i$ of $\alpha\Gamma\alpha^{-1}$ on the boundary $\mathbb{P}^1(\RR)$, and an arc of a Euclidean circle tangent to the real line at $x_i$.

   Since $M$ is compact, it can intersect only finitely many $\Gamma$-translates of $D$ (cf.~\cite[Theorem 3.5.1]{katok92}). Moreover, since $\alpha \in \operatorname{Comm}(D)$, the sets of parabolic points of $\Gamma$ and $\alpha \Gamma \alpha^{-1}$ are the same; so for each vertex-at-infinity $x$ of $\alpha D$, we may choose some $\gamma \in \Gamma$ which maps a vertex-at-infinity $y_i$ of $D$ to $x_i$, and it is clear that $N_i$ is contained in a finite union of translates of $\gamma \gamma' D$ where $\gamma'$ lies in the stabilizer of $y_i$. Each of these, in turn, intersects finitely many other translates of $D$ (since $D$ has finitely many sides).
  \end{proof}

  \begin{lemma}
   \label{lem:finitefixedpoints}
   Let $\Gamma$ be Fuchsian group of the first kind, and let $X \subset \operatorname{Comm}(\Gamma)$ be a finite union of double cosets $\Gamma \alpha \Gamma$. Then the set
   \[ \operatorname{Fix}(X) = \{ u \in \cH : \text{$\gamma u = u$ for some $\gamma \in X$, $\gamma \ne 1$}\}\]
   is a finite union of $\Gamma$-orbits in $\cH$.
  \end{lemma}

  \begin{proof}
   Since $X - \{\operatorname{id}\}$ is preserved by conjugation by $\Gamma$, the set $\operatorname{Fix}(X)$ is a union of orbits of $\Gamma$. So it suffices to show that $\operatorname{Fix}(X) \cap D$ is finite, where $D$ is a Dirichlet domain for $\Gamma$.

   We claim that there are only finitely many $x \in X$ such that $xD \cap D \ne \varnothing$. From Lemma \ref{lemma:comparable}, we know that for each $\alpha \in G$ there are finitely many $\gamma \in \Gamma$ such that $\gamma D \cap \alpha D \ne \varnothing$, and hence finitely many $x \in \Gamma \alpha$ such that $x D \cap D \ne \varnothing$. Since $X$ is the union of finitely many left cosets $\Gamma \alpha_i$, this implies the claim.

   However, each non-identity element in the finite set $\{ x \in X : xD \cap D \ne \varnothing\}$ can only have finitely many fixed points in $\cH$, and in particular in $D$; so $\operatorname{Fix}(X) \cap D$ is finite, as required.
  \end{proof}

  \begin{lemma}
   Let $\Gamma_1, \Gamma_2$ be commensurable Fuchsian groups of the first kind. Then the set
   \[ \{ u \in \cH : \exists c \in \Gamma_1, d \in \Gamma_2 \text{ such that } cd \ne 1 \text{ and } cd u = u \}\]
   is a finite union of orbits under $\Gamma_1 \cap \Gamma_2$.
  \end{lemma}

  \begin{proof}
   This follows from the previous lemma applied to $\Gamma = \Gamma_1 \cap \Gamma_2$ and $X = \Gamma_1 \Gamma_2$.
  \end{proof}

  (Note that if $\Gamma_1 = \Gamma_2$, or more generally if the group generated by $\Gamma_1$ and $\Gamma_2$ is Fuchsian, this generalizes the well-known result that Fuchsian groups of the first kind have finitely many elliptic points in $\Gamma$.)

  In particular, we have the following:

  \begin{proposition}
   \label{prop:generically injective}
   Let $\Gamma_1, \Gamma_2$ be commensurable Fuchsian groups of the first kind. Then the natural map
   \[ (\Gamma_1 \cap \Gamma_2) \backslash \cH \to (\Gamma_1 \backslash \cH) \times (\Gamma_2 \backslash \cH)\]
   is injective away from a finite subset of its domain.
  \end{proposition}

  \begin{proof}
   Let $z, z'$ be two points of $\cH$ such that $z' \in \Gamma_1 z$ and $z' \in \Gamma_2 z$. Then we may write $z' = \gamma_1 z$ for some $\gamma_1 \in \Gamma_1$ and $z' = \gamma_2 z$ for some $\gamma_2 \in \Gamma_2$.

   Hence $\gamma_1^{-1} \gamma_2 z = z$. So either $z$ lies in the finite subset $\operatorname{Fix}(\Gamma_1 \Gamma_2)$ of $(\Gamma_1 \cap \Gamma_2) \backslash \cH$, or $\gamma_1^{-1} \gamma_2 = 1$, in which case $z$ and $z'$ are clearly in the same orbit under $\Gamma_1 \cap \Gamma_2$.
  \end{proof}


  \subsection{Unbounded Iwasawa cohomology}
   \label{sect:unbounded cohomology}

   In this section, we shall consider inverse systems of cohomology classes in $\Zp^\times$-extensions which are not bounded (as in the usual definition of Iwasawa cohomology) but satisfy a weaker growth condition.

   Let $K$ be a finite extension of either $\QQ$ or $\Qp$. If $K$ is global, suppose that either $p \ne 2$, or $K$ has no real places.

   \begin{notation}
    In order to handle the two cases in a uniform manner, we shall adopt a notation that is slightly abusive: for $T$ a $\Zp$-representation of $\Gal(\bar K / K)$, the notation $H^i(K, T)$ will mean either $H^i(K, T)$ as defined above if $G$ is local, or what we previously called $H^i_S(K, T)$ if $K$ is global, where $S$ is some fixed finite set of places containing all infinite places and all those dividing $p$. In the latter case, we will assume that $S$ contains all primes at which $T$ is ramified.
   \end{notation}

    As before, we let $K_n = K(\mu_{p^n})$ and $K_\infty = \bigcup_n K_n$, and define Iwasawa cohomology groups $H^i_{\Iw}(K_\infty, T)$ as the inverse limit of the $H^i(K_n, T)$ with respect to corestriction, with their natural module structure over $\Lambda = \Lambda_{\Zp}(\Gamma)$.

   \begin{proposition}[Nekovar]
    For any $j \in \{0, 1, 2\}$, we have a short exact sequence
    \begin{equation}
     \label{eq:nekovarSES}
     0 \rTo H^j_{\Iw}(K_\infty, T)_{\Gamma_n} \rTo H^j(K_n, T) \rTo H^{j+1}_{\Iw}(K_\infty, T)^{\Gamma_n} \rTo 0
    \end{equation}
   \end{proposition}

   \begin{proof}
    This is Corollary 8.4.8.2 of \cite{nekovar06}. We briefly recall the proof. There are natural isomorphisms
    \begin{align*}
     H^i_{\Iw}(K_\infty, T) &\cong H^i(K, \Lambda \otimes_{\Zp} T)\\
     H^i(K_n, T) &\cong H^i(K, \Zp[\Gamma/\Gamma_n] \otimes_{\Zp} T)
    \end{align*}
    (for a suitable $\Lambda$-linear action of $G_K$ on the tensor products); see Proposition 8.3.5 of \emph{op.cit.}. Then the result above follows from the long exact cohomology sequence of $K$-cohomology attached to the short exact sequence of $\Lambda[G_K]$-modules
    \[ 0 \rTo \Lambda \otimes_{\Zp} T \rTo^{[\gamma_n] - 1} \Lambda \otimes_{\Zp} T \rTo \Zp[\Gamma/\Gamma_n] \otimes_{\Zp} T \rTo 0.\]
   \end{proof}

   \begin{proposition}[Perrin-Riou]
    \label{prop:torsionpart}
    There is an exact sequence
    \[ 0 \rTo T^{G_{K_\infty}} \rTo H^1_{\Iw}(K_\infty, T) \rTo \Hom_{\Lambda}(H^1(K_\infty, V/T)^\vee, \Lambda)^\iota \rTo \text{(finite)} \rTo 0,\]
    where $\iota$ signifies that the $\Lambda$-module structure is composed with the automorphism $\gamma \mapsto \gamma^{-1}$. In both cases, this exact sequence identifies $T^{G_{K_\infty}}$ with the $\Lambda$-torsion submodule of $H^1_{\Iw}(K_\infty, T)$.
   \end{proposition}

   \begin{proof}
    The local case is \cite[Proposition 2.1.6]{perrinriou92}; note that in the local situation Tate duality furnishes an isomorphism $H^1(K_\infty, V/T)^\vee \cong H^1_{\Iw}(K_\infty, T^*(1))$ and the middle map can be interpreted as Perrin-Riou's pairing $H^1_{\Iw}(K_\infty, T) \times H^1_{\Iw}(K_\infty, T^*(1)) \to \Lambda$. The global case is \cite[Lemma 1.3.3]{perrinriou95}.
   \end{proof}

   The main object of study in this section is the following module. Let $V = \Qp \otimes_{\Zp} T$.

   \begin{definition}
    For $K, T, V$ as above, and $0 \le r < 1$, let $Y_r(K_\infty, V)$ be the space of sequences $(c_n)_{n \ge 0} \in \varprojlim_n H^1(K_n, V)$ such that there exists $\delta < \infty$ independent of $n$ for which $p^{\lfloor r n\rfloor + \delta} c_n$ is in the image of $H^1(K_n, T)$ in $H^1(K_n,V)$.
   \end{definition}

   \begin{proposition}
    For all $0\leq r<1$, the natural map $\lambda_r : \cH_r(\Gamma) \otimes_{\Lambda_{\Qp}(\Gamma)} H^1_{\Iw}(K_\infty, V) \to \varprojlim_n H^1(K_n, V)$ has image contained in $Y_r(K_\infty, V)$.
   \end{proposition}

   \begin{proof}
    This is clear from the definition of $\cH_r(\Gamma)$.
   \end{proof}

   \begin{remark}
    The map $\lambda_r$ is not necessarily injective, even for $r = 0$ (where $\cH_r(\Gamma)$ is just $\Lambda \otimes \Qp$). A counterexample is provided by the representation $T = \Zp(1)$. Then the cocycle $c_n$ given by $\sigma \mapsto \frac{\chi(\sigma) - 1}{p^n}$ is well-defined as an element of $H^1(K_n, T)$ (for either local or global $K$). The sequence $(c_n)$ defines an element of $H^1_{\Iw}(K_\infty, T)$ which is not $p$-torsion, and thus is non-zero as an element of $H^1_{\Iw}(K_\infty, V)$. But $p^n c_n$ is a coboundary for all $n$, so the image of $c_n$ in $H^1(K_n, V)$ is zero for all $n$. Thus $(c_n)$ lies in the kernel of the above map.
   \end{remark}

   \begin{proposition}
    The kernel of $\lambda_r$ is contained in $H^1_{\Iw}(K_\infty, V)_{\mathrm{tors}} \cong V^{G_{K_\infty}}$.
   \end{proposition}

   \begin{remark}
    We note first that this statement does make sense, since for any $\Lambda_{\Qp}(\Gamma)$-torsion module $M$, tensoring with $1 \in \cH_r(\Gamma)$ gives an isomorphism $\cH_r(\Gamma) \otimes_{\Lambda_{\Qp}(\Gamma)} M \cong M$.
   \end{remark}

   \begin{proof}
    Tensoring \eqref{eq:nekovarSES} with $\Qp$, we find that the map
    \[ H^1_{\Iw}(K_\infty, V)_{\Gamma_n} \to H^1(K_n, V)\]
    is injective. Thus the kernel of $\lambda_r$ consists of those elements lying in
    \[ \bigcap_{n \ge 0} (\gamma_n - 1) \left(\cH_r(\Gamma) \otimes_{\Lambda_{\Qp}(\Gamma)} H^1_{\Iw}(K_\infty, V)\right).\]

    Since $H^1_{\Iw}(K_\infty, V)$ is a finitely-generated module over the subring $\Lambda_{\Qp}(\Gamma_1) \subset \Lambda_{\Qp}(\Gamma)$, which is a PID, we may write it as the direct sum of its torsion submodule and a complementary free submodule. Since $r < 1$, we find that
    \[ \bigcap_{n \ge 1} (\gamma_n - 1) \cH_r(\Gamma) = 0,\]
    and hence the kernel of $\lambda_r$ is contained in the torsion part of $H^1_{\Iw}(K_\infty,V)$, which is equal to $V^{G_{K_\infty}}$ by Proposition \ref{prop:torsionpart}.
   \end{proof}

   \begin{remark}
    Although we shall not need this, it clearly follows that the kernel of $\lambda_r$ is equal to $\bigcup_{n \ge 0} (\gamma_n - 1) V^{G_{K_\infty}}$, which is the unique $\Gamma$-invariant complement of $\bigcup_n H^0(K_n, V)$ in $H^0(K_\infty, V)$.
   \end{remark}

   \begin{proposition}
    \label{prop:local unbounded coho}
    Let $K$ be a $p$-adic field and suppose that $V^{G_{K_\infty}} = 0$. Then the map
    \[ \cH_r(\Gamma) \otimes_{\Lambda_{\Qp}(\Gamma)} H^1_{\Iw}(K_\infty, V) \to Y_r(K_\infty, V)\]
    is an isomorphism.
   \end{proposition}

   \begin{proof}
    By Proposition \ref{prop:torsionpart}, our hypotheses imply that $H^1_{\Iw}(K_\infty, V)$ is a torsion-free $\Lambda_{\Qp}(\Gamma)$-module; hence it is free, since $\Lambda_{\Qp}(\Gamma)$ is a finite product of principal ideal domains (and the ranks of the $\Gamma_{\mathrm{tors}}$-isotypical direct summands of $H^1_{\Iw}(K_\infty, V)$ are all equal). Thus there exists a free basis $x_1, \dots, x_d$ of $H^1_{\Iw}(K_\infty, V)$, where $d = \dim_{\Qp}(V)$.

    For each $n$, the cokernel of the projection map
    \[H^1_{\Iw}(K, T) \to H^1(K_n, T)\]
    is finite, and its order is bounded independently of $n$. (In fact, the cokernel of this map is isomorphic to the $\Gamma_n$-invariants of
    \[H^2_{\Iw}(K_\infty, T) \cong H^0(K_\infty, (V/T)^*(1))^\vee,\]
    and $H^0(K_\infty, (V/T)^*(1))$ is finite, since $H^0(K_\infty, V^*(1)) = H^0(K_\infty, V)^*(1) = 0$.) Thus the map
    \[H^1_{\Iw}(K_\infty, V) \to H^1(K_n, V)\]
    is surjective for all $n$, and there is $\nu < \infty$ independent of $n$ such that the $\Zp$-submodule spanned by the images of $x_1, \dots, x_d$ in $H^1(K_n,V)$ contains $p^\nu \cdot \frac{H^1(K_n, T)}{\mathrm{torsion}}$.

    Consequently, given any sequence $(c_n)_{n \ge 0} \in Y_r(K_\infty, V)$, we have for each $n$ uniquely determined elements $b_1^{(n)}, \dots, b_{d}^{(n)} \in \Qp[\Gamma / \Gamma_n]$ such that $\sum_{i = 1}^d b_i^{(n)} x_i^{(n)} = c_n$, where $x_i^{(n)}$ is the image of $x_i$ in $H^1(K_n, V)$. Also, for each $i$ the sequence $(b_i^{(n)})_{n \ge 0}$ is compatible under projection (by uniqueness), and its valuation is bounded below by $-\lfloor rn \rfloor - \delta - \nu$; so (as $r < 1$) there is a unique element $b_i \in \cH_r(\Gamma)$ whose image at level $n$ is $b_i^{(n)}$ for all $n$. Then it is clear that $c = \sum_i b_i \otimes x_i \in \cH_r(\Gamma) \otimes_{\Lambda_{\Qp}(\Gamma)} H^1_{\Iw}(K_\infty, V)$ is a preimage of $(c_n)_{n \ge 0}$; by the previous proposition, it is unique.
   \end{proof}

   In the global case we cannot prove quite such a strong result, as we do not have such good control over $H^2_{\Iw}(K_\infty, T)$; the following rather more specific result (which applies to both local and global cases) will suffice for our purposes:

   \begin{proposition}\label{prop:lifttoIwasawacohomology}
    Let $r < 1$, and suppose $T$ has the structure of a module over $\cO_E$, for some finite extension $E/\Qp$, and that $T^{G_{K_\infty}} = 0$. Let $\alpha \in \cO_E$ such that $v_p(\alpha) \le r$, and suppose we are given elements $x_n \in H^1(K_n, T)$ for $n\geq 0$ satisfying
    \[\cores_{n}^{n+1}(x_{n+1}) = \alpha x_n.\]
    Then there is a unique element $x \in \cH_r(\Gamma) \otimes_{\Lambda} H^1_{\Iw}(K_\infty, T)$ whose image in $H^1(K_n, V)$ is equal to $\alpha^{-n} c_n$ for all $n$.
   \end{proposition}

   \begin{proof}
    We claim that the hypotheses of the theorem force each $c_n$ to be the image of an element of $H^1_{\Iw}(K_\infty, V)_{\Gamma_n}$. To prove this, we shall argue much as in the proof of Proposition \ref{prop:unram outside p}. We begin by noting that $H^2_{\Iw}(K_\infty, T)$ is a finitely-generated $\Lambda$-module, and the subgroups $M_n \coloneqq H^2_{\Iw}(K_\infty, T)^{\Gamma_n}$ are $\Lambda$-submodules (since $\Gamma$ is abelian). As $\Lambda$ is a Noetherian ring, the ascending chain of submodules $(M_n)_{n \ge 0}$ must eventually stabilize; that is, there is an $n_0$ such that $M_n = M_{n_0}$ for all $n \ge n_0$.

    The corestriction map $H^1(K_{n+1}, T) \to H^1(K_n, T)$ corresponds to the trace map $M_{n+1} \to M_n$; when $n \ge n_0$ this is simply multiplication by $p$ on $M_{n_0}$. Since $v_p(\alpha) < 1$, we deduce as in Proposition \ref{prop:unram outside p} that for all $n \ge 0$, the image of $x_n$ is contained in the torsion submodule of $M_n$. Inverting $p$, the torsion is killed, and the image of $x_n$ in $H^2_{\Iw}(K_\infty, V)^{\Gamma_n}$ is 0; so $x_n$ lies in the submodule $H^1_{\Iw}(K_\infty, V)_{\Gamma_n} \subseteq H^1(K_n, V)$.

    Now let us choose a basis of the free $\Lambda_{\Qp}(\Gamma)$-module $H^1_{\Iw}(K_\infty, V)$. In order to apply the argument of the previous proposition, we need only check that the order of the torsion subgroup of $M_n$ is bounded independently of $n$; but this is immediate from the fact that the $M_n$ stabilize for large $n$ (and are all finitely-generated as $\Zp$-modules). This shows that there exists a $\mu$ such that $\alpha^{-n} c_n$ lies in $p^{-\lfloor rn \rfloor - \mu} H^1_{\Iw}(K_\infty, T)$, and the argument proceeds as in Proposition \ref{prop:local unbounded coho}.
   \end{proof}

\renewcommand{\MR}[1]{MR \href{http://www.ams.org/mathscinet-getitem?mr=#1}{#1}.}
\providecommand{\bysame}{\leavevmode\hbox to3em{\hrulefill}\thinspace}

\end{document}